\documentclass[a4paper,12pt,openany]{book}

\usepackage{tikz}
\usepackage{amsthm}
\usepackage{amsmath}
\usepackage{amssymb}
\usepackage{mathrsfs}
\usepackage{graphicx}
\usepackage{amsfonts}
\usepackage{hyperref}
\usepackage{geometry}
\usepackage{fancyhdr} 
\usepackage{enumerate}
\usepackage{subfigure}

\allowdisplaybreaks

\newcommand{\R}{\mathbb{R}}

\newcommand{\myindi}{\mathbf{1}}

\newcommand{\loc}{{\mathrm{loc}}}
\newcommand{\rref}{\mathrm{ref}}
\newcommand{\indi}{\mathbf{1}}
\newcommand{\vphi}{\varphi}
\newcommand{\veps}{\varepsilon}

\newcommand{\bbE}{\mathbb{E}}

\newcommand{\bbP}{\mathbb{P}}

\newcommand{\calD}{\mathcal{D}}
\newcommand{\calE}{\mathcal{E}}
\newcommand{\calF}{\mathcal{F}}

\newcommand{\calU}{\mathcal{U}}
\newcommand{\calV}{\mathcal{V}}
\newcommand{\calW}{\mathcal{W}}

\newcommand{\scrE}{\mathscr{E}}

\newcommand{\frakD}{\mathfrak{D}}
\newcommand{\frakE}{\mathfrak{E}}
\newcommand{\frakF}{\mathfrak{F}}

\newcommand{\frakR}{\mathfrak{R}}

\newcommand{\md}{\mathrm{d}}
\newcommand{\dd}{\partial}

\newcommand{\mybar}[1]{\overline{#1}}
\newcommand{\myset}[1]{\left\{#1\right\}}

\newcommand{\rmnum}[1]{\romannumeral#1}
\newcommand{\Rmnum}[1]{\uppercase\expandafter{\romannumeral#1}}

\newtheorem{mythm}{Theorem}[chapter]
\newtheorem{myprop}[mythm]{Proposition}
\newtheorem{mylem}[mythm]{Lemma}
\newtheorem{mycor}[mythm]{Corollary}
\newtheorem{myrmk}[mythm]{Remark}

\newtheorem{mydef}[mythm]{Definition}

\newcommand{\DATUM}{December 19, 2018}

\begin{document}

\title{Local and Non-Local Dirichlet Forms on the Sierpi\'nski Gasket and the Sierpi\'nski Carpet}
\author{Meng Yang}
\date{\DATUM}

\pagenumbering{Alph}
\setcounter{page}{1}

\maketitle



\thispagestyle{empty}
\vfill
\begin{center}
\vskip 1cm
{\Large\textbf{Local and Non-Local Dirichlet Forms on the Sierpi\'nski Gasket and the Sierpi\'nski Carpet}}
\vskip 2cm
{\Large Meng Yang}
\vskip 9cm
{\large A Dissertation Submitted for the Degree of Doctor\\
\emph{at}\\
the Department of Mathematics Bielefeld University}\\[1cm]

\date{\DATUM}
\end{center}
\newpage
\cleardoublepage

\thispagestyle{empty}
\begin{center}
{\Large\textbf{Local and Non-Local Dirichlet Forms on the Sierpi\'nski Gasket and the Sierpi\'nski Carpet}}
\vskip 5cm
\bigskip
{\large
Dissertation zur Erlangung des Doktorgrades\\
der Fakult\"at f\"ur Mathematik\\
der Universit\"at Bielefeld\\[3.5cm]
\bigskip
vorgelegt von\\
Meng Yang\\[0.6em]
\vskip 1cm
\date{\DATUM}
}
\end{center}

\newpage
\thispagestyle{empty}
\begin{center}
{\ }\\
\vspace{10cm}
\vfill
Gedruckt auf alterungsbest\"andigem Papier nach DIN--ISO 9706
\end{center}

\newpage

\pagenumbering{roman}

\chapter*{Preface}
\addcontentsline{toc}{chapter}{Preface}
This thesis is about local and non-local Dirichlet forms on the Sierpi\'nski gasket and the Sierpi\'nski carpet. We are concerned with the following three problems in analysis on the Sierpi\'nski gasket and the Sierpi\'nski carpet.

\begin{enumerate}[(1)]
\item A unified purely \emph{analytic} construction of local regular Dirichlet forms on the Sierpi\'n-ski gasket and the Sierpi\'nski carpet. We give a purely analytic construction of a self-similar local regular Dirichlet form on the Sierpi\'nski carpet using $\Gamma$-convergence of stable-like non-local closed forms which gives an answer to an open problem in analysis on fractals. We also apply this construction on the Sierpi\'nski gasket.
\item Determination of walk dimension \emph{without} using diffusion. Although the walk dimension is a parameter that determines the behaviour of diffusion, we give two approaches to the determination of the walk dimension \emph{prior} to the construction of diffusion.
\begin{itemize}
\item We construct non-local regular Dirichlet forms on the Sierpi\'nski gasket from regular Dirichlet forms on certain augmented rooted tree whose certain boundary at infinity is the Sierpi\'nski gasket. Then the walk dimension is determined by a critical value of a certain parameter of the random walk on the augmented rooted tree.
\item We determine a critical value of the index of a non-local quadratic form by finding a more convenient equivalent semi-norm.
\end{itemize}
\item Approximation of local Dirichlet forms by non-local Dirichlet forms. We prove that non-local Dirichlet forms can approximate local Dirichlet forms as direct consequences of our construction of local Dirichlet forms. We also prove that on the Sierpi\'nski gasket the local Dirichlet form can be obtained as a Mosco limit of non-local Dirichlet forms. Let us emphasize that we do \emph{not} need subordination technique based on heat kernel estimates.
\end{enumerate}
\medskip

{\textbf{Acknowledgement}}
Firstly, I would like to express my sincere gratitude to my supervisor Prof. Alexander Grigor'yan. During the three years of my PhD program, he gave me a lot of valuable advice and continuous encouragement on my research. The most important thing I learnt from him is to solve difficult problems using simple ideas.

Secondly, I would like to thank Prof. Jiaxin Hu from Tsinghua University. Without his recommendation, I could not have the opportunity to do research in Bielefeld with many excellent colleagues.

Thirdly, I would like to thank Prof. Ka-Sing Lau from the Chinese University of Hong Kong. Due to his invitation, I was able to discuss fractals with many talented researchers in Hong Kong.

Fourthly, I would like to thank Prof. Martin T. Barlow, Prof. Zhen-Qing Chen, Prof. Wolfhard Hansen, Dr. Michael Hinz, Prof. Moritz Ka{\ss}mann, Prof. Jun Kigami, Prof. Takashi Kumagai and Prof. Alexander Teplyaev for their valuable suggestions and helpful discussions.

Fifthly, I would like to thank my friends, Eryan Hu, Yuhua Sun, Qingsong Gu, Shilei Kong and Jun Cao. They gave me a lot of help in my research and life. In particular, I would like to thank Eryan Hu for the discussions of basic theory of Dirichlet forms and Qingsong Gu for the discussions of two classical papers about the Sierpi\'nski carpet.

Finally, I would like to thank my parents, Gaoyong Yang and Yanfang Zhang for their consistent support.
\begin{flushright}
Bielefeld, \date{\DATUM},
\hfill Meng Yang
\end{flushright}

\tableofcontents
\newpage

\pagenumbering{arabic}
\setcounter{page}{1}

\chapter{Introduction and Main Results}\label{ch_intro}

\section{Motivation and History}

This thesis is about local and non-local Dirichlet forms on the Sierpi\'nski gasket and the Sierpi\'nski carpet. Both the Sierpi\'nski gasket and the Sierpi\'nski carpet can be regarded as two-dimensional generalizations of the Cantor set.

The Sierpi\'nski gasket (SG) is a typical example of p.c.f. (post-critically finite) self-similar sets. The SG is the simplest self-similar set in some sense, see Figure \ref{fig_SG}.

\begin{figure}[ht]
\centering
\includegraphics[width=0.5\textwidth]{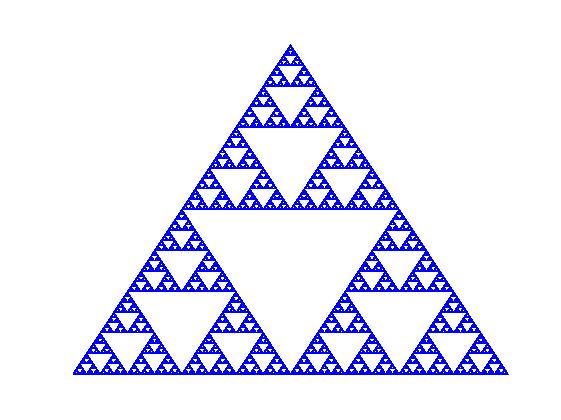}
\caption{The Sierpi\'nski Gasket}\label{fig_SG}
\end{figure}

The SG can be obtained as follows. Given an equilateral triangle with sides of length 1, divide the triangle into four congruent triangles, each with sides of length $1/2$, and remove the central one. Then divide each of the three remaining triangles into four congruent triangles, each with sides of length $1/4$, and remove the central ones, see Figure \ref{fig_SG_construction}. The SG is the compact connected set that remains after repeating the above procedure infinitely many times.

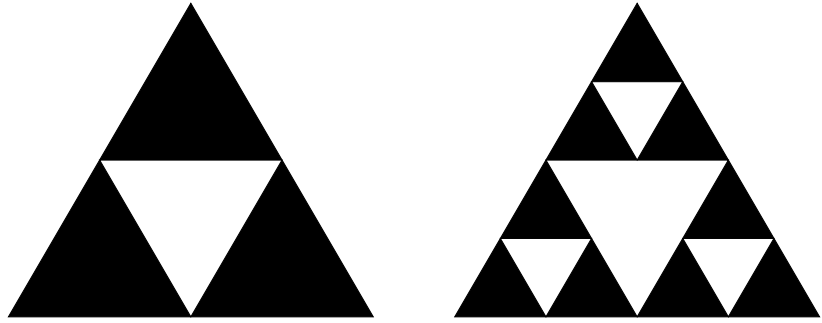
\begin{figure}[ht]
\centering
\subfigure{
\begin{tikzpicture}[scale=1.2]
\draw[fill=black] (0,0)--(4,0)--(2,2*1.7320508076)--cycle;
\draw[fill=white] (2,0)--(1,1*1.7320508076)--(3,1*1.7320508076)--cycle;

\end{tikzpicture}
}
\hspace{0.2in}
\subfigure{
\begin{tikzpicture}[scale=1.2]
\draw[fill=black] (0,0)--(4,0)--(2,2*1.7320508076)--cycle;
\draw[fill=white] (2,0)--(1,1*1.7320508076)--(3,1*1.7320508076)--cycle;
\draw[fill=white] (1,0)--(0.5,0.5*1.7320508076)--(1.5,0.5*1.7320508076)--cycle;
\draw[fill=white] (3,0)--(2.5,0.5*1.7320508076)--(3.5,0.5*1.7320508076)--cycle;
\draw[fill=white] (2,1.7320508076)--(1.5,1.5*1.7320508076)--(2.5,1.5*1.7320508076)--cycle;
\end{tikzpicture}
}
\caption{The Construction of the Sierpi\'nski Gasket}\label{fig_SG_construction}
\end{figure}

The Sierpi\'nski carpet (SC) is a typical example of \emph{non-p.c.f.} self-similar sets. It was first introduced by Wac\l aw Sierpi\'nski in 1916, see Figure \ref{fig_SC}.

\begin{figure}[ht]
\centering
\includegraphics[width=0.5\textwidth]{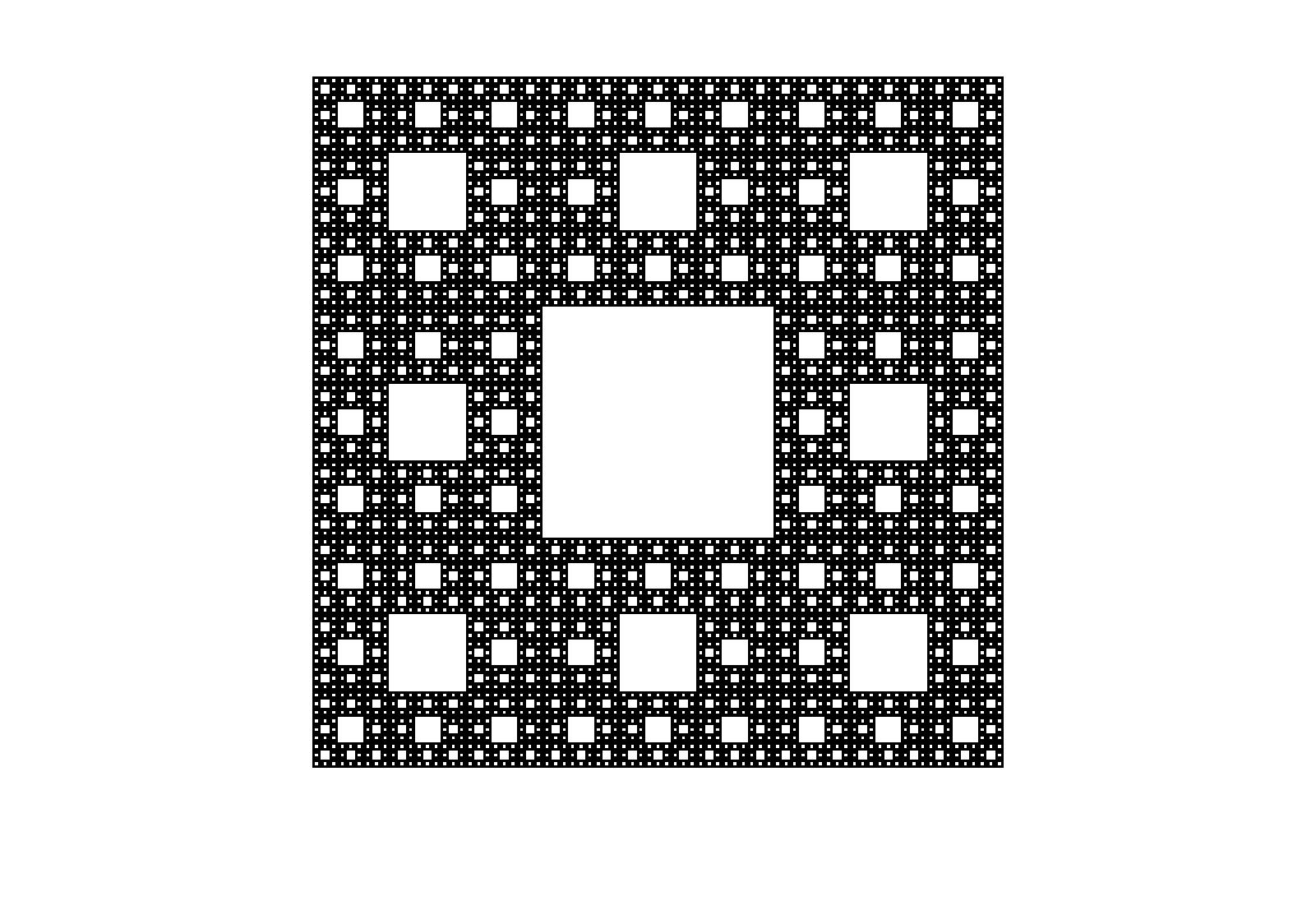}
\caption{The Sierpi\'nski Carpet}\label{fig_SC}
\end{figure}

The SC can be obtained as follows. Divide the unit square into nine congruent squares, each with sides of length $1/3$, and remove the central one. Then divide each of the eight remaining squares into nine congruent squares, each with sides of length $1/9$, and remove the central ones, see Figure \ref{fig_SC_construction}. Repeating the above procedure infinitely many times, we obtain the SC.

\begin{figure}[ht]
\centering
\begin{tikzpicture}[scale=0.4]
\draw[fill=black] (0,0)--(9,0)--(9,9)--(0,9)--cycle;
\draw[fill=white] (3,3)--(6,3)--(6,6)--(3,6)--cycle;

\draw[fill=black] (11,0)--(20,0)--(20,9)--(11,9)--cycle;
\draw[fill=white] (14,3)--(17,3)--(17,6)--(14,6)--cycle;
\draw[fill=white] (12,1)--(13,1)--(13,2)--(12,2)--cycle;
\draw[fill=white] (15,1)--(16,1)--(16,2)--(15,2)--cycle;
\draw[fill=white] (18,1)--(19,1)--(19,2)--(18,2)--cycle;
\draw[fill=white] (12,4)--(13,4)--(13,5)--(12,5)--cycle;
\draw[fill=white] (18,4)--(19,4)--(19,5)--(18,5)--cycle;
\draw[fill=white] (12,7)--(13,7)--(13,8)--(12,8)--cycle;
\draw[fill=white] (15,7)--(16,7)--(16,8)--(15,8)--cycle;
\draw[fill=white] (18,7)--(19,7)--(19,8)--(18,8)--cycle;

\end{tikzpicture}
\caption{The Construction of the Sierpi\'nski Carpet}\label{fig_SC_construction}
\end{figure}
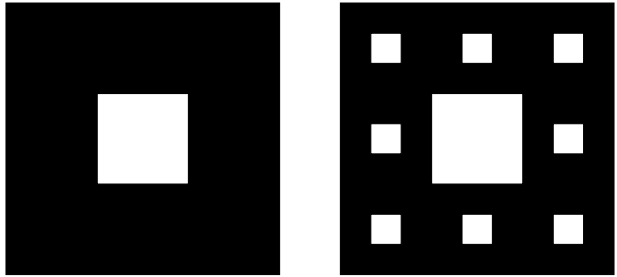

In recent decades, self-similar sets have been regarded as underlying spaces for analysis and probability. Apart from classical Hausdorff measures, this approach requires the introduction of Dirichlet forms (DFs) including local ones and non-local ones.

\subsection*{Construction of Local Regular Dirichlet Forms}

Local regular Dirichlet forms and associated diffusions (also called Brownian motion (BM)) have been constructed in many fractals, see \cite{BP88,BB89,Lin90,KZ92,Kig93,Bar98,Kig01}. In p.c.f. self-similar sets including the SG, the construction is relatively transparent, while similar construction on the SC is much more involved.

The construction of BM on the SG was given by Barlow and Perkins \cite{BP88}. The construction of a local regular Dirichlet form on the SG was given by Kigami \cite{Kig89} using difference quotients method which was generalized to p.c.f. self-similar sets in \cite{Kig93,Kig01}. Subsequently, Strichartz \cite{Str01} gave the characterization of the Dirichlet form and the Laplacian using the averaging method.

For the first time, BM on the SC was constructed by Barlow and Bass \cite{BB89} using \emph{extrinsic} approximation domains in $\R^2$ (see black domains in Figure \ref{fig_SC_construction}) and time-changed reflected BMs in those domains. Technically, \cite{BB89} is based on the following two ingredients in approximation domains:
\begin{enumerate}[(a)]
\item\label{SC_enum_a} Certain resistance estimates.
\item\label{SC_enum_b} Uniform Harnack inequality for harmonic functions with Neumann boundary condition.
\end{enumerate}
For the proof of the uniform Harnack inequality, Barlow and Bass used certain probabilistic techniques based on Knight move argument (this argument was generalized later in \cite{BB99a} to deal also with similar problems in higher dimensions).

Subsequently, Kusuoka and Zhou \cite{KZ92} gave an alternative construction of BM on the SC using \emph{intrinsic} approximation graphs and Markov chains in those graphs. However, in order to prove the convergence of Markov chains to a diffusion, they used the two aforementioned ingredients of \cite{BB89}, reformulated in terms of approximation graphs.

An important fact about the local regular Dirichlet forms on the SG and the SC is that these Dirichlet forms are resistance forms in the sense of Kigami whose existence gives many important corollaries, see \cite{Kig01,Kig03,Kig12}.

\subsection*{Heat Kernel Estimates and Walk Dimension}

Let $K$ be the SG or the SC and $\calE_\loc$ the self-similar local regular Dirichlet form on $K$. The heat semigroup associated with $\calE_\loc$ has a heat kernel $p_t(x,y)$ satisfying the following estimates: for all $x,y\in K$, $t\in(0,1)$
\begin{equation}\label{eqn_hk}
p_t(x,y)\asymp\frac{C}{t^{\alpha/\beta^*}}\exp\left(-c\left(\frac{|x-y|}{t^{1/\beta^*}}\right)^{\frac{\beta^*}{\beta^*-1}}\right),
\end{equation}
where $\alpha$ is the Hausdorff dimension of $K$ and $\beta^*$ is a new parameter called the \emph{walk dimension of the BM}. It is frequently denoted also by $d_w$. The estimates (\ref{eqn_hk}) on the SG were obtained by Barlow and Perkins \cite{BP88}. The estimates (\ref{eqn_hk}) on the SC were obtained by Barlow and Bass \cite{BB92,BB99a} and by Hambly, Kumagai, Kusuoka and Zhou \cite{HKKZ00}. Equivalent conditions of sub-Gaussian heat kernel estimates for local regular Dirichlet forms on metric measure spaces were explored by many authors, see Andres and Barlow \cite{AB15}, Grigor'yan and Hu \cite{GH14a,GH14b}, Grigor'yan, Hu and Lau \cite{GHL10,GHL15}, Grigor'yan and Telcs \cite{GT12}.

It is known that
\begin{equation}\label{eqn_alpha}
\alpha=
\begin{cases}
\log3/\log2,&\text{for the SG},\\
\log8/\log3,&\text{for the SC},
\end{cases}
\end{equation}
and
\begin{equation}\label{eqn_beta_up}
\beta^*=
\begin{cases}
\log5/\log2,&\text{for the SG},\\
\log(8\rho)/\log3,&\text{for the SC},
\end{cases}
\end{equation}
where $\rho>1$ is a parameter from the aforementioned resistance estimates, whose exact value remains still unknown. Barlow, Bass and Sherwood \cite{BB90,BBS90} gave two bounds as follows:
\begin{itemize}
\item $\rho\in[7/6,3/2]$ based on shorting and cutting technique.
\item $\rho\in[1.25147,1.25149]$ based on numerical calculation.
\end{itemize}
McGillivray \cite{McG02} generalized the above estimates to higher dimensions.

Although the walk dimension $\beta^*$ of the BM appears as a parameter in the heat kernel estimates (\ref{eqn_hk}), it was proved in \cite{GHL03} by Grigor'yan, Hu and Lau that $\beta^*$ is in fact an invariant of the underlying metric measure space.

\subsection*{Approximation of Local DFs by Non-Local DFs}

Consider the following stable-like non-local quadratic form
\begin{equation}\label{eqn_nonlocal}
\begin{aligned}
&\calE_\beta(u,u)=\int_K\int_K\frac{(u(x)-u(y))^2}{|x-y|^{\alpha+\beta}}\nu(\md x)\nu(\md y),\\
&\calF_\beta=\myset{u\in L^2(K;\nu):\calE_\beta(u,u)<+\infty},
\end{aligned}
\end{equation}
where $\alpha=\mathrm{dim}_{\mathcal{H}}K$ as above, $\nu$ is the normalized Hausdorff measure on $K$ of dimension $\alpha$, and $\beta>0$ is so far arbitrary.

Using the heat kernel estimates (\ref{eqn_hk}) and subordination technique, it was proved in \cite{Pie08} that
\begin{equation}\label{eqn_approximation}
\varliminf_{\beta\uparrow\beta^*}(\beta^*-\beta)\calE_\beta(u,u)\asymp\calE_\loc(u,u)\asymp\varlimsup_{\beta\uparrow\beta^*}(\beta^*-\beta)\calE_\beta(u,u)
\end{equation}
for all $u\in\calF_\loc$.
This is similar to the following classical result
\begin{equation}\label{eqn_classical}
\lim_{\beta\uparrow2}(2-\beta)\int_{\R^n}\int_{\R^n}\frac{(u(x)-u(y))^2}{|x-y|^{n+\beta}}\md x\md y=C(n)\int_{\R^n}|\nabla u(x)|^2\md x,
\end{equation}
for all $u\in W^{1,2}(\R^n)$, where $C(n)$ is some positive constant (see \cite[Example 1.4.1]{FOT11}).

\section{Goals of the Thesis}

In this thesis, we are concerned with the following three problems in analysis on the SG and the SC.

\begin{enumerate}[(1)]
\item A unified purely \emph{analytic} construction of local regular DFs on the SG and the SC.
\item Determination of walk dimension \emph{without} using diffusion.
\item Approximation of local DFs by non-local DFs.
\end{enumerate}

\subsection*{Analytic Construction of Local Regular Dirichlet Forms}

The problem of a purely analytic construction of a local regular Dirichlet form on the SC (similar to that on p.c.f. self-similar sets) has been open until now and was explicitly raised by Hu \cite{Hu13}.

We give a direct purely \emph{analytic} construction of local regular Dirichlet forms which works on the SC and the SG. Note that Kigami's construction can not be applied on the SC because it relies on certain monotonicity result and harmonic extension result which originate from certain compatible condition.

The most essential ingredient of our construction on the SC is a certain resistance estimate in approximation graphs which is similar to the ingredient (\ref{SC_enum_a}). We obtain the second ingredient---the uniform Harnack inequality on approximation graphs as a consequence of (\ref{SC_enum_a}). A possibility of such an approach was mentioned in \cite{BCK05}. In fact, in order to prove a uniform Harnack inequality on approximation graphs, we extend resistance estimates from finite graphs to the infinite graphical Sierpi\'nski carpet (see Figure \ref{fig_graphSC}) and then deduce from them a uniform Harnack inequality---first on the infinite graph and then also on finite graphs. By this argument, we avoid the most difficult part of the proof in \cite{BB89}.

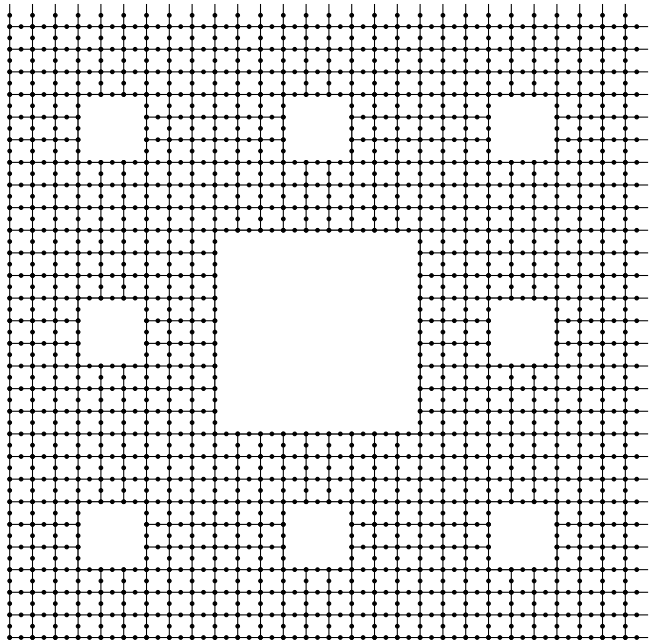
\begin{figure}[ht]
\centering
\begin{tikzpicture}[scale=0.3]

\foreach \x in {0,1,...,27}
\draw (\x,0)--(\x,28);

\foreach \y in {0,1,...,27}
\draw (0,\y)--(28,\y);

\foreach \x in {0,1,2}
\foreach \y in {0,1,2}
\draw[fill=white] (9*\x+3,9*\y+3)--(9*\x+6,9*\y+3)--(9*\x+6,9*\y+6)--(9*\x+3,9*\y+6)--cycle;

\draw[fill=white] (9,9)--(18,9)--(18,18)--(9,18)--cycle;

\foreach \x in {0,1,...,27}
\foreach \y in {0,0.5,1,...,27.5}
\draw[fill=black] (\x,\y) circle (0.08);

\foreach \y in {0,1,...,27}
\foreach \x in {0,0.5,1,...,27.5}
\draw[fill=black] (\x,\y) circle (0.08);

\draw[fill=white,draw=white] (9.25,9.25)--(17.75,9.25)--(17.75,17.75)--(9.25,17.75)--cycle;

\foreach \x in {0,1,2}
\foreach \y in {0,1,2}
\draw[fill=white,draw=white] (9*\x+3.25,9*\y+3.25)--(9*\x+5.75,9*\y+3.25)--(9*\x+5.75,9*\y+5.75)--(9*\x+3.25,9*\y+5.75)--cycle;

\end{tikzpicture}
\caption{The Infinite Graphical Sierpi\'nski Carpet}\label{fig_graphSC}
\end{figure}

\subsection*{Determination of Walk Dimension Without Using Diffusion}

We develop the techniques of determination of walk dimensions of fractal spaces without using diffusions. Using quadratic form $(\calE_\beta,\calF_\beta)$ defined in (\ref{eqn_nonlocal}), we define the walk dimension of the fractal space $K$ by
\begin{equation}\label{eqn_beta_low}
\beta_*:=\sup\myset{\beta>0:(\calE_\beta,\calF_\beta)\text{ is a regular Dirichlet form on }L^2(K;\nu)}.
\end{equation}
In fact, this definition applies to any metric measure space and does not require a priori construction of any Dirichlet form.

However, if the local regular Dirichlet form with heat kernel estimates (\ref{eqn_hk}) is available, then by means of subordination technique, it was proved in \cite{Pie00,GHL03} that $(\calE_\beta,\calF_\beta)$ is a regular Dirichlet form on $L^2(K;\nu)$ if $\beta\in(0,\beta^*)$ and that $\calF_\beta$ consists only of constant functions if $\beta\in(\beta^*,+\infty)$, which implies the identity
$$\beta_*=\beta^*.$$
We provide in this thesis direct approaches of computation of $\beta_*$ based on (\ref{eqn_beta_low}) without using heat kernels.

We prove by means of these approaches that
$$\beta_*=
\begin{cases}
\log5/\log2,&\text{for the SG},\\
\log(8\rho)/\log3,&\text{for the SC},
\end{cases}
$$
thus matching (\ref{eqn_beta_up}).

On the SG, we provide an approach by using reflection technique and trace technique from abstract Dirichlet form theory to construct non-local regular Dirichlet forms on the SG from regular Dirichlet forms on certain augmented rooted tree whose certain boundary at infinity is the SG.

\subsection*{Approximation of Local DFs by Non-Local DFs}

We prove on the SG and the SC the relations (\ref{eqn_approximation}) between $\calE_\loc$ and $\calE_\beta$ without using subordination technique. In fact, (\ref{eqn_approximation}) follows as direct consequences of our construction of $\calE_\loc$. Moreover, on the SG, we prove that $\calE_\loc$ can be obtained as a Mosco limit of non-local Dirichlet forms $E_\beta$ as $\beta\uparrow\beta_*$, where $E_\beta\asymp\calE_\beta$.

\vspace{1.5em}

Hence, in this thesis, we develop an alternative approach to analysis on fractals that is based on systematic use of the quadratic form $\calE_\beta$ and the notion of the walk dimension $\beta_*$ defined in (\ref{eqn_beta_low}). Although this approach has been implemented on the SG and the SC, there are indications that it may work in more general spaces.

An ultimate goal of this approach would be to construct local regular Dirichlet forms on rather general fractal spaces as renormalized limits of $\calE_\beta$ as $\beta\uparrow\beta_*$. However, this will be a subject of another work.

\section{Basic Notions of the SG and the SC}\label{sec_notion}

First, we give basic notions of the SG as follows.

Consider the following points in $\R^2$:
$$p_0=(0,0),p_1=(1,0),p_2=(\frac{1}{2},\frac{\sqrt{3}}{2}).$$
Let $f_i(x)=(x+p_i)/2,x\in\R^2,i=0,1,2$, then the SG is the unique non-empty compact set $K$ in $\R^2$ satisfying $K=\cup_{i=0}^2f_i(K)$. Let $\nu$ be the normalized Hausdorff measure on $K$ of dimension $\alpha=\log3/\log2$. Then $(K,|\cdot|,\nu)$ is a metric measure space, where $|\cdot|$ is the Euclidean metric in $\R^2$.

Let
$$V_0=\myset{p_0,p_1,p_2},V_{n+1}=f_0(V_n)\cup f_1(V_n)\cup f_2(V_n)\text{ for all }n\ge0.$$
Then $\myset{V_n}$ is an increasing sequence of finite sets and $K$ is the closure of $V^*=\cup_{n=0}^\infty V_n$.

Let $W_0=\myset{\emptyset}$ and
$$W_n=\myset{w=w_1\ldots w_n:w_i=0,1,2,i=1,\ldots,n}\text{ for all }n\ge1.$$

For all
\begin{align*}
w^{(1)}&=w^{(1)}_1\ldots w^{(1)}_m\in W_m,\\
w^{(2)}&=w^{(2)}_1\ldots w^{(2)}_n\in W_n,
\end{align*}
denote $w^{(1)}w^{(2)}\in W_{m+n}$ by
$$w^{(1)}w^{(2)}=w^{(1)}_1\ldots w^{(1)}_mw^{(2)}_1\ldots w^{(2)}_n.$$
For all $i=0,1,2$, denote
$$i^n=\underbrace{i\ldots i}_{n\ \text{times}}.$$
For all $w=w_1\ldots w_{n-1}w_n\in W_n$, denote $w^-=w_1\ldots w_{n-1}\in W_{n-1}$.

For all $w=w_1\ldots w_n\in W_n$, let
\begin{align*}
f_w&=f_{w_1}\circ\ldots\circ f_{w_n},\\
V_w&=f_{w_1}\circ\ldots\circ f_{w_n}(V_0),\\
K_w&=f_{w_1}\circ\ldots\circ f_{w_n}(K),\\
P_w&=f_{w_1}\circ\ldots\circ f_{w_{n-1}}(p_{w_n}), 
\end{align*}
where $f_\emptyset=\mathrm{id}$ is the identity map.

For all $n\ge1$, let $X_n$ be the graph with vertex set $W_n$ and edge set $H_n$ given by
$$H_n=\myset{(w^{(1)},w^{(2)}):w^{(1)},w^{(2)}\in W_n,w^{(1)}\ne w^{(2)},K_{w^{(1)}}\cap K_{w^{(2)}}\ne\emptyset}.$$
For example, we have the figure of $X_3$ in Figure \ref{fig_X3}. Denote $w^{(1)}\sim_n w^{(2)}$ if $(w^{(1)},w^{(2)})\in H_n$. If $w^{(1)}\sim_nw^{(2)}$ satisfies $P_{w^{(1)}}\ne P_{w^{(2)}}$, we say that $w^{(1)}\sim_nw^{(2)}$ is of type \Rmnum{1}. If $w^{(1)}\sim_nw^{(2)}$ satisfies $P_{w^{(1)}}=P_{w^{(2)}}$, we say that $w^{(1)}\sim_nw^{(2)}$ is of type \Rmnum{2}. For example, $000\sim_3001$ is of type \Rmnum{1}, $001\sim_3010$ is of type \Rmnum{2}.




\begin{figure}[ht]
  \centering
  \begin{tikzpicture}
  \draw (0,0)--(9,0);
  \draw (0,0)--(9/2,9/2*1.7320508076);
  \draw (9,0)--(9/2,9/2*1.7320508076);
  
  \draw (3,0)--(3/2,3/2*1.7320508076);
  \draw (6,0)--(15/2,3/2*1.7320508076);
  \draw (3,3*1.7320508076)--(6,3*1.7320508076);
  
  \draw (4,4*1.7320508076)--(5,4*1.7320508076);
  \draw (7/2,7/2*1.7320508076)--(4,3*1.7320508076);
  \draw (11/2,7/2*1.7320508076)--(5,3*1.7320508076);
  
  \draw (1,1*1.7320508076)--(2,1.7320508076);
  \draw (1/2,1/2*1.7320508076)--(1,0);
  \draw (2,0)--(5/2,1/2*1.7320508076);
  
  \draw (7,1.7320508076)--(8,1.7320508076);
  \draw (13/2,1.7320508076/2)--(7,0);
  \draw (17/2,1.7320508076/2)--(8,0);
  
  \draw[fill=black] (0,0) circle (0.06);
  \draw[fill=black] (1,0) circle (0.06);
  \draw[fill=black] (2,0) circle (0.06);
  \draw[fill=black] (3,0) circle (0.06);
  \draw[fill=black] (6,0) circle (0.06);
  \draw[fill=black] (7,0) circle (0.06);
  \draw[fill=black] (8,0) circle (0.06);
  \draw[fill=black] (9,0) circle (0.06);
  \draw[fill=black] (1/2,1.7320508076/2) circle (0.06);
  \draw[fill=black] (5/2,1.7320508076/2) circle (0.06);
  \draw[fill=black] (13/2,1.7320508076/2) circle (0.06);
  \draw[fill=black] (17/2,1.7320508076/2) circle (0.06);
  \draw[fill=black] (1,1.7320508076) circle (0.06);
  \draw[fill=black] (2,1.7320508076) circle (0.06);
  \draw[fill=black] (7,1.7320508076) circle (0.06);
  \draw[fill=black] (8,1.7320508076) circle (0.06);
  \draw[fill=black] (3/2,3/2*1.7320508076) circle (0.06);
  \draw[fill=black] (15/2,3/2*1.7320508076) circle (0.06);
  \draw[fill=black] (3,3*1.7320508076) circle (0.06);
  \draw[fill=black] (4,3*1.7320508076) circle (0.06);
  \draw[fill=black] (5,3*1.7320508076) circle (0.06);
  \draw[fill=black] (6,3*1.7320508076) circle (0.06);
  \draw[fill=black] (7/2,7/2*1.7320508076) circle (0.06);
  \draw[fill=black] (11/2,7/2*1.7320508076) circle (0.06);
  \draw[fill=black] (4,4*1.7320508076) circle (0.06);
  \draw[fill=black] (5,4*1.7320508076) circle (0.06);
  \draw[fill=black] (9/2,9/2*1.7320508076) circle (0.06);
  
  \draw (0,-0.3) node {$000$};
  \draw (1,-0.3) node {$001$};
  \draw (2,-0.3) node {$010$};
  \draw (3,-0.3) node {$011$};
  \draw (6,-0.3) node {$100$};
  \draw (7,-0.3) node {$101$};
  \draw (8,-0.3) node {$110$};
  \draw (9,-0.3) node {$111$};
  
  \draw (0.1,1/2*1.7320508076) node {$002$};
  \draw (2.9,1/2*1.7320508076) node {$012$};
  \draw (6.1,1/2*1.7320508076) node {$102$};
  \draw (8.9,1/2*1.7320508076) node {$112$};
  
  \draw (0.6,1*1.7320508076) node {$020$};
  \draw (2.4,1*1.7320508076) node {$021$};
  \draw (6.6,1*1.7320508076) node {$120$};
  \draw (8.4,1*1.7320508076) node {$121$};
  
  \draw (1.1,3/2*1.7320508076) node {$022$};
  \draw (7.9,3/2*1.7320508076) node {$122$};
  
  \draw (2.6,3*1.7320508076) node {$200$};
  \draw (4,3*1.7320508076-0.3) node {$201$};
  \draw (5,3*1.7320508076-0.3) node {$210$};
  \draw (6.4,3*1.7320508076) node {$211$};
  
  \draw (3.1,7/2*1.7320508076) node {$202$};
  \draw (5.9,7/2*1.7320508076) node {$212$};
  
  \draw (3.6,4*1.7320508076) node {$220$};
  \draw (5.4,4*1.7320508076) node {$221$};
  
  \draw (4.5,9/2*1.7320508076+0.3) node {$222$};
  
  \end{tikzpicture}
  \caption{$X_3$}\label{fig_X3}
\end{figure}


For all $n\ge1,u\in L^2(K;\nu)$, let $P_nu:W_n\to\R$ be given by
$$P_nu(w)=\frac{1}{\nu(K_{w})}\int_{K_w}u(x)\nu(\md x)=\int_K(u\circ f_w)(x)\nu(\md x),w\in W_n.$$

Then, we give basic notions of the SC as follows.

Consider the following points in $\R^2$:
$$p_0=(0,0),p_1=(\frac{1}{2},0),p_2=(1,0),p_3=(1,\frac{1}{2}),$$
$$p_4=(1,1),p_5=(\frac{1}{2},1),p_6=(0,1),p_7=(0,\frac{1}{2}).$$
Let $f_i(x)=(x+2p_i)/3$, $x\in\R^2$, $i=0,\ldots,7$. Then the SC is the unique non-empty compact set $K$ in $\R^2$ satisfying $K=\cup_{i=0}^7f_i(K)$. Let $\nu$ be the normalized Hausdorff measure on $K$ of dimension $\alpha=\log8/\log3$. Then $(K,|\cdot|,\nu)$ is a metric measure space, where $|\cdot|$ is the Euclidean metric in $\R^2$.

Let
$$V_0=\myset{p_0,\ldots,p_7},V_{n+1}=\cup_{i=0}^7f_i(V_n)\text{ for all }n\ge0.$$
Then $\myset{V_n}$ is an increasing sequence of finite sets and $K$ is the closure of $V^*=\cup_{n=0}^\infty V_n$.

Let $W_0=\myset{\emptyset}$ and 
$$W_n=\myset{w=w_1\ldots w_n:w_i=0,\ldots,7,i=1,\ldots,n}\text{ for all }n\ge1.$$
Similar to the SG, define $w^{(1)}w^{(2)}\in W_{m+n}$ and $i^n$ for all $w^{(1)}\in W_m,w^{(2)}\in W_n,i=0,\ldots,7$.

For all $w=w_1\ldots w_n\in W_n$, let
\begin{align*}
f_w&=f_{w_1}\circ\ldots\circ f_{w_n},\\
V_w&=f_{w_1}\circ\ldots\circ f_{w_n}(V_0),\\
K_w&=f_{w_1}\circ\ldots\circ f_{w_n}(K),\\
P_w&=f_{w_1}\circ\ldots\circ f_{w_{n-1}}(p_{w_n}),
\end{align*}
where $f_\emptyset=\mathrm{id}$ is the identity map.

\section{Statement of the Main Results}

We list the main results of this thesis. We use the notions introduced in Section \ref{sec_notion}.

\subsection*{Analytic Construction of Local Regular Dirichlet Forms}

We give a unified purely analytic construction of local regular Dirichlet forms on the SG and the SC using $\Gamma$-convergence of non-local closed forms. On the SG, this construction is much more complicated than that of Kigami, but this construction can be applied to more general fractal spaces, in particular, the SC.

\begin{mythm}\label{thm_main_SG_con}
Let $K$ be the SG. There exists a self-similar strongly local regular Dirichlet form $(\calE_\loc,\calF_\loc)$ on $L^2(K;\nu)$ satisfying
\begin{align*}
&\calE_\loc(u,u)\asymp\sup_{n\ge1}\left(\frac{5}{3}\right)^n\sum_{w^{(1)}\sim_nw^{(2)}}\left(P_nu(w^{(1)})-P_nu(w^{(2)})\right)^2,\\
&\calF_\loc=\myset{u\in L^2(K;\nu):\calE_\loc(u,u)<+\infty}.
\end{align*}
\end{mythm}

\begin{mythm}\label{thm_main_SC_con}
Let $K$ be the SC. There exists a self-similar strongly local regular Dirichlet form $(\calE_{\loc},\calF_{\loc})$ on $L^2(K;\nu)$ satisfying
\begin{align*}
&\calE_{\loc}(u,u)\asymp\sup_{n\ge1}3^{(\beta^*-\alpha)n}\sum_{w\in W_n}
{\sum_{\mbox{\tiny
$
\begin{subarray}{c}
p,q\in V_w\\
|p-q|=2^{-1}\cdot3^{-n}
\end{subarray}
$
}}}
(u(p)-u(q))^2,\\
&\calF_{\loc}=\myset{u\in C(K):\calE_\loc(u,u)<+\infty}.
\end{align*}
\end{mythm}

Theorem \ref{thm_main_SG_con} is Theorem \ref{SG_con_thm_BM}. Theorem \ref{thm_main_SC_con} is Theorem \ref{SC_con_thm_BM}.

\subsection*{Determination of Walk Dimension Without Using Diffusion}

We give the determination of the walk dimensions of the SG and the SC. The point of our approach is that we do \emph{not} need the diffusion. This gives a partial answer to a problem raised by Pietruska-Pa\l uba \cite[PROBLEM 3]{Pie09}.

Denote that
$$\beta^*:=
\begin{cases}
\log5/\log2,&\text{for the SG},\\
\log(8\rho)/\log3,&\text{for the SC},
\end{cases}
$$
where $\rho$ is some parameter in resistance estimates.

Recall that
$$\beta_*:=\sup\myset{\beta>0:(\calE_\beta,\calF_\beta)\text{ is a regular Dirichlet form on }L^2(K;\nu)}.$$

\begin{mythm}\label{thm_main_det}
Let $K$ be the SG or the SC. For all $\beta\in(\alpha,\beta^*)$, the quadratic form $(\calE_\beta,\calF_\beta)$ is a regular Dirichlet form on $L^2(K;\nu)$. For all $\beta\in[\beta^*,+\infty)$, the space $\calF_\beta$ consists only of constant functions. Consequently, $\beta_*=\beta^*$.
\end{mythm}

For the SG, this is Theorem \ref{SG_det_thm_E_K} and Theorem \ref{SG_det_thm_ub} (alternatively, see also Theorem \ref{SG_app_thm_det} and Theorem \ref{SG_con_thm_nonlocal}). For the SC, this is Theorem \ref{SC_con_thm_walk}.

We give bound of the walk dimension of the SC as follows.

\begin{mythm}\label{thm_main_SC_bound}
For the SC, we have
$$\beta_*\in\left[\frac{\log\left(8\cdot\frac{7}{6}\right)}{\log3},\frac{\log\left(8\cdot\frac{3}{2}\right)}{\log3}\right].$$
\end{mythm}

This is Theorem \ref{SC_con_thm_bound}. This gives a partial answer to an open problem raised by Barlow \cite[Open Problem 2]{Bar13}.

\subsection*{Approximation of Local DFs by Non-Local DFs}

We give the approximation of local Dirichlet forms by non-local Dirichlet forms on the SG and the SC which are direct consequences of our construction. Pietruska-Pa\l uba mentioned in \cite{Pie09} after Theorem 8 stating (\ref{eqn_approximation}) that ``how to do this without appealing to stochastic processes is unknown so far". The point of our approach is that we do \emph{not} need the diffusion.

\begin{mythm}\label{thm_main_app}
Both on the SG and the SC, there exists some positive constant $C$ such that for all $u\in\calF_\loc$, we have
$$\frac{1}{C}\calE_\loc(u,u)\le\varliminf_{\beta\uparrow\beta^*}(\beta^*-\beta)\calE_\beta(u,u)\le\varlimsup_{\beta\uparrow\beta^*}(\beta^*-\beta)\calE_{\beta}(u,u)\le C\calE_\loc(u,u).$$
\end{mythm}

For the SG, this is Corollary \ref{SG_con_cor_conv}. For the SC, this is Corollary \ref{SC_con_cor_approx}.

On the SG, we give a new semi-norm $E_\beta$ by
$$E_\beta(u,u):=\sum_{n=1}^\infty2^{(\beta-\alpha)n}\sum_{w\in W_n}\sum_{p,q\in V_w}(u(p)-u(q))^2.$$

\begin{mythm}\label{thm_main_Mosco}
Let $K$ be the SG.
\begin{enumerate}[(a)]
\item For all $\beta\in(\alpha,+\infty)$, for all $u\in C(K)$, we have
$$E_\beta(u,u)\asymp\calE_\beta(u,u).$$
\item For all $u\in L^2(K;\nu)$, we have
$$(1-5^{-1}\cdot 2^{\beta})E_\beta(u,u)\uparrow\frakE_\loc(u,u)$$
as $\beta\uparrow\beta^*=\log5/\log2$.
\item For all sequence $\myset{\beta_n}\subseteq(\alpha,\beta^*)$ with $\beta_n\uparrow\beta^*$, we have $(1-5^{-1}\cdot 2^{\beta_n})E_{\beta_n}\to\frakE_\loc$ in the sense of Mosco.
\end{enumerate}
\end{mythm}

Part (a) is Theorem \ref{SG_app_thm_main}. Part (b) is Theorem \ref{SG_app_thm_incre}. Part (c) is Theorem \ref{SG_app_thm_conv_main}.

The intrinsic idea of this thesis is the \emph{discretization} of non-local quadratic form on self-similar set using a quadratic form on an augmented rooted tree or a scaled summation of quadratic forms on infinitely many finite graphs. This enables us to investigate many interesting problems related to Dirichlet forms on self-similar sets.

\section{Structure of the Thesis}
This thesis is organized as follows.

In Chapter \ref{ch_pre}, we collect some preliminaries for later chapters.

In Chapter \ref{ch_SG_det}, we determine the walk dimension of the SG without using diffusion. We construct non-local regular Dirichlet forms on the SG from regular Dirichlet forms on certain augmented rooted tree whose certain boundary at infinity is the SG. This chapter is based on my work \cite{GY16} joint with Prof. Alexander Grigor'yan.

In Chapter \ref{ch_SG_app}, we consider approximation of the local Dirichlet form by non-local Dirichlet forms on the SG. This chapter is based on my work \cite{MY17}.

In Chapter \ref{ch_SG_con} and Chapter \ref{ch_SC_con}, we give a purely analytic construction of self-similar local regular Dirichlet forms on the SG and the SC using approximation of stable-like non-local closed forms. Chapter \ref{ch_SG_con} is based on my work \cite{MY18}. Chapter \ref{ch_SC_con} is based on my work \cite{GY17} joint with Prof. Alexander Grigor'yan.

NOTATION. The letters $c,C$ will always refer to some positive constants and may change at each occurrence. The sign $\asymp$ means that the ratio of the two sides is bounded from above and below by positive constants. The sign $\lesssim$ ($\gtrsim$) means that the LHS is bounded by positive constant times the RHS from above (below).

\chapter{Preliminary}\label{ch_pre}

In this chapter, we collect some preliminaries for later chapters.

\section{Dirichlet Form Theory}

The book \cite{FOT11} by Fukushima, Oshima and Takeda is a standard reference.

Let $H$ be a real Hilbert space with inner product $(\cdot,\cdot)$.

We say that $(\calE,\calD[\calE])$ is a symmetric form on $H$ if
\begin{itemize}
\item $\calD[\calE]$ is a dense subspace of $H$.
\item For all $u,v\in\calD[\calE]$, we have $\calE(u,v)=\calE(v,u)$.
\item For all $u,v,u_1,u_2,v_1,v_2\in\calD[\calE]$, $c\in\R$, we have
$$\calE(u_1+u_2,v)=\calE(u_1,v)+\calE(u_2,v),\calE(u,v_1+v_2)=\calE(u,v_1)+\calE(u,v_2),$$
$$\calE(cu,v)=\calE(u,cv)=c\calE(u,v).$$
\item For all $u\in\calD[\calE]$, we have $\calE(u,u)\ge0$.
\end{itemize}

For all $\alpha\in(0,+\infty)$, we denote $\calE_\alpha(u,v)=\calE(u,v)+\alpha(u,v)$ for all $u,v\in\calD[\calE]$.

We say that a symmetric form $(\calE,\calD[\calE])$ on $H$ is closed or a closed (symmetric) form if $(\calD[\calE],\calE_1)$ is a Hilbert space.

We say that $\myset{T_t:t\in(0,+\infty)}$ is a semi-group on $H$ if
\begin{enumerate}[(1)]
\item For all $t\in(0,+\infty)$, we have $T_t$ is a symmetric operator with domain $\calD(T_t)=H$.
\item For all $t\in(0,+\infty)$, $u\in H$, we have $(T_tu,T_tu)\le(u,u)$.
\item For all $t,s\in(0,+\infty)$, we have $T_t\circ T_s=T_{t+s}$.
\end{enumerate}
We say that a semi-group $\myset{T_t:t\in(0,+\infty)}$ on $H$ is strongly continuous if
\begin{itemize}
\item For all $u\in H$, we have $\lim_{t\downarrow0}(T_tu-u,T_tu-u)=0$.
\end{itemize}

We say that $\myset{G_\alpha:\alpha\in(0,+\infty)}$ is a resolvent on $H$ if
\begin{enumerate}[(1)]
\item For all $\alpha\in(0,+\infty)$, we have $G_\alpha$ is a symmetric operator with domain $\calD(G_\alpha)=H$.
\item For all $\alpha\in(0,+\infty)$, $u\in H$, we have $(\alpha G_\alpha u,\alpha G_\alpha u)\le(u,u)$.
\item For all $\alpha,\beta\in(0,+\infty)$, we have $G_\alpha-G_\beta+(\alpha-\beta)G_\alpha\circ G_\beta=0$.
\end{enumerate}
We say that a resolvent $\myset{G_\alpha:\alpha\in(0,+\infty)}$ on $H$ is strongly continuous if
\begin{itemize}
\item For all $u\in H$, we have $\lim_{\alpha\to+\infty}(\alpha G_\alpha u-u,\alpha G_\alpha u-u)=0$.
\end{itemize}

\begin{myprop}(\cite[Exercise 1.3.1, Lemma 1.3.1, Lemma 1.3.2, Exercise 1.3.2, Theorem 1.3.1]{FOT11})
There exists a one-to-one correspondence among the family of closed forms, the family of non-positive definite self-adjoint operators, the family of strongly continuous semi-groups and the family of strongly continuous resolvents.
\end{myprop}

Let $(X,m)$ be a measure space. Let $L^2(X;m)$ be the space of all $L^2$-integrable extended-real-valued functions on $(X,m)$, then $L^2(X;m)$ is a real Hilbert space.

Let $(\calE,\calD[\calE])$ be a symmetric form on $L^2(X;m)$. We say that $(\calE,\calD[\calE])$ on $L^2(X;m)$ is Markovian or has Markovian property if for all $\veps\in(0,+\infty)$, there exists a function $\phi_\veps:\R\to\R$ satisfying
$$\phi_\veps(t)=t\text{ for all }t\in[0,1],$$
$$\phi_\veps(t)\in[-\veps,1+\veps]\text{ for all }t\in\R,$$
$$0\le\phi_\veps(t)-\phi_\veps(s)\le t-s\text{ for all }t,s\in\R\text{ with }t\ge s,$$
such that for all $u\in\calD[\calE]$, we have $\phi_\veps(u)\in\calD[\calE]$ and $\calE(\phi_\veps(u),\phi_\veps(u))\le\calE(u,u)$.

A Dirichlet form is a Markovian closed symmetric form.

Let $u$ be a function on $X$. We say that $(u\vee0)\wedge1$ is the unit contraction of $u$. We say that $v$ is a normal contraction of $u$ if $|v(x)|\le|u(x)|$ and $|v(x)-v(y)|\le|u(x)-u(y)|$ for all $x,y\in X$.

We say that
\begin{enumerate}[(1)]
\item The unit contraction operates on $(\calE,\calD[\calE])$ on $L^2(X;m)$ if for all $u\in\calD[\calE]$, we have the unit contraction $v=(u\vee0)\wedge1\in\calD[\calE]$ and $\calE(v,v)\le\calE(u,u)$.
\item Every normal contraction operates on $(\calE,\calD[\calE])$ on $L^2(X;m)$ if for all $u\in\calD[\calE]$, for all normal contraction $v$ of $u$, we have $v\in\calD[\calE]$ and $\calE(v,v)\le\calE(u,u)$.
\end{enumerate}

Let $S$ be a linear operator with domain $\calD(S)=L^2(X;m)$. We say that $S$ is Markovian if for all $u\in L^2(X;m)$ with $0\le u\le 1$ $m$-a.e., we have $0\le Su\le 1$ $m$-a.e..

We say that a semi-group $\myset{T_t:t\in(0,+\infty)}$ on $L^2(X;m)$ is Markovian if for all $t\in(0,+\infty)$, we have $T_t$ is Markovian.

We say that a resolvent $\myset{G_\alpha:\alpha\in(0,+\infty)}$ on $L^2(X;m)$ is Markovian if for all $\alpha\in(0,+\infty)$, we have $\alpha G_\alpha$ is Markovian.

Let $(X,d,m)$ be a metric measure space, that is, $(X,d)$ is a locally compact separable metric space and $m$ is a Radon measure on $X$ with full support. 

\begin{myprop}(\cite[Theorem 1.4.1]{FOT11})
Let $(\calE,\calD[\calE])$ be a closed form on $L^2(X;m)$, $\myset{T_t:t\in(0,+\infty)}$ on $L^2(X;m)$ its corresponding strongly continuous semi-group and $\myset{G_\alpha:\alpha\in(0,+\infty)}$ on $L^2(X;m)$ its corresponding strongly continuous resolvent. Then the followings are equivalent.
\begin{enumerate}[(1)]
\item $(\calE,\calD[\calE])$ on $L^2(X;m)$ is Markovian.
\item The unit contraction operates on $(\calE,\calD[\calE])$ on $L^2(X;m)$.
\item Every normal contraction operates on $(\calE,\calD[\calE])$ on $L^2(X;m)$.
\item $\myset{T_t:t\in(0,+\infty)}$ on $L^2(X;m)$ is Markovian.
\item $\myset{G_\alpha:\alpha\in(0,+\infty)}$ on $L^2(X;m)$ is Markovian.
\end{enumerate}
\end{myprop}

Denote $C_c(X)$ as the space of all continuous functions with compact supports.

Let $(\calE,\calD[\calE])$ be a symmetric form on $L^2(X;m)$. We say that $(\calE,\calD[\calE])$ on $L^2(X;m)$ is regular if $\calD[\calE]\cap C_c(X)$ is $\calE_1$-dense in $\calD[\calE]$ and uniformly dense in $C_c(X)$. We say that $(\calE,\calD[\calE])$ on $L^2(X;m)$ is local if for all $u,v\in\calD[\calE]$ with compact supports, we have $\calE(u,v)=0$. We say that $(\calE,\calD[\calE])$ on $L^2(X;m)$ is strongly local if for all $u,v\in\calD[\calE]$ with compact supports and $v$ is constant in an open neighborhood of $\mathrm{supp}(u)$, we have $\calE(u,v)=0$.

Let $(\calE,\calD[\calE])$ on $L^2(X;m)$ be a Dirichlet form and $\myset{T_t:t\in(0,+\infty)}$ on $L^2(X;m)$ its corresponding strongly continuous Markovian semi-group. Then for all $t\in(0,+\infty)$, $T_t$ can be extended to be a contractive linear operator on $L^\infty(X;m)$. We say that $(\calE,\calD[\calE])$ on $L^2(X;m)$ or $\myset{T_t:t\in(0,+\infty)}$ on $L^2(X;m)$ is conservative if $T_t1=1$ $m$-a.e. for all (or equivalently, for some) $t\in(0,+\infty)$.

We say that $(\calE,\calD[\calE])$ is a closed form on $L^2(X;m)$ in the wide sense if $\calD[\calE]$ is complete under the inner product $\calE_1$ but $\calD[\calE]$ is not necessary to be dense in $L^2(X;m)$. If $(\calE,\calD[\calE])$ is a closed form on $L^2(X;m)$ in the wide sense, we extend $\calE$ to be $+\infty$ outside $\calD[\calE]$, hence the information of $\calD[\calE]$ is encoded in $\calE$.

Note that a closed form is a closed form in the wide sense.

We collect the definitions and some results about $\Gamma$-convergence and Mosco convergence as follows.

\begin{mydef}\label{def_gamma}
Let $\calE^n,\calE$ be closed forms on $L^2(X;m)$ in the wide sense. We say that $\calE^n$ is $\Gamma$-convergent to $\calE$ if the following conditions are satisfied.
\begin{enumerate}[(1)]
\item For all $\myset{u_n}\subseteq L^2(X;m)$ that converges \emph{strongly} to $u\in L^2(X;m)$, we have
$$\varliminf_{n\to+\infty}\calE^n(u_n,u_n)\ge\calE(u,u).$$
\item For all $u\in L^2(X;m)$, there exists a sequence $\myset{u_n}\subseteq L^2(X;m)$ converging \emph{strongly} to $u$ in $L^2(X;m)$ such that
$$\varlimsup_{n\to+\infty}\calE^n(u_n,u_n)\le\calE(u,u).$$
\end{enumerate}
\end{mydef}

We can see that $\Gamma$-convergence is very weak from the following result.

\begin{myprop}\label{prop_gamma}(\cite[Proposition 6.8, Theorem 8.5, Theorem 11.10, Proposition\\
\noindent 12.16]{Dal93})
Let $\myset{\calE^n}$ be a sequence of closed forms on $L^2(X;m)$ in the wide sense, then there exist some subsequence $\myset{\calE^{n_k}}$ and some closed form $(\calE,\calD[\calE])$ on $L^2(X;m)$ in the wide sense such that $\calE^{n_k}$ is $\Gamma$-convergent to $\calE$.
\end{myprop}

\begin{mydef}\label{def_Mosco}
Let $\calE^n$, $\calE$ be closed forms on $L^2(X;m)$. We say that $\calE^n$ converges to $\calE$ in the sense of Mosco if the following conditions are satisfied.
\begin{enumerate}[(1)]
\item\label{def_Mosco_1} For all $\myset{u_n}\subseteq L^2(X;m)$ that converges \emph{weakly} to $u\in L^2(X;m)$, we have
$$\varliminf_{n\to+\infty}\calE^n(u_n,u_n)\ge\calE(u,u).$$
\item\label{def_Mosco_2} For all $u\in L^2(X;m)$, there exists a sequence $\myset{u_n}\subseteq L^2(X;m)$ converging \emph{strongly} to $u$ in $L^2(X;m)$ such that
$$\varlimsup_{n\to+\infty}\calE^n(u_n,u_n)\le\calE(u,u).$$
\end{enumerate}
\end{mydef}

Let $\myset{T_t:t\in(0,+\infty)}$, $\myset{T^n_t:t\in(0,+\infty)}$ be the strongly continuous semi-groups on $L^2(X;m)$ and $\myset{G_\alpha:\alpha\in(0,+\infty)}$, $\myset{G^n_\alpha:\alpha\in(0,+\infty)}$ the strongly continuous resolvents on $L^2(X;m)$ corresponding to closed forms $(\calE,\calD[\calE])$, $(\calE^n,\calD[\calE^n])$ on $L^2(X;m)$. We have the following equivalence.

\begin{myprop}(\cite[Theorem 2.4.1, Corollary 2.6.1]{Mos94})\label{prop_Mosco}
The followings are equivalent.
\begin{enumerate}[(1)]
\item $\calE^n$ converges to $\calE$ in the sense of Mosco.
\item $T^n_tu\to T_tu$ in $L^2(X;m)$ for all $t\in(0,+\infty)$, $u\in L^2(X;m)$.
\item $G^n_\alpha u\to G_\alpha u$ in $L^2(X;m)$ for all $\alpha\in(0,+\infty)$, $u\in L^2(X;m)$.
\end{enumerate}
\end{myprop}

We have the following corollary.

\begin{mycor}\label{cor_Mosco}
Let $(\calE,\calD[\calE])$ be a closed form on $L^2(X;m)$, then for all $\myset{u_n}\subseteq L^2(X;m)$ that converges \emph{weakly} to $u\in L^2(X;m)$, we have
\begin{equation}\label{eqn_Mosco}
\calE(u,u)\le\varliminf_{n\to+\infty}\calE(u_n,u_n).
\end{equation}
\end{mycor}

\begin{proof}
Let $\calE^n=\calE$ for all $n\ge1$, then by Proposition \ref{prop_Mosco}, $\calE^n$ is trivially convergent to $\calE$ in the sense of Mosco. By Definition \ref{def_Mosco}, Equation (\ref{eqn_Mosco}) is obvious.
\end{proof}

Note that it is tedious to prove Corollary \ref{cor_Mosco} directly.

\section{Some Results on the SG}\label{sec_SG}

We use the notions of the SG introduced in Section \ref{sec_notion}.

Let us recall the classical Kigami's construction of the self-similar local regular Dirichlet form on the SG.

\begin{mythm}\label{thm_SG_con}(\cite{Kig89,Kig93,Kig01})
Let
$$\frakE_n(u,u)=\left(\frac{5}{3}\right)^n\sum_{w\in W_n}\sum_{p,q\in V_w}(u(p)-u(q))^2,n\ge0,u\in l(K),$$
where $l(S)$ is the set of all real-valued functions on the set $S$. Then $\frakE_n(u,u)$ is monotone increasing in $n$ for all $u\in l(K)$. Let
$$
\begin{aligned}
&\frakE_\loc(u,u)=\lim_{n\to+\infty}\frakE_n(u,u)=\lim_{n\to+\infty}
\left(\frac{5}{3}\right)^n\sum_{w\in W_n}\sum_{p,q\in V_w}(u(p)-u(q))^2,\\
&\frakF_\loc=\myset{u\in C(K):\frakE_\loc(u,u)<+\infty},
\end{aligned}
$$
then $(\frakE_\loc,\frakF_\loc)$ is a self-similar local regular Dirichlet form on $L^2(K;\nu)$.
\end{mythm}

Given $x_0,x_1,x_2\in\R$, we define $U=U^{(x_0,x_1,x_2)}:K\to\R$ as follows. We define $U$ on $V^*$ by induction. Let $U(p_i)=x_i,i=0,1,2$. Assume that we have defined $U$ on $P_{w}$ for all $w\in W_{n+1}$. Then for all $w\in W_n$, note that $P_{wii}=P_{wi}$ and $P_{wij}=P_{wji}$ for all $i,j=0,1,2$, define 
$$
\begin{aligned}
U(P_{w01})&=U(P_{w10})=\frac{2U(P_{w0})+2U(P_{w1})+U(P_{w2})}{5},\\
U(P_{w12})&=U(P_{w21})=\frac{U(P_{w0})+2U(P_{w1})+2U(P_{w2})}{5},\\
U(P_{w02})&=U(P_{w20})=\frac{2U(P_{w0})+U(P_{w1})+2U(P_{w2})}{5}.\\
\end{aligned}
$$
Hence we have the definition of $U$ on $P_w$ for all $w\in W_{n+2}$. Then $U$ is well-defined and uniformly continuous on $V^*$. We extend $U$ on $V^*$ to a continuous function $U$ on $K$.

Let
$$\calU=\myset{U^{(x_0,x_1,x_2)}:x_0,x_1,x_2\in\R}.$$

We have energy property and separation property as follows.

\begin{mythm}(\cite{Kig89,Kig93,Kig01})\label{thm_SG_fun}
\begin{enumerate}[(1)]
\item For all $U=U^{(x_0,x_1,x_2)}\in\calU$, $n\ge0$, we have
$$\sum_{w\in W_n}\sum_{p,q\in V_w}(U(p)-U(q))^2=\left(\frac{3}{5}\right)^n\left((x_0-x_1)^2+(x_1-x_2)^2+(x_0-x_2)^2\right).$$
\item $\calU$ separates points, that is, for all $x,y\in K$ with $x\ne y$, there exists $U\in\calU$ such that $U(x)\ne U(y)$.
\end{enumerate}
\end{mythm}

\begin{myrmk}
In Kigami's construction, $U^{(x_0,x_1,x_2)}$ is the standard harmonic function with boundary values $x_0,x_1,x_2$ on $p_0,p_1,p_2$, respectively.
\end{myrmk}

\begin{mylem}\label{lem_SG_holder}(\cite[Theorem 4.11 (\rmnum{3})]{GHL03})
For all $u\in L^2(K;\nu)$, let
$$
\begin{aligned}
E(u)&=\sum_{n=1}^\infty2^{(\beta-\alpha)n}\sum_{w^{(1)}\sim_nw^{(2)}}\left(P_nu(w^{(1)})-P_nu(w^{(2)})\right)^2,\\
F(u)&=\sup_{n\ge1}2^{(\beta-\alpha)n}\sum_{w^{(1)}\sim_nw^{(2)}}\left(P_nu(w^{(1)})-P_nu(w^{(2)})\right)^2.
\end{aligned}
$$
Then for all $\beta\in(\alpha,+\infty)$, there exists some positive constant $c$ such that
\begin{equation}\label{eqn_holder_E}
|u(x)-u(y)|^2\le cE(u)|x-y|^{\beta-\alpha},
\end{equation}
\begin{equation}\label{eqn_holder_F}
|u(x)-u(y)|^2\le cF(u)|x-y|^{\beta-\alpha},
\end{equation}
for $\nu$-almost every $x,y\in K$, for all $u\in L^2(K;\nu)$.
\end{mylem}

\begin{myrmk}
If $u\in L^2(K;\nu)$ satisfies $E(u)<+\infty$ or $F(u)<+\infty$, then $u$ has a continuous version in $C^{\frac{\beta-\alpha}{2}}(K)$. The proof of the above lemma does not rely on heat kernel.
\end{myrmk}

Let us introduce the Besov spaces on the SG as follows. Define
$$
\begin{aligned}
\left[u\right]_{B^{2,2}_{\alpha,\beta}(K)}&=\sum_{n=1}^\infty2^{(\alpha+\beta)n}\int\limits_K\int\limits_{B(x,2^{-n})}(u(x)-u(y))^2\nu(\md y)\nu(\md x),\\
\left[u\right]_{B^{2,\infty}_{\alpha,\beta}(K)}&=\sup_{n\ge1}2^{(\alpha+\beta)n}\int\limits_K\int\limits_{B(x,2^{-n})}(u(x)-u(y))^2\nu(\md y)\nu(\md x),
\end{aligned}
$$
and 
$$
\begin{aligned}
B_{\alpha,\beta}^{2,2}(K)&=\myset{u\in L^2(K;\nu):\left[u\right]_{B^{2,2}_{\alpha,\beta}(K)}<+\infty},\\
B_{\alpha,\beta}^{2,\infty}(K)&=\myset{u\in L^2(K;\nu):\left[u\right]_{B^{2,\infty}_{\alpha,\beta}(K)}<+\infty}.
\end{aligned}
$$

\section{Some Results on the SC}\label{sec_SC}

We use the notions of the SC introduced in Section \ref{sec_notion}.

\begin{mylem}\label{lem_SC_holder}(\cite[Theorem 4.11 (\rmnum{3})]{GHL03})
For all $u\in L^2(K;\nu)$, let
$$E(u)=\int_K\int_K\frac{(u(x)-u(y))^2}{|x-y|^{\alpha+\beta}}\nu(\md x)\nu(\md y),$$
$$F(u)=\sup_{n\ge1}3^{(\alpha+\beta)n}\int_K\int_{B(x,3^{-n})}(u(x)-u(y))^2\nu(\md y)\nu(\md x).$$
Then for all $\beta\in(\alpha,+\infty)$, there exists some positive constant $c$ such that
$$
\begin{aligned}
|u(x)-u(y)|^2&\le cE(u)|x-y|^{\beta-\alpha},\\
|u(x)-u(y)|^2&\le cF(u)|x-y|^{\beta-\alpha},
\end{aligned}
$$
for $\nu$-almost every $x,y\in K$, for all $u\in L^2(K;\nu)$.
\end{mylem}

\begin{myrmk}
If $u\in L^2(K;\nu)$ satisfies $E(u)<+\infty$ or $F(u)<+\infty$, then $u$ has a continuous version in $C^{\frac{\beta-\alpha}{2}}(K)$. The proof of the above lemma does not rely on heat kernel.
\end{myrmk}

Let us introduce the Besov spaces on the SC as follows. Define
$$
\begin{aligned}
&\left[u\right]_{B^{2,2}_{\alpha,\beta}(K)}&=\sum_{n=1}^\infty3^{(\alpha+\beta)n}\int\limits_K\int\limits_{B(x,3^{-n})}(u(x)-u(y))^2\nu(\md y)\nu(\md x),\\
&\left[u\right]_{B^{2,\infty}_{\alpha,\beta}(K)}&=\sup_{n\ge1}3^{(\alpha+\beta)n}\int\limits_K\int\limits_{B(x,3^{-n})}(u(x)-u(y))^2\nu(\md y)\nu(\md x),\\
\end{aligned}
$$
and
$$
\begin{aligned}
B_{\alpha,\beta}^{2,2}(K)&=\myset{u\in L^2(K;\nu):[u]_{B^{2,2}_{\alpha,\beta}(K)}<+\infty},\\
B_{\alpha,\beta}^{2,\infty}(K)&=\myset{u\in L^2(K;\nu):[u]_{B^{2,\infty}_{\alpha,\beta}(K)}<+\infty}.\\
\end{aligned}
$$

\section{Some Auxiliary Results}

We give two techniques from electrical networks.

The first is $\Delta$-Y transform (see \cite[Lemma 2.1.15]{Kig01}).

\begin{mylem}\label{lem_DeltaY}
The electrical networks in Figure \ref{fig_DeltaY} are equivalent, where
$$
\begin{aligned}
R_1&=\frac{R_{12}R_{31}}{R_{12}+R_{23}+R_{31}},\\
R_2&=\frac{R_{12}R_{23}}{R_{12}+R_{23}+R_{31}},\\
R_3&=\frac{R_{23}R_{31}}{R_{12}+R_{23}+R_{31}},
\end{aligned}
$$
and
$$
\begin{aligned}
R_{12}&=\frac{R_1R_2+R_2R_3+R_3R_1}{R_3},\\
R_{23}&=\frac{R_1R_2+R_2R_3+R_3R_1}{R_1},\\
R_{31}&=\frac{R_1R_2+R_2R_3+R_3R_1}{R_2}.
\end{aligned}
$$




\begin{figure}[ht]
\centering
\subfigure[$\Delta$-circuit]{
\begin{tikzpicture}
\draw (0,0)--(3,0)--(1.5,1.5*1.7320508076)--cycle;

\draw[fill=black] (0,0) circle (0.06);
\draw[fill=black] (3,0) circle (0.06);
\draw[fill=black] (1.5,1.5*1.7320508076) circle (0.06);

\draw (0,-0.3) node {$p_2$};
\draw (3,-0.3) node {$p_3$};
\draw (1.5,1.5*1.7320508076+0.3) node {$p_1$};

\draw (1.5,-0.3) node {$R_{23}$};
\draw (0.5,0.75*1.7320508076+0.3) node {$R_{12}$};
\draw (2.5,0.75*1.7320508076+0.3) node {$R_{31}$};

\end{tikzpicture}
}
\hspace{1in}
\subfigure[Y-circuit]{
\begin{tikzpicture}
\draw (0,0)--(1.5,1.5/1.7320508076);
\draw (3,0)--(1.5,1.5/1.7320508076);
\draw (1.5,1.5*1.7320508076)--(1.5,1.5/1.7320508076);

\draw[fill=black] (0,0) circle (0.06);
\draw[fill=black] (3,0) circle (0.06);
\draw[fill=black] (1.5,1.5*1.7320508076) circle (0.06);
\draw[fill=black] (1.5,1.5/1.7320508076) circle (0.06);

\draw (0,-0.3) node {$p_2$};
\draw (3,-0.3) node {$p_3$};
\draw (1.5,1.5*1.7320508076+0.3) node {$p_1$};
\draw (1.5,1.5/1.7320508076-0.3) node {$p_0$};

\draw (1.5+0.3,0.75*1.7320508076+0.3) node {$R_{1}$};
\draw (0.7,0.7) node {$R_2$};
\draw (2.3,0.7) node {$R_3$};

\end{tikzpicture}
}
\caption{$\Delta$-Y Transform}\label{fig_DeltaY}
\end{figure}
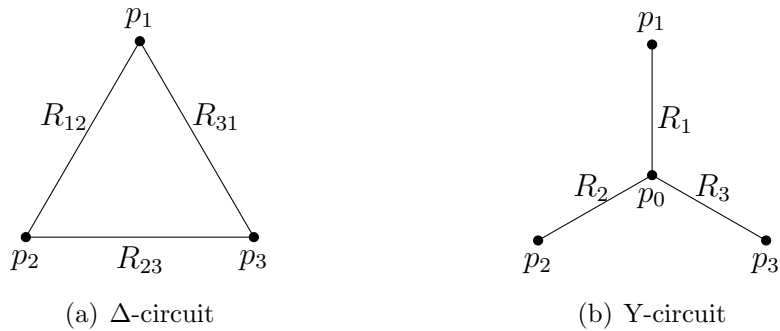


\end{mylem}

The second is shorting and cutting technique (see \cite{DS84}). Shorting certain sets of vertices will decrease the resistance between arbitrary two vertices. Cutting certain sets of vertices will increase the resistance between arbitrary two vertices.

We give two elementary results as follows.

\begin{myprop}\label{prop_ele1}
Let $\myset{x_n}$ be a \emph{monotone increasing} sequence in $[0,+\infty]$. Then $(1-\lambda)\sum_{n=1}^\infty\lambda^nx_n$ is monotone increasing in $\lambda\in(0,1)$ and
$$\lim_{\lambda\uparrow1}(1-\lambda)\sum_{n=1}^\infty\lambda^nx_n=\lim_{n\to+\infty}x_n=\sup_{n\ge1}x_n.$$
\end{myprop}

\begin{proof}
For all $\lambda_1,\lambda_2\in(0,1)$ with $\lambda_1<\lambda_2$. We show that
$$(1-\lambda_1)\sum_{n=1}^\infty\lambda_1^nx_n\le(1-\lambda_2)\sum_{n=1}^\infty\lambda_2^nx_n.$$

If $\sum_{n=1}^\infty\lambda_2^nx_n=+\infty$, then this result is obvious.

If $\sum_{n=1}^\infty\lambda_2^nx_n<+\infty$, then $\myset{x_n}\subseteq[0,+\infty)$. For all $\lambda\in(0,\lambda_2]$, we have
$$\sum_{n=1}^\infty\lambda^nx_n\le\sum_{n=1}^\infty\lambda_2^nx_n<+\infty,$$
hence
$$
\begin{aligned}
&(1-\lambda)\sum_{n=1}^\infty\lambda^nx_n=\sum_{n=1}^\infty\lambda^nx_n-\sum_{n=1}^\infty\lambda^{n+1}x_n=\sum_{n=1}^\infty\lambda^nx_n-\sum_{n=2}^\infty\lambda^{n}x_{n-1}\\
=&\lambda^1x_1+\sum_{n=2}^\infty\lambda^nx_n-\sum_{n=2}^\infty\lambda^{n}x_{n-1}=\lambda x_1+\sum_{n=2}^\infty\lambda^n(x_n-x_{n-1}).
\end{aligned}
$$
Since $\myset{x_n}$ is a monotone increasing sequence in $[0,+\infty)$, we have $x_n-x_{n-1}\ge0$ for all $n\ge2$. Hence $(1-\lambda)\sum_{n=1}^\infty\lambda^nx_n$ is monotone increasing in $\lambda\in(0,\lambda_2]$. In particular
$$(1-\lambda_1)\sum_{n=1}^\infty\lambda_1^nx_n\le(1-\lambda_2)\sum_{n=1}^\infty\lambda_2^nx_n.$$
Hence $(1-\lambda)\sum_{n=1}^\infty\lambda^nx_n$ is monotone increasing in $\lambda\in(0,1)$.

Since $\myset{x_n}$ is a monotone increasing sequence in $[0,+\infty]$, we denote
$$x_\infty=\lim_{n\to+\infty}x_n=\sup_{n\ge1}x_n\in[0,+\infty].$$
It is obvious that for all $\lambda\in(0,1)$, we have
$$(1-\lambda)\sum_{n=1}^\infty\lambda^nx_n\le(1-\lambda)\sum_{n=1}^\infty\lambda^nx_\infty=(1-\lambda)\frac{\lambda}{1-\lambda}x_\infty=\lambda x_\infty,$$
hence
$$\varlimsup_{\lambda\uparrow1}(1-\lambda)\sum_{n=1}^\infty\lambda^nx_n\le x_\infty.$$

On the other hand, for all $A<x_\infty$, there exists some positive integer $N\ge1$ such that for all $n>N$, we have $x_n>A$, hence
$$(1-\lambda)\sum_{n=1}^\infty\lambda^nx_n\ge(1-\lambda)\sum_{n=N+1}^\infty\lambda^nx_n\ge(1-\lambda)\sum_{n=N+1}^\infty\lambda^nA=(1-\lambda)\frac{\lambda^{N+1}}{1-\lambda}A=\lambda^{N+1}A,$$
hence
$$\varliminf_{\lambda\uparrow1}(1-\lambda)\sum_{n=1}^\infty\lambda^nx_n\ge A.$$
Since $A<x_\infty$ is arbitrary, we have
$$\varliminf_{\lambda\uparrow1}(1-\lambda)\sum_{n=1}^\infty\lambda^nx_n\ge x_\infty.$$

Therefore
$$\lim_{\lambda\uparrow1}(1-\lambda)\sum_{n=1}^\infty\lambda^nx_n=x_\infty=\lim_{n\to+\infty}x_n=\sup_{n\ge1}x_n.$$
\end{proof}

\begin{myprop}\label{prop_ele2}
Let $\myset{x_n}$ be a sequence in $[0,+\infty]$. Then
\begin{enumerate}[(1)]
\item $$\varliminf_{n\to+\infty}x_n\le\varliminf_{\lambda\uparrow1}(1-\lambda)\sum_{n=1}^\infty\lambda^nx_n\le\varlimsup_{\lambda\uparrow1}(1-\lambda)\sum_{n=1}^\infty\lambda^nx_n\le\varlimsup_{n\to+\infty}x_n\le\sup_{n\ge1}x_n.$$
\item If there exists some positive constant $C$ such that
$$x_n\le Cx_{n+m}\text{ for all }n,m\ge1,$$
then
$$\sup_{n\ge1}x_n\le C\varliminf_{n\to+\infty}x_n.$$
\end{enumerate}
\end{myprop}

\begin{proof}
(1) It is obvious that
$$\varliminf_{\lambda\uparrow1}(1-\lambda)\sum_{n=1}^\infty\lambda^nx_n\le\varlimsup_{\lambda\uparrow1}(1-\lambda)\sum_{n=1}^\infty\lambda^nx_n,$$
$$\varlimsup_{n\to+\infty}x_n\le\sup_{n\ge1}x_n.$$

We show that
$$\varlimsup_{\lambda\uparrow1}(1-\lambda)\sum_{n=1}^\infty\lambda^nx_n\le\varlimsup_{n\to+\infty}x_n.$$
If $\varlimsup_{n\to+\infty}x_n=+\infty$, then the result is obvious. Assume that $\varlimsup_{n\to+\infty}x_n<+\infty$.

For all $A>\varlimsup_{n\to+\infty}x_n$, there exists some positive integer $N\ge1$ such that for all $n>N$, we have $x_n<A$, hence
$$
\begin{aligned}
&(1-\lambda)\sum_{n=1}^\infty\lambda^nx_n\le(1-\lambda)\sum_{n=1}^N\lambda^nx_n+(1-\lambda)\sum_{n=N+1}^\infty\lambda^nA\\
=&(1-\lambda)\sum_{n=1}^N\lambda^nx_n+(1-\lambda)\frac{\lambda^{N+1}}{1-\lambda}A=(1-\lambda)\sum_{n=1}^N\lambda^nx_n+\lambda^{N+1}A,
\end{aligned}
$$
hence
$$\varlimsup_{\lambda\uparrow1}(1-\lambda)\sum_{n=1}^\infty\lambda^nx_n\le A.$$
Since $A>\varlimsup_{n\to+\infty}x_n$ is arbitrary, we have
$$\varlimsup_{\lambda\uparrow1}(1-\lambda)\sum_{n=1}^\infty\lambda^nx_n\le\varlimsup_{n\to+\infty}x_n.$$
Similarly, we have
$$\varliminf_{\lambda\uparrow1}(1-\lambda)\sum_{n=1}^\infty\lambda^nx_n\ge\varliminf_{n\to+\infty}x_n.$$

(2) For all $A<\sup_{n\ge1}x_n$, there exists some positive integer $N\ge1$ such that $x_N>A$. By assumption, for all $n>N$, we have
$$x_n\ge\frac{1}{C}x_N\ge\frac{1}{C}A,$$
hence
$$\varliminf_{n\to+\infty}x_n\ge\frac{1}{C}A.$$
Since $A<\sup_{n\ge1}x_n$ is arbitrary, we have
$$\varliminf_{n\to+\infty}x_n\ge\frac{1}{C}\sup_{n\ge1}x_n,$$
that is,
$$\sup_{n\ge1}x_n\le C\varliminf_{n\to+\infty}x_n.$$
\end{proof}

\chapter{Determination of the Walk Dimension of the SG}\label{ch_SG_det}

This chapter is based on my work \cite{GY16} joint with Prof. Alexander Grigor'yan.

\section{Background and Statement}

Our approach of determination is based on a recent paper \cite{KLW17} of S.-L. Kong, K.-S. Lau and T.-K. Wong. They introduced conductances with parameter $\lambda\in(0,1)$ on the Sierpi\'nski graph $X$ to obtain a random walk (and a corresponding quadratic form) on $X$ and showed that the Martin boundary of that random walk is homeomorphic to the SG $K$. Let $\bar{X}$ be the Martin compactification of $X$. It was also proved in \cite{KLW17} that the quadratic form on $X$ induces an quadratic form on $K\cong\bar{X}\backslash X$ of the form (\ref{eqn_nonlocal}) with $\beta=-\log\lambda/\log2$. However, no restriction on $\beta$ was established, so that the above quadratic form on $K$ does not have to be a regular Dirichlet form.

In this chapter, we establish the exact restriction on $\lambda$ (hence on $\beta$) under which $(\calE_\beta,\calF_\beta)$ is a regular Dirichlet form on $L^2(K;\nu)$. Our method is as follows.

Firstly, we introduce a measure $m$ on $X$ to obtain a regular Dirichlet form $(\calE_X,\calF_X)$ on $L^2(X;m)$ associated with the above random walk on $X$. Then we extend this Dirichlet form to an \emph{active reflected} Dirichlet form $(\calE^\rref,\calF^\rref_a)$ on $L^2(X;m)$ which is not regular, though.

Secondly, we \emph{regularize} $(\calE^\rref,\calF^\rref_a)$ on $L^2(X;m)$ using the theory of \cite{Fuk71}. The result of regularization is a regular Dirichlet form $(\calE_{\bar{X}},\calF_{\bar{X}})$ on $L^2(\bar{X};m)$ that is an extension of $(\calE_{{X}},\calF_{{X}})$ on $L^2({X};m)$. By \cite{Fuk71}, regularization is always possible, but we show that the regularized form ``sits" on $\bar{X}$ provided $\lambda>1/5$ which is equivalent $\beta<\beta^*:=\log5/\log2$.

Thirdly, we take trace of $\calE_{\bar{X}}$ to $K$ and obtain a regular Dirichlet form $(\calE_K,\calF_K)$ on $L^2(K;\nu)$ of the form (\ref{eqn_nonlocal}).

If $\beta>\beta^*$, then we show directly that $\calF_K$ consists only of constant functions. Hence we conclude that $\beta_*=\beta^*=\log5/\log2$. This approach allows to detect the critical value $\beta_*$ of the index $\beta$ of the jump process without the construction of the diffusion.

This chapter is organized as follows. In section \ref{SG_det_sec_SG}, we review basic constructions of the SG $K$ and the Sierpi\'nski graph $X$. In section \ref{SG_det_sec_rw}, we give a transient reversible random walk $Z$ on $X$. In section \ref{SG_det_sec_df_X}, we construct a regular Dirichlet form $\calE_X$ on $X$ and its corresponding symmetric Hunt process $\myset{X_t}$. We prove that the Martin boundaries of $\myset{X_t}$ and $Z$ coincide. We show that $\calE_X$ is stochastically incomplete and $\myset{X_t}$ goes to infinity in finite time almost surely. In section \ref{SG_det_sec_ref}, we construct active reflected Dirichlet form $(\calE^{\rref},\calF^{\rref}_a)$ and show that $\calF_X\subsetneqq\calF^{\rref}_a$, hence $\calE^{\rref}$ is not regular. In section \ref{SG_det_sec_repre}, we construct a regular Dirichlet form $(\calE_{\bar{X}},\calF_{\bar{X}})$ on $L^2(\bar{X};m)$ which is a regular representation of Dirichlet form $(\calE^\rref,\calF^\rref_a)$ on $L^2(X;m)$, where $\bar{X}$ is the Martin compactification of $X$. In section \ref{SG_det_sec_trace}, we take trace of the regular Dirichlet form $(\calE_{\bar{X}},\calF_{\bar{X}})$ on $L^2(\bar{X};m)$ to $K$ to have a regular Dirichlet form $(\calE_K,\calF_K)$ on $L^2(K;\nu)$ with the form (\ref{eqn_nonlocal}). In section \ref{SG_det_sec_trivial}, we show that $\calF_K$ consists only of constant functions if $\lambda\in(0,1/5)$ or $\beta\in(\beta^*,+\infty)$. Hence $\beta_*=\beta^*=\log5/\log2$.

\section{The SG and the Sierpi\'nski Graph}\label{SG_det_sec_SG}

In this section, we review some basic constructions of the SG and the Sierpi\'nski graph.

Let
$$p_0=(0,0),p_1=(1,0),p_2=(\frac{1}{2},\frac{\sqrt{3}}{2}),$$
$$f_i(x)=\frac{1}{2}(x+p_i),x\in\R^2,i=0,1,2.$$
Then the SG is the unique nonempty compact set $K$ satisfying
$$K=f_0(K)\cup f_1(K)\cup f_2(K).$$
Let
$$V_1=\myset{p_0,p_1,p_2},V_{n+1}=f_0(V_n)\cup f_1(V_n)\cup f_2(V_n)\text{ for all }n\ge1,$$
then $\myset{V_n}$ is an increasing sequence of finite sets such that $K$ is the closure of $\cup_{n=1}^\infty V_n$.

Let $W_0=\myset{\emptyset}$ and
$$W_n=\myset{w=w_1\ldots w_n:w_i=0,1,2,i=1,\ldots,n}\text{ for all }n\ge1,$$
and $W=\cup_{n=0}^\infty W_n$. An element $w=w_1\ldots w_n\in W_n$ is called a finite word with length $n$ and we denote $|w|=n$ for all $n\ge1$. $\emptyset\in W_0$ is called empty word and we denote its length $|\emptyset|=0$, we use the convention that zero length word is empty word. An element in $W$ is called a finite word.

Let
$$W_\infty=\myset{w=w_1w_2\ldots:w_i=0,1,2,i=1,2,\ldots}$$
be the set of all infinite sequences with elements in $\myset{0,1,2}$, then an element $w\in W_\infty$ is called an infinite word. For all $w=w_1\ldots w_n\in W$ with $n\ge1$, we write
$$f_w=f_{w_1}\circ\ldots\circ f_{w_n}$$
and $f_{\emptyset}=\mathrm{id}$. It is obvious that $K_w=f_w(K)$ is a compact set for all $w\in W$. For all $w=w_1w_2\ldots\in W_\infty$, we write
$$K_w=\bigcap_{n=0}^\infty K_{w_1\ldots w_n}.$$
Since $K_{w_1\ldots w_{n+1}}\subseteq K_{w_1\ldots w_n}$ for all $n\ge0$ and $\mathrm{diam}(K_{w_1\ldots w_n})\to0$ as $n\to+\infty$, we have $K_w\subseteq K$ is a one-point set. On the other hand, for all $x\in K$, there exists $w\in W_\infty$ such that $\myset{x}=K_w$. But this $w$ in not unique. For example, for the midpoint $x$ of the segment connecting $p_0$ and $p_1$, we have $\myset{x}=K_{100\ldots}=K_{011\ldots}$, where $100\ldots$ is the element $w=w_1w_2\ldots\in W_\infty$ with $w_1=1,w_n=0$ for all $n\ge2$ and $011\ldots$ has similar meaning.

By representation of infinite words, we construct the Sierpi\'nski graph as follows. First, we construct a triple tree. Take the root $o$ as the empty word $\emptyset$. It has three child nodes, that is, the words in $W_1$, $0,1,2$. Then the nodes $0,1,2$ have child nodes, that is, the words in $W_2$, $0$ has child nodes $00,01,02$, $1$ has child nodes $10,11,12$, $2$ has child nodes $20,21,22$. In general, each node $w_1\ldots w_n$ has three child nodes in $W_{n+1}$, that is, $w_1\ldots w_n0,w_1\ldots w_n1,w_1\ldots w_n2$ for all $n\ge1$. We use node and finite word interchangeable hereafter. For all $n\ge1$ and node $w=w_1\ldots w_n$, the node $w_1\ldots w_{n-1}$ is called the father node of $w$ and denoted by $w^-$. We obtain vertex set $V$ consisting of all nodes. Next, we construct edge set $E$, a subset of $V\times V$. Let
\begin{align*}
E_v&=\myset{(w,w^-),(w^-,w):w\in W_n,n\ge1},\\
E_h&=\myset{(w_1,w_2):w_1,w_2\in W_n,w_1\ne w_2,K_{w_1}\cap K_{w_2}\ne\emptyset,n\ge1},
\end{align*}
and $E=E_v\cup E_h$. $E_v$ is the set of all vertical edges and $E_h$ is the set of all horizontal edges. Then $X=(V,E)$ is the Sierpi\'nski graph, see Figure \ref{SG_det_fig_Sierpinski_graph}. We write $X$ for simplicity.



\begin{figure}[ht]
\centering
\begin{tikzpicture}
\draw (0,0)--(4,0);
\draw (0,0)--(2,3.4641016151);
\draw (4,0)--(2,3.4641016151);

\draw (2,3.4641016151)--(2,-1.3333333333);
\draw (4,0)--(2,-1.3333333333);
\draw (0,0)--(2,-1.3333333333);

\draw (1,1.7320508076)--(3,1.7320508076);
\draw (2,1.7320508076-0.6666666667)--(3,1.7320508076);
\draw (1,1.7320508076)--(2,1.7320508076-0.6666666667);

\draw (1,1.7320508076)--(1.3333333333,0);
\draw (1,1.7320508076)--(0.6666666666,-0.4444444444);
\draw (0.6666666666,-0.4444444444)--(1.3333333333,0);

\draw (2,1.7320508076-0.6666666667)--(1.3333333333,-0.8888888888);
\draw (2,1.7320508076-0.6666666667)--(2.6666666666,-0.8888888888);
\draw (2.6666666666,-0.8888888888)--(1.3333333333,-0.8888888888);

\draw (3,1.7320508076)--(3.3333333333,-0.4444444444);
\draw (3,1.7320508076)--(2.6666666666,0);
\draw (2.6666666666,0)--(3.3333333333,-0.4444444444);

\draw[fill=black] (0,0) circle (0.06);
\draw[fill=black] (4,0) circle (0.06);
\draw[fill=black] (2,3.4641016151) circle (0.06);
\draw[fill=black] (2,-1.3333333333) circle (0.06);
\draw[fill=black] (1,1.7320508076) circle (0.06);
\draw[fill=black] (2,1.7320508076-0.6666666667) circle (0.06);
\draw[fill=black] (3,1.7320508076) circle (0.06);
\draw[fill=black] (1.3333333333,0) circle (0.06);
\draw[fill=black] (0.6666666666,-0.4444444444) circle (0.06);
\draw[fill=black] (1.3333333333,-0.8888888888) circle (0.06);
\draw[fill=black] (2.6666666666,-0.8888888888) circle (0.06);
\draw[fill=black] (3.3333333333,-0.4444444444) circle (0.06);
\draw[fill=black] (2.6666666666,0) circle (0.06);

\draw (2,3.8) node {$\emptyset$};
\draw (0.9,2) node {$0$};
\draw (2.1,1.4) node {$1$};
\draw (3.1,2) node {$2$};
\draw (-0.1,-0.3) node {$00$};
\draw (4.1,-0.3) node {$22$};
\draw (1.3,-0.3) node {$02$};
\draw (2.7,-0.3) node {$20$};
\draw (0.6,-0.75) node {$01$};
\draw (3.4,-0.75) node {$21$};
\draw (1.2,-1.2) node {$10$};
\draw (2.8,-1.2) node {$12$};
\draw (2,-1.6) node {$11$};
\end{tikzpicture}
\caption{The Sierpi\'nski graph}\label{SG_det_fig_Sierpinski_graph}
\end{figure}
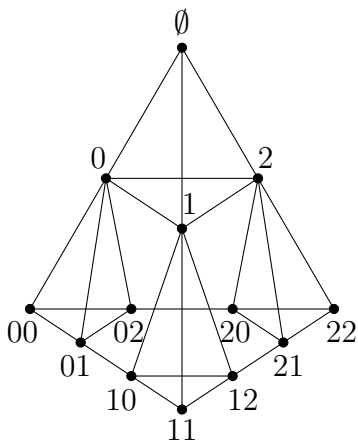


For all $x,y\in V$, if $(x,y)\in E$, then we write $x\sim y$ and say that $y$ is a neighbor of $x$. It is obvious that $\sim$ is an equivalence relation. A path in $X$ is a finite sequence $\pi=[x_0,\ldots,x_n]$ with distinct nodes and $x_0\sim x_1,\ldots,x_{n-1}\sim x_n$, $n$ is called the length of the path. For all $x,y\in V$, let $d(x,y)$ be the graph metric, that is, the minimum length of all paths connecting $x$ and $y$, if a path connecting $x$ and $y$ has length $d(x,y)$, then this path is called geodesic. Hereafter, we write $x\in X$ to mean that $x\in V$. It is obvious that $X$ is a connected and locally finite graph, that is, for all $x,y\in X$ with $x\ne y$, there exists a path connecting $x$ and $y$, for all $x\in X$, the set of its neighbors $\myset{y\in X:x\sim y}$ is a finite set. We write $S_n=\myset{x\in X:|x|=n}$, $B_n=\cup_{i=0}^nS_i$ as sphere and closed ball with radius $n$.

Roughly speaking, for all $n\ge1$, $S_n$ looks like some disconnected triangles, see Figure \ref{SG_det_fig_SG_S3} for $S_3$, and $V_n$ looks like some connected triangles, see Figure \ref{SG_det_fig_SG_V3} for $V_3$. We define a mapping $\Phi_n:S_n\to V_n$ as follows. For all $n\ge2$, $w=w_1\ldots w_n\in W_n$, write $p_{w}=p_{w_1\ldots w_n}=f_{w_1\ldots w_{n-1}}(p_{w_n})$. Write $p_1,p_2,p_3$ for $n=1$ and $w=0,1,2$, respectively. By induction, we have $V_n=\cup_{w\in W_n}p_w$ for all $n\ge1$. Define $\Phi_n(w)=p_w$. Then $\Phi_n$ is onto and many pairs of points are mapped into same points, such as $\Phi_3(001)=\Phi_3(010)$. This property can divide the edges in $S_n$ into two types. For an arbitrary edge in $S_n$ with end nodes $x,y$, it is called of type \Rmnum{1} if $\Phi_n(x)\ne\Phi_n(y)$ such as the edge in $S_3$ with end nodes $000$ and $001$, it is called of type \Rmnum{2} if $\Phi_n(x)=\Phi_n(y)$ such as the edge in $S_3$ with end nodes $001$ and $010$. By induction, it is obvious there exist only these two types of edges on each sphere $S_n$.




\begin{figure}[ht]
  \centering
  \begin{tikzpicture}[scale=0.7]
  \tikzstyle{every node}=[font=\small,scale=0.7]
  \draw (0,0)--(9,0);
  \draw (0,0)--(9/2,9/2*1.7320508076);
  \draw (9,0)--(9/2,9/2*1.7320508076);
  
  \draw (3,0)--(3/2,3/2*1.7320508076);
  \draw (6,0)--(15/2,3/2*1.7320508076);
  \draw (3,3*1.7320508076)--(6,3*1.7320508076);
  
  \draw (4,4*1.7320508076)--(5,4*1.7320508076);
  \draw (7/2,7/2*1.7320508076)--(4,3*1.7320508076);
  \draw (11/2,7/2*1.7320508076)--(5,3*1.7320508076);
  
  \draw (1,1*1.7320508076)--(2,1.7320508076);
  \draw (1/2,1/2*1.7320508076)--(1,0);
  \draw (2,0)--(5/2,1/2*1.7320508076);
  
  \draw (7,1.7320508076)--(8,1.7320508076);
  \draw (13/2,1.7320508076/2)--(7,0);
  \draw (17/2,1.7320508076/2)--(8,0);
  
  \draw[fill=black] (0,0) circle (0.06);
  \draw[fill=black] (1,0) circle (0.06);
  \draw[fill=black] (2,0) circle (0.06);
  \draw[fill=black] (3,0) circle (0.06);
  \draw[fill=black] (6,0) circle (0.06);
  \draw[fill=black] (7,0) circle (0.06);
  \draw[fill=black] (8,0) circle (0.06);
  \draw[fill=black] (9,0) circle (0.06);
  \draw[fill=black] (1/2,1.7320508076/2) circle (0.06);
  \draw[fill=black] (5/2,1.7320508076/2) circle (0.06);
  \draw[fill=black] (13/2,1.7320508076/2) circle (0.06);
  \draw[fill=black] (17/2,1.7320508076/2) circle (0.06);
  \draw[fill=black] (1,1.7320508076) circle (0.06);
  \draw[fill=black] (2,1.7320508076) circle (0.06);
  \draw[fill=black] (7,1.7320508076) circle (0.06);
  \draw[fill=black] (8,1.7320508076) circle (0.06);
  \draw[fill=black] (3/2,3/2*1.7320508076) circle (0.06);
  \draw[fill=black] (15/2,3/2*1.7320508076) circle (0.06);
  \draw[fill=black] (3,3*1.7320508076) circle (0.06);
  \draw[fill=black] (4,3*1.7320508076) circle (0.06);
  \draw[fill=black] (5,3*1.7320508076) circle (0.06);
  \draw[fill=black] (6,3*1.7320508076) circle (0.06);
  \draw[fill=black] (7/2,7/2*1.7320508076) circle (0.06);
  \draw[fill=black] (11/2,7/2*1.7320508076) circle (0.06);
  \draw[fill=black] (4,4*1.7320508076) circle (0.06);
  \draw[fill=black] (5,4*1.7320508076) circle (0.06);
  \draw[fill=black] (9/2,9/2*1.7320508076) circle (0.06);
  
  \draw (0,-0.3) node {$000$};
  \draw (1,-0.3) node {$001$};
  \draw (2,-0.3) node {$010$};
  \draw (3,-0.3) node {$011$};
  \draw (6,-0.3) node {$100$};
  \draw (7,-0.3) node {$101$};
  \draw (8,-0.3) node {$110$};
  \draw (9,-0.3) node {$111$};
  
  \draw (0.1,1/2*1.7320508076) node {$002$};
  \draw (2.9,1/2*1.7320508076) node {$012$};
  \draw (6.1,1/2*1.7320508076) node {$102$};
  \draw (8.9,1/2*1.7320508076) node {$112$};
  
  \draw (0.6,1*1.7320508076) node {$020$};
  \draw (2.4,1*1.7320508076) node {$021$};
  \draw (6.6,1*1.7320508076) node {$120$};
  \draw (8.4,1*1.7320508076) node {$121$};
  
  \draw (1.1,3/2*1.7320508076) node {$022$};
  \draw (7.9,3/2*1.7320508076) node {$122$};
  
  \draw (2.6,3*1.7320508076) node {$200$};
  \draw (4,3*1.7320508076-0.3) node {$201$};
  \draw (5,3*1.7320508076-0.3) node {$210$};
  \draw (6.4,3*1.7320508076) node {$211$};
  
  \draw (3.1,7/2*1.7320508076) node {$202$};
  \draw (5.9,7/2*1.7320508076) node {$212$};
  
  \draw (3.6,4*1.7320508076) node {$220$};
  \draw (5.4,4*1.7320508076) node {$221$};
  
  \draw (4.5,9/2*1.7320508076+0.3) node {$222$};
  
  \end{tikzpicture}
  \caption{$S_3$}\label{SG_det_fig_SG_S3}
\end{figure}






\begin{figure}[ht]
  \centering
  \begin{tikzpicture}[scale=1.125*0.7]
  \tikzstyle{every node}=[font=\small,scale=0.7]
  \draw (0,0)--(8,0);
  \draw (0,0)--(4,4*1.7320508076);
  \draw (8,0)--(4,4*1.7320508076);
  
  \draw (4,0)--(2,2*1.7320508076);
  \draw (2,2*1.7320508076)--(6,2*1.7320508076);
  \draw (6,2*1.7320508076)--(4,0);
  
  \draw (2,0)--(3,1*1.7320508076);
  \draw (3,1*1.7320508076)--(1,1*1.7320508076);
  \draw (1,1*1.7320508076)--(2,0);
  
  \draw (6,0)--(5,1*1.7320508076);
  \draw (5,1*1.7320508076)--(7,1*1.7320508076);
  \draw (7,1*1.7320508076)--(6,0);
  
  \draw (4,2*1.7320508076)--(3,3*1.7320508076);
  \draw (3,3*1.7320508076)--(5,3*1.7320508076);
  \draw (5,3*1.7320508076)--(4,2*1.7320508076);
  
  \draw[fill=black] (0,0) circle (0.06);
  \draw[fill=black] (2,0) circle (0.06);
  \draw[fill=black] (4,0) circle (0.06);
  \draw[fill=black] (6,0) circle (0.06);
  \draw[fill=black] (8,0) circle (0.06);
  \draw[fill=black] (1,1.7320508076) circle (0.06);
  \draw[fill=black] (3,1.7320508076) circle (0.06);
  \draw[fill=black] (5,1.7320508076) circle (0.06);
  \draw[fill=black] (7,1.7320508076) circle (0.06);
  \draw[fill=black] (2,2*1.7320508076) circle (0.06);
  \draw[fill=black] (4,2*1.7320508076) circle (0.06);
  \draw[fill=black] (6,2*1.7320508076) circle (0.06);
  \draw[fill=black] (3,3*1.7320508076) circle (0.06);
  \draw[fill=black] (5,3*1.7320508076) circle (0.06);
  \draw[fill=black] (4,4*1.7320508076) circle (0.06);
  
  \draw (0,-0.3) node {$p_{000}$};
  \draw (2,-0.3) node {$p_{001}=p_{010}$};
  \draw (4,-0.3) node {$p_{011}=p_{100}$};
  \draw (6,-0.3) node {$p_{101}=p_{110}$};
  \draw (8,-0.3) node {$p_{111}$};
  \draw (4,4*1.7320508076+0.3) node {$p_{222}$};
  
  \draw (0,1*1.7320508076) node {$p_{002}=p_{020}$};
  \draw (8,1*1.7320508076) node {$p_{112}=p_{121}$};
  \draw (2.95,1*1.7320508076+0.3) node {$p_{012}=p_{021}$};
  \draw (5.05,1*1.7320508076+0.3) node {$p_{102}=p_{120}$};
  
  \draw (1,2*1.7320508076) node {$p_{022}=p_{200}$};
  \draw (7,2*1.7320508076) node {$p_{122}=p_{211}$};
  
  \draw (4,2*1.7320508076-0.3) node {$p_{201}=p_{210}$};
  
  \draw (2,3*1.7320508076) node {$p_{202}=p_{220}$};
  \draw (6,3*1.7320508076) node {$p_{212}=p_{221}$};
  \end{tikzpicture}
  \caption{$V_3$}\label{SG_det_fig_SG_V3}
\end{figure}



The Sierpi\'nski graph is a hyperbolic graph, see \cite[Theorem 3.2]{LW09}. For arbitrary graph $X$, choose a node $o$ as root, define graph metric $d$ as above, write $|x|=d(o,x)$. For all $x,y\in X$, define Gromov product
$$|x\wedge y|=\frac{1}{2}(|x|+|y|-d(x,y)).$$
$X$ is called a hyperbolic graph if there exists $\delta>0$ such that for all $x,y,z\in X$, we have
$$|x\wedge y|\ge\mathrm{min}{\myset{|x\wedge z|,|z\wedge y|}}-\delta.$$
It is known that the definition is independent of the choice of root $o$. For a hyperbolic graph, we can introduce a metric as follows. Choose $a>0$ such that $a'=e^{3\delta a}-1<\sqrt{2}-1$. For all $x,y\in X$, define
$$
\rho_a(x,y)=
\begin{cases}
\exp{(-a|x\wedge y|)},&\text{if }x\ne y,\\
0,&\text{if }x=y,
\end{cases}
$$
then $\rho_a$ satisfies
$$\rho_a(x,y)\le(1+a')\max{\myset{\rho_a(x,z),\rho_a(z,y)}}\text{ for all }x,y,z\in X.$$
This means $\rho_a$ is an ultra-metric not a metric. But we can define
$$\theta_a(x,y)=\inf{\myset{\sum_{i=1}^n\rho_a(x_{i-1},x_i):x=x_0,\ldots,x_n=y,x_i\in X,i=0,\ldots,n,n\ge1}},$$
for all $x,y\in X$. $\theta_a$ is a metric and equivalent to $\rho_a$. So we use $\rho_a$ rather than $\theta_a$ for simplicity. It is known that a sequence $\myset{x_n}\subseteq X$ with $|x_n|\to+\infty$ is a Cauchy sequence in $\rho_a$ if and only if $|x_m\wedge x_n|\to+\infty$ as $m,n\to+\infty$. Let $\hat{X}$ be the completion of $X$ with respect to $\rho_a$, then $\dd_hX=\hat{X}\backslash X$ is called the hyperbolic boundary of $X$. By \cite[Corollary 22.13]{Wo00}, $\hat{X}$ is compact. It is obvious that hyperbolicity is only related to the graph structure of $X$. We introduce a description of hyperbolic boundary in terms of geodesic rays. A geodesic ray is a sequence $[x_0,x_1,\ldots]$ with distinct nodes, $x_n\sim x_{n+1}$ and path $[x_0,\ldots,x_n]$ is geodesic for all $n\ge0$. Two geodesic rays $\pi=[x_0,x_1,\ldots]$ and $\pi'=[y_0,y_1,\ldots]$ are called equivalent if $\varliminf_{n\to+\infty}d(y_n,\pi)<+\infty$, where $d(x,\pi)=\inf_{n\ge0}d(x,x_n)$. There exists a one-to-one correspondence between the family of all equivalent geodesic rays and hyperbolic boundary as follows.

By \cite[Proposition 22.12(b)]{Wo00}, equivalence geodesic rays is an equivalence relation. By \cite[Lemma 22.11]{Wo00}, for all geodesic ray $\pi=[x_0,x_1,\ldots]$, for all $u\in X$, there exist $k,l\ge0$, $u=u_0,\ldots,u_k=x_l$, such that
$$[u,u_1,\ldots,u_k,x_{l+1},x_{l+2},\ldots]$$
is a geodesic ray. It is obvious that this new geodesic ray is equivalent to $\pi$, hence we can take a geodesic ray in each equivalence class of the form $\pi=[x_0,x_1,\ldots]$, $|x_n|=n$, $x_n\sim x_{n+1}$ for all $n\ge0$. By \cite[Proposition 22.12(c)]{Wo00}, we can define a one-to-one mapping $\tau$ from the family of all equivalent geodesic rays to hyperbolic boundary,
$$\tau:[x_0,x_1,\ldots]\mapsto\text{the limit }\xi\text{ of }\myset{x_n}\text{ in }\rho_a.$$
By above, we can choose $[x_0,x_1,\ldots]$ of the form $|x_n|=n$, $x_n\sim x_{n+1}$ for all $n\ge0$, we say that $[x_0,x_1,\ldots]$ is a geodesic ray from $o$ to $\xi$.

For $y\in\hat{X}$, $x\in X$, we say that $y$ is in the subtree with root $x$ if $x$ lies on the geodesic path or some geodesic ray from $o$ to $y$. And if $y$ is in the subtree with root $x$, then it is obvious that $|x\wedge y|=|x|$, $\rho_a(x,y)=e^{-a|x|}$ if $x\ne y$. For more detailed discussion of hyperbolic graph, see \cite[Chapter \Rmnum{4}, \Rmnum{4}.22]{Wo00}.

\cite[Theorem 3.2, Theorem 4.3, Proposition 4.4]{LW09} showed that for a general class of fractals satisfying open set condition (OSC), we can construct an augmented rooted tree which is a hyperbolic graph and the hyperbolic boundary is H\"older equivalent to the fractal through canonical mapping. In particular, the SG satisfies OSC, the Sierpi\'nski graph is an augmented rooted tree hence hyperbolic. The canonical mapping $\Phi$ can be described as follows.

For all $\xi\in\dd_hX$, there corresponds a geodesic ray in the equivalence class corresponding to $\xi$ through the mapping $\tau$ of the form $[x_0,x_1,\ldots]$ with $|x_n|=n$ and $x_n\sim x_{n+1}$ for all $n\ge0$, then there exists an element $w\in W_\infty$ such that $w_1\ldots w_n=x_n$ for all $n\ge1$. Then $\myset{\Phi(\xi)}=K_w$ and
\begin{equation}\label{SG_det_eqn_Holder}
\lvert\Phi(\xi)-\Phi(\eta)\rvert\asymp\rho_a(\xi,\eta)^{\log2/a}\text{ for all }\xi,\eta\in\dd_hX.
\end{equation}

\section{Random Walk on \texorpdfstring{$X$}{X}}\label{SG_det_sec_rw}

In this section, we give a transient reversible random walk on $X$ from \cite{KLW17}. Let $c:X\times X\to[0,+\infty)$ be conductance satisfying
\begin{align*}
c(x,y)&=c(y,x),\\
\pi(x)&=\sum_{y\in X}c(x,y)\in(0,+\infty),\\
c(x,y)&>0\text{ if and only if }x\sim y,
\end{align*}
for all $x,y\in X$. Let $P(x,y)=c(x,y)/\pi(x)$, $x,y\in X$, then $P$ is a transition probability satisfying
$$\pi(x)P(x,y)=\pi(y)P(y,x)\text{ for all }x,y\in X.$$
We construct a reversible random walk $Z=\myset{Z_n}$ on $X$ with transition probability $P$. We introduce some related quantities. For all $x,y\in X$, let $P^{(0)}(x,y)=\delta_{xy}$ and
$$P^{(n+1)}(x,y)=\sum_{z\in X}P(x,z)P^{(n)}(z,y)\text{ for all }n\ge0.$$
Define
$$G(x,y)=\sum_{n=0}^\infty P^{(n)}(x,y),x,y\in X,$$
then $G$ is the Green function of $Z$ and $Z$ is called transient if $G(x,y)<+\infty$ for all or equivalently for some $x,y\in X$. Define
$$F(x,y)=\bbP_x\left[Z_n=y\text{ for some }n\ge0\right],$$
that is, the probability of ever reaching $y$ starting from $x$. By Markovian property, we have
$$G(x,y)=F(x,y)G(y,y).$$
For more detailed discussion of general theory of random walk, see \cite[Chapter \Rmnum{1}, \Rmnum{1}.1]{Wo00}.

Here, we take some specific random walk called $\lambda$-return ratio random walk introduced in \cite{KLW17}, that is,
$$\frac{c(x,x^-)}{\sum_{y:y^-=x}c(x,y)}=\frac{P(x,x^-)}{\sum_{y:y^-=x}P(x,y)}=\lambda\in(0,+\infty)\text{ for all }x\in X\text{ with }|x|\ge1.$$
For all $n\ge0$, $x\in S_n,y\in S_{n+1}$, we take $c(x,y)$ the same value denoted by $c(n,n+1)=c(n+1,n)$. Then
$$\lambda=\frac{c(n-1,n)}{3c(n,n+1)},$$
that is,
$$c(n,n+1)=\frac{c(n-1,n)}{3\lambda}=\ldots=\frac{1}{(3\lambda)^n}{c(0,1)}.$$
Take $c(0,1)=1$, then $c(n,n+1)=1/(3\lambda)^n$. Moreover, \cite[Definition 4.4]{KLW17} gave restrictions to conductance of horizontal edges. For all $n\ge1$, $x,y\in S_n$, $x\sim y$, let
$$
c(x,y)=\\
\begin{cases}
\frac{C_1}{(3\lambda)^n},&\text{ if the edge with end nodes }x,y\text{ is of type \Rmnum{1}},\\
\frac{C_2}{(3\lambda)^n},&\text{ if the edge with end nodes }x,y\text{ is of type \Rmnum{2}},
\end{cases}
$$
where $C_1,C_2$ are some positive constants.

\cite[Proposition 4.1, Lemma 4.2]{KLW17} showed that if $\lambda\in(0,1)$, then $Z$ is transient and
\begin{equation}\label{SG_det_eqn_G}
G(o,o)=\frac{1}{1-\lambda},
\end{equation}
\begin{equation}\label{SG_det_eqn_F}
F(x,0)=\lambda^{|x|}\text{ for all }x\in X.
\end{equation}

For a transient random walk, we can introduce Martin kernel given by
$$K(x,y)=\frac{G(x,y)}{G(o,y)},$$
and Martin compactification $\bar{X}$, that is, the smallest compactification such that $K(x,\cdot)$ can be extended continuously for all $x\in X$. Martin boundary is given by $\dd_MX=\bar{X}\backslash X$. Then Martin kernel $K$ can be defined on $X\times\bar{X}$.

\cite[Theorem 5.1]{KLW17} showed that the Martin boundary $\dd_MX$, the hyperbolic boundary $\dd_hX$ and the SG $K$ are homeomorphic. Hence the completion $\hat{X}$ of $X$ with respect to $\rho_a$ and Martin compactification $\bar{X}$ are homeomorphic. It is always convenient to consider $\hat{X}$ rather than $\bar{X}$. We use $\dd X$ to denote all these boundaries. We list some general results of Martin boundary for later use.

\begin{mythm}\label{SG_det_thm_conv}(\cite[Theorem 24.10]{Wo00})
Let $Z$ be transient, then $\myset{Z_n}$ converges to a $\dd_MX$-valued random variable $Z_\infty$, $\bbP_x$-a.s. for all $x\in X$. The hitting distribution of $\myset{Z_n}$ or the distribution of $Z_\infty$ under $\bbP_x$, denoted by $\nu_x$, satisfies
$$\nu_x(B)=\int_BK(x,\cdot)\md\nu_o\text{ for all Borel measurable set }B\subseteq\dd_M X,$$
that is, $\nu_x$ is absolutely continuous with respect to $\nu_o$ with Radon-Nikodym derivative $K(x,\cdot)$. 
\end{mythm}

For all $\nu_o$-integrable function $\vphi$ on $\dd_MX$, we have
$$h(x)=\int_{\dd_MX}\vphi\md\nu_x=\int_{\dd_MX}K(x,\cdot)\vphi\md\nu_o,x\in X,$$
is a harmonic function on $X$. It is called the Poisson integral of $\vphi$, denoted by $H\vphi$.

\cite[Theorem 5.6]{KLW17} showed that the hitting distribution $\nu_o$ is the normalized Hausdorff measure on $K$. We write $\nu$ for $\nu_o$ for simplicity.

Using conductance $c$, we construct an energy on $X$ given by
$$\calE_X(u,u)=\frac{1}{2}\sum_{x,y\in X}c(x,y)(u(x)-u(y))^2.$$

In \cite{Sil74}, Silverstein constructed Na\"im kernel $\Theta$ on $\bar{X}\times\bar{X}$ using Martin kernel to induce an energy on $\dd X$ given by
$$\calE_{\dd X}(u,u)=\calE_X(Hu,Hu)=\frac{1}{2}\pi(o)\int_{\dd X}\int_{\dd X}(u(x)-u(y))^2\Theta(x,y)\nu(\md x)\nu(\md y),$$
for all $u\in L^2(\dd_MX;\nu)$ with $\calE_{\dd X}(u,u)<+\infty$.

\cite[Theorem 6.3]{KLW17} calculated Na\"im kernel forcefully
\begin{equation}\label{SG_det_eqn_Theta}
\Theta(x,y)\asymp\frac{1}{|x-y|^{\alpha+\beta}},
\end{equation}
where $\alpha=\log3/\log2$ is the Hausdorff dimension of the SG, $\beta=-\log\lambda/\log2\in(0,+\infty)$, $\lambda\in(0,1)$. No message of upper bound for $\beta$ of walk dimension appeared in their calculation.

\section{Regular Dirichlet Form on \texorpdfstring{$X$}{X}}\label{SG_det_sec_df_X}

In this section, we construct a regular Dirichlet form $\calE_X$ on $X$ and its corresponding symmetric Hunt process $\myset{X_t}$. We prove that the Martin boundaries of $\myset{X_t}$ and $Z$ coincide. We show that $\calE_X$ is stochastically incomplete and $\myset{X_t}$ goes to infinity in finite time almost surely.

Let $m:X\to(0,+\infty)$ be a positive function given by
$$m(x)=\left(\frac{c}{3\lambda}\right)^{|x|},x\in X,$$
where $c\in(0,\lambda)\subseteq(0,1)$. Then $m$ can be regarded as a measure on $X$. Note that
$$m(X)=\sum_{x\in X}m(x)=\sum_{n=0}^\infty3^n\cdot\left(\frac{c}{3\lambda}\right)^n=\sum_{n=0}^{\infty}\left(\frac{c}{\lambda}\right)^n<+\infty,$$
we have $m$ is a finite measure on $X$. We construct a symmetric form on $L^2(X;m)$ given by
$$
\begin{cases}
&\calE_X(u,u)=\frac{1}{2}\sum_{x,y\in X}c(x,y)(u(x)-u(y))^2,\\
&\calF_X=\text{the }(\calE_X)_1\text{-closure of }C_c(X),
\end{cases}
$$
where $C_c(X)$ is the set of all functions with finite support. It is obvious that $(\calE_X,\calF_X)$ is a regular Dirichlet form on $L^2(X;m)$. By \cite[Theorem 7.2.1]{FOT11}, it corresponds to a symmetric Hunt process on $X$. Roughly speaking, this process is a variable speed continuous time random walk characterized by holding at one node with time distributed to exponential distribution and jumping according to random walk. For some discussion of continuous time random walk, see \cite[Chapter 2]{Nor98}. We give detailed construction as follows.

Let $(\Omega,\calF,\bbP)$ be a probability space on which given a random walk $\myset{Y_n}$ with transition probability $P$ and initial distribution $\sigma$ and a sequence of independent exponential distributed random variables $\myset{S_n}$ with parameter $1$, that is, $\bbP[S_n\in\md t]=e^{-t}\md t$. Assume that $\myset{S_n}$ is independent of $\myset{Y_n}$. Let $\alpha(x)=\pi(x)/m(x)$, $x\in X$. For all $n\ge1$, let
$$T_n=\frac{S_n}{\alpha(Y_{n-1})},$$
$$J_0=0,J_n=T_1+\ldots+T_n.$$
Then $T_n$ is called the $n$-th holding time and $J_n$ is called the $n$-th jumping time. Let
$$
X_t=
\begin{cases}
Y_n,&\text{if }J_n\le t<J_{n+1}\text{ for some }n\ge0,\\
\dd,&\text{otherwise},
\end{cases}
$$
where $\dd$ is a death point. This construction is similar to that of Poisson process and it is called variable speed continuous time random walk in some literature. $\myset{X_t}$ is a symmetric Hunt process with initial distribution $\sigma$. We claim that $\myset{X_t}$ is the symmetric Hunt process corresponding to $\calE_X$.

Indeed, we only need to show that their generators coincide. By \cite[Corollary 1.3.1]{FOT11}, the generator of $\calE_X$ is characterized by
$$\calE_X(u,v)=(-Au,v)\text{ for all }u,v\in C_c(X).$$
Noting that
\begin{align*}
\calE_X(u,v)&=\frac{1}{2}\sum_{x,y\in X}c(x,y)(u(x)-u(y))(v(x)-v(y))\\
&=\sum_{x\in X}\left(\frac{1}{m(x)}\sum_{y\in X}c(x,y)(u(x)-u(y))\right)v(x)m(x),
\end{align*}
we have
$$Au(x)=\frac{1}{m(x)}\sum_{y\in X}c(x,y)(u(y)-u(x))\text{ for all }u\in C_c(X).$$
On the other hand, the generator of $\myset{X_t}$ is characterized by $\lim_{t\downarrow0}\frac{1}{t}\left(\bbE_xu(X_t)-u(x)\right)$ for all $u\in C_c(X)$.

Since $X$ is locally finite, we have
\begin{align*}
&\lim_{t\downarrow0}\frac{1}{t}\left(\bbE_xu(X_t)-u(x)\right)=\lim_{t\downarrow0}\frac{1}{t}\sum_{y\in X}(u(y)-u(x))\bbP_x[X_t=y]\\
=&\lim_{t\downarrow0}\frac{1}{t}\sum_{y:y\ne x}(u(y)-u(x))\bbP_x[X_t=y]=\sum_{y:y\ne x}(u(y)-u(x))\lim_{t\downarrow0}\frac{1}{t}\bbP_x[X_t=y].
\end{align*}
Let $p_{xy}(t)=\bbP_x[X_t=y]$, then for all $x\ne y$, we have $p_{xy}(0)=0$ and
$$\lim_{t\downarrow0}\frac{1}{t}\bbP_x[X_t=y]=\lim_{t\downarrow0}\frac{1}{t}p_{xy}(t)=p_{xy}'(0)$$
if the derivative exists. We calculate some equation that $p_{xy}$ satisfies. The idea of the following calculation is from \cite[Theorem 2.8.4]{Nor98}.

Note that
$$p_{xy}(t)=\bbP_x[X_t=y]=\bbP_x[X_t=y,t<J_1]+\bbP_x[X_t=y,t\ge J_1],$$
where
\begin{align*}
\bbP_x[X_t=y,t<J_1]&=\bbP_x[Y_0=y,t<T_1]=\bbP_x[Y_0=y]\bbP_x[t<\frac{S_1}{\alpha(x)}]=\delta_{xy}e^{-\alpha(x)t},\\
\bbP_x[X_t=y,t\ge J_1]&=\sum_{z\in X}\bbP_x[X_t=y,t\ge J_1,Y_1=z]\\
&=\sum_{z\in X}\int_0^t\bbP_x[Y_1=z]\bbP_x[J_1\in\md s]\bbP_z[X_{t-s}=y]\\
&=\sum_{z\in X}\int_0^tP(x,z)\alpha(x)e^{-\alpha(x)s}p_{zy}(t-s)\md s.
\end{align*}
Hence
\begin{equation}\label{SG_det_eqn_pxy}
p_{xy}(t)=\delta_{xy}e^{-\alpha(x)t}+\sum_{z\in X}\int_0^tP(x,z)\alpha(x)e^{-\alpha(x)s}p_{zy}(t-s)\md s.
\end{equation}
Since $X$ is locally finite and $p_{xy}\in[0,1]$, we have $p_{xy}$ is continuous, $p_{xy}$ is continuous differentiable,\ldots, $p_{xy}$ is infinitely differentiable. Note that Equation (\ref{SG_det_eqn_pxy}) is equivalent to
$$e^{\alpha(x)t}p_{xy}(t)=\delta_{xy}+\sum_{z\in X}\alpha(x)P(x,z)\int_0^te^{\alpha(x)s}p_{zy}(s)\md s.$$
Differentiating both sides with respect to $t$, we have
$$e^{\alpha(x)t}\left(\alpha(x)p_{xy}(t)+p_{xy}'(t)\right)=\sum_{z\in X}\alpha(x)P(x,z)e^{\alpha(x)t}p_{zy}(t),$$
that is,
$$p_{xy}'(t)=\sum_{z\in X}\alpha(x)P(x,z)p_{zy}(t)-\alpha(x)p_{xy}(t).$$
Letting $x\ne y$ and $t=0$, we have
$$p_{xy}'(0)=\alpha(x)P(x,y)=\frac{\pi(x)}{m(x)}\frac{c(x,y)}{\pi(x)}=\frac{c(x,y)}{m(x)},$$
hence
\begin{align*}
&\lim_{t\downarrow0}\frac{1}{t}\left(\bbE_xu(X_t)-u(x)\right)=\sum_{y:y\ne x}\left(u(y)-u(x)\right)p_{xy}'(0)\\
=&\frac{1}{m(x)}\sum_{y:y\ne x}c(x,y)\left(u(y)-u(x)\right)=\frac{1}{m(x)}\sum_{y\in X}c(x,y)\left(u(y)-u(x)\right),
\end{align*}
which coincides with $Au(x)$.

By the construction of $\myset{X_t}$ in terms of $\myset{Y_n}$, the Martin boundary of $\myset{X_t}$ is the same as the Martin boundary of $Z$.

Indeed, we calculate the Green function of $\myset{X_t}$ explicitly. By the correspondence between $\calE_X$ and $\myset{X_t}$, we only need to calculate the Green function of $\calE_X$. By \cite[Theorem 1.5.4]{FOT11}, the Green operator $G$ is characterized by $\calE_X(Gu,v)=(u,v)$ for all $u,v\in C_c(X)$. Note that
$$Gu(x)=\int_XG(x,\md y)u(y)=\sum_{y\in X}g(x,y)u(y)m(y)$$
where $g$ is the Green function. Taking $u=\indi_{x_0}$, $v=\indi_{y_0}$, $x_0,y_0\in X$, then we have
$$(u,v)=\sum_{x\in X}u(x)v(x)m(x)=\delta_{x_0y_0}m(x_0),$$
$$Gu(x)=g(x,x_0)m(x_0),$$
and
\begin{align*}
\calE_X(Gu,v)&=\frac{1}{2}\sum_{x,y\in X}c(x,y)(Gu(x)-Gu(y))(v(x)-v(y))\\
&=\frac{1}{2}\sum_{x,y\in X}c(x,y)(g(x,x_0)-g(y,x_0))m(x_0)(v(x)-v(y))\\
&=\sum_{x,y\in X}c(x,y)(g(x,x_0)-g(y,x_0))m(x_0)v(x)\\
&=\sum_{y\in X}c(y_0,y)(g(y_0,x_0)-g(y,x_0))m(x_0).
\end{align*}
Letting $g(y,x_0)=G(y,x_0)/C(x_0)$, where $C$ is some function to be determined, we have
\begin{align*}
\sum_{y\in X}c(y_0,y)(g(y_0,x_0)-g(y,x_0))m(x_0)&=\sum_{y\in X}\frac{\pi(y_0)}{C(x_0)}P(y_0,y)(G(y_0,x_0)-G(y,x_0))m(x_0)\\
&=\frac{\pi(y_0)}{C(x_0)}(G(y_0,x_0)-G(y_0,x_0)+\delta_{x_0y_0})m(x_0)\\
&=\frac{\pi(y_0)}{C(x_0)}\delta_{x_0y_0}m(x_0)\\
&=\delta_{x_0y_0}m(x_0),
\end{align*}
hence $C(x_0)=\pi(x_0)$ and
$$g(x,y)=\frac{G(x,y)}{\pi(y)}.$$
Hence the Martin kernel of $\myset{X_t}$ is given by
$$k(x,y)=\frac{g(x,y)}{g(o,y)}=\frac{G(x,y)/\pi(y)}{G(o,y)/\pi(y)}=\frac{G(x,y)}{G(o,y)}=K(x,y)\text{ for all }x,y\in X.$$
Hence the Martin boundaries of $\myset{X_t}$ and $Z$ coincide. Moreover, $\calE_X$ is transient.

\begin{mythm}\label{SG_det_thm_sto}
$(\calE_X,\calF_X)$ on $L^2(X;m)$ is stochastically incomplete.
\end{mythm}

We prove stochastic incompleteness by considering the lifetime
$$\zeta=\sum_{n=1}^\infty T_n=\lim_{n\to+\infty}J_n.$$
This quantity is called the (first) explosion time in \cite[Chapter 2, 2.2]{Nor98}. We need a proposition for preparation.

\begin{myprop}\label{SG_det_prop_jump}
The jumping times $J_n$ are stopping times of $\myset{X_t}$ for all $n\ge0$.
\end{myprop}

\begin{proof}
Let $\myset{\calF_t}$ be the minimum completed admissible filtration with respect to $\myset{X_t}$. It is obvious that $J_0=0$ is a stopping time of $\myset{X_t}$. Assume that $J_n$ is a stopping time of $\myset{X_t}$, then for all $t\ge0$, we have
$$\myset{J_{n+1}\le t}=\cup_{s\in\mathbb{Q},s\le t}\left(\myset{J_n\le t}\cap\myset{X_s\ne X_{J_n}}\right)\in\calF_t,$$
hence $J_{n+1}$ is a stopping time of $\myset{X_t}$. By induction, it follows that $J_n$ are stopping times of $\myset{X_t}$ for all $n\ge0$.
\end{proof}

\begin{proof}[Proof of Theorem \ref{SG_det_thm_sto}]
By Equation (\ref{SG_det_eqn_F}), we have
\begin{align*}
\bbE_o\zeta&=\bbE_o\sum_{n=1}^\infty T_n=\sum_{n=1}^\infty\bbE_o\left[\frac{S_n}{\alpha(Y_{n-1})}\right]=\sum_{n=1}^\infty\bbE_o[{S_n}]\bbE_o\left[\frac{1}{\alpha(Y_{n-1})}\right]\\
&=\sum_{n=1}^\infty\bbE_o\frac{m}{\pi}(Y_{n-1})=\sum_{n=0}^\infty\bbE_o\frac{m}{\pi}(Y_{n})=\sum_{n=0}^\infty\sum_{x\in X}\frac{m(x)}{\pi(x)}P^{(n)}(o,x)\\
&=\sum_{x\in X}\frac{m(x)}{\pi(x)}G(o,x)=\sum_{x\in X}\frac{m(x)}{\pi(x)}\frac{\pi(x)G(x,o)}{\pi(o)}=\sum_{x\in X}\frac{m(x)G(x,o)}{\pi(o)}\\
&=\sum_{x\in X}\frac{m(x)F(x,o)G(o,o)}{\pi(o)}=\frac{G(o,o)}{\pi(o)}\sum_{n=0}^\infty3^n\cdot\left(\frac{c}{3\lambda}\right)^n\cdot\lambda^n\\
&=\frac{G(o,o)}{\pi(o)}\sum_{n=0}^\infty c^n.
\end{align*}
Since $c\in(0,\lambda)\subseteq(0,1)$, we have $\bbE_o\zeta<+\infty$, hence
$$\bbP_o[\zeta<+\infty]=1.$$

For all $x\in X$, let $n=|x|$, note that $P^{(n)}(o,x)>0$, by Proposition \ref{SG_det_prop_jump} and strong Markov property, we have
\begin{align*}
\bbE_o\zeta&\ge\bbE_o\left[\zeta\myindi{\myset{X_{J_n}=x}}\right]=\bbE_o\left[\bbE_o\left[\zeta\myindi{\myset{X_{J_n}=x}}|X_{J_n}\right]\right]\\
&=\bbE_o\left[\myindi{\myset{X_{J_n}=x}}\bbE_o\left[\zeta|X_{J_n}\right]\right]=\bbE_o\left[\myindi{\myset{X_{J_n}=x}}\bbE_{X_{J_n}}\left[\zeta\right]\right]\\
&=P^{(n)}(o,x)\bbE_x\left[\zeta\right].
\end{align*}
Hence $\bbE_x\zeta<+\infty$, hence
$$\bbP_x\left[\zeta<+\infty\right]=1.$$
Therefore, $\calE_X$ is stochastically incomplete.
\end{proof}

By \cite[Proposition 1.17(b)]{Wo00}, for a transient random walk $Z$ on $X$, for all finite set $A\subseteq X$, we have $\bbP_x\left[Z_n\in A\text{ for infinitely many }n\right]=0$ for all $x\in X$. Roughly speaking, a transient random walk will go to infinity almost surely. For variable speed continuous time random walk $\myset{X_t}$ on $X$, we have following theorem.

\begin{mythm}\label{SG_det_thm_infty}
$\myset{X_t}$ goes to infinity in finite time almost surely, that is,
$$\bbP_x\left[\lim_{t\uparrow\zeta}\lvert X_t\rvert=+\infty,\zeta<+\infty\right]=1\text{ for all }x\in X.$$
\end{mythm}
\begin{proof}
There exists $\Omega_0$ with $\bbP_x(\Omega_0)=1$ such that $\zeta(\omega)<+\infty$ for all $\omega\in\Omega_0$. For all $m\ge1$, we have
$$\bbP_x\left[Y_n\in B_m\text{ for infinitely many }n\right]=0,$$
hence there exists $\Omega_m$ with $\bbP_x(\Omega_m)=1$ such that for all $\omega\in\Omega_m$, there exist $N=N(\omega)\ge1$ such that for all $n\ge N$, we have $Y_n(\omega)\notin B_m$. Moreover
$$\bbP_x\left(\Omega_0\cap\bigcap_{m=1}^\infty\Omega_m\right)=1.$$
For all $\omega\in\Omega_0\cap\cap_{m=1}^\infty\Omega_m$, we have
$$J_n(\omega)\le J_{n+1}(\omega)<\zeta(\omega)<+\infty.$$
Since $\omega\in\Omega_m$, there exists $N=N(\omega)\ge1$ such that for all $n>N$, we have $Y_n(\omega)\notin B_m$. By definition, we have
$$X_t(\omega)=Y_n(\omega)\text{ if }J_n(\omega)\le t<J_{n+1}(\omega).$$
Let $T=J_{N(\omega)}(\omega)$, then for all $t>T$, there exists $n\ge N$ such that $J_n(\omega)\le t<J_{n+1}(\omega)$, hence $X_t(\omega)=Y_n(\omega)\not\in B_m$, that is,
$$\lim_{t\uparrow\zeta(\omega)}\lvert X_t(\omega)\rvert=+\infty.$$
\end{proof}

\section{Active Reflected Dirichlet Space \texorpdfstring{$(\calE^{\rref},\calF^{\rref}_a)$}{(Eref,Frefa)}}\label{SG_det_sec_ref}

In this section, we construct active reflected Dirichlet form $(\calE^{\rref},\calF^{\rref}_a)$ and show that $\calF_X\subsetneqq\calF^{\rref}_a$, hence $\calE^{\rref}$ is not regular.

Reflected Dirichlet space was introduced by Chen \cite{Chen92}. This is a generalization of reflected Brownian motion in Euclidean space. He considered abstract Dirichlet form instead of constructing reflection path-wisely from probabilistic viewpoint. More detailed discussion is incorporated into his book with Fukushima \cite[Chapter 6]{CF12}.

Given a regular transient Dirichlet form $(\calE,\calF)$ on $L^2(X;m)$, we can do reflection in the following two ways:
\begin{itemize}
\item The linear span of $\calF$ and all harmonic functions of finite ``$\calE$-energy".
\item All functions that are ``locally" in $\calF$ and have finite ``$\calE$-energy".
\end{itemize}
We use the second way which is more convenient. Recall the Beurling-Deny decomposition. Since $(\calE,\calF)$ is regular, we have
$$\calE(u,u)=\frac{1}{2}\mu^c_{<u>}(X)+\int_{X\times X\backslash d}(u(x)-u(y))^2J(\md x\md y)+\int_Xu(x)^2k(\md x)$$
for all $u\in\calF_e$, here we use the convention that all functions in $\calF_e$ are quasi-continuous. By this formula, we can define
$$\hat{\calE}(u,u)=\frac{1}{2}\mu^c_{<u>}(X)+\int_{X\times X\backslash d}(u(x)-u(y))^2J(\md x\md y)+\int_Xu(x)^2k(\md x)$$
for all $u\in\calF_{\mathrm{loc}}$, where
$$\calF_{\mathrm{loc}}=\myset{u:\forall G\subseteq X\text{ relatively compact open,}\exists v\in\calF,\text{s.t. }u=v,m\text{-a.e. on }G}.$$
We give the definition of reflected Dirichlet space as follows. \cite[Theorem 6.2.5]{CF12} gave
$$
\begin{cases}
&\calF^{\rref}=\myset{u:\text{ finite }m\text{-a.e.},\exists\myset{u_n}\subseteq\calF_{\mathrm{loc}}\text{ that is }\hat{\calE}\text{-Cauchy},u_n\to u,m\text{-a.e. on }X},\\
&\hat{\calE}(u,u)=\lim_{n\to+\infty}\hat{\calE}(u_n,u_n).
\end{cases}
$$
Let $\tau_ku=\left((-k)\vee u\right)\wedge k$, $k\ge1$, then \cite[Theorem 6.2.13]{CF12} gave
$$
\begin{cases}
&\calF^{\rref}=\myset{u:|u|<+\infty,m\text{-a.e.},\tau_ku\in\calF_{\mathrm{loc}}\forall k\ge1,\sup_{k\ge1}\hat{\calE}(\tau_ku,\tau_ku)<+\infty},\\
&\calE^\rref(u,u)=\lim_{k\to+\infty}\hat{\calE}(\tau_ku,\tau_ku).
\end{cases}
$$
Let $\calF^\rref_a=\calF^\rref\cap L^2(X;m)$, then $(\calF^\rref_a,\calE^\rref)$ is called active reflected Dirichlet space. \cite[Theorem 6.2.14]{CF12} showed that $(\calE^\rref,\calF^\rref_a)$ is a Dirichlet form on $L^2(X;m)$.

Returning to our case, since
$$\calE_X(u,u)=\frac{1}{2}\sum_{x,y\in X}c(x,y)(u(x)-u(y))^2\text{ for all }u\in\calF_X,$$
$\calE_X$ has only jumping part, we have
$$\hat{\calE}_X(u,u)=\frac{1}{2}\sum_{x,y\in X}c(x,y)(u(x)-u(y))^2\text{ for all }u\in(\calF_X)_{\mathrm{loc}}.$$
By the definition of local Dirichlet space
$$(\calF_X)_{\mathrm{loc}}=\myset{u:\forall G\subseteq X\text{ relatively compact open,}\exists v\in\calF_X,\text{s.t. }u=v,m\text{-a.e. on }G}.$$
For all $G\subseteq X$ relatively compact open, we have $G$ is a finite set, for all function $u$ on $X$, let
$$v(x)=
\begin{cases}
u(x),&\text{if }x\in G,\\
0,&\text{if }x\in X\backslash G,
\end{cases}
$$
then $v\in C_c(X)\subseteq\calF_X$ and $u=v$ on $G$, hence $(\calF_X)_{\mathrm{loc}}=\myset{u:u\text{ is a finite function on }X}$.
$$\calF^\rref=\myset{u:|u(x)|<+\infty,\forall x\in X,\sup_{k\ge1}\frac{1}{2}\sum_{x,y\in X}c(x,y)\left(\tau_ku(x)-\tau_ku(y)\right)^2<+\infty}.$$
By monotone convergence theorem, we have
$$\frac{1}{2}\sum_{x,y\in X}c(x,y)\left(\tau_ku(x)-\tau_ku(y)\right)^2\uparrow\frac{1}{2}\sum_{x,y\in X}c(x,y)\left(u(x)-u(y)\right)^2,$$
hence
$$\sup_{k\ge1}\frac{1}{2}\sum_{x,y\in X}c(x,y)\left(\tau_ku(x)-\tau_ku(y)\right)^2=\frac{1}{2}\sum_{x,y\in X}c(x,y)\left(u(x)-u(y)\right)^2,$$
and
$$
\begin{cases}
&\calF^\rref=\myset{u:\text{ finite function},\frac{1}{2}\sum_{x,y\in X}c(x,y)\left(u(x)-u(y)\right)^2<+\infty},\\
&\calE^\rref(u,u)=\frac{1}{2}\sum_{x,y\in X}c(x,y)\left(u(x)-u(y)\right)^2.
\end{cases}
$$
$$\calF^\rref_a=\myset{u\in L^2(X;m):\frac{1}{2}\sum_{x,y\in X}c(x,y)\left(u(x)-u(y)\right)^2<+\infty}.$$
Indeed, we can show that $(\calE^\rref,\calF^\rref_a)$ is a Dirichlet form on $L^2(X;m)$ directly. In general, $(\calE^\rref,\calF^\rref_a)$ on $L^2(X;m)$ is not regular, $\calF_X\subsetneqq\calF^\rref_a$. This is like $H^1_0(D)\subsetneqq H^1(D)$, where $D$ is the open unit disk in $\R^d$. We need to show $\calF_X\ne\calF^\rref_a$, otherwise reflection is meaningless. Then we do regular representation of $(\calE^\rref,\calF^\rref_a)$ on $L^2(X;m)$ to enlarge the space $X$ to Martin compactification $\bar{X}$ and Martin boundary $\dd X$ will appear.

\begin{mythm}
$\calF_X\subsetneqq\calF^{\rref}_a$, hence $\calE^{\rref}$ is not regular.
\end{mythm}

\begin{proof}
Since $m(X)<+\infty$, we have $1\in\calF^{\rref}_a$ and $\calE^{\rref}(1,1)=0$, by \cite[Theorem 1.6.3]{FOT11}, $\calE^{\rref}$ is recurrent, by \cite[Lemma 1.6.5]{FOT11}, $\calE^{\rref}$ is conservative or stochastically complete. Since $\calE_X$ is transient and stochastically incomplete, we have $\calF_X\ne\calF^{\rref}_a$. Note that $\calE^{\rref}$ is not regular, there is no corresponding Hunt process, but recurrent and conservative properties are still well-defined, see \cite[Chapter 1, 1.6]{FOT11}.
\end{proof}

\section{Regular Representation of \texorpdfstring{$(\calE^\rref,\calF^\rref_a)$}{(Eref,Frefa)}}\label{SG_det_sec_repre}

In this section, we construct a regular Dirichlet form $(\calE_{\bar{X}},\calF_{\bar{X}})$ on $L^2(\bar{X};m)$ which is a regular representation of Dirichlet form $(\calE^\rref,\calF^\rref_a)$ on $L^2(X;m)$, where $\bar{X}$ is the Martin compactification of $X$ and $m$ is given as above.

Recall that $(\frac{1}{2}\mathbf{D},H^1(D))$ on $L^2(D)$ is not regular and $(\frac{1}{2}\mathbf{D},H^1(D))$ on $L^2(\bar{D},\indi_D(\md x))$ is a regular representation. Our construction is very simple and similar to this case. Let
$$
\begin{cases}
&\calE_{\bar{X}}(u,u)=\frac{1}{2}\sum_{x,y\in X}c(x,y)(u(x)-u(y))^2,\\
&\calF_{\bar{X}}=\myset{u\in C(\bar{X}):\sum_{x,y\in X}c(x,y)(u(x)-u(y))^2<+\infty}.
\end{cases}
$$
We show that $(\calE_{\bar{X}},\calF_{\bar{X}})$ is a regular Dirichlet form on $L^2({\bar{X}};m)$.

\begin{mythm}\label{SG_det_thm_main}
If $\lambda\in(1/5,1/3)$, then $(\calE_{\bar{X}},\calF_{\bar{X}})$ is a regular Dirichlet form on $L^2({\bar{X}};m)$.
\end{mythm}

First, we need a lemma.

\begin{mylem}\label{SG_det_lem_ext}
If $\lambda<1/3$, then for all $u$ on $X$ with
$$C=\frac{1}{2}\sum_{x,y\in X}c(x,y)(u(x)-u(y))^2<+\infty,$$
$u$ can be extended continuously to $\bar{X}$.
\end{mylem}

\begin{proof}
Since $\hat{X}$ is homeomorphic to $\bar{X}$, we consider $\hat{X}$ instead. For all $\xi\in\dd X$, take geodesic ray $[x_0,x_1,\ldots]$ with $|x_n|=n$, $x_n\sim x_{n+1}$ for all $n\ge0$ such that $x_n\to\xi$ in $\rho_a$. Then
$$\lvert u(x_n)-u(x_{n+1})\rvert\le\sqrt{\frac{2C}{c(x_n,x_{n+1})}}=\sqrt{2C}(\sqrt{3\lambda})^n,$$
since $\lambda<1/3$, we have $\myset{u(x_n)}$ is a Cauchy sequence, define $u(\xi)=\lim_{n\to+\infty}u(x_n)$.

First, we show that this is well-defined. Indeed, for all equivalent geodesic rays $[x_0,x_1,\ldots]$ and $[y_0,y_1,\ldots]$ with $|x_0|=|y_0|=0$, by \cite[Proposition 22.12(a)]{Wo00}, for all $n\ge0$, $d(x_n,y_n)\le2\delta$. Take an integer $M\ge 2\delta$, then for arbitrary fixed $n\ge0$, there exist $z_0=x_n,\ldots,z_M=y_n$ with $|z_i|=n$ for all $i=0,\ldots,M$, $z_i=z_{i+1}$ or $z_{i}\sim z_{i+1}$ for all $i=0,\ldots, M-1$, we have
$$\lvert u(x_n)-u(y_n)\rvert\le\sum_{i=0}^{M-1}|u(z_{i})-u(z_{i+1})|\le\sum_{i=0}^{M-1}\sqrt{\frac{2C}{c(z_i,z_{i+1})}}\le M\sqrt{\frac{2C}{\min{\myset{C_1,C_2}}}}(\sqrt{3\lambda})^n.$$
Since $\lambda<1/3$, letting $n\to+\infty$, we have $\lvert u(x_n)-u(y_n)\rvert\to0$, $u(\xi)$ is well-defined and
$$|u(\xi)-u(x_n)|\le\sum_{i=n}^\infty|u(x_i)-u(x_{i+1})|\le\sum_{i=n}^\infty\sqrt{2C}(\sqrt{3\lambda})^n=\frac{\sqrt{2C}}{1-\sqrt{3\lambda}}(\sqrt{3\lambda})^n.$$

Next, we show that the extended function $u$ is continuous on $\hat{X}$. We only need to show that for all sequence $\myset{\xi_n}\subseteq\dd X$ with $\xi_n\to\xi\in\dd X$ in $\rho_a$, we have $u(\xi_n)\to u(\xi)$. Since $\dd X$ with $\rho_a$ is H\"older equivalent to $K$ with Euclidean metric by Equation (\ref{SG_det_eqn_Holder}), we use them interchangeably, $\myset{\xi_n}\subseteq K$ and $\xi_n\to\xi\in K$ in Euclidean metric.

For all $\veps>0$, there exists $M\ge1$ such that $(\sqrt{3\lambda})^M<\veps$. Take $w\in W_M$ such that $\xi\in K_w$, there are at most 12 numbers of $\tilde{w}\in W_M$, $\tilde{w}\ne w$ such that $\tilde{w}\sim w$, see Figure \ref{SG_det_figure_nbhd}. Indeed, if we analyze geometric property of the SG carefully, we will see there are at most 3.




\begin{figure}[ht]
  \centering
  \begin{tikzpicture}[scale=0.7]
  \tikzstyle{every node}=[font=\small,scale=0.7]
  \draw (0,2*1.7320508076)--(2,2*1.7320508076);
  \draw (-1,1.7320508076)--(3,1.7320508076);
  \draw (-2,0)--(4,0);
  \draw (-1,-1.7320508076)--(3,-1.7320508076);
  
  \draw (-2,0)--(0,2*1.7320508076);
  \draw (-1,-1.7320508076)--(2,2*1.7320508076);
  \draw (1,-1.7320508076)--(3,1.7320508076);
  \draw (3,-1.7320508076)--(4,0);
  
  \draw (-2,0)--(-1,-1.7320508076);
  \draw (-1,1.7320508076)--(1,-1.7320508076);
  \draw (0,2*1.7320508076)--(3,-1.7320508076);
  \draw (2,2*1.7320508076)--(4,0);
  
  \draw[line width=2pt] (0,0)--(1,1.7320508076)--(2,0)--cycle;
  \draw (1,0.7) node {$K_w$};
  
  \end{tikzpicture}
  \caption{A Neighborhood of $K_w$}\label{SG_det_figure_nbhd}
\end{figure}



Let
$$U=\left(\bigcup_{\tilde{w}:\tilde{w}\in W_M,\tilde{w}\sim w}K_{\tilde{w}}\right)\cup K_w,$$
there exists $N\ge1$ for all $n>N$, $\xi_n\in U$. For all $n>N$. If $\xi_n\in K_w$, then
$$\lvert u(\xi_n)-u(\xi)\rvert\le\lvert u(\xi_n)-u(w)\rvert+\lvert u(\xi)-u(w)\rvert\le\frac{2\sqrt{2C}}{1-\sqrt{3\lambda}}(\sqrt{3\lambda})^M<\frac{2\sqrt{2C}}{1-\sqrt{3\lambda}}\veps.$$
If $\xi_n\in K_{\tilde{w}}$, $\tilde{w}\in K_M$, $\tilde{w}\sim w$, then
\begin{align*}
\lvert u(\xi_n)-u(\xi)\rvert&\le\lvert u(\xi_n)-u(\tilde{w})\rvert+\lvert u(\tilde{w})-u(w)\rvert+\lvert u(w)-u(\xi)\rvert\\
&\le\frac{2\sqrt{2C}}{1-\sqrt{3\lambda}}(\sqrt{3\lambda})^M+\sqrt{\frac{2C}{\min{\myset{C_1,C_2}}}}(\sqrt{3\lambda})^M\\
&<\left(\frac{2\sqrt{2C}}{1-\sqrt{3\lambda}}+\sqrt{\frac{2C}{\min{\myset{C_1,C_2}}}}\right)\veps.
\end{align*}
Hence
$$|u(\xi_n)-u(\xi)|<\left(\frac{2\sqrt{2C}}{1-\sqrt{3\lambda}}+\sqrt{\frac{2C}{\min{\myset{C_1,C_2}}}}\right)\veps,$$
for all $n>N$. $\lim_{n\to+\infty}u(\xi_n)=u(\xi)$. The extended function $u$ is continuous on $\hat{X}$.
\end{proof}

\begin{proof}[Proof of Theorem \ref{SG_det_thm_main}]
Since $C_c(X)\subseteq\calF_{\bar{X}}$ is dense in $L^2({\bar{X}};m)$, we have $\calE_{\bar{X}}$ is a symmetric form on $L^2({\bar{X}};m)$.

We show closed property of $\calE_{\bar{X}}$. Let $\myset{u_k}\subseteq\calF_{\bar{X}}$ be an $(\calE_{\bar{X}})_1$-Cauchy sequence. Then there exists $u\in L^2(\bar{X};m)$ such that $u_k\to u$ in $L^2(\bar{X};m)$, hence $u_k(x)\to u(x)$ for all $x\in X$. By Fatou's lemma, we have
\begin{align*}
&\frac{1}{2}\sum_{x,y\in X}c(x,y)\left((u_k-u)(x)-(u_k-u)(y)\right)^2\\
=&\frac{1}{2}\sum_{x,y\in X}c(x,y)\lim_{l\to+\infty}\left((u_k-u_l)(x)-(u_k-u_l)(y)\right)^2\\
\le&\varliminf_{l\to+\infty}\frac{1}{2}\sum_{x,y\in X}c(x,y)\left((u_k-u_l)(x)-(u_k-u_l)(y)\right)^2\\
=&\varliminf_{l\to+\infty}\calE_{\bar{X}}(u_k-u_l,u_k-u_l).
\end{align*}
Letting $k\to+\infty$, we have
$$\frac{1}{2}\sum_{x,y\in X}c(x,y)\left((u_k-u)(x)-(u_k-u)(y)\right)^2\to0,$$
and
$$\frac{1}{2}\sum_{x,y\in X}c(x,y)\left(u(x)-u(y)\right)^2<+\infty.$$
By Lemma \ref{SG_det_lem_ext}, $u$ can be extended continuously to $\bar{X}$, hence $u\in C(\bar{X})$, $u\in\calF_{\bar{X}}$. $\calE_{\bar{X}}$ is closed.

It is obvious that $\calE_{\bar{X}}$ is Markovian. Hence $\calE_{\bar{X}}$ is a Dirichlet form on $L^2(\bar{X};m)$.

Since ${\bar{X}}$ is compact, we have $C_c({\bar{X}})=C({\bar{X}})$. To show $\calE_{\bar{X}}$ is regular, we need to show $C_c({\bar{X}})\cap\calF_{\bar{X}}=C({\bar{X}})\cap\calF_{\bar{X}}=\calF_{\bar{X}}$ is $(\calE_{\bar{X}})_1$-dense in $\calF_{\bar{X}}$ and uniformly dense in $C_c({\bar{X}})=C({\bar{X}})$. $\calF_{\bar{X}}$ is trivially $(\calE_{\bar{X}})_1$-dense in $\calF_{\bar{X}}$. We need to show that $\calF_{\bar{X}}$ is uniformly dense in $C({\bar{X}})$. Since $\bar{X}$ is compact, we have $\calF_{\bar{X}}$ is a sub algebra of $C(\bar{X})$. By Stone-Weierstrass theorem, we only need to show that $\calF_{\bar{X}}$ separates points. The idea of our proof is from classical construction of local regular Dirichlet form on the SG.

For all $p,q\in\bar{X}$ with $p\ne q$, we only need to show that there exists $v\in\calF_{\bar{X}}$ such that $v(p)\ne v(q)$.

If $p\in X$, then letting $v(p)=1$ and $v(x)=0$ for all $x\in X\backslash\myset{p}$, we have
$$\sum_{x,y\in X}c(x,y)(v(x)-v(y))^2<+\infty.$$
By Lemma \ref{SG_det_lem_ext}, $v$ can be extended to a function in $C(\bar{X})$ still denoted by $v$, hence $v\in\calF_{\bar{X}}$. Moreover, $v(q)=0\ne1=v(p)$.

If $q\in X$, then we have the similar proof.

If $p,q\in\bar{X}\backslash X=\dd X=K$, then there exists sufficiently large $m\ge1$ and $w^{(1)},w^{(2)}\in S_m$ with $p\in K_{w^{(1)}}$, $q\in K_{w^{(2)}}$ and $K_{w^{(1)}}\cap K_{w^{(2)}}=\emptyset$, hence $w^{(1)}\not\sim w^{(2)}$. Let $v=0$ in $B_m$ and
$$v(w^{(1)}0)=v(w^{(1)}1)=v(w^{(1)}2)=1.$$
For all $w\in S_{m+1}\backslash\myset{w^{(1)}0,w^{(1)}1,w^{(1)}2}$, let
$$
v(w)=
\begin{cases}
1,&\text{if }w\sim w^{(1)}0\text{ or }w\sim w^{(1)}1\text{ or }w\sim w^{(1)}2,\\
0,&\text{otherwise},
\end{cases}
$$
then
$$v(w^{(2)}0)=v(w^{(2)}1)=v(w^{(2)}2)=0.$$
In the summation $\sum_{x,y\in S_{m+1}}c(x,y)(v(x)-v(y))^2$, horizontal edges of type \Rmnum{2} make no contribution since $v$ takes the same values at end nodes of each such edge. Assume that we have constructed $v$ on $B_n$ such that in the summation $\sum_{x,y\in S_i}c(x,y)(v(x)-v(y))^2$, $i=m+1,\ldots,n$, horizontal edges of type \Rmnum{2} make no contribution, that is, $v$ takes the same values at end nodes of each edge. We construct $v$ on $S_{n+1}$ as follows.

Consider $\sum_{x,y\in S_n}c(x,y)(v(x)-v(y))^2$, nonzero terms all come from edges of smallest triangles in $S_n$. Pick up one such triangle in $S_n$, it generates three triangles in $S_{n+1}$, nine triangles in $S_{n+2}$, \ldots. See Figure \ref{SG_det_figure_gene}.




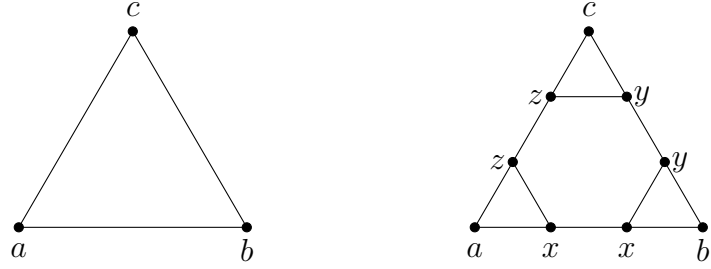
\begin{figure}[ht]
  \centering
  \begin{tikzpicture}
  \draw (0,0)--(3,0);
  \draw (3,0)--(3/2,3/2*1.7320508076);
  \draw (0,0)--(3/2,3/2*1.7320508076);
  
  \draw (6,0)--(9,0);
  \draw (9,0)--(7.5,3/2*1.7320508076);
  \draw (6,0)--(7.5,3/2*1.7320508076);
  
  \draw (6.5,1/2*1.7320508076)--(7,0);
  \draw (8.5,1/2*1.7320508076)--(8,0);
  \draw (7,1.7320508076)--(8,1.7320508076);
  
  \draw[fill=black] (0,0) circle (0.06);
  \draw[fill=black] (3,0) circle (0.06);
  \draw[fill=black] (3/2,3/2*1.7320508076) circle (0.06);
  \draw[fill=black] (6,0) circle (0.06);
  \draw[fill=black] (7,0) circle (0.06);
  \draw[fill=black] (8,0) circle (0.06);
  \draw[fill=black] (9,0) circle (0.06);
  \draw[fill=black] (6.5,1.7320508076/2) circle (0.06);
  \draw[fill=black] (8.5,1.7320508076/2) circle (0.06);
  \draw[fill=black] (7,1.7320508076) circle (0.06);
  \draw[fill=black] (8,1.7320508076) circle (0.06);
  \draw[fill=black] (7.5,3/2*1.7320508076) circle (0.06);
  
  \draw (0,-0.3) node {$a$};
  \draw (3,-0.3) node {$b$};
  \draw (1.5,3/2*1.7320508076+0.3) node {$c$};
  
  \draw (6,-0.3) node {$a$};
  \draw (9,-0.3) node {$b$};
  \draw (7.5,3/2*1.7320508076+0.3) node {$c$};
  
  \draw (7,-0.3) node {$x$};
  \draw (8,-0.3) node {$x$};
  
  \draw (6.3,1/2*1.7320508076) node {$z$};
  \draw (6.8,1.7320508076) node {$z$};
  
  \draw (8.7,1/2*1.7320508076) node {$y$};
  \draw (8.2,1.7320508076) node {$y$};
  
  \end{tikzpicture}
  \caption{Generation of triangles}\label{SG_det_figure_gene}
\end{figure}



We only need to assign values of $v$ on the three triangles in $S_{n+1}$ from the values of $v$ on the triangle in $S_n$. As in Figure \ref{SG_det_figure_gene}, $x,y,z$ are the values of $v$ at corresponding nodes to be determined from $a,b,c$. The contribution of this one triangle in $S_n$ to $\sum_{x,y\in S_n}c(x,y)(v(x)-v(y))^2$ is
$$A_1=\frac{C_1}{(3\lambda)^n}\left[(a-b)^2+(b-c)^2+(a-c)^2\right].$$
The contribution of these three triangles in $S_{n+1}$ to $\sum_{x,y\in S_{n+1}}c(x,y)(v(x)-v(y))^2$ is
\begin{align*}
A_2=\frac{C_1}{(3\lambda)^{n+1}}&\left[(a-x)^2+(a-z)^2+(x-z)^2\right.\\
&+(b-x)^2+(b-y)^2+(x-y)^2\\
&\left.+(c-y)^2+(c-z)^2+(y-z)^2\right].\\
\end{align*}
Regarding $A_2$ as a function of $x,y,z$, by elementary calculation, we have $A_2$ takes the minimum value when
$$
\begin{cases}
x=\frac{2a+2b+c}{5},\\
y=\frac{a+2b+2c}{5},\\
z=\frac{2a+b+2c}{5},
\end{cases}
$$
and
\begin{align*}
A_2&=\frac{C_1}{(3\lambda)^{n+1}}\cdot\frac{3}{5}\left[(a-b)^2+(b-c)^2+(a-c)^2\right]\\
&=\frac{1}{5\lambda}\left(\frac{C_1}{(3\lambda)^n}\left[(a-b)^2+(b-c)^2+(a-c)^2\right]\right)=\frac{1}{5\lambda}A_1.
\end{align*}
By the above construction, horizontal edges of type \Rmnum{2} in $S_{n+1}$ also make no contribution to
$$\sum_{x,y\in S_{n+1}}c(x,y)(v(x)-v(y))^2$$
and
$$\sum_{x,y\in S_{n+1}}c(x,y)(v(x)-v(y))^2=\frac{1}{5\lambda}\sum_{x,y\in S_{n}}c(x,y)(v(x)-v(y))^2.$$
Since $\lambda>1/5$, we have
$$\sum_{n=0}^\infty\sum_{x,y\in S_{n}}c(x,y)(v(x)-v(y))^2<+\infty,$$
this is the contribution of all horizontal edges to $\sum_{x,y\in X}c(x,y)(v(x)-v(y))^2$.

We consider the contribution of all vertical edges as follows. For all $n\ge m$, by construction $v|_{S_{n+1}}$ is uniquely determined by $v|_{S_n}$, hence the contribution of vertical edges between $S_n$ and $S_{n+1}$ is uniquely determined by $v|_{S_n}$. As above, we pick up one smallest triangle in $S_n$ and consider the contribution of the vertical edges connecting it to $S_{n+1}$. There are nine vertical edges between $S_n$ and $S_{n+1}$ connecting this triangle. These nine vertical edges make contribution
\begin{align*}
A_3&=\frac{1}{(3\lambda)^n}\left[(a-x)^2+(a-z)^2+(a-a)^2\right.\\
&+(b-x)^2+(b-y)^2+(b-b)^2\\
&+(c-y)^2+(c-z)^2+(c-c)^2\left.\right]\\
&=\frac{14}{25C_1}\left(\frac{C_1}{(3\lambda)^n}\left[(a-b)^2+(b-c)^2+(a-c)^2\right]\right)=\frac{14}{25C_1}A_1.
\end{align*}
Hence
$$\sum_{x\in S_n,y\in S_{n+1}}c(x,y)(v(x)-v(y))^2=\frac{14}{25C_1}\sum_{x,y\in S_n}c(x,y)(v(x)-v(y))^2,$$
and
$$\sum_{n=0}^\infty\sum_{x\in S_n,y\in S_{n+1}}c(x,y)(v(x)-v(y))^2<+\infty\Leftrightarrow\sum_{n=0}^\infty\sum_{x,y\in S_n}c(x,y)(v(x)-v(y))^2<+\infty.$$
Since $\lambda>1/5$, we have both summations converge and
$$\sum_{x,y\in X}c(x,y)(v(x)-v(y))^2<+\infty.$$
By Lemma \ref{SG_det_lem_ext}, $v$ can be extended to a function in $C(\bar{X})$ still denoted by $v$, hence $v\in\calF_{\bar{X}}$. Since $v$ is constructed by convex interpolation on $X\backslash B_{m+1}$, we have $v(p)=1\ne0=v(q)$.

Therefore, $(\calE_{\bar{X}},\calF_{\bar{X}})$ is a regular Dirichlet form on $L^2({\bar{X}};m)$.
\end{proof}

\begin{mythm}
$\calE_{\bar{X}}$ on $L^2({\bar{X}};m)$ is a regular representation of $\calE^\rref$ on $L^2(X;m)$.
\end{mythm}

Regular representation theory was developed by Fukushima \cite{Fuk71} and incorporated into his book \cite[Appendix, A4]{FOT11}.
\begin{proof}
We only need to construct an algebraic isomorphism $\Phi:(\calF^\rref_a)_b\to(\calF_{\bar{X}})_b$ such that for all $u\in(\calF^\rref_a)_b$, we have
\begin{equation}\label{SG_det_eqn_iso}
\lVert u\rVert_{L^\infty(X;m)}=\lVert\Phi(u)\rVert_{L^\infty(\bar{X};m)},(u,u)_X=(\Phi(u),\Phi(u))_{\bar{X}}, \calE^\rref(u,u)=\calE_{\bar{X}}(\Phi(u),\Phi(u)).
\end{equation}
Indeed, for all $u\in(\calF^\rref_a)_b$, we have $\sum_{x,y\in X}c(x,y)(u(x)-u(y))^2<+\infty$, by Lemma \ref{SG_det_lem_ext}, we define $\Phi(u)$ as the continuous extension of $u$ to ${\bar{X}}$. Since $\calE^\rref$, $\calE_{\bar{X}}$ have the same expression for energy and $m(\dd X)=0$, Equation (\ref{SG_det_eqn_iso}) is obvious.
\end{proof}

Moreover we have

\begin{mythm}\label{SG_det_thm_part}
$(\calE_X,\calF_X)$ on $L^2(X;m)$ is the part form on $X$ of $(\calE_{\bar{X}},\calF_{\bar{X}})$ on $L^2({\bar{X}};m)$.
\end{mythm}

\begin{proof}
By \cite[Theorem 3.3.9]{CF12}, since $X\subseteq{\bar{X}}$ is an open subset and $\calF_{\bar{X}}$ is a special standard core of $(\calE_{\bar{X}},\calF_{\bar{X}})$ on $L^2(\bar{X};m)$, we have
$$(\calF_{\bar{X}})_X=\myset{u\in\calF_{\bar{X}}:\mathrm{supp}(u)\subseteq X}=\myset{u\in\calF_{\bar{X}}:u\in C_c(X)}=C_c(X).$$
Since $\calF_X$ is the $(\calE_X)_1$-closure of $C_c(X)$, we have the part form of $(\calE_{\bar{X}},\calF_{\bar{X}})$ on $L^2({\bar{X}};m)$ on $X$ is exactly $(\calE_X,\calF_X)$ on $L^2(X;m)$.
\end{proof}

From probabilistic viewpoint, $(\calE_X,\calF_X)$ on $L^2(X;m)$ corresponds to absorbed process $\myset{X_t}$ and $(\calE_{\bar{X}},\calF_{\bar{X}})$ on $L^2({\bar{X}};m)$ corresponds to reflected process $\myset{\bar{X}_t}$. By \cite[Theorem 3.3.8]{CF12}, $\myset{X_t}$ is the part process of $\myset{\bar{X}_t}$ on $X$ which can be described as follows.

Let
$$\tau_X=\inf{\myset{t>0:\bar{X}_t\notin X}}=\inf{\myset{t>0:\bar{X}_t\in\dd X}}=\sigma_{\dd X},$$
then
$$X_t=
\begin{cases}
\bar{X}_t,&0\le t<\tau_X,\\
\dd,&t\ge\tau_X,
\end{cases}
$$
and
$$\zeta=\tau_X=\sigma_{\dd X}.$$

\section{Trace Form on \texorpdfstring{$\dd X$}{partial X}}\label{SG_det_sec_trace}

In this section, we take trace of the regular Dirichlet form $(\calE_{\bar{X}},\calF_{\bar{X}})$ on $L^2(\bar{X};m)$ to $K$ to have a regular Dirichlet form $(\calE_K,\calF_K)$ on $L^2(K;\nu)$ with the form (\ref{eqn_nonlocal}).

Firstly, we show that $\nu$ is of finite energy with respect to $\calE_{\bar{X}}$, that is, 
\begin{equation*}
\int_{\bar{X}}\lvert u(x)\rvert\nu(\md x)\le C\sqrt{(\calE_{\bar{X}})_1(u,u)}\text{ for all }u\in\calF_{\bar{X}}\cap C_c({\bar{X}})=\calF_{\bar{X}},
\end{equation*}
where $C$ is some positive constant. Since $\nu(\dd X)=1$, we only need to show that
\begin{mythm}\label{SG_det_thm_trace}
\begin{equation}\label{SG_det_eqn_trace}
\left(\int_{\bar{X}}\lvert u(x)\rvert^2\nu(\md x)\right)^{1/2}\le C\sqrt{(\calE_{\bar{X}})_1(u,u)}\text{ for all }u\in\calF_{\bar{X}}.
\end{equation}
\end{mythm}

We need some preparation.
\begin{mythm}(\cite[Theorem 1.1]{ALP99})\label{thm_ALP}
Suppose that a reversible random walk $\myset{Z_n}$ is transient, then for all $f$ with
$$D(f)=\frac{1}{2}\sum_{x,y\in X}c(x,y)(f(x)-f(y))^2<+\infty,$$
we have $\myset{f(Z_n)}$ converges almost surely and in $L^2$ under $\bbP_x$ for all $x\in X$.
\end{mythm}

For all $f$ with $D(f)<+\infty$, under $\bbP_o$, $\myset{f(Z_n)}$ converges almost surely and in $L^2$ to a random variable $W$, that is
$$f(Z_n)\to W,\bbP_o\text{-a.s.},\bbE_o\left[\left(f(Z_n)-W\right)^2\right]\to0,$$
then $W$ is a terminal random variable. By Theorem \ref{SG_det_thm_conv}, we have $Z_n\to Z_\infty$, $\bbP_o$-a.s.. By \cite[Corollary 7.65]{Wo09}, we have $W$ is of the form $W=\vphi(Z_\infty)$, $\bbP_o$-a.s., where $\vphi$ is a measurable function on $\dd X$. Hence we define a map $f\mapsto\vphi$, this is the operation of taking boundary value in some sense.

Let $\mathbf{D}=\myset{f:D(f)<+\infty}$. The Dirichlet norm of $f\in\mathbf{D}$ is given by $\lVert f\rVert^2=D(f)+\pi(o)f(o)^2$. Let $\mathbf{D}_0$ be the family of all functions that are limits in the Dirichlet norm of functions with finite support. We have the following Royden decomposition.

\begin{mythm}(\cite[Theorem 3.69]{Soa94})\label{SG_det_thm_Soa}
For all $f\in\mathbf{D}$, there exist unique harmonic Dirichlet function $f_{HD}$ and $f_0\in\mathbf{D}_0$ such that $f=f_{HD}+f_0$. Moreover, $D(f)=D(f_{HD})+D(f_0)$.
\end{mythm}

\begin{mylem}(\cite[Lemma 2.1]{ALP99})\label{SG_det_lem_ALP}
For all $f\in\mathbf{D}_0$, $x\in X$, we have
$$\pi(x)f(x)^2\le D(f)G(x,x).$$
Furthermore, there exists a superharmonic function $h\in\mathbf{D}_0$ such that $h\ge|f|$ pointwise and $D(h)\le D(f)$.
\end{mylem}

\begin{proof}[Proof of Theorem \ref{SG_det_thm_trace}]
Since $\calF_{\bar{X}}\subseteq C(\bar{X})$, for all $u\in\calF_{\bar{X}}$, it is trivial to take boundary value just as $u|_{\dd X}$. We still use notions $f,\vphi$. Then
$$f(Z_n)\to\vphi(Z_\infty),\bbP_o\text{-a.s.},$$
$$\bbE_o\left[\left(f(Z_n)-\vphi(Z_\infty)\right)^2\right]\to0.$$
Under $\bbP_o$, the hitting distribution of $\myset{Z_n}$ or the distribution of $Z_\infty$ is $\nu$, the normalized Hausdorff measure on $K$, we have
$$\int_{\dd X}|\vphi|^2\md\nu=\bbE_o\left[\vphi(Z_\infty)^2\right]=\lim_{n\to+\infty}\bbE_o\left[f(Z_n)^2\right].$$
We only need to estimate $\bbE_o\left[f(Z_n)^2\right]$ in terms of
$$D(f)+(f,f)=\frac{1}{2}\sum_{x,y\in X}c(x,y)(f(x)-f(y))^2+\sum_{x\in X}f(x)^2m(x).$$
By Theorem \ref{SG_det_thm_Soa}, we only need to consider harmonic Dirichlet functions and functions in $\mathbf{D}_0$.

For all $f\in\mathbf{D}$, we have
\begin{align*}
\sum_{k=0}^\infty\bbE_o\left[\left(f(Z_{k+1})-f(Z_k)\right)^2\right]=&\sum_{k=0}^\infty\bbE_o\left[\bbE_o\left[\left(f(Z_{k+1})-f(Z_k)\right)^2|Z_k\right]\right]\\
=&\sum_{k=0}^\infty\bbE_o\left[\bbE_{Z_k}\left[\left(f(Z_{1})-f(Z_0)\right)^2\right]\right]\\
=&\sum_{k=0}^\infty\sum_{x\in X}P^{(k)}(o,x)\bbE_{x}\left[\left(f(Z_{1})-f(Z_0)\right)^2\right]\\
=&\sum_{x,y\in X}\left(\sum_{k=0}^\infty P^{(k)}(o,x)\right)P(x,y)\left(f(x)-f(y)\right)^2\\
=&\sum_{x,y\in X}G(o,x)\frac{c(x,y)}{\pi(x)}(f(x)-f(y))^2\\
=&\sum_{x,y\in X}\frac{\pi(x)G(x,o)}{\pi(o)}\frac{c(x,y)}{\pi(x)}(f(x)-f(y))^2\\
=&\sum_{x,y\in X}\frac{F(x,o)G(o,o)}{\pi(o)}c(x,y)(f(x)-f(y))^2\\
\le&\frac{G(o,o)}{\pi(o)}\sum_{x,y\in X}c(x,y)(f(x)-f(y))^2\\
=&\frac{2G(o,o)}{\pi(o)}D(f).
\end{align*}

Let $f$ be a harmonic Dirichlet function, then $\myset{f(Z_n)}$ is a martingale. For all $n\ge1$
\begin{align*}
\bbE_o\left[f(Z_n)^2\right]&=\bbE_o\left[\left(\sum_{k=0}^{n-1}(f(Z_{k+1})-f(Z_k))+f(Z_0)\right)^2\right]\\
&=\sum_{k=0}^{n-1}\bbE_o\left[(f(Z_{k+1})-f(Z_k))^2\right]+f(o)^2\\
&\le\sum_{k=0}^{\infty}\bbE_o\left[(f(Z_{k+1})-f(Z_k))^2\right]+f(o)^2\\
&\le\frac{2G(o,o)}{\pi(o)}D(f)+f(o)^2,
\end{align*}
hence
\begin{equation}\label{SG_det_eqn_hd}
\bbE_o\left[f_{HD}(Z_n)^2\right]\le\frac{2G(o,o)}{\pi(o)}D(f_{HD})+f_{HD}(o)^2.
\end{equation}

Let $f\in\mathbf{D}_0$. Let $h$ be as in Lemma \ref{SG_det_lem_ALP}. Then $h\ge0$. Since $h$ is superharmonic, we have
$$\bbE_o\left[h(Z_{k+1})-h(Z_k)|Z_0,\ldots,Z_k\right]\le0,$$
hence
\begin{align*}
&\bbE_o\left[h(Z_{k+1})^2-h(Z_k)^2\right]\\
=&\bbE_o\left[(h(Z_{k+1})-h(Z_{k}))^2\right]+2\bbE_o\left[h(Z_k)(h(Z_{k+1})-h(Z_k))\right]\\
=&\bbE_o\left[(h(Z_{k+1})-h(Z_{k}))^2\right]+2\bbE_o\left[\bbE_o\left[h(Z_k)(h(Z_{k+1})-h(Z_k))|Z_0,\ldots,Z_k\right]\right]\\
=&\bbE_o\left[(h(Z_{k+1})-h(Z_{k}))^2\right]+2\bbE_o\left[h(Z_k)\bbE_o\left[h(Z_{k+1})-h(Z_k)|Z_0,\ldots,Z_k\right]\right]\\
\le&\bbE_o\left[(h(Z_{k+1})-h(Z_{k}))^2\right].
\end{align*}
We have
\begin{align*}
\bbE_o\left[h(Z_n)^2\right]&=\sum_{k=0}^{n-1}\bbE_o\left[h(Z_{k+1})^2-h(Z_k)^2\right]+h(o)^2\\
&\le\sum_{k=0}^{n-1}\bbE_o\left[(h(Z_{k+1})-h(Z_{k}))^2\right]+\frac{G(o,o)}{\pi(o)}D(h)\\
&\le\sum_{k=0}^{\infty}\bbE_o\left[(h(Z_{k+1})-h(Z_{k}))^2\right]+\frac{G(o,o)}{\pi(o)}D(h)\\
&\le\frac{2G(o,o)}{\pi(o)}D(h)+\frac{G(o,o)}{\pi(o)}D(h)\\
&=\frac{3G(o,o)}{\pi(o)}D(h),
\end{align*}
hence
$$\bbE_o\left[f(Z_n)^2\right]\le\bbE_o\left[h(Z_n)^2\right]\le\frac{3G(o,o)}{\pi(o)}D(h)\le\frac{3G(o,o)}{\pi(o)}D(f).$$
We have
\begin{equation}\label{SG_det_eqn_d0}
\bbE_o\left[f_0(Z_n)^2\right]\le\frac{3G(o,o)}{\pi(o)}D(f_0).
\end{equation}
Combining Equation (\ref{SG_det_eqn_hd}) and Equation (\ref{SG_det_eqn_d0}), we have
\begin{align*}
\bbE_o\left[f(Z_n)^2\right]&=\bbE_o\left[(f_{HD}(Z_n)+f_0(Z_n))^2\right]\le2\bbE_o\left[f_{HD}(Z_n)^2+f_0(Z_n)^2\right]\\
&\le2\left(\frac{2G(o,o)}{\pi(o)}D(f_{HD})+f_{HD}(o)^2+\frac{3G(o,o)}{\pi(o)}D(f_0)\right)\\
&\le2\left(\frac{5G(o,o)}{\pi(o)}D(f)+(f(o)-f_0(o))^2\right)\\
&\le2\left(\frac{5G(o,o)}{\pi(o)}D(f)+2f(o)^2+2f_0(o)^2\right)\\
&\le2\left(\frac{5G(o,o)}{\pi(o)}D(f)+2\frac{1}{m(o)}f(o)^2m(o)+2\frac{G(o,o)}{\pi(o)}D(f_0)\right)\\
&\le2\left(\frac{7G(o,o)}{\pi(o)}D(f)+2\frac{1}{m(o)}\sum_{x\in X}f(x)^2m(x)\right)\\
&\le\max{\myset{\frac{14G(o,o)}{\pi(o)},\frac{4}{m(o)}}}\left(D(f)+\sum_{x\in X}f(x)^2m(x)\right).
\end{align*}
Letting $C^2=\max{\myset{\frac{14G(o,o)}{\pi(o)},\frac{4}{m(o)}}}$ be a constant only depending on the conductance $c$ and the measure $m$, we have
$$\int_{\dd X}|\vphi|^2\md\nu=\lim_{n\to+\infty}\bbE_o\left[f(Z_n)^2\right]\le C^2(D(f)+\sum_{x\in X}f(x)^2m(x)).$$
In the notion of $u$, we obtain Equation (\ref{SG_det_eqn_trace}).
\end{proof}

Secondly, we obtain a regular Dirichlet form on $L^2(\dd X;\nu)$ by abstract theory of trace form. For more detailed discussion of trace form, see \cite[Chapter 5, 5.2]{CF12} and \cite[Chapter 6, 6.2]{FOT11}. We introduce some results used here.

Taking trace with respect to a regular Dirichlet form corresponds to taking time-change with respect to the corresponding Hunt process. Taking trace is realized by smooth measure. The family of all smooth measures is denoted by $S$. Taking time-change is realized by positive continuous additive functional, abbreviated as PCAF. The family of all PCAFs is denoted by $A_c^+$. The family of all equivalent classes of $A_c^+$ and the family $S$ are in one-to-one correspondence, see \cite[Theorem 5.1.4]{FOT11}.

We fix a regular Dirichlet form $(\calE,\calF)$ on $L^2(E;m)$ and its corresponding Hunt process $X=\myset{X_t}$.
\begin{itemize}
\item Firstly, we introduce the basic setup of time-change. Given a PCAF $A\in A_c^+$, define its support $F$, then $F$ is quasi closed and nearly Borel measurable. Define the right-continuous inverse $\tau$ of $A$, let $\check{X}_t=X_{\tau_t}$, then $\check{X}$ is a right process with state space $F$ and called the time-changed process of $X$ by $A$.
\item Secondly, we introduce the basic setup of trace form. For arbitrary non-polar, quasi closed, nearly Borel measurable, finely closed set $F$, define hitting distribution $H_F$ of ${X}$ for $F$ as follows:
$$H_Fg(x)=\bbE_x\left[g(X_{\sigma_F})\indi_{\sigma_F<+\infty}\right],x\in E,g\text{ is nonnegative Borel function.}$$
By \cite[Theorem 3.4.8]{CF12}, for all $u\in\calF_e$, we have $H_F|u|(x)<+\infty$, q.e. and $H_Fu\in\calF_e$. Define
$$\check{\calF}_e=\calF_e|_F,\check{\calE}(u|_F,v|_F)={\calE}(H_Fu,H_Fv),u,v\in\calF_e.$$
Two elements in $\check{\calF}_e$ can be identified if they coincide q.e. on $F$. We still need a measure on $F$. Let
$$S_F=\myset{\mu\in S:\text{the quasi support of }\mu=F},$$
where the quasi support of a Borel measure is the smallest (up to q.e. equivalence) quasi closed set outside which the measure vanishes. Let $\mu\in S_F$, by \cite[Theorem 3.3.5]{CF12}, two elements of $\check{\calF}_e$ coincide q.e. on $F$ if and only if they coincide $\mu$-a.e.. Define $\check{\calF}=\check{\calF}_e\cap L^2(F;\mu)$. Then $(\check{\calE},\check{\calF})$ is a symmetric form on $L^2(F;\mu)$.
\item Thirdly, the relation of trace form and time-change process is as follows. Given $A\in A_c^+$ or equivalently $\mu\in S$, let $F$ be the support of $A$, then $F$ satisfies the conditions in the second setup and by \cite[Theorem 5.2.1(\rmnum{1})]{CF12}, $\mu\in S_F$. We obtain $(\check{\calE},\check{\calF})$ on $L^2(F;\mu)$. By \cite[Theorem 5.2.2]{CF12}, the regular Dirichlet form corresponding to $\check{X}$ is exactly $(\check{\calE},\check{\calF})$ on $L^2(F;\mu)$.
\end{itemize}

We have $F\subseteq\mathrm{supp}(\mu)$ q.e.. But the point is that $F$ can be strictly contained in $\mathrm{supp}(\mu)$ q.e., usually we indeed need a trace form on $\mathrm{supp}(\mu)$. \cite{CF12} provided a solution not for all smooth measures, but some subset
$$\mathring{S}=\myset{\mu:\text{positive Radon measure charging no }\calE\text{-polar set}}.$$
For non-$\calE$-polar, quasi closed subset $F$ of $E$, let
$$\mathring{S}_F=\myset{\mu\in\mathring{S}:\text{the quasi support of }\mu\text{ is }F}.$$
Note that if $\mu\in\mathring{S}_F$, it may happen that $\mathrm{supp}(\mu)\supsetneqq F$ q.e.. We want some $\mu\in\mathring{S}_F$ such that $\mathrm{supp}(\mu)=F$ q.e.. We have a criterion as follows.

\begin{mylem}(\cite[Lemma 5.2.9(\rmnum{2})]{CF12})\label{SG_det_lem_ac}
Let $F$ be a non-$\calE$-polar, nearly Borel, finely closed set. Let $\nu\in\mathring{S}$ satisfy $\nu(E\backslash F)=0$. Assume the 1-order hitting distribution $H_F^1(x,\cdot)$ of $X$ for $F$ is absolutely continuous with respect to $\nu$ for $m$-a.e. $x\in E$. Then $\nu\in\mathring{S}_F$.
\end{mylem}

\begin{mycor}(\cite[Corollary 5.2.10]{CF12})\label{SG_det_cor_cf}
Let $F$ be a closed set. If there exists $\nu\in\mathring{S}_F$ such that the topological support $\mathrm{supp}(\nu)=F$, then for all $\mu\in\mathring{S}_F$, we have $(\check{\calE},\check{\calF})$ is a regular Dirichlet form on $L^2(F;\mu)$.
\end{mycor}

Roughly speaking, given a positive Radon measure $\mu$ charging no $\calE$-polar set, let $F=\mathrm{supp}(\mu)$. Firstly, we check Lemma \ref{SG_det_lem_ac} to have $\mu\in\mathring{S}_F$, then the quasi support of $\mu$ is $F$ and the support of corresponding PCAF $A$ can be taken as $F$. Secondly, by Corollary \ref{SG_det_cor_cf}, the time-changed process $\check{X}$ of $X$ by $A$ corresponds to the regular Dirichlet form $(\check{\calE},\check{\calF})$ on $L^2(F;\mu)$.

Return to our case, $\nu$ is a probability measure of finite energy integral, hence $\nu$ is a positive Radon measure charging no $\calE_{\bar{X}}$-polar set. We need to check absolutely continuous condition in Lemma \ref{SG_det_lem_ac}. We give a theorem as follows.
\begin{mythm}\label{SG_det_thm_hit}
The hitting distributions of $\myset{\bar{X}_t}$ and $\myset{Z_n}$ for $\dd X$ coincide.
\end{mythm}

\begin{proof}
Recall that $\myset{X_t}$ is characterized by random walk $\myset{Y_n}$ and jumping times $\myset{J_n}$, $\myset{X_t}$ is part process of $\myset{\bar{X}_t}$ on $X$ and $\zeta=\tau_X=\sigma_{\dd X}<+\infty$, $\bbP_x$-a.s. for all $x\in X$.

First, we show that jumping times $J_n$ are stopping times of $\myset{\bar{X}_t}$ for all $n\ge0$. Let $\myset{\calF_t}$ and $\myset{\bar{\calF}_t}$ be the minimum completed admissible filtration with respect to $\myset{X_t}$ and $\myset{\bar{X}_t}$, respectively. By Proposition \ref{SG_det_prop_jump}, $J_n$ are stopping times of $\myset{X_t}$. Since for all Borel measurable set $B\subseteq\bar{X}$, we have
$$\myset{X_t\in B}=\myset{\bar{X}_t\in B\cap X}\cap\myset{t<\zeta}\in\bar{\calF}_t,$$
$\calF_t\subseteq\bar{\calF}_t$ for all $t\ge0$. $J_n$ are stopping times of $\myset{\bar{X}_t}$ for all $n\ge0$.

Then, since $J_n\uparrow\zeta=\sigma_{\dd X}$, by quasi left continuity of $\myset{\bar{X}_t}$, we have for all $x\in X$
$$\bbP_x\left[\lim_{n\to+\infty}\bar{X}_{J_n}=\bar{X}_{\sigma_{\dd X}},\sigma_{\dd X}<+\infty\right]=\bbP_x\left[\sigma_{\dd X}<+\infty\right],$$
that is,
$$\bbP_x[\lim_{n\to+\infty}\bar{X}_{J_n}=\bar{X}_{\sigma_{\dd X}}]=1.$$
Noting that $J_n<\zeta=\sigma_{\dd X}$, we have $\bar{X}_{J_n}=X_{J_n}=Y_n$, hence
$$\bbP_x\left[\lim_{n\to+\infty}Y_n=\bar{X}_{\sigma_{\dd X}}\right]=1.$$
Hence the hitting distributions of $\myset{\bar{X}_t}$ and $\myset{Z_n}$ for $\dd X$ coincide under $\bbP_x$ for all $x\in X$.
\end{proof}
By Theorem \ref{SG_det_thm_hit}, the hitting distribution of $\myset{\bar{X}_t}$ for $\dd X$ under $\bbP_x$ is exactly $\nu_x$, hence
\begin{equation}\label{SG_det_eqn_hit}
\begin{aligned}
H_{\dd X}g(x)&=\bbE_x\left[g(\bar{X}_{\sigma_{\dd X}})\indi_{\myset{\sigma_{\dd X}<+\infty}}\right]=\bbE_x\left[g(\bar{X}_{\sigma_{\dd X}})\right]\\
&=\int_{\dd X}g\md\nu_x=\int_{\dd X}K(x,\cdot)g\md\nu=Hg(x),
\end{aligned}
\end{equation}
for all $x\in X$ and nonnegative Borel function $g$.

By Theorem \ref{SG_det_thm_hit} and Theorem \ref{SG_det_thm_conv}, $\nu$ satisfies the condition of Lemma \ref{SG_det_lem_ac} with $F=\mathrm{supp}(\nu)=\dd X$. By the above remark, we obtain a regular Dirichlet form $\check{\calE}$ on $L^2(\dd X;\nu)$.

We have explicit representation of $\check{\calE}$ as follows.

\begin{mythm}\label{SG_det_thm_check_E}
We have
$$
\begin{cases}
\check{\calE}(u,u)\asymp\int_K\int_K\frac{(u(x)-u(y))^2}{|x-y|^{\alpha+\beta}}\nu(\md x)\nu(\md y)<+\infty,\\
\check{\calF}=\myset{u\in C(K):\int_K\int_K\frac{(u(x)-u(y))^2}{|x-y|^{\alpha+\beta}}\nu(\md x)\nu(\md y)<+\infty},
\end{cases}
$$
where $\beta\in(\alpha,\beta^*)$.
\end{mythm}

To prove this theorem, we need some preparation.

\begin{mylem}\label{SG_det_lem_bdy_har}
If $\lambda<1/3$, then for all $u\in C(\dd X)=C(K)$ with
$$\int_K\int_K\frac{(u(x)-u(y))^2}{|x-y|^{\alpha+\beta}}\nu(\md x)\nu(\md y)<+\infty,$$
let $v\in C(\bar{X})$ be the extended function of $Hu$ in Lemma \ref{SG_det_lem_ext}, we have $v|_{\dd X}=u$.
\end{mylem}

We need a calculation result from \cite[Theorem 5.3]{KLW17} as follows.
\begin{equation}\label{SG_det_eqn_Martin}
K(x,\xi)\asymp\lambda^{|x|-|x\wedge\xi|}(\frac{1}{2})^{-\frac{\log3}{\log2}|x\wedge\xi|}=\lambda^{|x|}\left(\frac{3}{\lambda}\right)^{|x\wedge\xi|},
\end{equation}
where $x\in X$ and $\xi\in\dd X$.

\begin{proof}
By the estimate of Na\"im kernel, we have
$$\sum_{x,y\in X}c(x,y)(Hu(x)-Hu(y))^2<+\infty,$$
hence Lemma \ref{SG_det_lem_ext} can be applied here and $v$ is well-defined. We only need to show that for all $\myset{x_n}\subseteq X$ and $\xi\in\dd X$ with $x_n\to\xi$, then $Hu(x_n)\to u(\xi)$ as $n\to+\infty$. Indeed, since
$$Hu(x)=\int_{\dd X}K(x,\eta)u(\eta)\nu(\md\eta)=\int_{\dd X}u(\eta)\nu_x(\md\eta)=\bbE_x\left[u(Z_\infty)\right],$$
we have
$$H1(x)=\int_{\dd X}K(x,\eta)\nu(\md\eta)=1$$
for all $x\in X$. Then
\begin{align*}
\lvert Hu(x_n)-u(\xi)\rvert&=\lvert\int_{\dd X}K(x_n,\eta)u(\eta)\nu(\md\eta)-u(\xi)\rvert=\lvert\int_{\dd X}K(x,\eta)(u(\eta)-u(\xi))\nu(\md\eta)\rvert\\
&\le\int_{\dd X}K(x_n,\eta)|u(\eta)-u(\xi)|\nu(\md\eta).
\end{align*}
Since $u\in C(\dd X)$, for all $\veps>0$, there exists $\delta>0$ such that for all $\eta,\xi\in\dd X$ with $\theta_a(\eta,\xi)<\delta$, we have $|u(\eta)-u(\xi)|<\veps$. Assume that $|u(x)|\le M<+\infty$ for all $x\in\dd X$, then
\begin{align*}
&\lvert Hu(x_n)-u(\xi)\rvert\\
&\le\int_{\theta_a(\eta,\xi)<\delta}K(x_n,\eta)|u(\eta)-u(\xi)|\nu(\md\eta)+\int_{\theta_a(\eta,\xi)\ge\delta}K(x_n,\eta)|u(\eta)-u(\xi)|\nu(\md\eta)\\
&<\veps\int_{\theta_a(\eta,\xi)<\delta}K(x_n,\eta)\nu(\md\eta)+2M\int_{\theta_a(\eta,\xi)\ge\delta}K(x_n,\eta)\nu(\md\eta)\\
&\le\veps+2M\int_{\theta_a(\eta,\xi)\ge\delta}K(x_n,\eta)\nu(\md\eta).
\end{align*}
There exists $N\ge1$ such that for all $n>N$, we have $\theta_a(x_n,\xi)<\delta/2$, then for all $\theta_a(\eta,\xi)\ge\delta$
$$\theta_a(x_n,\eta)\ge\theta_a(\eta,\xi)-\theta_a(x_n,\xi)\ge\delta-\frac{\delta}{2}=\frac{\delta}{2}.$$
By Equation (\ref{SG_det_eqn_Martin}), we have
\begin{align*}
K(x_n,\eta)&\asymp \lambda^{|x_n|}\left(\frac{3}{\lambda}\right)^{|x_n\wedge\eta|}=\lambda^{|x_n|}e^{|x_n\wedge\eta|\log(\frac{3}{\lambda})}=\lambda^{|x_n|}e^{-a|x_n\wedge\eta|\frac{1}{a}\log(\frac{\lambda}{3})}\\
&=\lambda^{|x_n|}\rho_a(x_n,\eta)^{\frac{1}{a}\log(\frac{\lambda}{3})}=\frac{\lambda^{|x_n|}}{\rho_a(x_n,\eta)^{\frac{1}{a}\log(\frac{3}{\lambda})}}.
\end{align*}
Since $\rho_a$ and $\theta_a$ are equivalent, there exists some positive constant $C$ independent of $x_n$ and $\eta$ such that
$$K(x_n,\eta)\le C\frac{\lambda^{|x_n|}}{\delta^{\frac{1}{a}\log(\frac{3}{\lambda})}}.$$
Hence
$$\lvert Hu(x_n)-u(\xi)\rvert\le\veps+2M\int_{\theta_a(\eta,\xi)\ge\delta}C\frac{\lambda^{|x_n|}}{\delta^{\frac{1}{a}\log(\frac{3}{\lambda})}}\nu(\md\eta)\le\veps+2MC\frac{\lambda^{|x_n|}}{\delta^{\frac{1}{a}\log(\frac{3}{\lambda})}},$$
letting $n\to+\infty$, we have $|x_n|\to+\infty$, hence
$$\varlimsup_{n\to+\infty}\lvert Hu(x_n)-u(\xi)\rvert\le\veps.$$
Since $\veps>0$ is arbitrary, we have $\lim_{n\to+\infty}Hu(x_n)=u(\xi)$.
\end{proof}

\begin{mythm}\label{SG_det_thm_f_e}
$(\calF_{\bar{X}})_e=\calF_{\bar{X}}$, here we use the convention that all functions in extended Dirichlet spaces are quasi continuous.
\end{mythm}

\begin{proof}
It is obvious that $(\calF_{\bar{X}})_e\supseteq\calF_{\bar{X}}$. For all $u\in(\calF_{\bar{X}})_e$, by definition, there exists an $\calE_{\bar{X}}$-Cauchy sequence $\myset{u_n}\subseteq\calF_{\bar{X}}$ that converges to $u$ $m$-a.e.. Hence $u_n(x)\to u(x)$ for all $x\in X$. By Fatou's lemma, we have
\begin{align*}
\frac{1}{2}\sum_{x,y\in X}c(x,y)(u(x)-u(y))^2&=\frac{1}{2}\sum_{x,y\in X}\lim_{n\to+\infty}c(x,y)(u(x)-u(y))^2\\
&\le\varliminf_{n\to+\infty}\frac{1}{2}\sum_{x,y\in X}c(x,y)(u_n(x)-u_n(y))^2\\
&=\varliminf_{n\to+\infty}\calE_{\bar{X}}(u_n,u_n)\\
&<+\infty.
\end{align*}
By Lemma \ref{SG_det_lem_ext}, $u|_X$ can be extended to a continuous function $v$ on $\bar{X}$. Since $u,v$ are quasi continuous on $\bar{X}$ and $u=v,m$-a.e., we have $u=v$ q.e., we can take $u$ as $v$. Hence $u$ can be taken continuous, $u\in\calF_{\bar{X}}$, $(\calF_{\bar{X}})_e\subseteq\calF_{\bar{X}}$.
\end{proof}

\begin{myrmk}
It was proved in \cite[Proposition 2.9]{HK06} that a result of the above type holds in much more general frameworks.
\end{myrmk}

\begin{proof}[Proof of Theorem \ref{SG_det_thm_check_E}]
By Equation (\ref{SG_det_eqn_hit}) and Equation (\ref{SG_det_eqn_Theta}), we have
\begin{align*}
\check{\calE}(u|_{\dd X},u|_{\dd X})&=\calE_{\bar{X}}(H_{\dd X}u,H_{\dd X}u)=\calE_{\bar{X}}(Hu,Hu)\\
&=\frac{1}{2}\sum_{x,y\in X}c(x,y)(Hu(x)-Hu(y))^2\\
&\asymp\int_K\int_K\frac{(u(x)-u(y))^2}{|x-y|^{\alpha+\beta}}\nu(\md x)\nu(\md y),
\end{align*}
and
\begin{align*}
\check{\calF}&=(\calF_{\bar{X}})_e|_{\dd X}\cap L^2(\dd X;\nu)=\calF_{\bar{X}}|_{\dd X}\cap L^2(\dd X;\nu)\\
&=\myset{u|_{\dd X}\in L^2(\dd X;\nu):u\in C(\bar{X}),\sum_{x,y\in X}c(x,y)(u(x)-u(y))^2<+\infty}\\
&=\myset{u|_{\dd X}:u\in C(\bar{X}),\sum_{x,y\in X}c(x,y)(u(x)-u(y))^2<+\infty}.
\end{align*}
For all $u|_{\dd X}\in\check{\calF}$, we have $H_{\dd X}u=Hu\in(\calF_{\bar{X}})_e=\calF_{\bar{X}}$, $u|_{\dd	X}\in C(\dd X)=C(K)$ and
\begin{align*}
\int_K\int_K\frac{(u(x)-u(y))^2}{|x-y|^{\alpha+\beta}}\nu(\md x)\nu(\md y)&\asymp\frac{1}{2}\sum_{x,y\in X}c(x,y)(Hu(x)-Hu(y))^2\\
&=\calE_{\bar{X}}(Hu,Hu)<+\infty,
\end{align*}
that is, $\check{\calF}\subseteq\text{RHS}$.

On the other hand, for all $u\in\text{RHS}$, we have $Hu$ satisfies
$$\frac{1}{2}\sum_{x,y\in X}c(x,y)(Hu(x)-Hu(y))^2\asymp\int_K\int_K\frac{(u(x)-u(y))^2}{|x-y|^{\alpha+\beta}}\nu(\md x)\nu(\md y)<+\infty.$$
By Lemma \ref{SG_det_lem_ext}, $Hu$ can be extended to a continuous function $v$ on $\bar{X}$, then $v\in C(\bar{X})$. By Lemma \ref{SG_det_lem_bdy_har}, we have $v|_{\dd X}=u$.
$$\frac{1}{2}\sum_{x,y\in X}c(x,y)(v(x)-v(y))^2=\frac{1}{2}\sum_{x,y\in X}c(x,y)(Hu(x)-Hu(y))^2<+\infty,$$
hence $v\in\calF_{\bar{X}}$, $u\in\check{\calF}$, $\text{RHS}\subseteq\check{\calF}$.
\end{proof}

Then we have the following result.

\begin{mythm}\label{SG_det_thm_E_K}
Let
$$
\begin{cases}
&\calE_K(u,u)=\int_K\int_K\frac{(u(x)-u(y))^2}{|x-y|^{\alpha+\beta}}\nu(\md x)\nu(\md y),\\
&\calF_K=\myset{u\in L^2(K;\nu):\int_K\int_K\frac{(u(x)-u(y))^2}{|x-y|^{\alpha+\beta}}\nu(\md x)\nu(\md y)<+\infty}.
\end{cases}
$$
If $\beta\in(\alpha,\beta^*)$, then $(\calE_K,\calF_K)$ is a regular Dirichlet form on $L^2(K;\nu)$.
\end{mythm}

\begin{proof}
Let
$$
\begin{cases}
&\calE_K(u,u)=\int_K\int_K\frac{(u(x)-u(y))^2}{|x-y|^{\alpha+\beta}}\nu(\md x)\nu(\md y),\\
&\calF_K=\myset{u\in C(K):\int_K\int_K\frac{(u(x)-u(y))^2}{|x-y|^{\alpha+\beta}}\nu(\md x)\nu(\md y)<+\infty}.
\end{cases}
$$
By Theorem \ref{SG_det_thm_check_E}, if $\beta\in(\alpha,\beta^*)$, then $(\calE_K,\calF_K)$ is a regular Dirichlet form on $L^2(K;\nu)$. We only need to show that
$$\calF_K=\myset{u\in L^2(K;\nu):\int_K\int_K\frac{(u(x)-u(y))^2}{|x-y|^{\alpha+\beta}}\nu(\md x)\nu(\md y)<+\infty}.$$
Indeed, it is obvious that $\calF_K\subseteq\text{RHS}$. On the other hand, since $\beta\in(\alpha,\beta^*)$, by \cite[Theorem 4.11 (\rmnum{3})]{GHL03}, the $\text{RHS}$ can be embedded into a H\"older space with parameter $(\beta-\alpha)/2$, hence the functions in the $\text{RHS}$ can be modified to be continuous, $\text{RHS}\subseteq\calF_K$.
\end{proof}

\begin{myrmk}
A more general result was proved in \cite{Kum03}.
\end{myrmk}

\section{Triviality of \texorpdfstring{$\calF_K$}{FK} when \texorpdfstring{$\beta\in(\beta^*,+\infty)$}{beta in (beta*,+infty)}}\label{SG_det_sec_trivial}

In this section, we show that $\calF_K$ consists only of constant functions if $\lambda\in(0,1/5)$ or $\beta\in(\beta^*,+\infty)$. Hence $\beta_*=\beta^*=\log5/\log2$.

\begin{mythm}\label{SG_det_thm_constant}
If $\lambda<{1}/{5}$, then for all continuous function $u$ on $\bar{X}$ with
\begin{equation}\label{SG_det_eqn_finite}
C=\frac{1}{2}\sum_{x,y\in X}c(x,y)(u(x)-u(y))^2<+\infty,
\end{equation}
we have $u|_{\dd X}$ is constant.
\end{mythm}

\begin{proof}
By Lemma \ref{SG_det_lem_ext}, Equation (\ref{SG_det_eqn_finite}) implies that $u|_X$ can be extended continuously to $\bar{X}$ which is exactly $u$ on $\bar{X}$. Assume that $u|_{\dd X}$ is not constant. First, we consider
$$\sum_{x,y\in S_n}c(x,y)(u(\Phi_n(x))-u(\Phi_n(y)))^2.$$
By the proof of Theorem \ref{SG_det_thm_main}, we have
$$\sum_{x,y\in S_{n+1}}c(x,y)(u(\Phi_{n+1}(x))-u(\Phi_{n+1}(y)))^2\ge\frac{1}{5\lambda}\sum_{x,y\in S_n}c(x,y)(u(\Phi_n(x))-u(\Phi_n(y)))^2.$$
Since $u|_{\dd X}$ is continuous on $\dd X$ and $u|_{\dd X}$ is not constant, there exists $N\ge1$ such that
$$\sum_{x,y\in S_N}c(x,y)(u(\Phi_N(x))-u(\Phi_N(y)))^2>0.$$
Since $\lambda<1/5$, for all $n\ge N$, we have
\begin{align*}
&\sum_{x,y\in S_n}c(x,y)(u(\Phi_n(x))-u(\Phi_n(y)))^2\\
&\ge\frac{1}{(5\lambda)^{n-N}}\sum_{x,y\in S_N}c(x,y)(u(\Phi_N(x))-u(\Phi_N(y)))^2\to+\infty,
\end{align*}
as $n\to+\infty$. Next, we consider the relation between
$$\sum_{x,y\in S_n}c(x,y)(u(x)-u(y))^2\text{ and }\sum_{x,y\in S_n}c(x,y)(u(\Phi_n(x))-u(\Phi_n(y)))^2.$$
Indeed
\begin{align*}
&\sum_{x,y\in S_n}c(x,y)(u(x)-u(y))^2\\
\le&\sum_{x,y\in S_n}c(x,y)\left(\lvert u(x)-u(\Phi_n(x))\rvert+\lvert u(\Phi_n(x))-u(\Phi_n(y))\rvert+\lvert u(\Phi_n(y))-u(y)\rvert\right)^2\\
\le&3\sum_{x,y\in S_n}c(x,y)\left((u(x)-u(\Phi_n(x)))^2+(u(\Phi_n(x))-u(\Phi_n(y)))^2+(u(\Phi_n(y))-u(y))^2\right)\\
=&3\sum_{x,y\in S_n}c(x,y)(u(\Phi_n(x))-u(\Phi_n(y)))^2\\
&+3\sum_{x,y\in S_n}c(x,y)\left((u(x)-u(\Phi_n(x)))^2+(u(\Phi_n(y))-u(y))^2\right).
\end{align*}
For all $x\in S_n$, there are at most 3 elements $y\in S_n$ such that $c(x,y)>0$ and for all $x,y\in S_n$, $c(x,y)\le\max{\myset{C_1,C_2}}/(3\lambda)^n$. By symmetry, we have
\begin{align*}
&\sum_{x,y\in S_n}c(x,y)\left((u(x)-u(\Phi_n(x)))^2+(u(\Phi_n(y))-u(y))^2\right)\\
=&2\sum_{x,y\in S_n}c(x,y)(u(x)-u(\Phi_n(x)))^2\le6\sum_{x\in S_n}\frac{\max{\myset{C_1,C_2}}}{(3\lambda)^n}(u(x)-u(\Phi_n(x)))^2.
\end{align*}
For all $x\in S_n$, there exists a geodesic ray $[x_0,x_1,\ldots]$ with $|x_k|=k$, $x_k\sim x_{k+1}$ for all $k\ge0$ and $x_n=x$, $x_k\to\Phi_n(x)$ as $k\to+\infty$. For distinct $x,y\in S_n$, the corresponding geodesic rays $[x_0,x_1,\ldots]$, $[y_0,y_1,\ldots]$ satisfy $x_k\ne y_k$ for all $k\ge n$. Then
\begin{align*}
\frac{1}{(3\lambda)^n}(u(x)-u(\Phi_n(x)))^2&\le\frac{1}{(3\lambda)^n}\left(\sum_{k=n}^\infty|u(x_k)-u(x_{k+1})|\right)^2\\
&=\left(\sum_{k=n}^\infty\frac{1}{(3\lambda)^{n/2}}|u(x_k)-u(x_{k+1})|\right)^2\\
&=\left(\sum_{k=n}^\infty{(3\lambda)^{(k-n)/2}}\frac{1}{(3\lambda)^{k/2}}|u(x_k)-u(x_{k+1})|\right)^2\\
&\le\sum_{k=n}^\infty(3\lambda)^{k-n}\sum_{k=n}^\infty\frac{1}{(3\lambda)^k}(u(x_k)-u(x_{k+1}))^2\\
&=\frac{1}{1-3\lambda}\sum_{k=n}^\infty c(x_k,x_{k+1})(u(x_k)-u(x_{k+1}))^2,
\end{align*}
hence
\begin{align*}
\sum_{x\in S_n}\frac{1}{(3\lambda)^n}(u(x)-u(\Phi_n(x)))^2&\le\frac{1}{1-3\lambda}\sum_{x\in S_n}\sum_{k=n}^\infty c(x_k,x_{k+1})(u(x_k)-u(x_{k+1}))^2\\
&\le\frac{1}{1-3\lambda}\left(\frac{1}{2}\sum_{x,y\in X}c(x,y)(u(x)-u(y))^2\right)\\
&=\frac{1}{1-3\lambda}C.
\end{align*}
We have
$$
\sum_{x,y\in S_n}c(x,y)\left((u(x)-u(\Phi_n(x)))^2+(u(\Phi_n(y))-u(y))^2\right)\le\frac{6\max{\myset{C_1,C_2}}}{1-3\lambda}C,
$$
and
$$
\sum_{x,y\in S_n}c(x,y)(u(x)-u(y))^2\le 3\sum_{x,y\in S_n}c(x,y)(u(\Phi_n(x))-u(\Phi_n(y)))^2+\frac{18\max{\myset{C_1,C_2}}}{1-3\lambda}C.
$$
Similarly, we have
$$
\sum_{x,y\in S_n}c(x,y)(u(\Phi_n(x))-u(\Phi_n(y)))^2\le 3\sum_{x,y\in S_n}c(x,y)(u(x)-u(y))^2+\frac{18\max{\myset{C_1,C_2}}}{1-3\lambda}C.
$$
Since
$$\lim_{n\to+\infty}\sum_{x,y\in S_n}c(x,y)(u(\Phi_n(x))-u(\Phi_n(y)))^2=+\infty,$$
we have
$$\lim_{n\to+\infty}\sum_{x,y\in S_n}c(x,y)(u(x)-u(y))^2=+\infty.$$
Therefore
$$C=\frac{1}{2}\sum_{x,y\in X}c(x,y)(u(x)-u(y))^2=+\infty,$$
contradiction! Hence $u|_{\dd X}$ is constant.
\end{proof}

\begin{mythm}\label{SG_det_thm_ub}
If $\lambda\in(0,1/5)$ or $\beta\in(\beta^*,+\infty)$, then $(\calE_K,\calF_K)$ on $L^2(K;\nu)$ is trivial, that is, $\calF_K$ consists only of constant functions. Hence the walk dimension of the SG $\beta_*=\beta^*=\log5/\log2$.
\end{mythm}

\begin{proof}
For all $u\in\calF_K$, let $v=Hu$ on $X$, then we have
$$\frac{1}{2}\sum_{x,y\in X}c(x,y)(v(x)-v(y))^2<+\infty.$$
Since $\lambda<1/5<1/3$, by Lemma \ref{SG_det_lem_ext}, $v$ on $X$ can be extended continuously to $\bar{X}$ still denoted by $v$. By Lemma \ref{SG_det_lem_bdy_har}, we have $v|_{\dd X}=u$. By Theorem \ref{SG_det_thm_constant}, we have $v|_{\dd X}$ is constant, hence $u$ is constant, $\calF_K$ consists only of constant functions.
\end{proof}

\begin{myrmk}
We can not have the triviality of $\calF_K$ when $\beta=\beta^*$ from the above proof. We will consider the case $\beta=\beta^*$ in Theorem \ref{SG_app_thm_det} and Theorem \ref{SG_con_thm_nonlocal}.
\end{myrmk}

\chapter{Approximation of Dirichlet Forms on the SG}\label{ch_SG_app}

This chapter is based on my work \cite{MY17}.

\section{Background and Statement}

We use the notions of the SG introduced in Section \ref{sec_notion}.

Let $\nu$ be the normalized Hausdorff measure on the SG $K$.

Let $(\calE_\beta,\calF_\beta)$ be given by
\begin{align*}
&\calE_\beta(u,u)=\int_K\int_K\frac{(u(x)-u(y))^2}{|x-y|^{\alpha+\beta}}\nu(\md x)\nu(\md y),\\
&\calF_\beta=\myset{u\in C(K):\calE_\beta(u,u)<+\infty}.
\end{align*}
where $\alpha=\log3/\log2$ is the Hausdorff dimension of the SG. By Chapter \ref{ch_SG_det}, $(\calE_\beta,\calF_\beta)$ is a non-local regular Dirichlet form on $L^2(K;\nu)$ for all $\beta\in(\alpha,\beta^*)$, where $\beta^*=\log5/\log2$ is the walk dimension of the SG (see also \cite{Kum03} using heat kernel estimates and subordination technique and \cite{KL16} using effective resistance on graph). We give another simple proof in this chapter.

Let $(\frakE_{\loc},\frakF_{\loc})$ be given by
\begin{align*}
&\frakE_\loc(u,u)=\lim_{n\to+\infty}\left(\frac{5}{3}\right)^n\sum_{w\in W_n}\sum_{p,q\in V_w}(u(p)-u(q))^2,\\
&\frakF_\loc=\myset{u\in C(K):\frakE_\loc(u,u)<+\infty}.
\end{align*}
Then $(\frakE_\loc,\frakF_\loc)$ is a self-similar local regular Dirichlet form on $L^2(K;\nu)$ which corresponds to the diffusion on the SG, see Theorem \ref{thm_SG_con}.

Analog to (\ref{eqn_classical}), one may expect that $(\beta^*-\beta)\calE_\beta(u,u)$ converges to $\frakE_\loc(u,u)$. However, this is not known. Using the sub-Gaussian estimates for the heat kernel of $\frakE_\loc$, it was shown in \cite[Theorem 3.1]{Pie08}, \cite[2.1]{Kum03} that the Dirichlet form $\tilde{\calE}_\beta$ that is obtained from $\frakE_\loc$ by subordination of order $\beta/\beta^*$ has the following properties
\begin{equation}\label{SG_app_eqn_sub}
\begin{aligned}
&\tilde{\calE}_\beta(u,u)\asymp(\beta^*-\beta)\calE_\beta(u,u),\\
&\tilde{\calE}_\beta(u,u)\to\frakE_\loc(u,u)\text{ as }\beta\uparrow\beta^*.
\end{aligned}
\end{equation}
Moreover, the jump kernel of $\tilde{\calE}_\beta$ is of the order $|x-y|^{-(\alpha+\beta)}$ for all $\beta\in(0,\beta^*)$.

In this chapter, we construct explicitly a different semi-norm $E_\beta$ of jump type that has properties similar to (\ref{SG_app_eqn_sub}). Our construction has the following two advantages. Firstly, our construction is independent of any knowledge about the local Dirichlet form $\frakE_\loc$ except for its definition. Secondly, we obtain a monotone convergence result for all functions in $L^2(K;\nu)$ which implies a Mosco convergence. While \cite[Theorem 3.1]{Pie08} only gave a convergence result for the functions in $\frakF_\loc$.

The new semi-norm $E_\beta$ is defined as follows.
$$E_\beta(u,u):=\sum_{n=1}^\infty2^{(\beta-\alpha)n}\sum_{w\in W_n}\sum_{p,q\in V_w}(u(p)-u(q))^2.$$
We state the main results in the next two theorems. Our first main result is as follows.
\begin{mythm}\label{SG_app_thm_main}
For all $\beta\in(\alpha,+\infty)$, for all $u\in C(K)$, we have
$$E_\beta(u,u)\asymp\calE_\beta(u,u).$$
\end{mythm}

Recall that a similar result for the unit interval was proved in \cite{Kam97} as follows. Let $I=[0,1]$. Then for all $\beta\in(1,+\infty)$, for all $u\in C(I)$, we have
\begin{equation}\label{SG_app_eqn_interval}
\sum_{n=1}^\infty2^{(\beta-1)n}\sum_{i=0}^{2^n-1}(u(\frac{i}{2^n})-u(\frac{i+1}{2^n}))^2\asymp\int_{I}\int_I\frac{(u(x)-u(y))^2}{|x-y|^{1+\beta}}\md x\md y.
\end{equation}

Consider the convergence problem. Assume that $(\calE,\calF)$ is a quadratic form on $L^2(K;\nu)$ where the energy $\calE$ has an explicit expression and the domain $\calF\subseteq C(K)$. We use the convention to extent $\calE$ to $L^2(K;\nu)$ as follows. For all $u\in L^2(K;\nu)$, $u$ has at most one continuous version. If $u$ has a continuous version $\tilde{u}$, then we define $\calE(u,u)$ as the energy of $\tilde{u}$ using its explicit expression which might be $+\infty$, if $u$ has no continuous version, then we define $\calE(u,u)$ as $+\infty$.

It is obvious that $\calF_{\beta_1}\supseteq\calF_{\beta_2}\supseteq\frakF_\loc$ for all $\alpha<\beta_1<\beta_2<\beta^*$. We use Theorem \ref{SG_app_thm_main} to answer the question about convergence as follows.

\begin{mythm}\label{SG_app_thm_incre}
For all $u\in L^2(K;\nu)$, we have
$$(1-5^{-1}\cdot 2^{\beta})E_\beta(u,u)\uparrow\frakE_\loc(u,u)$$
as $\beta\uparrow\beta^*=\log5/\log2$.
\end{mythm}

Moreover, we also have a Mosco convergence.

\begin{mythm}\label{SG_app_thm_conv_main}
For all sequence $\myset{\beta_n}\subseteq(\alpha,\beta^*)$ with $\beta_n\uparrow\beta^*$, we have $(1-5^{-1}\cdot 2^{\beta_n})E_{\beta_n}\to\frakE_\loc$ in the sense of Mosco.
\end{mythm}

As a byproduct of Theorem \ref{SG_app_thm_main}, we obtain the following result about a trace problem. Let us introduce the notion of Besov spaces. Let $(M,d,\mu)$ be a metric measure space and $\alpha,\beta>0$ two parameters. Let
$$B^{2,2}_{\alpha,\beta}(M)=\myset{u\in L^2(M;\mu):\sum_{n=0}^\infty2^{(\alpha+\beta)n}\int\limits_M\int\limits_{d(x,y)<2^{-n}}(u(x)-u(y))^2\mu(\md y)\mu(\md x)<+\infty}.$$
If $\beta>\alpha$, then $B^{2,2}_{\alpha,\beta}(M)$ can be embedded in $C^{\frac{\beta-\alpha}{2}}(M)$. We regard the SG $K$ and the unit interval $I$ as metric measure spaces with Euclidean metrics and normalized Hausdorff measures. Let $\alpha_1=\log3/\log2$, $\alpha_2=1$ be the Hausdorff dimensions, $\beta_1^*=\log5/\log2$, $\beta_2^*=2$ the walk dimensions of $K$ and $I$, respectively.

Let us identify $I$ as the segment $[p_0,p_1]\subseteq K$. Choose some $\beta_1\in(\alpha_1,\beta_1^*)$. Any function $u\in B^{2,2}_{\alpha_1,\beta_1}(K)$ is continuous on $K$ and, hence, has the trace $u|_I$ on $I$. The trace problem is the problem of identifying the space of all traces $u|_I$ of all functions $u\in B_{\alpha_1,\beta_1}^{2,2}(K)$. This problem was considered by A. Jonsson using general Besov spaces in $\R^n$, see remarks after \cite[Theorem 3.1]{Jon05}. The following result follows from \cite{Jon05}.

\begin{mythm}\label{SG_app_thm_trace_main}
Let $\beta_1,\beta_2$ satisfy $\beta_1\in(\alpha_1,\beta_1^*)$ and $\beta_1-\alpha_1=\beta_2-\alpha_2$. Then the trace space of $B^{2,2}_{\alpha_1,\beta_1}(K)$ to $I$ is $B^{2,2}_{\alpha_2,\beta_2}(I)$.
\end{mythm}

We give here a new short proof of Theorem \ref{SG_app_thm_trace_main} using Theorem \ref{SG_app_thm_main}.

Finally, we construct explicitly a sequence of non-local Dirichlet forms with jumping kernels equivalent to $|x-y|^{-\alpha-\beta}$ that converges \emph{exactly} to the local Dirichlet form. We need some notions as follows. For all $n\ge1$, $w=w_1\ldots w_n\in W_n$ and $p\in V_w$, we have $p=P_{w_1\ldots w_nw_{n+1}}$ for some $w_{n+1}\in\myset{0,1,2}$. Let $\gamma\ge1$ be an integer, define
$$K^{(i)}_{p,n}=K_{w_1\ldots w_nw_{n+1}\ldots w_{n+1}},i\ge1,$$
with $\gamma ni$ terms of $w_{n+1}$.

\begin{mythm}\label{SG_app_thm_jumping_kernel}
For all sequence $\myset{\beta_i}\subseteq(\alpha,\beta^*)$ with $\beta_i\uparrow\beta^*$, there exist positive functions $a_i$ bounded from above and below by positive constants given by
$$a_i=\delta_iC_i+(1-\delta_i),$$
where $\myset{\delta_i}\subseteq(0,1)$ is an arbitrary sequence with $\delta_i\uparrow1$ and
$$C_i(x,y)=\sum_{n=1}^{\Phi(i)}2^{-2\alpha n}\sum_{w\in W_n}\sum_{p,q\in V_w}\frac{1}{\nu(K^{(i)}_{p,n})\nu(K^{(i)}_{q,n})}1_{K^{(i)}_{p,n}}(x)1_{K^{(i)}_{q,n}}(y),$$
where $\Phi:\mathbb{N}\to\mathbb{N}$ is increasing and $(1-5^{-1}\cdot2^{\beta_i})\Phi(i)\ge i$ for all $i\ge1$. Then for all $u\in \mathcal{F}_\loc$, we have
$$\lim_{i\to+\infty}(1-5^{-1}\cdot2^{\beta_i})\iint_{K\times K\backslash\mathrm{diag}}\frac{a_i(x,y)(u(x)-u(y))^2}{|x-y|^{\alpha+\beta_i}}\nu(\md x)\nu(\md y)=\calE_\loc(u,u).$$
\end{mythm}

\begin{myrmk}
The shape of the function $C_i$ reflects the inhomogeneity of the fractal structure with respect to the Euclidean structure. Of course, subordination technique in \cite{Pie08} ensures the existence of the functions $a_i$, but Theorem \ref{SG_app_thm_jumping_kernel} provides them explicitly.
\end{myrmk}

\section{Proof of Theorem \ref{SG_app_thm_main}}

First, we give other equivalent semi-norms which are more convenient for later use.
\begin{mylem}\label{SG_app_lem_equiv1}
For all $u\in L^2(K;\nu)$, we have
$$\int_K\int_K\frac{(u(x)-u(y))^2}{|x-y|^{\alpha+\beta}}\nu(\md x)\nu(\md y)\asymp\sum_{n=0}^\infty2^{(\alpha+\beta)n}\int_K\int_{B(x,2^{-n})}(u(x)-u(y))^2\nu(\md y)\nu(\md x).$$
\end{mylem}
\begin{proof}
On the one hand, we have
\begin{align*}
&\int_K\int_K\frac{(u(x)-u(y))^2}{|x-y|^{\alpha+\beta}}\nu(\md x)\nu(\md y)\\
=&\int_K\int_{B(x,1)}\frac{(u(x)-u(y))^2}{|x-y|^{\alpha+\beta}}\nu(\md y)\nu(\md x)\\
=&\sum_{n=0}^\infty\int_K\int_{B(x,2^{-n})\backslash{B(x,2^{-(n+1)})}}\frac{(u(x)-u(y))^2}{|x-y|^{\alpha+\beta}}\nu(\md y)\nu(\md x)\\
\le&\sum_{n=0}^\infty2^{(\alpha+\beta)(n+1)}\int_K\int_{B(x,2^{-n})\backslash{B(x,2^{-(n+1)})}}{(u(x)-u(y))^2}\nu(\md y)\nu(\md x)\\
\le&2^{\alpha+\beta}\sum_{n=0}^\infty2^{(\alpha+\beta)n}\int_K\int_{B(x,2^{-n})}{(u(x)-u(y))^2}\nu(\md y)\nu(\md x).
\end{align*}
On the other hand, we have
\begin{align*}
&\sum_{n=0}^\infty2^{(\alpha+\beta)n}\int_K\int_{B(x,2^{-n})}{(u(x)-u(y))^2}\nu(\md y)\nu(\md x)\\
=&\sum_{n=0}^\infty\sum_{k=n}^\infty2^{(\alpha+\beta)n}\int_K\int_{B(x,2^{-k})\backslash{B(x,2^{-(k+1)})}}{(u(x)-u(y))^2}\nu(\md y)\nu(\md x)\\
=&\sum_{k=0}^\infty\sum_{n=0}^k 2^{(\alpha+\beta)n}\int_K\int_{B(x,2^{-k})\backslash{B(x,2^{-(k+1)})}}{(u(x)-u(y))^2}\nu(\md y)\nu(\md x)\\
\le&\sum_{k=0}^\infty\frac{2^{(\alpha+\beta)(k+1)}}{2^{\alpha+\beta}-1}\int_K\int_{B(x,2^{-k})\backslash{B(x,2^{-(k+1)})}}{(u(x)-u(y))^2}\nu(\md y)\nu(\md x)\\
\le&\frac{2^{\alpha+\beta}}{2^{\alpha+\beta}-1}\sum_{k=0}^\infty\int_K\int_{B(x,2^{-k})\backslash{B(x,2^{-(k+1)})}}\frac{(u(x)-u(y))^2}{|x-y|^{\alpha+\beta}}\nu(\md y)\nu(\md x)\\
=&\frac{2^{\alpha+\beta}}{2^{\alpha+\beta}-1}\int_K\int_{K}\frac{(u(x)-u(y))^2}{|x-y|^{\alpha+\beta}}\nu(\md x)\nu(\md y).
\end{align*}
\end{proof}

Moreover, we have

\begin{mycor}\label{SG_app_cor_arbi}
Fix arbitrary integer $N\ge0$ and real number $c>0$. For all $u\in L^2(K;\nu)$, we have
$$\int_K\int_K\frac{(u(x)-u(y))^2}{|x-y|^{\alpha+\beta}}\nu(\md x)\nu(\md y)\asymp\sum_{n=N}^\infty2^{(\alpha+\beta)n}\int_K\int_{B(x,c2^{-n})}(u(x)-u(y))^2\nu(\md y)\nu(\md x).$$
\end{mycor}

\begin{proof}
We only need to show that for all $n\ge1$, there exists some positive constant $C=C(n)$ such that
$$\int_K\int_K(u(x)-u(y))^2\nu(\md x)\nu(\md y)\le C\int_K\int_{B(x,{2^{-n}})}(u(x)-u(y))^2\nu(\md y)\nu(\md x).$$
Indeed, since the SG satisfies the chain condition, see \cite[Definition 3.4]{GHL03}, that is, there exists a positive constant $C_1$ such that for all $x,y\in K$, for all integer $N\ge1$ there exist $z_0,\ldots,z_N\in K$ with $z_0=x,z_N=y$ and
$$|z_i-z_{i+1}|\le C_1\frac{|x-y|}{N}\text{ for all }i=0,\ldots,N-1.$$
Take integer $N\ge2^{n+2}C_1+1$. Fix $x,y\in K$, there exist $z_0,\ldots,z_N$ with $z_0=x,z_N=y$ and
$$|z_i-z_{i+1}|\le C_1\frac{|x-y|}{N}\le 2^{-(n+2)}\text{ for all }i=0,\ldots,N-1.$$
For all $i=0,\ldots,N-1$, for all $x_i\in B(z_i,2^{-(n+2)})$, $x_{i+1}\in B(z_{i+1},2^{-(n+2)})$, we have
$$|x_i-x_{i+1}|\le|x_i-z_i|+|z_i-z_{i+1}|+|z_{i+1}-x_{i+1}|\le{3}\cdot{2^{-(n+2)}}<2^{-n}.$$
Fix $x_0=z_0=x$, $x_N=z_N=y$, note that
$$(u(x)-u(y))^2=(u(x_0)-u(x_N))^2\le N\sum_{i=0}^{N-1}(u(x_i)-u(x_{i+1}))^2.$$
Integrating with respect to $x_1\in B(z_1,2^{-(n+2)}),\ldots,x_{N-1}\in B(z_{N-1},2^{-(n+2)})$ and dividing by $\nu(B(z_1,2^{-(n+2)})),\ldots,\nu(B(z_{N-1},2^{-(n+2)}))$, we have
\begin{align*}
(u(x)-u(y))^2&\le N\left(\frac{1}{\nu(B(z_1,2^{-(n+2)}))}\int_{B(z_1,2^{-(n+2)})}(u(x_0)-u(x_1))^2\nu(\md x_1)\right.\\
&+\frac{1}{\nu(B(z_{N-1},2^{-(n+2)}))}\int_{B(z_{N-1},2^{-(n+2)})}(u(x_{N-1})-u(x_N))^2\nu(\md x_{N-1})\\
&+\sum_{i=1}^{N-2}\frac{1}{\nu(B(z_i,2^{-(n+2)}))\nu(B(z_{i+1},2^{-(n+2)}))}\\
&\left.\int_{B(z_i,2^{-(n+2)})}\int_{B(z_{i+1},2^{-(n+2)})}(u(x_i)-u(x_{i+1}))^2\nu(\md x_i)\nu(\md x_{i+1})\right).
\end{align*}
Noting that $\nu(B(z_i,2^{-(n+2)}))\asymp 2^{-\alpha n}$ for all $i=1,\ldots,N-1$, we have
\begin{align*}
(u(x)-u(y))^2&\le C_2\left(\int_{B(x,2^{-n})}(u(x)-u(x_1))^2\nu(\md x_1)\right.\\
&+\int_{B(y,2^{-n})}(u(y)-u(x_{N-1}))^2\nu(\md x_{N-1})\\
&\left.+\int_K\int_{B(x,2^{-n})}(u(x)-u(y))^2\nu(\md y)\nu(\md x)\right),
\end{align*}
where $C_2=C_2(n)$ is some positive constant. Since $\nu(K)=1$, integrating with respect to $x,y\in K$, we have
$$\int_K\int_K(u(x)-u(y))^2\nu(\md x)\nu(\md y)\le 4C_2\int_K\int_{B(x,{2^{-n}})}(u(x)-u(y))^2\nu(\md y)\nu(\md x).$$
Letting $C=4C_2$, then we have the desired result.
\end{proof}

The following result states that a Besov space can be embedded in some H\"older space.
\begin{mylem}\label{SG_app_lem_holder}(\cite[Theorem 4.11 (iii)]{GHL03})
For all $u\in C(K)$, let
$$E(u)=\int_K\int_K\frac{(u(x)-u(y))^2}{|x-y|^{\alpha+\beta}}\nu(\md x)\nu(\md y).$$
Then there exists some positive constant $c$ such that
$$|u(x)-u(y)|^2\le cE(u)|x-y|^{\beta-\alpha},$$
for all $x,y\in K$, for all $u\in C(K)$.
\end{mylem}

Note that the proof of the above lemma does not rely on heat kernel.

We divide Theorem \ref{SG_app_thm_main} into the following Theorem \ref{SG_app_thm_equiv2_1} and Theorem \ref{SG_app_thm_equiv2_2}. The idea of the proofs of these theorems comes from \cite{Jon96} where the case of local Dirichlet form was considered.

\begin{mythm}\label{SG_app_thm_equiv2_1}
For all $u\in C(K)$, we have
$$\sum_{n=1}^\infty2^{(\beta-\alpha)n}\sum_{w\in W_n}\sum_{p,q\in V_w}(u(p)-u(q))^2\lesssim\int_K\int_K\frac{(u(x)-u(y))^2}{|x-y|^{\alpha+\beta}}\nu(\md x)\nu(\md y).$$
\end{mythm}

\begin{proof}
First, fix $n\ge1$ and $w=w_1\ldots w_n\in W_n$, consider $\sum_{p,q\in V_w}(u(p)-u(q))^2$. For all $x\in K_w$, we have
$$(u(p)-u(q))^2\le 2(u(p)-u(x))^2+2(u(x)-u(q))^2.$$
Integrating with respect to $x\in K_w$ and dividing by $\nu(K_w)$, we have
$$(u(p)-u(q))^2\le\frac{2}{\nu(K_w)}\int_{K_w}(u(p)-u(x))^2\nu(\md x)+\frac{2}{\nu(K_w)}\int_{K_w}(u(q)-u(x))^2\nu(\md x),$$
hence
\begin{align*}
&\sum_{p,q\in V_w}(u(p)-u(q))^2\\
\le&\sum_{p,q\in V_w,p\ne q}\left[\frac{2}{\nu(K_w)}\int_{K_w}(u(p)-u(x))^2\nu(\md x)+\frac{2}{\nu(K_w)}\int_{K_w}(u(q)-u(x))^2\nu(\md x)\right]\\
\le&2\cdot2\cdot2\sum_{p\in V_w}\frac{1}{\nu(K_w)}\int_{K_w}(u(p)-u(x))^2\nu(\md x).
\end{align*}
Consider $(u(p)-u(x))^2$, $p\in V_w$, $x\in K_w$. There exists $w_{n+1}\in\myset{0,1,2}$ such that $p=f_{w_1}\circ\ldots\circ f_{w_n}(p_{w_{n+1}})$. Let $k,l\ge1$ be integers to be determined later, let
$$w^{(i)}=w_1\ldots w_nw_{n+1}\ldots w_{n+1}$$
with $ki$ terms of $w_{n+1}$, $i=0,\ldots,l$. For all $x^{(i)}\in K_{w^{(i)}}$, $i=0,\ldots,l$, we have
\begin{align*}
(u(p)-u(x^{(0)}))^2&\le2(u(p)-u(x^{(l)}))^2+2(u(x^{(0)})-u(x^{(l)}))^2\\
&\le2(u(p)-u(x^{(l)}))^2+2\left[2(u(x^{(0)})-u(x^{(1)}))^2+2(u(x^{(1)})-u(x^{(l)}))^2\right]\\
&=2(u(p)-u(x^{(l)}))^2+2^2(u(x^{(0)})-u(x^{(1)}))^2+2^2(u(x^{(1)})-u(x^{(l)}))^2\\
&\le\ldots\le2(u(p)-u(x^{(l)}))^2+2^2\sum_{i=0}^{l-1}2^i(u(x^{(i)})-u(x^{(i+1)}))^2.
\end{align*}
Integrating with respect to $x^{(0)}\in K_{w^{(0)}}$, \ldots, $x^{(l)}\in K_{w^{(l)}}$ and dividing by $\nu(K_{w^{(0)}})$, \ldots, $\nu(K_{w^{(l)}})$, we have
\begin{align*}
&\frac{1}{\nu(K_{w^{(0)}})}\int_{K_{w^{(0)}}}(u(p)-u(x^{(0)}))^2\nu(\md x^{(0)})\\
\le&\frac{2}{\nu(K_{w^{(l)}})}\int_{K_{w^{(l)}}}(u(p)-u(x^{(l)}))^2\nu(\md x^{(l)})\\
&+2^2\sum_{i=0}^{l-1}\frac{2^i}{\nu(K_{w^{(i)}})\nu(K_{w^{(i+1)}})}\int_{K_{w^{(i)}}}\int_{K_{w^{(i+1)}}}(u(x^{(i)})-u(x^{(i+1)}))^2\nu(\md x^{(i)})\nu(\md x^{(i+1)}).
\end{align*}
Now let us use $\nu(K_{w^{(i)}})=(1/3)^{n+ki}=2^{-\alpha(n+ki)}$. For the first term, by Lemma \ref{SG_app_lem_holder}, we have
\begin{align*}
\frac{1}{\nu(K_{w^{(l)}})}\int_{K_{w^{(l)}}}(u(p)-u(x^{(l)}))^2\nu(\md x^{(l)})&\le \frac{cE(u)}{\nu(K_{w^{(l)}})}\int_{K_{w^{(l)}}}|p-x^{(l)}|^{\beta-\alpha}\nu(\md x^{(l)})\\
&\le cE(u){2}^{-(\beta-\alpha)(n+kl)}.
\end{align*}
For the second term, for all $x^{(i)}\in K_{w^{(i)}},x^{(i+1)}\in K_{w^{(i+1)}}$, we have $|x^{(i)}-x^{(i+1)}|\le2^{-(n+ki)}$, hence
\begin{align*}
&\sum_{i=0}^{l-1}\frac{2^i}{\nu(K_{w^{(i)}})\nu(K_{w^{(i+1)}})}\int_{K_{w^{(i)}}}\int_{K_{w^{(i+1)}}}(u(x^{(i)})-u(x^{(i+1)}))^2\nu(\md x^{(i)})\nu(\md x^{(i+1)})\\
\le&\sum_{i=0}^{l-1}{2^{i+\alpha(n+ki+n+ki+k)}}\int_{K_{w^{(i)}}}\int_{|x^{(i+1)}-x^{(i)}|\le2^{-n-ki}}(u(x^{(i)})-u(x^{(i+1)}))^2\nu(\md x^{(i)})\nu(\md x^{(i+1)})\\
=&\sum_{i=0}^{l-1}{2^{i+\alpha k+2\alpha(n+ki)}}\int_{K_{w^{(i)}}}\int_{|x^{(i+1)}-x^{(i)}|\le2^{-(n+ki)}}(u(x^{(i)})-u(x^{(i+1)}))^2\nu(\md x^{(i)})\nu(\md x^{(i+1)}),
\end{align*}
and
\begin{align*}
&\frac{1}{\nu(K_w)}\int_{K_w}(u(p)-u(x))^2\nu(\md x)=\frac{1}{\nu(K_{w^{(0)}})}\int_{K_{w^{(0)}}}(u(p)-u(x^{(0)}))^2\nu(\md x^{(0)})\\
\le& 2cE(u)2^{-(\beta-\alpha)(n+kl)}\\
&+4\sum_{i=0}^{l-1}{2^{i+\alpha k+2\alpha(n+ki)}}\int_{K_{w^{(i)}}}\int_{|x^{(i+1)}-x^{(i)}|\le2^{-n-ki}}(u(x^{(i)})-u(x^{(i+1)}))^2\nu(\md x^{(i)})\nu(\md x^{(i+1)}).
\end{align*}
Hence
\begin{align*}
&\sum_{w\in W_n}\sum_{p,q\in V_w}(u(p)-u(q))^2\\
\le&8\sum_{w\in W_n}\sum_{p\in V_w}\frac{1}{\nu(K_w)}\int_{K_w}(u(p)-u(x))^2\nu(\md x)\\
\le&8\sum_{w\in W_n}\sum_{p\in V_w}\left(2cE(u)2^{-(\beta-\alpha)(n+kl)}\right.\\
&\left.+4\sum_{i=0}^{l-1}{2^{i+\alpha k+2\alpha(n+ki)}}\int_{K_{w^{(i)}}}\int_{|x^{(i+1)}-x^{(i)}|\le2^{-n-ki}}(u(x^{(i)})-u(x^{(i+1)}))^2\nu(\md x^{(i)})\nu(\md x^{(i+1)})\right).
\end{align*}
For the first term, we have
$$\sum_{w\in W_n}\sum_{p\in V_w}2^{-(\beta-\alpha)(n+kl)}=3\cdot3^n\cdot2^{-(\beta-\alpha)(n+kl)}=3\cdot2^{\alpha n-(\beta-\alpha)(n+kl)}.$$
For the second term, fix $i=0,\ldots,l-1$, different $p\in V_w$, $w\in W_n$ correspond to different $K_{w^{(i)}}$, hence
\begin{align*}
&\sum_{i=0}^{l-1}\sum_{w\in W_n}\sum_{p\in V_w}2^{i+\alpha k+2\alpha(n+ki)}\\
&\cdot\int_{K_{w^{(i)}}}\int_{|x^{(i+1)}-x^{(i)}|\le2^{-n-ki}}(u(x^{(i)})-u(x^{(i+1)}))^2\nu(\md x^{(i)})\nu(\md x^{(i+1)})\\
\le&\sum_{i=0}^{l-1}2^{i+\alpha k+2\alpha(n+ki)}\int_{K}\int_{|x^{(i+1)}-x^{(i)}|\le2^{-(n+ki)}}(u(x^{(i)})-u(x^{(i+1)}))^2\nu(\md x^{(i)})\nu(\md x^{(i+1)})\\
=&2^{\alpha k}\sum_{i=0}^{l-1}2^{i-(\beta-\alpha)ki-(\beta-\alpha)n}\\
&\cdot\left(2^{(\alpha+\beta)(n+ki)}\int_{K}\int_{|x^{(i+1)}-x^{(i)}|\le2^{-(n+ki)}}(u(x^{(i)})-u(x^{(i+1)}))^2\nu(\md x^{(i)})\nu(\md x^{(i+1)})\right).
\end{align*}
For simplicity, denote
$$E_{n}(u)=2^{(\alpha+\beta)n}\int_{K}\int_{|x-y|\le2^{-n}}(u(x)-u(y))^2\nu(\md x)\nu(\md y).$$
We have
\begin{align*}
&\sum_{w\in W_n}\sum_{p,q\in V_w}(u(p)-u(q))^2\\
\le&48cE(u)\cdot2^{\alpha n-(\beta-\alpha)(n+kl)}+32\cdot2^{\alpha k}\sum_{i=0}^{l-1}2^{i-(\beta-\alpha)ki-(\beta-\alpha)n}E_{n+ki}(u).
\end{align*}
Hence
\begin{align*}
&\sum_{n=1}^\infty2^{(\beta-\alpha)n}\sum_{w\in W_n}\sum_{p,q\in V_w}(u(p)-u(q))^2\\
\le&48cE(u)\sum_{n=1}^\infty2^{\beta n-(\beta-\alpha)(n+kl)}+32\cdot2^{\alpha k}\sum_{n=1}^\infty\sum_{i=0}^{l-1}2^{i-(\beta-\alpha)ki}E_{n+ki}(u).
\end{align*}
Take $l=n$, then
\begin{align*}
&\sum_{n=1}^\infty2^{(\beta-\alpha)n}\sum_{w\in W_n}\sum_{p,q\in V_w}(u(p)-u(q))^2\\
\le&48cE(u)\sum_{n=1}^\infty2^{[\beta-(\beta-\alpha)(k+1)]n}+32\cdot2^{\alpha k}\sum_{n=1}^\infty\sum_{i=0}^{n-1}2^{i-(\beta-\alpha)ki}E_{n+ki}(u)\\
=&48cE(u)\sum_{n=1}^\infty2^{[\beta-(\beta-\alpha)(k+1)]n}+32\cdot2^{\alpha k}\sum_{i=0}^\infty2^{i-(\beta-\alpha)ki}\sum_{n=i+1}^\infty E_{n+ki}(u)\\
\le&48cE(u)\sum_{n=1}^\infty2^{[\beta-(\beta-\alpha)(k+1)]n}+32\cdot2^{\alpha k}C_1E(u)\sum_{i=0}^\infty 2^{[1-(\beta-\alpha)k]i},
\end{align*}
where $C_1$ is some positive constant from Lemma \ref{SG_app_lem_equiv1}. Take $k\ge1$ such that
$$\beta-(\beta-\alpha)(k+1)<0$$
and
$$1-(\beta-\alpha)k<0,$$
then the above two series converge, hence
$$\sum_{n=1}^\infty2^{(\beta-\alpha)n}\sum_{w\in W_n}\sum_{p,q\in V_w}(u(p)-u(q))^2\lesssim\int_K\int_K\frac{(u(x)-u(y))^2}{|x-y|^{\alpha+\beta}}\nu(\md x)\nu(\md y).$$
\end{proof}

\begin{mythm}\label{SG_app_thm_equiv2_2}
For all $u\in C(K)$, we have
\begin{equation}\label{SG_app_eqn_equiv2_1}
\int_K\int_K\frac{(u(x)-u(y))^2}{|x-y|^{\alpha+\beta}}\nu(\md x)\nu(\md y)\lesssim\sum_{n=1}^\infty2^{(\beta-\alpha)n}\sum_{w\in W_n}\sum_{p,q\in V_w}(u(p)-u(q))^2,
\end{equation}
or equivalently
\begin{equation}\label{SG_app_eqn_equiv2_2}
\sum_{n=1}^\infty2^{(\alpha+\beta)n}\int\limits_K\int\limits_{B(x,2^{-n-1})}(u(x)-u(y))^2\nu(\md y)\nu(\md x)\lesssim\sum_{n=1}^\infty2^{(\beta-\alpha)n}\sum_{w\in W_n}\sum_{p,q\in V_w}(u(p)-u(q))^2.
\end{equation}
\end{mythm}

\begin{proof}
Note $V_n=\cup_{w\in W_n}V_w$, it is obvious that its cardinal $\#V_n\asymp3^n=2^{\alpha n}$. Let $\nu_n$ be the measure on $V_n$ which assigns $1/\#V_n$ on each point of $V_n$, then $\nu_n$ converges weakly to $\nu$.

First, fix $n\ge1$ and $m\ge n$, we estimate
$$2^{(\alpha+\beta)n}\int_K\int_{B(x,2^{-n-1})}(u(x)-u(y))^2\nu_m(\md y)\nu_m(\md x).$$
Note that
\begin{align*}
&\int_K\int_{B(x,2^{-n-1})}(u(x)-u(y))^2\nu_m(\md y)\nu_m(\md x)\\
=&\sum_{w\in W_n}\int_{K_w}\int_{B(x,2^{-n-1})}(u(x)-u(y))^2\nu_m(\md y)\nu_m(\md x).
\end{align*}
Fix $w\in W_n$, there exist at most four $\tilde{w}\in W_n$ such that $K_{\tilde{w}}\cap K_w\ne\emptyset$, let
$$K_w^*=\cup_{\tilde{w}\in W_n,K_{\tilde{w}}\cap K_w\ne\emptyset}K_{\tilde{w}}.$$
For all $x\in K_w$, $y\in B(x,2^{-n-1})$, we have $y\in K_w^*$, hence
$$\int_{K_w}\int_{B(x,2^{-n-1})}(u(x)-u(y))^2\nu_m(\md y)\nu_m(\md x)\le\int_{K_w}\int_{K_w^*}(u(x)-u(y))^2\nu_m(\md y)\nu_m(\md x).$$
For all $x\in K_w$, $y\in K_w^*$, there exists $\tilde{w}\in W_n$ such that $y\in K_{\tilde{w}}$ and $K_{\tilde{w}}\cap K_w\ne\emptyset$. Take $z\in V_w\cap V_{\tilde{w}}$, then
$$(u(x)-u(y))^2\le2(u(x)-u(z))^2+2(u(z)-u(y))^2,$$
and
\begin{align*}
&\int_{K_w}\int_{K_w^*}(u(x)-u(y))^2\nu_m(\md y)\nu_m(\md x)\\
\le&\sum_{\tilde{w}\in W_n,K_{\tilde{w}}\cap K_w\ne\emptyset}\int_{K_w}\int_{K_{\tilde{w}}}(u(x)-u(y))^2\nu_m(\md y)\nu_m(\md x)\\
\le&\sum_{\tilde{w}\in W_n,K_{\tilde{w}}\cap K_w\ne\emptyset}2\int_{K_w}\int_{K_{\tilde{w}}}\left((u(x)-u(z))^2+(u(z)-u(y))^2\right)\nu_m(\md y)\nu_m(\md x).
\end{align*}
Hence
\begin{align}
&\sum_{w\in W_n}\int_{K_w}\int_{K_w^*}(u(x)-u(y))^2\nu_m(\md y)\nu_m(\md x)\nonumber\\
\le&2\cdot2\cdot4\cdot2\sum_{w\in W_n}\sum_{z\in V_w}\int_{K_w}(u(x)-u(z))^2\nu_m(\md x)\left(\int_{K_w}\nu_m(\md y)\right)\nonumber\\
=&32\sum_{w\in W_n}\sum_{z\in V_w}\int_{K_w}(u(x)-u(z))^2\nu_m(\md x)\frac{\#(V_m\cap K_w)}{\#V_m}\nonumber\\
=&32\sum_{w\in W_n}\sum_{z\in V_w}\sum_{x\in V_m\cap K_w}(u(x)-u(z))^2\frac{1}{\#V_m}\frac{\#(V_m\cap K_w)}{\#V_m}\nonumber\\
=&32\frac{\#V_{m-n}}{(\#V_m)^2}\sum_{w\in W_n}\sum_{z\in V_w}\sum_{x\in V_m\cap K_w}(u(x)-u(z))^2.\label{SG_app_eqn_tmp1}
\end{align}
Let us estimate $(u(x)-u(z))^2$ for $z\in V_w$, $x\in V_m\cap K_w$, $w\in W_n$. We construct a finite sequence $p_n,\ldots,p_{m+1}$ as follows. If $w=w_1\ldots w_n\in W_n$, then
\begin{align*}
z&=P_{w_1\ldots w_nw_{n+1}},\\
x&=P_{w_1\ldots w_n\tilde{w}_{n+1}\ldots\tilde{w}_m\tilde{w}_{m+1}}.
\end{align*}
Let
\begin{align*}
p_n&=P_{w_1\ldots w_nw_{n+1}}=z,\\
p_{n+1}&=P_{w_1\ldots w_n\tilde{w}_{n+1}},\\
p_{n+2}&=P_{w_1\ldots w_n\tilde{w}_{n+1}\tilde{w}_{n+2}},\\
&\ldots\\
p_{m+1}&=P_{w_1\ldots w_n\tilde{w}_{n+1}\ldots\tilde{w}_{m}\tilde{w}_{m+1}}=x,
\end{align*}
then $|p_i-p_{i+1}|=0$ or $2^{-i}$, $i=n,\ldots,m$ and
\begin{align*}
&(u(x)-u(z))^2=(u(p_n)-u(p_{m+1}))^2\\
\le&2(u(p_n)-u(p_{n+1}))^2+2(u(p_{n+1})-u(p_{m+1}))^2\\
\le&2(u(p_n)-u(p_{n+1}))^2+2\left[2(u(p_{n+1})-u(p_{n+2}))^2+2(u(p_{n+2})-u(p_{m+1}))^2\right]\\
=&2(u(p_n)-u(p_{n+1}))^2+2^2(u(p_{n+1})-u(p_{n+2}))^2+2^2(u(p_{n+2})-u(p_{m+1}))^2\\
\le&\ldots\le\sum_{i=n}^{m}2^{i-n+1}(u(p_i)-u(p_{i+1}))^2.
\end{align*}
Let us sum up the resulting inequality for all $z\in V_w$, $x\in V_m\cap K_w$, $w\in W_n$. For all $i=n,\ldots,m$, $p,q\in V_i\cap K_w$ with $|p-q|=2^{-i}$, the term $(u(p)-u(q))^2$ occurs in the sum with times of the order $3^{m-i}$, hence
$$\sum_{w\in W_n}\sum_{z\in V_w}\sum_{x\in V_m\cap K_w}(u(x)-u(z))^2\le c\sum_{i=n}^m\sum_{w\in W_i}\sum_{p,q\in V_w}(u(p)-u(q))^2\cdot3^{m-i}\cdot2^{i-n}.$$
It follows from Equation (\ref{SG_app_eqn_tmp1}) that
\begin{align*}
&\sum_{w\in W_n}\int_{K_w}\int_{K_{{w}}^*}(u(x)-u(y))^2\nu_m(\md y)\nu_m(\md x)\\
\le& c\frac{3^{m-n}}{3^{2m}}\sum_{i=n}^m\sum_{w\in W_i}\sum_{p,q\in V_w}(u(p)-u(q))^2\cdot3^{m-i}\cdot2^{i-n}\\
=&c\sum_{i=n}^m\sum_{w\in W_i}\sum_{p,q\in V_w}(u(p)-u(q))^2\cdot 3^{-n-i}\cdot2^{i-n}.
\end{align*}
Letting $m\to+\infty$, we obtain
$$\int_K\int_{B(x,2^{-n-1})}(u(x)-u(y))^2\nu(\md y)\nu(\md x)\le c\sum_{i=n}^\infty\sum_{w\in W_i}\sum_{p,q\in V_w}(u(p)-u(q))^2\cdot 3^{-n-i}\cdot2^{i-n},$$
and
\begin{align*}
&2^{(\alpha+\beta)n}\int_K\int_{B(x,2^{-n-1})}(u(x)-u(y))^2\nu(\md y)\nu(\md x)\\
\le& c\sum_{i=n}^\infty\sum_{w\in W_i}\sum_{p,q\in V_w}(u(p)-u(q))^2\cdot 2^{-(\alpha-1)i}\cdot2^{(\beta-1)n},
\end{align*}
and hence
\begin{align*}
&\sum_{n=1}^\infty2^{(\alpha+\beta)n}\int_K\int_{B(x,2^{-n-1})}(u(x)-u(y))^2\nu(\md y)\nu(\md x)\\
\le& c\sum_{n=1}^\infty\sum_{i=n}^\infty\sum_{w\in W_i}\sum_{p,q\in V_w}(u(p)-u(q))^2\cdot 2^{-(\alpha-1)i}\cdot2^{(\beta-1)n}\\
=&c\sum_{i=1}^\infty\sum_{n=1}^i\sum_{w\in W_i}\sum_{p,q\in V_w}(u(p)-u(q))^2\cdot 2^{-(\alpha-1)i}\cdot2^{(\beta-1)n}\\
\le&\frac{2^{\beta-1}c}{2^{\beta-1}-1}\sum_{i=1}^\infty2^{(\beta-\alpha)i}\sum_{w\in W_i}\sum_{p,q\in V_w}(u(p)-u(q))^2,
\end{align*}
which proves Equation (\ref{SG_app_eqn_equiv2_2}). Applying Corollary \ref{SG_app_cor_arbi}, we obtain Equation (\ref{SG_app_eqn_equiv2_1}).
\end{proof}

\section{Another Approach of Determination of the Walk Dimension of the SG}

In this section, we give another determination of the walk dimension of the SG which is much simpler than that given in Chapter \ref{ch_SG_det}.

We prove the following result.

\begin{mythm}\label{SG_app_thm_det}
For all $\beta\in(\alpha,\beta^*)$, we have $(\calE_\beta,\calF_\beta)$ is a regular Dirichlet form on $L^2(K;\nu)$. For all $\beta\in[\beta^*,+\infty)$, we have $\calF_\beta$ consists only of constant functions.
\end{mythm}

\begin{proof}
By Fatou's lemma, it is obvious that $(\calE_\beta,\calF_\beta)$ is a closed form on $L^2(K;\nu)$ in the wide sense.

For all $\beta\in(\alpha,\beta^*)$. By \cite[Theorem 4.11 (\rmnum{3})]{GHL03}, we have $\calF_\beta\subseteq C(K)$. We only need to show that $\calF_\beta$ is uniformly dense in $C(K)$, then $\calF_\beta$ is dense in $L^2(K;\nu)$, hence $(\calE_\beta,\calF_\beta)$ is a regular closed form on $L^2(K;\nu)$.

Indeed, for all $U=U^{(x_0,x_1,x_2)}\in\calU$, by Theorem \ref{thm_SG_fun}, we have
\begin{align*}
E_\beta(U,U)&=\sum_{n=1}^\infty2^{(\beta-\alpha)n}\sum_{w\in W_n}\sum_{p,q\in V_w}(U(p)-U(q))^2\\
&=\sum_{n=1}^\infty2^{(\beta-\alpha)n}\left(\frac{3}{5}\right)^{n}\left((x_0-x_1)^2+(x_1-x_2)^2+(x_0-x_2)^2\right)<+\infty,
\end{align*}
hence $U\in\calF_\beta$, $\calU\subseteq\calF_\beta$. By Theorem \ref{thm_SG_fun} again, we have $\calF_\beta$ separates points. It is obvious that $\calF_\beta$ is a sub-algebra of $C(K)$, that is, for all $u,v\in\calF_\beta,c\in\R$, we have $u+v,cu,uv\in\calF_\beta$. By Stone-Weierstrass theorem, we have $\calF_\beta$ is uniformly dense in $C(K)$.

It is obvious that $\calE_\beta$ does have Markovian property, hence $(\calE_\beta,\calF_\beta)$ is a regular Dirichlet form on $L^2(K;\nu)$.

For all $\beta\in[\beta^*,+\infty)$. Assume that $u\in\calF_\beta$ is not constant, then there exists some integer $N\ge1$ such that $\frakE_N(u,u)>0$. By Theorem \ref{thm_SG_con}, we have
\begin{align*}
E_\beta(u,u)&=\sum_{n=1}^\infty2^{(\beta-\alpha)n}\left(\frac{3}{5}\right)^n\frakE_n(u,u)\ge\sum_{n=N+1}^\infty2^{(\beta-\alpha)n}\left(\frac{3}{5}\right)^n\frakE_n(u,u)\\
&\ge\sum_{n=N+1}^\infty2^{(\beta-\alpha)n}\left(\frac{3}{5}\right)^n\frakE_N(u,u)=+\infty,
\end{align*}
contradiction! Hence $\calF_\beta$ consists only of constant functions.
\end{proof}

\section{Proof of Theorem \ref{SG_app_thm_incre} and \ref{SG_app_thm_conv_main}}

First, we prove Theorem \ref{SG_app_thm_incre}.
\begin{proof}[Proof of Theorem \ref{SG_app_thm_incre}]
For simplicity, let $\lambda=5^{-1}\cdot2^{\beta}$ or $\beta=\log(5\lambda)/\log2$, where $\beta\in(\alpha,\beta^*)$ or $\lambda\in(3/5,1)$. Then
\begin{align*}
(1-5^{-1}\cdot 2^{\beta})E_\beta(u,u)&=\left(1-\frac{2^{\beta}}{5}\right)\sum_{n=1}^\infty\left(\frac{2^\beta}{5}\right)^n\left[\left(\frac{5}{3}\right)^n\sum_{w\in W_n}\sum_{p,q\in V_w}(u(p)-u(q))^2\right]\\
&=(1-\lambda)\sum_{n=1}^\infty\lambda^n\frakE_n(u,u).
\end{align*}

If $u$ has no continuous version, then this result is obvious. Hence, we may assume that $u$ is continuous. Let
$$x_n=\frakE_n(u,u)=\left(\frac{5}{3}\right)^n\sum_{w\in W_n}\sum_{p,q\in V_w}(u(p)-u(q))^2.$$
By Theorem \ref{thm_SG_con}, we have $\myset{x_n}$ is a monotone increasing sequence in $[0,+\infty]$ and
$$\lim_{n\to+\infty}x_n=\lim_{n\to+\infty}\frakE_n(u,u)=\frakE_\loc(u,u)=\lim_{n\to+\infty}\left(\frac{5}{3}\right)^n\sum_{w\in W_n}\sum_{p,q\in V_w}(u(p)-u(q))^2.$$
By Proposition \ref{prop_ele1}, we have $(1-\lambda)\sum_{n=1}^\infty\lambda^nx_n$ is monotone increasing in $\lambda\in(3/5,1)$, that is, $(1-5^{-1}\cdot2^\beta)E_\beta(u,u)$ is monotone increasing in $\beta\in(\alpha,\beta^*)$. Moreover
$$\lim_{\beta\uparrow\beta^*}(1-5^{-1}\cdot2^\beta)E_\beta(u,u)=\lim_{\lambda\uparrow1}(1-\lambda)\sum_{n=1}^\infty\lambda^nx_n=\lim_{n\to+\infty}x_n=\lim_{n\to+\infty}\frakE_n(u,u)=\frakE_\loc(u,u).$$
\end{proof}

Next, we prove Theorem \ref{SG_app_thm_conv_main}.

\begin{proof}[Proof of Theorem \ref{SG_app_thm_conv_main}]
First, we check (\ref{def_Mosco_2}) in Definition \ref{def_Mosco}. For all $u\in L^2(K;\nu)$, let $u_n=u$ for all $n\ge1$, then $u_n$ is trivially convergent to $u$ in $L^2(K;\nu)$ and by Theorem \ref{SG_app_thm_incre}, we have
$$\frakE_\loc(u,u)=\lim_{n\to+\infty}(1-5^{-1}\cdot2^{\beta_n})E_{\beta_n}(u,u)=\lim_{n\to+\infty}(1-5^{-1}\cdot2^{\beta_n})E_{\beta_n}(u_n,u_n).$$

Then, we check (\ref{def_Mosco_1}) in Definition \ref{def_Mosco}. For all $\myset{u_n}\subseteq L^2(K;\nu)$ that converges weakly to $u\in L^2(K;\nu)$. For all $m\ge1$, by Corollary \ref{cor_Mosco}, we have
$$(1-5^{-1}\cdot2^{\beta_m})E_{\beta_m}(u,u)\le\varliminf_{n\to+\infty}(1-5^{-1}\cdot2^{\beta_m})E_{\beta_m}(u_n,u_n).$$
By Theorem \ref{SG_app_thm_incre}, for all $n\ge m$, we have
$$(1-5^{-1}\cdot2^{\beta_m})E_{\beta_m}(u_n,u_n)\le(1-5^{-1}\cdot2^{\beta_n})E_{\beta_n}(u_n,u_n),$$
hence
$$(1-5^{-1}\cdot2^{\beta_m})E_{\beta_m}(u,u)\le\varliminf_{n\to+\infty}(1-5^{-1}\cdot2^{\beta_m})E_{\beta_m}(u_n,u_n)\le\varliminf_{n\to+\infty}(1-5^{-1}\cdot2^{\beta_n})E_{\beta_n}(u_n,u_n).$$
By Theorem \ref{SG_app_thm_incre} again, we have
$$\frakE_\loc(u,u)=\lim_{m\to+\infty}(1-5^{-1}\cdot2^{\beta_m})E_{\beta_m}(u,u)\le\varliminf_{n\to+\infty}(1-5^{-1}\cdot2^{\beta_n})E_{\beta_n}(u_n,u_n).$$
Hence $(1-5^{-1}\cdot2^{\beta_n})E_{\beta_n}$ converges to $\frakE_\loc$ in the sense of Mosco.
\end{proof}
Mosco convergence in Theorem \ref{SG_app_thm_conv_main} implies that appropriate time-changed jump processes can approximate the diffusion at least in the sense of finite-dimensional distribution.

\section{Proof of Theorem \ref{SG_app_thm_trace_main}}

Similar to Lemma \ref{SG_app_lem_equiv1}, we have the following result for the unit interval. For all $u\in L^2(I)$, we have
$$\int_{I}\int_I\frac{(u(x)-u(y))^2}{|x-y|^{1+\beta}}\md x\md y\asymp\sum_{n=0}^\infty2^{n}2^{\beta n}\int_I\int_{B(x,2^{-n})}(u(x)-u(y))^2\md y\md x.$$
Combining this result with Equation (\ref{SG_app_eqn_interval}), we obtain that for all $u\in C(I)$
$$\sum_{n=1}^\infty2^{-n}2^{\beta n}\sum_{i=0}^{2^n-1}(u(\frac{i}{2^n})-u(\frac{i+1}{2^n}))^2\asymp\sum_{n=0}^\infty2^{n}2^{\beta n}\int_I\int_{B(x,2^{-n})}(u(x)-u(y))^2\md y\md x.$$
Since $\beta_1\in(\alpha_1,\beta_1^*)$, we have
$$\beta_2=\beta_1-(\alpha_1-\alpha_2)\in(\alpha_2,\beta_1^*-\alpha_1+\alpha_2)\subseteq(\alpha_2,\beta_2^*).$$
For all $u\in B^{2,2}_{\alpha_1,\beta_1}(K)$, we have $u\in C(K)$, hence $u|_I\in C(I)$. Note that
$$\sum_{n=1}^\infty2^{-\alpha_2n}2^{\beta_2 n}\sum_{i=0}^{2^n-1}(u(\frac{i}{2^n})-u(\frac{i+1}{2^n}))^2\le\sum_{n=1}^\infty2^{-\alpha_1 n}2^{\beta_1 n}\sum_{w\in W_n}\sum_{p,q\in V_w}(u(p)-u(q))^2<+\infty,$$
hence $u|_I\in B^{2,2}_{\alpha_2,\beta_2}(I)$.

\section{Proof of Theorem \ref{SG_app_thm_jumping_kernel}}

Firstly, we construct equivalent semi-norms with jumping kernels that converge exactly to the local Dirichlet form.

For all $\beta\in(\alpha,\beta^*)$, $((1-5^{-1}\cdot2^{\beta})E_\beta,\calF_{\beta})$ is a non-local regular Dirichlet form on $L^2(K;\nu)$, by Beurling-Deny formula, there exists a unique jumping measure $J_\beta$ on $K\times K\backslash\mathrm{diag}$ such that for all $u\in\calF_{\beta}$, we have
$$(1-5^{-1}\cdot2^\beta)E_{\beta}(u,u)=\iint_{K\times K\backslash\mathrm{diag}}(u(x)-u(y))^2J_\beta(\md x\md y).$$
It is obvious that
$$J_{\beta}(\md x\md y)=(1-5^{-1}\cdot2^{\beta})\sum_{n=1}^\infty2^{(\beta-\alpha)n}\sum_{w\in W_n}\sum_{p,q\in V_w}\delta_p(\md x)\delta_q(\md y),$$
where $\delta_p,\delta_q$ are Dirac measures at $p,q$, respectively. Hence $J_\beta$ is singular with respect to $\nu\times\nu$ and no jumping kernel exists. Note that
\begin{align*}
&\sum_{n=1}^\infty2^{(\beta-\alpha)n}\sum_{w\in W_n}\sum_{p,q\in V_w}(u(p)-u(q))^2\\
=&\iint_{K\times K\backslash\mathrm{diag}}(u(x)-u(y))^2\left(\sum_{n=1}^\infty2^{(\beta-\alpha)n}\sum_{w\in W_n}\sum_{p,q\in V_w}\delta_p(\md x)\delta_q(\md y)\right),
\end{align*}
where
\begin{align*}
&\sum_{n=1}^\infty2^{(\beta-\alpha)n}\sum_{w\in W_n}\sum_{p,q\in V_w}\delta_p(\md x)\delta_q(\md y)=\sum_{n=1}^\infty\sum_{w\in W_n}\sum_{p,q\in V_w}\frac{1}{|p-q|^{\alpha+\beta}}{|p-q|^{2\alpha}}\delta_p(\md x)\delta_q(\md y)\\
=&\frac{1}{|x-y|^{\alpha+\beta}}\sum_{n=1}^\infty2^{-2\alpha n}\sum_{w\in W_n}\sum_{p,q\in V_w}\delta_p(\md x)\delta_q(\md y)=\frac{1}{|x-y|^{\alpha+\beta}}J(\md x\md y),
\end{align*}
and
$$J(\md x\md y)=\sum_{n=1}^\infty2^{-2\alpha n}\sum_{w\in W_n}\sum_{p,q\in V_w}\delta_p(\md x)\delta_q(\md y).$$

\begin{myprop}\label{SG_app_prop_kernel1}
Let
$$c_i(x,y)=\sum_{n=1}^\infty2^{-2\alpha n}\sum_{w\in W_n}\sum_{p,q\in V_w}\frac{1}{\nu(K^{(i)}_{p,n})\nu(K^{(i)}_{q,n})}1_{K^{(i)}_{p,n}}(x)1_{K^{(i)}_{q,n}}(y),$$
then for all $u\in C(K)$, we have
\begin{equation}\label{SG_app_eqn_kernel1}
\begin{aligned}
&\left(1-C(\frac{2^{\alpha-\gamma i}}{1-2^{\alpha-\gamma i}}+\frac{2^{\alpha-\frac{\beta-\alpha}{2}\gamma i}}{1-2^{\alpha-\frac{\beta-\alpha}{2}\gamma i}})\right)\sum_{n=1}^\infty 2^{(\beta-\alpha)n}\sum_{w\in W_n}\sum_{p,q\in V_w}(u(p)-u(q))^2\\
\le&\iint_{K\times K\backslash\mathrm{diag}}\frac{c_i(x,y)(u(x)-u(y))^2}{|x-y|^{\alpha+\beta}}\nu(\md x)\nu(\md y)\\
\le&\left(1+C(\frac{2^{\alpha-\gamma i}}{1-2^{\alpha-\gamma i}}+\frac{2^{\alpha-\frac{\beta-\alpha}{2}\gamma i}}{1-2^{\alpha-\frac{\beta-\alpha}{2}\gamma i}})\right)\sum_{n=1}^\infty 2^{(\beta-\alpha)n}\sum_{w\in W_n}\sum_{p,q\in V_w}(u(p)-u(q))^2.
\end{aligned}
\end{equation}
\end{myprop}

\begin{proof}
Note that
\begin{align*}
&\iint_{K\times K\backslash\mathrm{diag}}\frac{c_i(x,y)(u(x)-u(y))^2}{|x-y|^{\alpha+\beta}}\nu(\md x)\nu(\md y)\\
=&\sum_{n=1}^\infty2^{-2\alpha n}\sum_{w\in W_n}\sum_{p,q\in V_w}\frac{1}{\nu(K^{(i)}_{p,n})\nu(K^{(i)}_{q,n})}\int_{K^{(i)}_{p,n}}\int_{K^{(i)}_{q,n}}\frac{(u(x)-u(y))^2}{|x-y|^{\alpha+\beta}}\nu(\md x)\nu(\md y).
\end{align*}
Since
$$\frac{1}{\nu(K^{(i)}_{p,n})\nu(K^{(i)}_{q,n})}1_{K^{(i)}_{p,n}}(x)1_{K^{(i)}_{q,n}}(y)\nu(\md x)\nu(\md y)\text{ converges weakly to }\delta_p(\md x)\delta_q(\md y),$$
for all $u\in C(K)$, we have
\begin{align*}
&\lim_{i\to+\infty}\frac{1}{\nu(K^{(i)}_{p,n})\nu(K^{(i)}_{q,n})}\int_{K^{(i)}_{p,n}}\int_{K^{(i)}_{q,n}}\frac{(u(x)-u(y))^2}{|x-y|^{\alpha+\beta}}\nu(\md x)\nu(\md y)\\
=&\frac{(u(p)-u(q))^2}{|p-q|^{\alpha+\beta}}=2^{(\alpha+\beta)n}(u(p)-u(q))^2.
\end{align*}
By Fatou's lemma, we have
\begin{align*}
&\sum_{n=1}^\infty 2^{(\beta-\alpha)n}\sum_{w\in W_n}\sum_{p,q\in V_w}(u(p)-u(q))^2\\
\le&\varliminf_{i\to+\infty}\iint_{K\times K\backslash\mathrm{diag}}\frac{c_i(x,y)(u(x)-u(y))^2}{|x-y|^{\alpha+\beta}}\nu(\md x)\nu(\md y).
\end{align*}
If $\text{LHS}=+\infty$, then $E(u)=+\infty$, the limit in RHS exists and equals to $+\infty$. Hence, we may assume that $E(u)<+\infty$, by Lemma \ref{SG_app_lem_holder}, we have
$$|u(x)-u(y)|^2\le cE(u)|x-y|^{\beta-\alpha}\text{ for all }x,y\in K.$$
Consider
\begin{align*}
&\lvert\frac{1}{\nu(K^{(i)}_{p,n})\nu(K^{(i)}_{q,n})}\int_{K^{(i)}_{p,n}}\int_{K^{(i)}_{q,n}}\frac{(u(x)-u(y))^2}{|x-y|^{\alpha+\beta}}\nu(\md x)\nu(\md y)-\frac{(u(p)-u(q))^2}{|p-q|^{\alpha+\beta}}\rvert\\
\le&\frac{1}{\nu(K^{(i)}_{p,n})\nu(K^{(i)}_{q,n})}\int_{K^{(i)}_{p,n}}\int_{K^{(i)}_{q,n}}\lvert\frac{(u(x)-u(y))^2}{|x-y|^{\alpha+\beta}}-\frac{(u(p)-u(q))^2}{|p-q|^{\alpha+\beta}}\rvert\nu(\md x)\nu(\md y).
\end{align*}
For all $x\in K^{(i)}_{p,n}, y\in K^{(i)}_{q,n}$, we have
\begin{align*}
&\lvert\frac{(u(x)-u(y))^2}{|x-y|^{\alpha+\beta}}-\frac{(u(p)-u(q))^2}{|p-q|^{\alpha+\beta}}\rvert\\
\le&\frac{1}{|x-y|^{\alpha+\beta}|p-q|^{\alpha+\beta}}\left((u(x)-u(y))^2\lvert |p-q|^{\alpha+\beta}-|x-y|^{\alpha+\beta}\rvert\right.\\
&\left.+\lvert(u(x)-u(y))^2-(u(p)-u(q))^2\rvert\cdot |x-y|^{\alpha+\beta}\right).
\end{align*}
Since
\begin{align*}
&|p-q|\ge|x-y|\ge|p-q|-|x-p|-|y-q|\\
\ge&|p-q|-\frac{2}{2^{\gamma ni}}|p-q|=\left(1-\frac{2}{2^{\gamma ni}}\right)|p-q|,
\end{align*}
for all $i\ge2$, we have $|x-y|\ge|p-q|/2$, for all $i\ge1$, we have
$$|p-q|^{\alpha+\beta}\ge|x-y|^{\alpha+\beta}\ge\left(1-\frac{2}{2^{\gamma ni}}\right)^{\alpha+\beta}|p-q|^{\alpha+\beta}.$$
Therefore, we have
$$\frac{1}{|x-y|^{\alpha+\beta}|p-q|^{\alpha+\beta}}\le\frac{2^{\alpha+\beta}}{|p-q|^{2(\alpha+\beta)}}=2^{\alpha+\beta}2^{2(\alpha+\beta)n},$$
\begin{align*}
&(u(x)-u(y))^2\le cE(u)|x-y|^{\beta-\alpha}\\
\le&cE(u)|p-q|^{\beta-\alpha}=cE(u)2^{-(\beta-\alpha)n},
\end{align*}
\begin{align*}
\lvert |p-q|^{\alpha+\beta}-|x-y|^{\alpha+\beta}\rvert&\le|p-q|^{\alpha+\beta}\left[1-(1-\frac{2}{2^{\gamma ni}})^{\alpha+\beta}\right]\\
&\le 2(\alpha+\beta)2^{-(\alpha+\beta)n-\gamma ni},
\end{align*}
\begin{align*}
&\lvert(u(x)-u(y))^2-(u(p)-u(q))^2\rvert\\
=&\lvert(u(x)-u(y))+(u(p)-u(q))\rvert\cdot\lvert(u(x)-u(y))-(u(p)-u(q))\rvert\\
\le&\left(|u(x)-u(y)|+|u(p)-u(q)|\right)\left(|u(x)-u(p)|+|u(y)-u(q)|\right)\\
\le& cE(u)\left(|x-y|^{\frac{\beta-\alpha}{2}}+|p-q|^{\frac{\beta-\alpha}{2}}\right)\left(|x-p|^{\frac{\beta-\alpha}{2}}+|y-q|^{\frac{\beta-\alpha}{2}}\right)\\
\le& 4cE(u)2^{-\frac{\beta-\alpha}{2}(2n+\gamma ni)},
\end{align*}
$$|x-y|^{\alpha+\beta}\le|p-q|^{\alpha+\beta}=2^{-(\alpha+\beta)n}.$$
It follows that
$$\lvert\frac{(u(x)-u(y))^2}{|x-y|^{\alpha+\beta}}-\frac{(u(p)-u(q))^2}{|p-q|^{\alpha+\beta}}\rvert\le2^{\alpha+\beta}cE(u)\left(2(\alpha+\beta)2^{2\alpha n-\gamma ni}+4\cdot 2^{2\alpha n-\frac{\beta-\alpha}{2}\gamma ni}\right),$$
and, hence,
\begin{align*}
&|\sum_{n=1}^\infty 2^{(\beta-\alpha)n}\sum_{w\in W_n}\sum_{p,q\in V_w}(u(p)-u(q))^2\\
&-\sum_{n=1}^\infty2^{-2\alpha n}\sum_{w\in W_n}\sum_{p,q\in V_w}\frac{1}{\nu(K^{(i)}_{p,n})\nu(K^{(i)}_{q,n})}\int_{K^{(i)}_{p,n}}\int_{K^{(i)}_{q,n}}\frac{(u(x)-u(y))^2}{|x-y|^{\alpha+\beta}}\nu(\md x)\nu(\md y)|\\
\le&\sum_{n=1}^\infty 2^{-2\alpha n}2^{\alpha n}2^{\alpha+\beta}cE(u)\left(2(\alpha+\beta)2^{2\alpha n-\gamma ni}+4\cdot 2^{2\alpha n-\frac{\beta-\alpha}{2}\gamma ni}\right)\\
\le& CE(u)\sum_{n=1}^\infty\left(2^{\alpha n-\gamma ni}+2^{\alpha n-\frac{\beta-\alpha}{2}\gamma ni}\right)=CE(u)\sum_{n=1}^\infty\left(2^{(\alpha-\gamma i)n}+2^{(\alpha-\frac{\beta-\alpha}{2}\gamma i)n}\right).
\end{align*}
Choose $\gamma\ge1$ such that $\alpha-\gamma<0$ and $\alpha-\frac{\beta-\alpha}{2}\gamma<0$, then
$$\sum_{n=1}^\infty\left(2^{(\alpha-\gamma i)n}+2^{(\alpha-\frac{\beta-\alpha}{2}\gamma i)n}\right)=\frac{2^{\alpha-\gamma i}}{1-2^{\alpha-\gamma i}}+\frac{2^{\alpha-\frac{\beta-\alpha}{2}\gamma i}}{1-2^{\alpha-\frac{\beta-\alpha}{2}\gamma i}}\to0,$$
as $i\to+\infty$. Hence
\begin{align*}
&|\iint_{K\times K\backslash\mathrm{diag}}\frac{c_i(x,y)(u(x)-u(y))^2}{|x-y|^{\alpha+\beta}}\nu(\md x)\nu(\md y)-\sum_{n=1}^\infty 2^{(\beta-\alpha)n}\sum_{w\in W_n}\sum_{p,q\in V_w}(u(p)-u(q))^2|\\
\le& CE(u)\left(\frac{2^{\alpha-\gamma i}}{1-2^{\alpha-\gamma i}}+\frac{2^{\alpha-\frac{\beta-\alpha}{2}\gamma i}}{1-2^{\alpha-\frac{\beta-\alpha}{2}\gamma i}}\right),
\end{align*}
hence we have Equation (\ref{SG_app_eqn_kernel1}).
\end{proof}

Secondly, we do appropriate cutoff to have bounded jumping kernels.

\begin{myprop}\label{SG_app_prop_kernel2}
For all sequence $\myset{\beta_i}\subseteq(\alpha,\beta^*)$ with $\beta_i\uparrow\beta^*$. Let
$$C_i(x,y)=\sum_{n=1}^{\Phi(i)}2^{-2\alpha n}\sum_{w\in W_n}\sum_{p,q\in V_w}\frac{1}{\nu(K^{(i)}_{p,n})\nu(K^{(i)}_{q,n})}1_{K^{(i)}_{p,n}}(x)1_{K^{(i)}_{q,n}}(y),$$
where $\Phi:\mathbb{N}\to\mathbb{N}$ is increasing and $(1-5^{-1}\cdot2^{\beta_i})\Phi(i)\ge i$ for all $i\ge1$. Then for all $u\in\frakF_\loc$, we have
$$\lim_{i\to+\infty}(1-5^{-1}\cdot 2^{\beta_i})\iint_{K\times K\backslash\mathrm{diag}}\frac{C_i(x,y)(u(x)-u(y))^2}{|x-y|^{\alpha+\beta_i}}\nu(\md x)\nu(\md y)=\frakE_\loc(u,u).$$
\end{myprop}

\begin{proof}
By the proof of Proposition \ref{SG_app_prop_kernel1}, for all $u\in\frakF_\loc$, we have
\begin{align*}
&\lvert\iint_{K\times K\backslash\mathrm{diag}}\frac{C_i(x,y)(u(x)-u(y))^2}{|x-y|^{\alpha+\beta_i}}\nu(\md x)\nu(\md y)-\sum_{n=1}^{\Phi(i)}2^{(\beta_i-\alpha)n}\sum_{w\in W_n}\sum_{p,q\in V_w}(u(p)-u(q))^2\rvert\\
\le& CE_{\beta_i}(u,u)\sum_{n=1}^{\Phi(i)}\left(2^{(\alpha-\gamma i)n+2^{(\alpha-\frac{\beta_i-\alpha}{2}\gamma i)n}}\right)\le CE_{\beta_i}(u,u)\left(\frac{2^{\alpha-\gamma i}}{1-2^{\alpha-\gamma i}}+\frac{2^{\alpha-\frac{\beta_i-\alpha}{2}\gamma i}}{1-2^{\alpha-\frac{\beta_i-\alpha}{2}\gamma i}}\right),
\end{align*}
hence
\begin{align*}
&\lvert(1-5^{-1}\cdot2^{\beta_i})\iint_{K\times K\backslash\mathrm{diag}}\frac{C_i(x,y)(u(x)-u(y))^2}{|x-y|^{\alpha+\beta_i}}\nu(\md x)\nu(\md y)\\
&-(1-5^{-1}\cdot2^{\beta_i})\sum_{n=1}^{\Phi(i)}2^{(\beta_i-\alpha)n}\sum_{w\in W_n}\sum_{p,q\in V_w}(u(p)-u(q))^2\rvert\\
\le& C(1-5^{-1}\cdot2^{\beta_i})E_{\beta_i}(u,u)\left(\frac{2^{\alpha-\gamma i}}{1-2^{\alpha-\gamma i}}+\frac{2^{\alpha-\frac{\beta_i-\alpha}{2}\gamma i}}{1-2^{\alpha-\frac{\beta_i-\alpha}{2}\gamma i}}\right)\to0,
\end{align*}
as $i\to+\infty$. Hence we only need to show that for all $u\in\frakF_\loc$
$$\lim_{i\to+\infty}(1-5^{-1}\cdot2^{\beta_i})\sum_{n=1}^{\Phi(i)}2^{(\beta_i-\alpha)n}\sum_{w\in W_n}\sum_{p,q\in V_w}(u(p)-u(q))^2=\frakE_\loc(u,u).$$
Let $\lambda_i=5^{-1}\cdot 2^{\beta_i}$, then $\myset{\lambda_i}\subseteq(3/5,1)$ and $\lambda_i\uparrow1$. We use the notions in the proof of Theorem \ref{SG_app_thm_incre}. We only need to show that for all $u\in\frakF_\loc$
$$\lim_{i\to+\infty}(1-\lambda_i)\sum_{n=1}^{\Phi(i)}\lambda_i^nx_n=\lim_{n\to+\infty}x_n.$$
It is obvious that
$$\varlimsup_{i\to+\infty}(1-\lambda_i)\sum_{n=1}^{\Phi(i)}\lambda_i^nx_n\le\varlimsup_{i\to+\infty}(1-\lambda_i)\sum_{n=1}^{\infty}\lambda_i^nx_n=\lim_{n\to+\infty}x_n.$$

On the other hand, for all $A<\lim_{n\to+\infty}x_n$, there exists some positive integer $N_0\ge1$ such that for all $n>N_0$, we have $a_n>A$, hence
\begin{align*}
&(1-\lambda_i)\sum_{n=1}^{\Phi(i)}\lambda_i^nx_n\ge(1-\lambda_i)\sum_{n=N_0+1}^{\Phi(i)}\lambda_i^nA\\
=&(1-\lambda_i)\frac{\lambda_i^{N_0+1}\left(1-\lambda_i^{\Phi(i)-N_0}\right)}{1-\lambda_i}A=\lambda_i^{N_0+1}\left(1-\lambda_i^{\Phi(i)-N_0}\right)A.\\
\end{align*}
It is obvious that $\lambda_i^{N_0}\to1$ as $i\to+\infty$. Since $(1-\lambda_i)\Phi(i)\ge i$, we have
\begin{align*}
\lambda_i^{-\Phi(i)}&=(1+\lambda_i^{-1}-1)^{\Phi(i)}=\left[(1+\lambda_i^{-1}-1)^{\frac{1}{1-\lambda_i}}\right]^{(1-\lambda_i)\Phi(i)}\\
&\ge\left[(1+\lambda_i^{-1}-1)^{\frac{\lambda_i^{-1}}{\lambda_i^{-1}-1}}\right]^{i}\to+\infty.
\end{align*}
Hence
$$\varliminf_{i\to+\infty}(1-\lambda_i)\sum_{n=1}^{\Phi(i)}\lambda_i^nx_n\ge A.$$
Since $A<\lim_{n\to+\infty}x_n$ is arbitrary, we have
$$\lim_{i\to+\infty}(1-\lambda_i)\sum_{n=1}^{\Phi(i)}\lambda_i^nx_n=\lim_{n\to+\infty}x_n.$$
\end{proof}

Now we give the proof of Theorem \ref{SG_app_thm_jumping_kernel}.

\begin{proof}[Proof of Theorem \ref{SG_app_thm_jumping_kernel}]
For all $u\in\frakF_\loc$, by Proposition \ref{SG_app_prop_kernel2}, we have
$$\lim_{i\to+\infty}(1-5^{-1}\cdot 2^{\beta_i})\iint_{K\times K\backslash\mathrm{diag}}\frac{C_i(x,y)(u(x)-u(y))^2}{|x-y|^{\alpha+\beta_i}}\nu(\md x)\nu(\md y)=\frakE_\loc(u,u).$$
By Theorem \ref{SG_app_thm_main} and Theorem \ref{SG_app_thm_incre}, we have
\begin{align*}
\frac{1}{C}\frakE_\loc(u,u)&\le\varliminf_{i\to+\infty}(1-5^{-1}\cdot 2^{\beta_i})\iint_{K\times K\backslash\mathrm{diag}}\frac{(u(x)-u(y))^2}{|x-y|^{\alpha+\beta_i}}\nu(\md x)\nu(\md y)\\
&\le\varlimsup_{i\to+\infty}(1-5^{-1}\cdot 2^{\beta_i})\iint_{K\times K\backslash\mathrm{diag}}\frac{(u(x)-u(y))^2}{|x-y|^{\alpha+\beta_i}}\nu(\md x)\nu(\md y)\le C\frakE_\loc(u,u),
\end{align*}
hence
$$\lim_{i\to+\infty}(1-\delta_i)(1-5^{-1}\cdot 2^{\beta_i})\iint_{K\times K\backslash\mathrm{diag}}\frac{(u(x)-u(y))^2}{|x-y|^{\alpha+\beta_i}}\nu(\md x)\nu(\md y)=0,$$
hence
$$\lim_{i\to+\infty}(1-5^{-1}\cdot 2^{\beta_i})\iint_{K\times K\backslash\mathrm{diag}}\frac{a_i(x,y)(u(x)-u(y))^2}{|x-y|^{\alpha+\beta_i}}\nu(\md x)\nu(\md y)=\frakE_\loc(u,u).$$
It is obvious that $a_i=\delta_iC_i+(1-\delta_i)$ is bounded from above and below by positive constants.
\end{proof}

\chapter{Construction of Local Regular Dirichlet Form on the SG}\label{ch_SG_con}

This chapter is based on my work \cite{MY18}.

\section{Background and Statement}

We have used the local regular Dirichlet form $\frakE_\loc$ given by Kigami in Chapter \ref{ch_SG_app}. Here we give another construction of local regular Dirichlet form. This construction is much more complicated than that of Kigami, but this construction can be applied to more general spaces. We will apply this construction to the SC in Chapter \ref{ch_SC_con}.

We use the notions of the SG introduced in Section \ref{sec_notion}.

Our main result is as follows.

\begin{mythm}\label{SG_con_thm_BM}
There exists a self-similar strongly local regular Dirichlet form $(\calE_\loc,\calF_\loc)$ on $L^2(K;\nu)$ satisfying
\begin{align*}
&\calE_\loc(u,u)\asymp\sup_{n\ge1}\left(\frac{5}{3}\right)^n\sum_{w^{(1)}\sim_nw^{(2)}}\left(P_nu(w^{(1)})-P_nu(w^{(2)})\right)^2,\\
&\calF_\loc=\myset{u\in L^2(K;\nu):\calE_\loc(u,u)<+\infty}.
\end{align*}
\end{mythm}

\begin{myrmk}
This theorem was also proved by Kusuoka and Zhou \cite[Theorem 7.19, Example 8.4]{KZ92} using approximation of Markov chains. Here we use $\Gamma$-convergence of stable-like non-local closed forms.
\end{myrmk}

The Besov spaces $B^{2,2}_{\alpha,\beta}(K)$ and $B^{2,\infty}_{\alpha,\beta}(K)$ have the following equivalent semi-norms.
\begin{mylem}(\cite[Lemma 3.1]{HK06}, Lemma \ref{SG_app_lem_equiv1})\label{SG_con_lem_equiv}
For all $\beta\in(0,+\infty)$, for all $u\in L^2(K;\nu)$, we have
$$\calE_\beta(u,u)\asymp\frakE_\beta(u,u)\asymp[u]_{B^{2,2}_{\alpha,\beta}(K)},$$
$$\sup_{n\ge1}2^{(\beta-\alpha)n}\sum_{w^{(1)}\sim_nw^{(2)}}\left(P_nu(w^{(1)})-P_nu(w^{(2)})\right)^2\asymp[u]_{B^{2,\infty}_{\alpha,\beta}(K)},$$
where
\begin{align*}
\calE_\beta(u,u)&=\int_K\int_K\frac{(u(x)-u(y))^2}{|x-y|^{\alpha+\beta}}\nu(\md x)\nu(\md y),\\
\frakE_\beta(u,u)&=\sum_{n=1}^\infty2^{(\beta-\alpha)n}\sum_{w^{(1)}\sim_nw^{(2)}}\left(P_nu(w^{(1)})-P_nu(w^{(2)})\right)^2.
\end{align*}
\end{mylem}

We have the following two corollaries whose proofs are obvious by Lemma \ref{SG_con_lem_equiv} and the proof of Theorem \ref{SG_con_thm_BM}.

Firstly, we have the characterization of the domain of the local Dirichlet form.

\begin{mycor}\label{SG_con_cor_chara}
$\calF_\loc=B^{2,\infty}_{\alpha,\beta^*}(K)$ and $\calE_\loc(u,u)\asymp[u]_{B^{2,\infty}_{\alpha,\beta^*}(K)}$ for all $u\in\calF_\loc$, where $\alpha=\log3/\log2$ is the Hausdorff dimension of the SG and $\beta^*=\log5/\log2$ is the walk dimension of the BM.
\end{mycor}

Secondly, we have the approximation of non-local Dirichlet forms to the local Dirichlet form.

\begin{mycor}\label{SG_con_cor_conv}
There exists some positive constant $C$ such that for all $u\in\calF_\loc$
$$\frac{1}{C}\calE_\loc(u,u)\le\varliminf_{\beta\uparrow\beta^*}(\beta^*-\beta)\calE_\beta(u,u)\le\varlimsup_{\beta\uparrow\beta^*}(\beta^*-\beta)\calE_\beta(u,u)\le C\calE_\loc(u,u),$$
$$\frac{1}{C}\calE_\loc(u,u)\le\varliminf_{\beta\uparrow\beta^*}(\beta^*-\beta)\frakE_\beta(u,u)\le\varlimsup_{\beta\uparrow\beta^*}(\beta^*-\beta)\frakE_\beta(u,u)\le C\calE_\loc(u,u),$$
$$\frac{1}{C}\calE_\loc(u,u)\le\varliminf_{\beta\uparrow\beta^*}(\beta^*-\beta)[u]_{B^{2,2}_{\alpha,\beta}(K)}\le\varlimsup_{\beta\uparrow\beta^*}(\beta^*-\beta)[u]_{B^{2,2}_{\alpha,\beta}(K)}\le C\calE_\loc(u,u).$$
\end{mycor}

This chapter is organized as follows. In Section \ref{SG_con_sec_resist}, we give resistance estimates and introduce good functions. In Section \ref{SG_con_sec_monotone}, we give weak monotonicity result. In Section \ref{SG_con_sec_BM}, we prove Theorem \ref{SG_con_thm_BM}.

\section{Resistance Estimates and Good Functions}\label{SG_con_sec_resist}

Firstly, we give resistance estimates.
 
For all $n\ge1$, let us introduce an energy on $X_n$ given by
$$E_n(u,u)=\sum_{w^{(1)}\sim_nw^{(2)}}\left(u(w^{(1)})-u(w^{(2)})\right)^2,u\in l(W_n).$$
For all $w^{(1)},w^{(2)}\in W_n$ with $w^{(1)}\ne w^{(2)}$, we define the resistance
\begin{align*}
R_n(w^{(1)},w^{(2)})&=\inf\myset{E_n(u,u):u(w^{(1)})=1,u(w^{(2)})=0,u\in l(W_n)}^{-1}\\
&=\sup\myset{\frac{\left(u(w^{(1)})-u(w^{(2)})\right)^2}{E_n(u,u)}:E_n(u,u)\ne0,u\in l(W_n)}.
\end{align*}
It is obvious that
$$\left(u(w^{(1)})-u(w^{(2)})\right)^2\le R_n(w^{(1)},w^{(2)})E_n(u,u)\text{ for all }w^{(1)},w^{(2)}\in W_n,u\in l(W_n),$$
and $R_n$ is a metric on $W_n$, hence
$$R_n(w^{(1)},w^{(2)})\le R_n(w^{(1)},w^{(3)})+R_n(w^{(3)},w^{(2)})\text{ for all }w^{(1)},w^{(2)},w^{(3)}\in W_n.$$

\begin{mythm}\label{SG_con_thm_resist}
Considering effective resistances between any two of $0^n,1^n,2^n$, we have the electrical network $X_n$ is equivalent to the electrical network in Figure \ref{SG_con_fig_resist}, where
$$r_n=\frac{1}{2}\left(\frac{5}{3}\right)^n-\frac{1}{2}.$$




\begin{figure}[ht]
\centering
\begin{tikzpicture}
\draw (0,0)--(1.5,1.5/1.7320508076);
\draw (3,0)--(1.5,1.5/1.7320508076);
\draw (1.5,1.5*1.7320508076)--(1.5,1.5/1.7320508076);

\draw[fill=black] (0,0) circle (0.06);
\draw[fill=black] (3,0) circle (0.06);
\draw[fill=black] (1.5,1.5*1.7320508076) circle (0.06);
\draw[fill=black] (1.5,1.5/1.7320508076) circle (0.06);

\draw (0,-0.3) node {$0^n$};
\draw (3,-0.3) node {$1^n$};
\draw (1.5,1.5*1.7320508076+0.3) node {$2^n$};

\draw (1.5+0.3,0.75*1.7320508076+0.3) node {$r_{n}$};
\draw (0.7,0.7) node {$r_n$};
\draw (2.3,0.7) node {$r_n$};

\end{tikzpicture}
\caption{An Equivalent Electrical Network}\label{SG_con_fig_resist}
\end{figure}
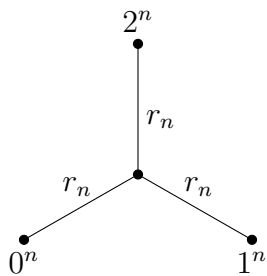


\end{mythm}

\begin{proof}
Using $\Delta$-Y transform directly, we have
$$r_1=\frac{1\cdot1}{1+1+1}=\frac{1}{3}=\frac{1}{2}\left(\frac{5}{3}\right)^1-\frac{1}{2}.$$
For $n+1$, using $\Delta$-Y transform again, we have the electrical network $X_{n+1}$ is equivalent to the electrical networks in Figure \ref{SG_con_fig_resist1}.




\begin{figure}[ht]
\centering
\subfigure{
\begin{tikzpicture}[scale=0.5]
\draw (0,0)--(1.5,1.5/1.7320508076);
\draw (3,0)--(1.5,1.5/1.7320508076);
\draw (1.5,1.5*1.7320508076)--(1.5,1.5/1.7320508076);

\draw (0+6,0)--(1.5+6,1.5/1.7320508076);
\draw (3+6,0)--(1.5+6,1.5/1.7320508076);
\draw (1.5+6,1.5*1.7320508076)--(1.5+6,1.5/1.7320508076);

\draw (0+3,0+3*1.7320508076)--(1.5+3,1.5/1.7320508076+3*1.7320508076);
\draw (3+3,0+3*1.7320508076)--(1.5+3,1.5/1.7320508076+3*1.7320508076);
\draw (1.5+3,1.5*1.7320508076+3*1.7320508076)--(1.5+3,1.5/1.7320508076+3*1.7320508076);

\draw (1.5,1.5*1.7320508076)--(3,3*1.7320508076);
\draw (6,3*1.7320508076)--(7.5,1.5*1.7320508076);
\draw (3,0)--(6,0);
  
\draw[fill=black] (0,0) circle (0.06);
\draw[fill=black] (3,0) circle (0.06);
\draw[fill=black] (3/2,3/2*1.7320508076) circle (0.06);
\draw[fill=black] (1.5,1.5/1.7320508076) circle (0.06);

\draw[fill=black] (0+6,0) circle (0.06);
\draw[fill=black] (3+6,0) circle (0.06);
\draw[fill=black] (3/2+6,3/2*1.7320508076) circle (0.06);
\draw[fill=black] (1.5+6,1.5/1.7320508076) circle (0.06);

\draw[fill=black] (0+3,0+3*1.7320508076) circle (0.06);
\draw[fill=black] (3+3,0+3*1.7320508076) circle (0.06);
\draw[fill=black] (3/2+3,3/2*1.7320508076+3*1.7320508076) circle (0.06);
\draw[fill=black] (1.5+3,1.5/1.7320508076+3*1.7320508076) circle (0.06);

\draw (0,-0.3) node {\tiny{$0^{n+1}$}};
\draw (9,-0.3) node {\tiny{$1^{n+1}$}};
\draw (4.5,9/2*1.7320508076+0.3) node {\tiny{$2^{n+1}$}};

\draw (1.5+0.3,0.75*1.7320508076+0.4) node {\tiny{$r_{n}$}};
\draw (0.7,0.7) node {\tiny{$r_n$}};
\draw (2.3,0.7) node {\tiny{$r_n$}};

\draw (1.5+0.3+6,0.75*1.7320508076+0.4) node {\tiny{$r_{n}$}};
\draw (0.7+6,0.7) node {\tiny{$r_n$}};
\draw (2.3+6,0.7) node {\tiny{$r_n$}};

\draw (1.5+0.3+3,0.75*1.7320508076+0.4+3*1.7320508076) node {\tiny{$r_{n}$}};
\draw (0.7+3,0.7+3*1.7320508076) node {\tiny{$r_n$}};
\draw (2.3+3,0.7+3*1.7320508076) node {\tiny{$r_n$}};

\draw (4.5,0.3) node {\tiny{$1$}};
\draw (2.2,2.25*1.7320508076+0.3) node {\tiny{$1$}};
\draw (6.8,2.25*1.7320508076+0.3) node {\tiny{$1$}};

\end{tikzpicture}
}
\hspace{0.2in}
\subfigure{
\begin{tikzpicture}[scale=0.5]
\draw (0,0)--(4.5,4.5/1.7320508076);
\draw (9,0)--(4.5,4.5/1.7320508076);
\draw (4.5,4.5*1.7320508076)--(4.5,4.5/1.7320508076);

\draw[fill=black] (4.5,4.5/1.7320508076) circle (0.06);
  
\draw[fill=black] (0,0) circle (0.06);
\draw[fill=black] (1.5,1.5/1.7320508076) circle (0.06);

\draw[fill=black] (3+6,0) circle (0.06);
\draw[fill=black] (1.5+6,1.5/1.7320508076) circle (0.06);

\draw[fill=black] (3/2+3,3/2*1.7320508076+3*1.7320508076) circle (0.06);
\draw[fill=black] (1.5+3,1.5/1.7320508076+3*1.7320508076) circle (0.06);

\draw (0,-0.3) node {\tiny{$0^{n+1}$}};
\draw (9,-0.3) node {\tiny{$1^{n+1}$}};
\draw (4.5,9/2*1.7320508076+0.3) node {\tiny{$2^{n+1}$}};

\draw (0.7,0.7) node {\tiny{$r_n$}};

\draw (2.3+6,0.7) node {\tiny{$r_n$}};

\draw (1.5+0.3+3,0.75*1.7320508076+0.4+3*1.7320508076) node {\tiny{$r_{n}$}};

\draw (5.5,4.33012701892219323) node {\tiny{$\frac{2}{3}r_n+\frac{1}{3}$}};
\draw (2,2.3) node {\tiny{$\frac{2}{3}r_n+\frac{1}{3}$}};
\draw (7,2.3) node {\tiny{$\frac{2}{3}r_n+\frac{1}{3}$}};

\end{tikzpicture}
}
\caption{Equivalent Electrical Networks}\label{SG_con_fig_resist1}
\end{figure}
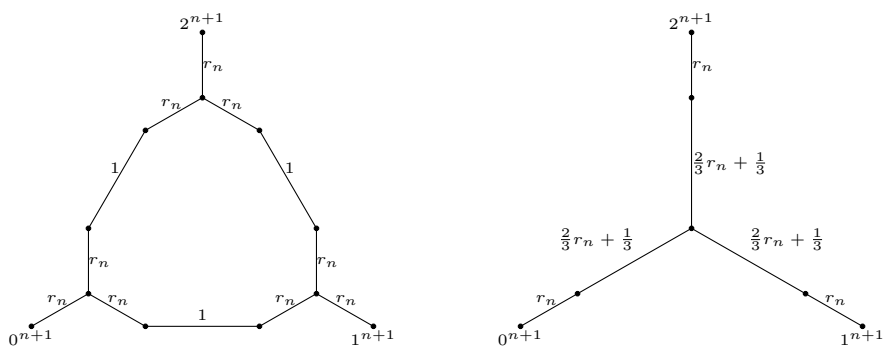


Hence
$$r_{n+1}=\frac{5}{3}r_n+\frac{1}{3}.$$
By elementary calculation, we have
$$r_n=\frac{1}{2}\left(\frac{5}{3}\right)^n-\frac{1}{2}\text{ for all }n\ge1.$$
\end{proof}
\begin{myrmk}\label{SG_con_rmk_resist012}
For all $n\ge1$, we have
$$R_n(0^n,1^n)=R_n(1^n,2^n)=R_n(0^n,2^n)=2r_n=\left(\frac{5}{3}\right)^n-1.$$
\end{myrmk}

\begin{myprop}\label{SG_con_prop_resist}
For all $n\ge1,w\in W_n$, we have
$$R_n(w,0^n),R_n(w,1^n,),R_n(w,2^n)\le\frac{5}{2}\left(\frac{5}{3}\right)^n.$$
\end{myprop}

\begin{proof}
By symmetry, we only need to consider $R_n(w,0^n)$. Letting
$$w=w_1\ldots w_{n-2}w_{n-1}w_n,$$
we construct a finite sequence in $W_n$ as follows.
\begin{align*}
w^{(1)}&=w_1\ldots w_{n-2}w_{n-1}w_{n}=w,\\
w^{(2)}&=w_1\ldots w_{n-2}w_{n-1}w_{n-1},\\
w^{(3)}&=w_1\ldots w_{n-2}w_{n-2}w_{n-2},\\
&\ldots\\
w^{(n)}&=w_1\ldots w_1w_1w_1,\\
w^{(n+1)}&=0\ldots000.
\end{align*}
For all $i=1,\ldots,n-1$, by cutting technique, we have
\begin{align*}
&R_n(w^{(i)},w^{(i+1)})\\
=&R_n(w_1\ldots w_{n-i-1}w_{n-i}w_{n-i+1}\ldots w_{n-i+1},w_1\ldots w_{n-i-1}w_{n-i}w_{n-i}\ldots w_{n-i})\\
\le& R_{i}(w_{n-i+1}\ldots w_{n-i+1},w_{n-i}\ldots w_{n-i})\le R_i(0^i,1^i)=\left(\frac{5}{3}\right)^i-1\le\left(\frac{5}{3}\right)^i.
\end{align*}
Since
$$R_n(w^{(n)},w^{(n+1)})=R_n(w_1^n,0^n)\le R_n(0^n,1^n)=\left(\frac{5}{3}\right)^n-1\le\left(\frac{5}{3}\right)^n,$$
we have
$$
R_n(w,0^n)=R_n(w^{(1)},w^{(n+1)})\le\sum_{i=1}^nR_n(w^{(i)},w^{(i+1)})\le\sum_{i=1}^n\left(\frac{5}{3}\right)^i\le\frac{5}{2}\left(\frac{5}{3}\right)^n.
$$
\end{proof}

Secondly, we introduce \emph{good} functions with energy property and separation property.

Let $U^{(x_0,x_1,x_2)}$ and $\calU$ be defined as in Section \ref{sec_SG}.

For all $n\ge1$, let 
$$A_n(u)=E_n(P_nu,P_nu)=\sum_{w^{(1)}\sim_nw^{(2)}}\left(P_nu(w^{(1)})-P_nu(w^{(2)})\right)^2,u\in L^2(K;\nu).$$

For all $n\ge0$, let
$$B_n(u)=\sum_{w\in W_n}\sum_{p,q\in V_w}(u(p)-u(q))^2,u\in l(K).$$

By Theorem \ref{thm_SG_fun}, for all $U=U^{(x_0,x_1,x_2)}\in\calU,n\ge0$, we have
$$B_n(U)=\left(\frac{3}{5}\right)^n\left((x_0-x_1)^2+(x_1-x_2)^2+(x_0-x_2)^2\right).$$

We calculate $A_n(U)$ as follows.

\begin{mythm}\label{SG_con_thm_AnU}
For all $n\ge1$, we have
$$A_n(U)=\frac{2}{3}\left[\left(\frac{3}{5}\right)^n-\left(\frac{3}{5}\right)^{2n}\right]\left((x_0-x_1)^2+(x_1-x_2)^2+(x_0-x_2)^2\right).$$
\end{mythm}

\begin{myrmk}
The above result was also obtained in \cite[Theorem 3.1]{Str01}.
\end{myrmk}

\begin{proof}
We observe the following facts.
\begin{itemize}
\item For all $w^{(1)}\sim_nw^{(2)}$ is of type \Rmnum{1}, $w^{(1)},w^{(2)}$ are of the form $wi,wj$ for some $w\in W_{n-1}$ and $i,j=0,1,2$ with $i\ne j$. On the other hand, for all $w\in W_{n-1}$ and $i,j=0,1,2$ with $i\ne j$, $wi\sim_nwj$ is of type \Rmnum{1}.
\item For all $w^{(1)},w^{(2)}\in W_n$ such that
$$w^{(1)}=w^{(1)}_1\ldots w^{(1)}_n\sim_nw^{(2)}=w^{(2)}_1\ldots w^{(2)}_n$$
is of type \Rmnum{2}, there exists $k=1,\ldots,n-1$ such that $w^{(1)}_1\ldots w^{(1)}_k\sim_kw^{(2)}_1\ldots w^{(2)}_k$ is of type \Rmnum{1} and
\begin{align*}
w^{(2)}_{k}&=w^{(1)}_{k+1}=\ldots=w^{(1)}_{n},\\
w^{(1)}_{k}&=w^{(2)}_{k+1}=\ldots=w^{(2)}_{n}.
\end{align*}
On the other hand, for all $w^{(1)},w^{(2)}\in W_k$ such that
$$w^{(1)}_1\ldots w^{(1)}_k\sim_kw^{(2)}_1\ldots w^{(2)}_k$$
is of type \Rmnum{1}, we have
$$w^{(1)}_1\ldots w^{(1)}_kw^{(2)}_k\ldots w^{(2)}_k\sim_nw^{(2)}_1\ldots w^{(2)}_kw^{(1)}_k\ldots w^{(1)}_{k}$$
is of type \Rmnum{2} for all $n=k+1,k+2,\ldots$.
\end{itemize}

It is obvious that for all $n\ge1,w\in W_n$, we have $V_{w}=\myset{P_{w0},P_{w1},P_{w2}}$ and 
$$P_nU(w)=\frac{1}{\nu(K_w)}\int_{K_w}U(x)\nu(\md x)=\frac{1}{3}\left(U(P_{w0})+U(P_{w1})+U(P_{w2})\right).$$




\begin{figure}[ht]
\centering
\begin{tikzpicture}
\draw (0,0)--(4,0)--(2,2*1.7320508076)--cycle;
\draw (2,0)--(1,1.7320508076)--(3,1.7320508076)--cycle;

\draw[fill=black] (0,0) circle (0.06);
\draw[fill=black] (4,0) circle (0.06);
\draw[fill=black] (2,2*1.7320508076) circle (0.06);

\draw[fill=black] (1,1*1.7320508076) circle (0.06);
\draw[fill=black] (3,1*1.7320508076) circle (0.06);
\draw[fill=black] (2,0) circle (0.06);

\draw (0,-0.5) node {$P_{w0}$};
\draw (4,-0.5) node {$P_{w1}$};
\draw (2,2*1.7320508076+0.5) node {$P_{w2}$};

\draw (-0.1,1.7320508076) node {$P_{w02}=P_{w20}$};
\draw (2,-0.5) node {$P_{w01}=P_{w10}$};
\draw (4.1,1.7320508076) node {$P_{w12}=P_{w21}$};

\draw (2,1.7320508076+0.5) node {$K_{w2}$};
\draw (1,0.5) node {$K_{w0}$};
\draw (3,0.5) node {$K_{w1}$};

\end{tikzpicture}
\caption{Cells and Nodes}\label{SG_con_fig_AnU1}
\end{figure}
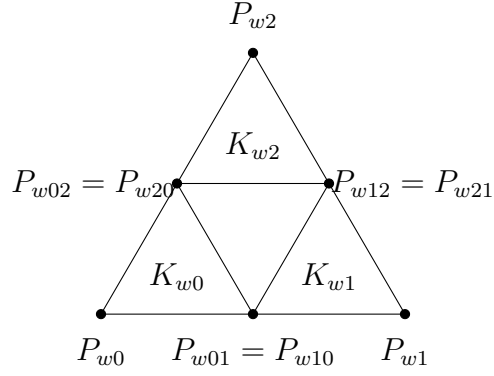


For all $n\ge1,w\in W_{n-1}$, we have
\begin{align*}
&P_{n}U(w0)=\frac{1}{3}\left(U(P_{w00})+U(P_{w01})+U(P_{w02})\right)\\
=&\frac{1}{3}\left(U(P_{w0})+\frac{2U(P_{w0})+2U(P_{w1})+U(P_{w2})}{5}+\frac{2U(P_{w0})+U(P_{w1})+2U(P_{w2})}{5}\right)\\
=&\frac{1}{3}\frac{9U(P_{w0})+3U(P_{w1})+3U(P_{w2})}{5}=\frac{3U(P_{w0})+U(P_{w1})+U(P_{w2})}{5}.
\end{align*}
Similarly
\begin{align*}
P_{n}U(w1)&=\frac{U(P_{w0})+3U(P_{w1})+U(P_{w2})}{5},\\
P_{n}U(w2)&=\frac{U(P_{w0})+U(P_{w1})+3U(P_{w2})}{5}.
\end{align*}
Hence
\begin{align*}
&\left(P_{n}U(w0)-P_{n}U(w1)\right)^2+\left(P_{n}U(w1)-P_{n}U(w2)\right)^2+\left(P_{n}U(w0)-P_{n}U(w2)\right)^2\\
=&\frac{4}{25}\left[\left(U(P_{w0})-U(P_{w1})\right)^2+\left(U(P_{w1})-U(P_{w2})\right)^2+\left(U(P_{w0})-U(P_{w2})\right)^2\right].
\end{align*}




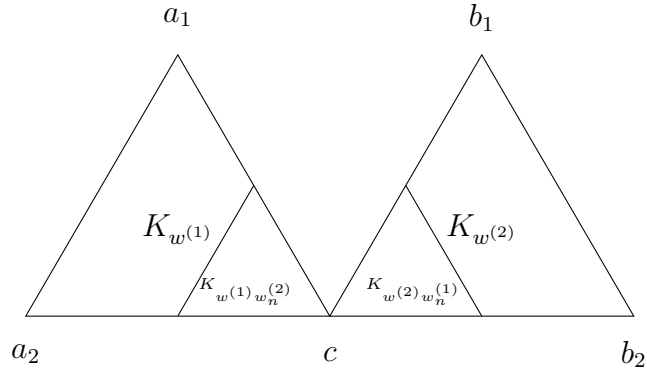
\begin{figure}[ht]
\centering
\begin{tikzpicture}
\draw (0,0)--(4,0)--(2,2*1.7320508076)--cycle;
\draw (4,0)--(8,0)--(6,2*1.7320508076)--cycle;

\draw (2,0)--(3,1.7320508076);

\draw (2,2/1.7320508076) node {$K_{w^{(1)}}$};
\draw (2.9,1/1.7320508076-0.25) node {\tiny{$K_{w^{(1)}w^{(2)}_n}$}};

\draw (6,0)--(5,1.7320508076);

\draw (6,2/1.7320508076) node {$K_{w^{(2)}}$};
\draw (5.1,1/1.7320508076-0.25) node {\tiny{$K_{w^{(2)}w^{(1)}_n}$}};

\draw (0,-0.5) node {$a_2$};
\draw (2,2*1.7320508076+0.5) node {$a_1$};
\draw (8,-0.5) node {$b_2$};
\draw (6,2*1.7320508076+0.5) node {$b_1$};

\draw (4,-0.5) node {$c$};

\end{tikzpicture}
\caption{Adjacent Cells}\label{SG_con_fig_AnU2}
\end{figure}


For all $n\ge1,w^{(1)}=w^{(1)}_1\ldots w^{(1)}_n\sim_nw^{(2)}=w^{(2)}_1\ldots w^{(2)}_n$. Assume that $U$ takes values $a_1,a_2,c$ and $b_1,b_2,c$ on $V_{w^{(1)}}$ and $V_{w^{(2)}}$, respectively, see Figure \ref{SG_con_fig_AnU2}. By above, we have
$$P_nU(w^{(1)})=\frac{a_1+a_2+c}{3},P_nU(w^{(2)})=\frac{b_1+b_2+c}{3},$$
$$P_{n+1}U(w^{(1)}w^{(2)}_n)=\frac{a_1+a_2+3c}{5},P_{n+1}U(w^{(2)}w^{(1)}_n)=\frac{b_1+b_2+3c}{5},$$
hence
$$P_nU(w^{(1)})-P_nU(w^{(2)})=\frac{1}{3}\left((a_1+a_2)-(b_1+b_2)\right),$$
$$P_{n+1}U(w^{(1)}w^{(2)}_n)-P_{n+1}U(w^{(2)}w^{(1)}_n)=\frac{1}{5}\left((a_1+a_2)-(b_1+b_2)\right).$$
Hence
$$P_{n+1}U(w^{(1)}w^{(2)}_n)-P_{n+1}U(w^{(2)}w^{(1)}_n)=\frac{3}{5}\left(P_nU(w^{(1)})-P_nU(w^{(2)})\right).$$
Therefore
\begin{align*}
A_n(U)&=\frac{4}{25}B_{n-1}(U)+\left(\frac{3}{5}\right)^2\left[\frac{4}{25}B_{n-2}(U)\right]+\ldots+\left(\frac{3}{5}\right)^{2(n-1)}\left[\frac{4}{25}B_{0}(U)\right]\\
&=\frac{4}{25}\left((x_0-x_1)^2+(x_1-x_2)^2+(x_0-x_2)^2\right)\\
&\cdot\left[\left(\frac{9}{25}\right)^0\left(\frac{3}{5}\right)^{n-1}+\left(\frac{9}{25}\right)^1\left(\frac{3}{5}\right)^{n-2}+\ldots+\left(\frac{9}{25}\right)^{n-1}\left(\frac{3}{5}\right)^{0}\right]\\
&=\frac{2}{3}\left[\left(\frac{3}{5}\right)^{n}-\left(\frac{3}{5}\right)^{2n}\right]\left((x_0-x_1)^2+(x_1-x_2)^2+(x_0-x_2)^2\right).
\end{align*}
\end{proof}

\section{Weak Monotonicity Result}\label{SG_con_sec_monotone}
In this section, we give weak monotonicity result using resistance estimates.

For all $n\ge1$, let
$$D_n(u)=\left(\frac{5}{3}\right)^nA_n(u)=\left(\frac{5}{3}\right)^n\sum_{w^{(1)}\sim_nw^{(2)}}\left(P_nu(w^{(1)})-P_nu(w^{(2)})\right)^2,u\in L^2(K;\nu).$$

The weak monotonicity result is as follows.

\begin{mythm}\label{SG_con_thm_monotone}
There exists some positive constant $C$ such that
$$D_n(u)\le CD_{n+m}(u)\text{ for all }u\in L^2(K;\nu),n,m\ge1.$$
Indeed, we can take $C=36$.
\end{mythm}

\begin{myrmk}
In Kigami's construction, the energies are monotone, that is, the constant $C=1$. Hence, the above result is called weak monotonicity.
\end{myrmk}

Theorem \ref{SG_con_thm_monotone} can be reduced as follows.

For all $n\ge1$, let
$$G_n(u)=\left(\frac{5}{3}\right)^nE_n(u,u)=\left(\frac{5}{3}\right)^n\sum_{w^{(1)}\sim_nw^{(2)}}\left(u(w^{(1)})-u(w^{(2)})\right)^2,u\in l(W_n).$$

For all $n,m\ge1$, let $M_{n,m}:l(W_{n+m})\to l(W_n)$ be a mean value operator given by
$$(M_{n,m}u)(w)=\frac{1}{3^m}\sum_{v\in W_m}u(wv),w\in W_n,u\in l(W_{n+m}).$$

\begin{mythm}\label{SG_con_thm_monotone_graph}
There exists some positive constant $C$ such that
$$G_n(M_{n,m}u)\le CG_{n+m}(u)\text{ for all }u\in l(W_{n+m}),n,m\ge1.$$
\end{mythm}

\begin{proof}[Proof of Theorem \ref{SG_con_thm_monotone} using Theorem \ref{SG_con_thm_monotone_graph}]
Note $P_nu=M_{n,m}(P_{n+m}u)$, hence
\begin{align*}
D_n(u)&=\left(\frac{5}{3}\right)^n\sum_{w^{(1)}\sim_nw^{(2)}}\left(P_nu(w^{(1)})-P_nu(w^{(2)})\right)^2=G_n(P_nu)\\
&=G_n(M_{n,m}(P_{n+m}u))\le CG_{n+m}(P_{n+m}u)\\
&=C\left(\frac{5}{3}\right)^{n+m}\sum_{w^{(1)}\sim_{n+m}w^{(2)}}\left(P_{n+m}u(w^{(1)})-P_{n+m}u(w^{(2)})\right)^2=CD_{n+m}(u).
\end{align*}
\end{proof}

\begin{myrmk}
The constant in Theorem \ref{SG_con_thm_monotone} can be taken as the one in Theorem \ref{SG_con_thm_monotone_graph}.
\end{myrmk}

\begin{proof}[Proof of Theorem \ref{SG_con_thm_monotone_graph}]
For all $n\ge1$. Assume that $W\subseteq W_n$ is connected, that is, for all $w^{(1)},w^{(2)}\in W$, there exists a finite sequence $\myset{v^{(1)},\ldots,v^{(k)}}\subseteq W$ with $v^{(1)}=w^{(1)},v^{(k)}=w^{(2)}$ and $v^{(i)}\sim_nv^{(i+1)}$ for all $i=1,\ldots,k-1$. Let 
$$E_W(u,u)=
\sum_{\mbox{\tiny
$
\begin{subarray}{c}
w^{(1)},w^{(2)}\in W\\
w^{(1)}\sim_nw^{(2)}
\end{subarray}
$}}
(u(w^{(1)})-u(w^{(2)}))^2,u\in l(W).$$
For all $w^{(1)},w^{(2)}\in W$, let 
\begin{align*}
R_W(w^{(1)},w^{(2)})&=\inf\myset{E_W(u,u):u(w^{(1)})=1,u(w^{(2)})=0,u\in l(W)}^{-1}\\
&=\sup\myset{\frac{(u(w^{(1)})-u(w^{(2)}))^2}{E_W(u,u)}:E_W(u,u)\ne0,u\in l(W)}.
\end{align*}
It is obvious that
$$\left(u(w^{(1)})-u(w^{(2)})\right)^2\le R_W(w^{(1)},w^{(2)})E_W(u,u)\text{ for all }w^{(1)},w^{(2)}\in W,u\in l(W),$$
and $R_W$ is a metric on $W$, hence
$$R_W(w^{(1)},w^{(2)})\le R_W(w^{(1)},w^{(3)})+R_W(w^{(3)},w^{(2)})\text{ for all }w^{(1)},w^{(2)},w^{(3)}\in W.$$
By definition, we have
\begin{align*}
G_n(M_{n,m}u)&=\left(\frac{5}{3}\right)^n\sum_{w^{(1)}\sim_nw^{(2)}}\left(M_{n,m}u(w^{(1)})-M_{n,m}u(w^{(2)})\right)^2\\
&=\left(\frac{5}{3}\right)^n\sum_{w^{(1)}\sim_nw^{(2)}}\left(\frac{1}{3^m}\sum_{v\in W_m}\left(u(w^{(1)}v)-u(w^{(2)}v)\right)\right)^2\\
&\le\left(\frac{5}{3}\right)^n\sum_{w^{(1)}\sim_nw^{(2)}}\frac{1}{3^m}\sum_{v\in W_m}\left(u(w^{(1)}v)-u(w^{(2)}v)\right)^2.
\end{align*}

Fix $w^{(1)}\sim_nw^{(2)}$, there exist $i,j=0,1,2$ with $i\ne j$ such that $w^{(1)}i^m\sim_{n+m}w^{(2)}j^m$, see Figure \ref{SG_con_fig_monotone}.




\begin{figure}[ht]
\centering
\begin{tikzpicture}
\draw (0,0)--(4,0)--(2,2*1.7320508076)--cycle;
\draw (5,0)--(9,0)--(7,2*1.7320508076)--cycle;

\draw (4,0)--(5,0);

\draw[fill=black] (4,0) circle (0.06);
\draw[fill=black] (5,0) circle (0.06);

\draw (2,2/1.7320508076) node {$w^{(1)}W_m$};
\draw (7,2/1.7320508076) node {$w^{(2)}W_m$};

\draw (3.7,-0.5) node {$w^{(1)}i^m$};
\draw (5.3,-0.5) node {$w^{(2)}j^m$};

\draw[fill=black] (1.3,1.7) circle (0.06);
\draw (2,2) node {$w^{(1)}v$};

\draw[fill=black] (6.3,1.7) circle (0.06);
\draw (7,2) node {$w^{(2)}v$};

\end{tikzpicture}
\caption{$w^{(1)}W_m$ and $w^{(2)}W_m$}\label{SG_con_fig_monotone}
\end{figure}
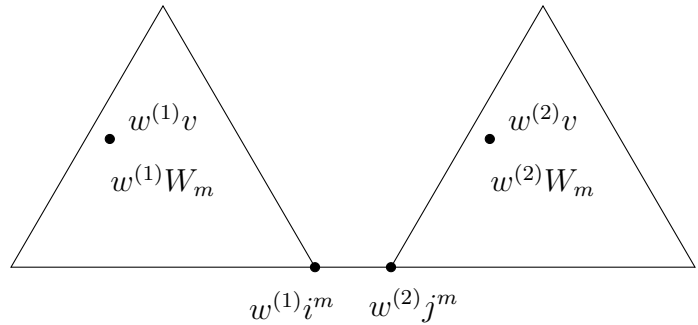


Fix $v\in W_m$, we have
$$\left(u(w^{(1)}v)-u(w^{(2)}v)\right)^2\le R_{w^{(1)}W_m\cup w^{(2)}W_m}(w^{(1)}v,w^{(2)}v)E_{w^{(1)}W_m\cup w^{(2)}W_m}(u,u).$$
By cutting technique and Proposition \ref{SG_con_prop_resist}, we have
\begin{align*}
&R_{w^{(1)}W_m\cup w^{(2)}W_m}(w^{(1)}v,w^{(2)}v)\\
\le& R_{w^{(1)}W_m\cup w^{(2)}W_m}(w^{(1)}v,w^{(1)}i^m)+R_{w^{(1)}W_m\cup w^{(2)}W_m}(w^{(1)}i^m,w^{(2)}j^m)\\
&+R_{w^{(1)}W_m\cup w^{(2)}W_m}(w^{(2)}j^m,w^{(2)}v)\\
\le& R_m(v,i^m)+1+R_m(v,j^m)\le5\left(\frac{5}{3}\right)^m+1\le6\left(\frac{5}{3}\right)^m,
\end{align*}
hence
\begin{align*}
&(u(w^{(1)}v)-u(w^{(2)}v))^2\le6\left(\frac{5}{3}\right)^mE_{w^{(1)}W_m\cup w^{(2)}W_m}(u,u)\\
=&6\left(\frac{5}{3}\right)^m\left(E_{w^{(1)}W_m}(u,u)+E_{w^{(2)}W_m}(u,u)+\left(u(w^{(1)}i^m)-u(w^{(2)}j^m)\right)^2\right).
\end{align*}
Hence
\begin{align*}
&\frac{1}{3^m}\sum_{v\in W_m}\left(u(w^{(1)}v)-u(w^{(2)}v)\right)^2\\
\le&6\left(\frac{5}{3}\right)^m\left(E_{w^{(1)}W_m}(u,u)+E_{w^{(2)}W_m}(u,u)+\left(u(w^{(1)}i^m)-u(w^{(2)}j^m)\right)^2\right).
\end{align*}
In the summation with respect to $w^{(1)}\sim_nw^{(2)}$, the terms $E_{w^{(1)}W_m}(u,u),E_{w^{(2)}W_m}(u,u)$ are summed at most 6 times, hence
\begin{align*}
&G_n(M_{n,m}u)\le\left(\frac{5}{3}\right)^n\sum_{w^{(1)}\sim_nw^{(2)}}\frac{1}{3^m}\sum_{v\in W_m}\left(u(w^{(1)}v)-u(w^{(2)}v)\right)^2\\
\le&6\left(\frac{5}{3}\right)^n6\left(\frac{5}{3}\right)^mE_{n+m}(u,u)=36\left(\frac{5}{3}\right)^{n+m}E_{n+m}(u,u)=36G_{n+m}(u).
\end{align*}
\end{proof}

\section{Proof of Theorem \ref{SG_con_thm_BM}}\label{SG_con_sec_BM}

For all $\beta>0$, let 
$$\frakE_\beta(u,u)=\sum_{n=1}^\infty2^{(\beta-\alpha)n}\sum_{w^{(1)}\sim_nw^{(2)}}\left(P_nu(w^{(1)})-P_nu(w^{(2)})\right)^2,$$
denote $\scrE_\beta(u,u)=[u]_{B^{2,2}_{\alpha,\beta}(K)}$ for simplicity.

We obtain non-local regular closed forms and Dirichlet forms as follows. 

\begin{mythm}\label{SG_con_thm_nonlocal}
For all $\beta\in(\alpha,\beta^*)$, $(\frakE_\beta,\calF_\beta)$ is a non-local regular closed form on $L^2(K;\nu)$, $(\calE_\beta,\calF_\beta)$ and $(\scrE_\beta,\calF_\beta)$ are non-local regular Dirichlet forms on $L^2(K;\nu)$. For all $\beta\in[\beta^*,+\infty)$, $\calF_\beta$ consists only of constant functions.
\end{mythm}

\begin{myrmk}
$\frakE_\beta$ does not have Markovian property but $\calE_\beta$ and $\scrE_\beta$ do have Markovian property.
\end{myrmk}

\begin{proof}
By Fatou's lemma, it is obvious that $(\frakE_\beta,\calF_\beta)$ is a closed form on $L^2(K;\nu)$ in the wide sense.

For all $\beta\in(\alpha,\beta^*)$. By Lemma \ref{lem_SG_holder}, we have $\calF_\beta\subseteq C(K)$. We only need to show that $\calF_\beta$ is uniformly dense in $C(K)$, then $\calF_\beta$ is dense in $L^2(K;\nu)$, hence $(\frakE_\beta,\calF_\beta)$ is a regular closed form on $L^2(K;\nu)$.

Indeed, by Theorem \ref{SG_con_thm_AnU}, for all $U=U^{(x_0,x_1,x_2)}\in\calU$, we have
\begin{align*}
&\frakE_\beta(U,U)=\sum_{n=1}^\infty2^{(\beta-\alpha)n}A_n(U)\\
=&\sum_{n=1}^\infty2^{(\beta-\alpha)n}\frac{2}{3}\left[\left(\frac{3}{5}\right)^{n}-\left(\frac{3}{5}\right)^{2n}\right]\left((x_0-x_1)^2+(x_1-x_2)^2+(x_0-x_2)^2\right)\\
\le&\frac{2}{3}\left((x_0-x_1)^2+(x_1-x_2)^2+(x_0-x_2)^2\right)\sum_{n=1}^\infty2^{(\beta-\alpha)n}\left(\frac{3}{5}\right)^n<+\infty,
\end{align*}
hence $U\in\calF_\beta$, $\calU\subseteq\calF_\beta$. By Theorem \ref{thm_SG_fun}, we have $\calF_\beta$ separates points. It is obvious that $\calF_\beta$ is a sub-algebra of $C(K)$, that is, for all $u,v\in\calF_\beta,c\in\R$, we have $u+v,cu,uv\in\calF_\beta$. By Stone-Weierstrass theorem, $\calF_\beta$ is uniformly dense in $C(K)$.

Since $\calE_\beta,\scrE_\beta$ do have Markovian property, by above, $(\calE_\beta,\calF_\beta),(\scrE_\beta,\calF_\beta)$ are non-local regular Dirichlet forms on $L^2(K;\nu)$.

For all $\beta\in[\beta^*,+\infty)$. Assume that $u\in\calF_\beta$ is not constant, then there exists some integer $N\ge1$ such that $D_N(u)>0$. By Theorem \ref{SG_con_thm_monotone}, we have
\begin{align*}
\frakE_\beta(u,u)&=\sum_{n=1}^\infty2^{(\beta-\alpha)n}\left(\frac{3}{5}\right)^nD_n(u)\ge\sum_{n=N+1}^\infty2^{(\beta-\alpha)n}\left(\frac{3}{5}\right)^nD_n(u)\\
&\ge\frac{1}{C}\sum_{n=N+1}^\infty2^{(\beta-\alpha)n}\left(\frac{3}{5}\right)^nD_N(u)=+\infty,
\end{align*}
contradiction! Hence $\calF_\beta$ consists only of constant functions.
\end{proof}

Take $\myset{\beta_n}\subseteq(\alpha,\beta^*)$ with $\beta_n\uparrow\beta^*$. By Proposition \ref{prop_gamma}, there exist some subsequence still denoted by $\myset{\beta_n}$ and some closed form $(\calE,\calF)$ on $L^2(K;\nu)$ in the wide sense such that $(\beta^*-\beta_n)\frakE_{\beta_n}$ is $\Gamma$-convergent to $\calE$. Without lose of generality, we may assume that
$$0<\beta^*-\beta_n<\frac{1}{n+1}\text{ for all }n\ge1.$$

We have the characterization of $(\calE,\calF)$ on $L^2(K;\nu)$ as follows.

\begin{mythm}\label{SG_con_thm_E}
\begin{align*}
&\calE(u,u)\asymp\sup_{n\ge1}D_n(u)=\sup_{n\ge1}\left(\frac{5}{3}\right)^n\sum_{w^{(1)}\sim_nw^{(2)}}\left(P_nu(w^{(1)})-P_nu(w^{(2)})\right)^2,\\
&\calF=\myset{u\in L^2(K;\nu):\sup_{n\ge1}\left(\frac{5}{3}\right)^n\sum_{w^{(1)}\sim_nw^{(2)}}\left(P_nu(w^{(1)})-P_nu(w^{(2)})\right)^2<+\infty}.
\end{align*}
Moreover, $(\calE,\calF)$ is a regular closed form on $L^2(K;\nu)$ and
$$\frac{1}{2(\log2)C^2}\sup_{n\ge1}D_n(u)\le\calE(u,u)\le\frac{1}{\log2}\sup_{n\ge1}D_n(u).$$
\end{mythm}

\begin{proof}
Recall that
$$\frakE_{\beta}(u,u)=\sum_{n=1}^\infty2^{(\beta-\alpha)n}A_n(u)=\sum_{n=1}^\infty2^{(\beta-\beta^*)n}D_n(u).$$

We use weak monotonicity result Theorem \ref{SG_con_thm_monotone} and elementary result Proposition \ref{prop_ele2}.

On the one hand, for all $u\in L^2(K;\nu)$
\begin{align*}
\calE(u,u)&\le\varliminf_{n\to+\infty}(\beta^*-\beta_n)\calE_{\beta_n}(u,u)=\varliminf_{n\to+\infty}(\beta^*-\beta_n)\sum_{k=1}^\infty2^{(\beta_n-\beta^*)k}D_k(u)\\
&=\frac{1}{\log2}\varliminf_{n\to+\infty}(1-2^{\beta_n-\beta^*})\sum_{k=1}^\infty2^{(\beta_n-\beta^*)k}D_k(u)\le\frac{1}{\log2}\sup_{k\ge1}D_k(u).
\end{align*}
On the other hand, for all $u\in L^2(K;\nu)$, there exists $\myset{u_n}\subseteq L^2(K;\nu)$ converging strongly to $u$ in $L^2(K;\nu)$ such that
\begin{align*}
&\calE(u,u)\ge\varlimsup_{n\to+\infty}(\beta^*-\beta_n)\calE_{\beta_n}(u_n,u_n)=\varlimsup_{n\to+\infty}(\beta^*-\beta_n)\sum_{k=1}^\infty2^{(\beta_n-\beta^*)k}D_k(u_n)\\
\ge&\varlimsup_{n\to+\infty}(\beta^*-\beta_n)\sum_{k=n+1}^\infty2^{(\beta_n-\beta^*)k}D_k(u_n)\ge\frac{1}{C}\varlimsup_{n\to+\infty}(\beta^*-\beta_n)\sum_{k=n+1}^\infty2^{(\beta_n-\beta^*)k}D_n(u_n)\\
=&\frac{1}{C}\varlimsup_{n\to+\infty}\left[(\beta^*-\beta_n)\frac{2^{(\beta_n-\beta^*)(n+1)}}{1-2^{\beta_n-\beta^*}}D_n(u_n)\right].
\end{align*}
Since $0<\beta^*-\beta_n<1/(n+1)$, we have $2^{(\beta_n-\beta^*)(n+1)}>1/2$. Since
$$\lim_{n\to+\infty}\frac{\beta^*-\beta_n}{1-2^{\beta_n-\beta^*}}=\frac{1}{\log2},$$
we have
$$\calE(u,u)\ge\frac{1}{2C}\varlimsup_{n\to+\infty}\frac{\beta^*-\beta_n}{1-2^{\beta_n-\beta^*}}D_n(u_n)\ge\frac{1}{2(\log2)C}\varlimsup_{n\to+\infty}D_n(u_n).$$
Since $u_n\to u$ in $L^2(K;\nu)$, for all $k\ge1$, we have
$$D_k(u)=\lim_{n\to+\infty} D_k(u_n)=\lim_{k\le n\to+\infty} D_k(u_n)\le C\varliminf_{n\to+\infty} D_n(u_n).$$
Taking supremum with respect to $k\ge1$, we have
$$\sup_{k\ge1}D_k(u)\le C\varliminf_{n\to+\infty}D_n(u_n)\le C\varlimsup_{n\to+\infty}D_n(u_n)\le 2(\log2)C^2\calE(u,u).$$

By Lemma \ref{lem_SG_holder}, we have $\calF\subseteq C(K)$. We only need to show that $\calF$ is uniformly dense in $C(K)$, then $\calF$ is dense in $L^2(K;\nu)$, hence $(\calE,\calF)$ is a regular closed form on $L^2(K;\nu)$.

Indeed, by Theorem \ref{SG_con_thm_AnU}, for all $U=U^{(x_0,x_1,x_2)}\in\calU$, we have
\begin{align*}
&\sup_{n\ge1}D_n(U)=\sup_{n\ge1}\left(\frac{5}{3}\right)^nA_n(U)\\
=&\sup_{n\ge1}\left(\frac{5}{3}\right)^n\frac{2}{3}\left[\left(\frac{3}{5}\right)^n-\left(\frac{3}{5}\right)^{2n}\right]\left((x_0-x_1)^2+(x_1-x_2)^2+(x_0-x_2)^2\right)\\
\le&\frac{2}{3}\left((x_0-x_1)^2+(x_1-x_2)^2+(x_0-x_2)^2\right)<+\infty,
\end{align*}
hence $U\in\calF$, $\calU\subseteq\calF$. By Theorem \ref{thm_SG_fun}, we have $\calF$ separates points. It is obvious that $\calF$ is a sub-algebra of $C(K)$. By Stone-Weierstrass theorem, $\calF$ is uniformly dense in $C(K)$.
\end{proof}

Now we prove Theorem \ref{SG_con_thm_BM} using a standard approach as follows.

\begin{proof}[Proof of Theorem \ref{SG_con_thm_BM}]
For all $u\in L^2(K;\nu),n,k\ge1$, we have
\begin{align*}
&\sum_{w^{(1)}\sim_{n+k}w^{(2)}}\left(P_{n+k}u(w^{(1)})-P_{n+k}u(w^{(2)})\right)^2\\
=&\sum_{w\in W_n}\sum_{w^{(1)}\sim_kw^{(2)}}\left(P_{n+k}u(ww^{(1)})-P_{n+k}u(ww^{(2)})\right)^2\\
&+\sum_{w^{(1)}=w^{(1)}_1\ldots w^{(1)}_n\sim_nw^{(2)}=w^{(2)}_1\ldots w^{(2)}_n}\left(P_{n+k}u(w^{(1)}w^{(2)}_n\ldots w^{(2)}_n)-P_{n+k}u(w^{(2)}w^{(1)}_n\ldots w^{(1)}_n)\right)^2,
\end{align*}
where for all $i=1,2$
\begin{align*}
&P_{n+k}u(ww^{(i)})=\int_K(u\circ f_{ww^{(i)}})(x)\nu(\md x)\\
=&\int_K(u\circ f_{w}\circ f_{w^{(i)}})(x)\nu(\md x)=P_k(u\circ f_w)(w^{(i)}),
\end{align*}
hence
\begin{align*}
\sum_{w\in W_n}A_k(u\circ f_w)=&\sum_{w\in W_n}\sum_{w^{(1)}\sim_kw^{(2)}}\left(P_{k}(u\circ f_w)(w^{(1)})-P_{k}(u\circ f_w)(w^{(2)})\right)^2\\
\le&\sum_{w^{(1)}\sim_{n+k}w^{(2)}}\left(P_{n+k}u(w^{(1)})-P_{n+k}u(w^{(2)})\right)^2=A_{n+k}(u),
\end{align*}
and
\begin{align*}
\left(\frac{5}{3}\right)^n\sum_{w\in W_n}D_k(u\circ f_w)&=\left(\frac{5}{3}\right)^{n+k}\sum_{w\in W_n}A_k(u\circ f_w)\\
&\le\left(\frac{5}{3}\right)^{n+k}A_{n+k}(u)=D_{n+k}(u).
\end{align*}
For all $u\in\calF,n\ge1,w\in W_n$, we have
\begin{align*}
&\sup_{k\ge1}D_k(u\circ f_w)\le\sup_{k\ge1}\sum_{w\in W_n}D_k(u\circ f_w)\\
\le&\left(\frac{3}{5}\right)^n\sup_{k\ge1}D_{n+k}(u)\le\left(\frac{3}{5}\right)^n\sup_{k\ge1}D_k(u)<+\infty,
\end{align*}
hence $u\circ f_w\in\calF$.

For all $u\in L^2(K;\nu),n\ge1$, let
\begin{align*}
\mybar{\calE}(u,u)&=\sum_{i=0}^2\left(u(p_i)-\int_Ku(x)\nu(\md x)\right)^2,\\
\mybar{\calE}^{(n)}(u,u)&=\left(\frac{5}{3}\right)^n\sum_{w\in W_n}\mybar{\calE}(u\circ f_w,u\circ f_w).
\end{align*}
By Lemma \ref{lem_SG_holder}, we have
\begin{align*}
\mybar{\calE}(u,u)&=\sum_{i=0}^2\left(\int_K\left(u(p_i)-u(x)\right)\nu(\md x)\right)^2\le\sum_{i=0}^2\int_K\left(u(p_i)-u(x)\right)^2\nu(\md x)\\
&\le\sum_{i=0}^2\int_Kc^2|p_i-x|^{\beta^*-\alpha}\left(\sup_{k\ge1}D_k(u)\right)\nu(\md x)\le3c^2\sup_{k\ge1}D_k(u),
\end{align*}
hence
\begin{equation}\label{SG_con_eqn_Ebar_upper}
\begin{aligned}
\mybar{\calE}^{(n)}(u,u)&\le\left(\frac{5}{3}\right)^n\sum_{w\in W_n}3c^2\sup_{k\ge1}D_k(u\circ f_w)\\
&\le3c^2C\left(\frac{5}{3}\right)^n\sum_{w\in W_n}\varliminf_{k\to+\infty}D_k(u\circ f_w)\\
\\
&\le3c^2C\left(\frac{5}{3}\right)^n\varliminf_{k\to+\infty}\sum_{w\in W_n}D_k(u\circ f_w)\\
&\le 3c^2C\varliminf_{k\to+\infty}D_{n+k}(u)\le3c^2C\sup_{k\ge1}D_k(u).
\end{aligned}
\end{equation}

On the other hand, for all $u\in L^2(K;\nu),n\ge1$, we have
$$D_n(u)=\left(\frac{5}{3}\right)^n\sum_{w^{(1)}\sim_nw^{(2)}}\left(\int_K(u\circ f_{w^{(1)}})(x)\nu(\md x)-\int_K(u\circ f_{w^{(2)}})(x)\nu(\md x)\right)^2.$$
For all $w^{(1)}\sim_n w^{(2)}$, there exist $i,j=0,1,2$ such that
$$K_{w^{(1)}}\cap K_{w^{(2)}}=\myset{f_{w^{(1)}}(p_i)}=\myset{f_{w^{(2)}}(p_j)}.$$
Hence
\begin{align}
D_n(u)=&\left(\frac{5}{3}\right)^n\sum_{w^{(1)}\sim_nw^{(2)}}\left[\left((u\circ f_{w^{(1)}})(p_i)-\int_K(u\circ f_{w^{(1)}})(x)\nu(\md x)\right)\right.\nonumber\\
&-\left.\left((u\circ f_{w^{(2)}})(p_j)-\int_K(u\circ f_{w^{(2)}})(x)\nu(\md x)\right)\right]^2\nonumber\\
\le&2\left(\frac{5}{3}\right)^n\sum_{w^{(1)}\sim_nw^{(2)}}\left[\left((u\circ f_{w^{(1)}})(p_i)-\int_K(u\circ f_{w^{(1)}})(x)\nu(\md x)\right)^2\right.\nonumber\\
&+\left.\left((u\circ f_{w^{(2)}})(p_j)-\int_K(u\circ f_{w^{(2)}})(x)\nu(\md x)\right)^2\right]\nonumber\\
\le&6\left(\frac{5}{3}\right)^n\sum_{w\in W_n}\sum_{i=0}^2\left((u\circ f_{w})(p_i)-\int_K(u\circ f_{w})(x)\nu(\md x)\right)^2\nonumber\\
=&6\left(\frac{5}{3}\right)^n\sum_{w\in W_n}\mybar{\calE}(u\circ f_w,u\circ f_w)=6\mybar{\calE}^{(n)}(u,u).\label{SG_con_eqn_Ebar_lower}
\end{align}

For all $u\in L^2(K;\nu),n\ge1$, we have
\begin{align}
\mybar{\calE}^{(n+1)}(u,u)&=\left(\frac{5}{3}\right)^{n+1}\sum_{w\in W_{n+1}}\mybar{\calE}(u\circ f_w,u\circ f_w)\nonumber\\
&=\left(\frac{5}{3}\right)^{n+1}\sum_{i=0}^2\sum_{w\in W_{n}}\mybar{\calE}(u\circ f_i\circ f_w,u\circ f_i\circ f_w)\nonumber\\
&=\frac{5}{3}\sum_{i=0}^2\mybar{\calE}^{(n)}(u\circ f_i,u\circ f_i).\label{SG_con_eqn_Ebar_ss}
\end{align}
Let
$$\tilde{\calE}^{(n)}(u,u)=\frac{1}{n}\sum_{l=1}^n\mybar{\calE}^{(l)}(u,u),u\in L^2(K;\nu),n\ge1.$$
By Equation (\ref{SG_con_eqn_Ebar_upper}), we have
$$\tilde{\calE}^{(n)}(u,u)\le3c^2C\sup_{k\ge1}D_k(u)\asymp\calE(u,u)\text{ for all }u\in\calF,n\ge1.$$
Since $(\calE,\calF)$ is a regular closed form on $L^2(K;\nu)$, by \cite[Definition 1.3.8, Remark 1.3.9, Definition 1.3.10, Remark 1.3.11]{CF12}, we have $(\calF,\calE_1)$ is a separable Hilbert space. Let $\{u_i\}_{i\ge1}$ be a dense subset of $(\calF,\calE_1)$. For all $i\ge1$, $\{\tilde{\calE}^{(n)}(u_i,u_i)\}_{n\ge1}$ is a bounded sequence. By diagonal argument, there exists a subsequence $\{n_k\}_{k\ge1}$ such that $\{\tilde{\calE}^{(n_k)}(u_i,u_i)\}_{k\ge1}$ converges for all $i\ge1$. Hence $\{\tilde{\calE}^{(n_k)}(u,u)\}_{k\ge1}$ converges for all $u\in\calF$. Let
$$\calE_{\loc}(u,u)=\lim_{k\to+\infty}\tilde{\calE}^{(n_k)}(u,u)\text{ for all }u\in\calF_{\loc}:=\calF.$$
Then
$$\calE_\loc(u,u)\le3c^2C\sup_{k\ge1}D_k(u)\text{ for all }u\in\calF_\loc=\calF.$$
By Equation (\ref{SG_con_eqn_Ebar_lower}), for all $u\in\calF_\loc=\calF$, we have
$$\calE_\loc(u,u)=\lim_{k\to+\infty}\tilde{\calE}^{(n_k)}(u,u)\ge\varliminf_{n\to+\infty}\mybar{\calE}^{(n)}(u,u)\ge\frac{1}{6}\varliminf_{k\to+\infty}D_k(u)\ge\frac{1}{6C}\sup_{k\ge1}D_k(u).$$
Hence
$$\calE_\loc(u,u)\asymp\sup_{k\ge1}D_k(u)\text{ for all }u\in\calF_\loc=\calF.$$
Hence $(\calE_\loc,\calF_\loc)$ is a regular closed form on $L^2(K;\nu)$. Since $1\in\calF_\loc$ and $\calE_\loc(1,1)=0$, by \cite[Lemma 1.6.5, Theorem 1.6.3]{FOT11}, $(\calE_\loc,\calF_\loc)$ on $L^2(K;\nu)$ is conservative.

For all $u\in\calF_\loc=\calF$, we have $u\circ f_i\in\calF=\calF_\loc$ for all $i=0,1,2$. Moreover, by Equation (\ref{SG_con_eqn_Ebar_ss}), we have
\begin{align*}
\frac{5}{3}\sum_{i=0}^2\calE_\loc(u\circ f_i,u\circ f_i)&=\frac{5}{3}\sum_{i=0}^2\lim_{k\to+\infty}\tilde{\calE}^{(n_k)}(u\circ f_i,u\circ f_i)\\
&=\frac{5}{3}\sum_{i=0}^2\lim_{k\to+\infty}\frac{1}{n_k}\sum_{l=1}^{n_k}\mybar{\calE}^{(l)}(u\circ f_i,u\circ f_i)\\
&=\lim_{k\to+\infty}\frac{1}{n_k}\sum_{l=1}^{n_k}\left[\frac{5}{3}\sum_{i=0}^2\mybar{\calE}^{(l)}(u\circ f_i,u\circ f_i)\right]\\
&=\lim_{k\to+\infty}\frac{1}{n_k}\sum_{l=1}^{n_k}\mybar{\calE}^{(l+1)}(u,u)=\lim_{k\to+\infty}\frac{1}{n_k}\sum_{l=2}^{n_k+1}\mybar{\calE}^{(l)}(u,u)\\
&=\lim_{k\to+\infty}\left[\frac{1}{n_k}\sum_{l=1}^{n_k}\mybar{\calE}^{(l)}(u,u)+\frac{1}{n_k}\mybar{\calE}^{(n_k+1)}(u,u)-\frac{1}{n_k}\mybar{\calE}^{(1)}(u,u)\right]\\
&=\lim_{k\to+\infty}\tilde{\calE}^{(n_k)}(u,u)=\calE_\loc(u,u).
\end{align*}
Hence $(\calE_\loc,\calF_\loc)$ on $L^2(K;\nu)$ is self-similar.

For all $u,v\in\calF_\loc$ satisfying $\mathrm{supp}(u),\mathrm{supp}(v)$ are compact and $v$ is constant in an open neighborhood $U$ of $\mathrm{supp}(u)$, we have $K\backslash U$ is compact and $\mathrm{supp}(u)\cap(K\backslash U)=\emptyset$, hence
$$\delta=\mathrm{dist}(\mathrm{supp}(u),K\backslash U)>0.$$
Taking sufficiently large $n\ge1$ such that $2^{1-n}<\delta$, by self-similarity, we have
$$\calE_\loc(u,v)=\left(\frac{5}{3}\right)^n\sum_{w\in W_n}\calE_\loc(u\circ f_w,v\circ f_w).$$
For all $w\in W_n$, we have $u\circ f_w=0$ or $v\circ f_w$ is constant, hence $\calE_\loc(u\circ f_w,v\circ f_w)=0$, hence $\calE_\loc(u,v)=0$, that is, $(\calE_\loc,\calF_\loc)$ on $L^2(K;\nu)$ is strongly local.

For all $u\in\calF_\loc$, it is obvious that $u^+,u^-,1-u,\mybar{u}=(0\vee u)\wedge 1\in\calF_\loc$ and
$$\calE_\loc(u,u)=\calE_\loc(1-u,1-u).$$
Since $u^+u^-=0$ and $(\calE_\loc,\calF_\loc)$ on $L^2(K;\nu)$ is strongly local, we have $\calE_\loc(u^+,u^-)=0$. Hence
\begin{align*}
\calE_\loc(u,u)&=\calE_\loc(u^+-u^-,u^+-u^-)=\calE_\loc(u^+,u^+)+\calE_\loc(u^-,u^-)-2\calE_\loc(u^+,u^-)\\
&=\calE_\loc(u^+,u^+)+\calE_\loc(u^-,u^-)\ge\calE_\loc(u^+,u^+)=\calE_\loc(1-u^+,1-u^+)\\
&\ge\calE_\loc((1-u^+)^+,(1-u^+)^+)=\calE_\loc(1-(1-u^+)^+,1-(1-u^+)^+)\\
&=\calE_\loc(\mybar{u},\mybar{u}),
\end{align*}
that is, $(\calE_\loc,\calF_\loc)$ on $L^2(K;\nu)$ is Markovian. Hence $(\calE_\loc,\calF_\loc)$ is a self-similar strongly local regular Dirichlet form on $L^2(K;\nu)$.
\end{proof}

\begin{myrmk}
The idea of the standard approach is from \cite[Section 6]{KZ92}. The proof of Markovian property is from the proof of \cite[Theorem 2.1]{BBKT10}.
\end{myrmk}

\chapter{Construction of Local Regular Dirichlet Form on the SC}\label{ch_SC_con}

This chapter is based on my work \cite{GY17} joint  with Prof. Alexander Grigor'yan.

\section{Background and Statement}

We apply the method introduced in Chapter \ref{ch_SG_con} to the SC. We use the notions of the SC introduced in Section \ref{sec_notion}.

Let $\nu$ be the normalized Hausdorff measure on the SC $K$.

Let $(\calE_\beta,\calF_\beta)$ be given by
$$
\begin{aligned}
&\calE_\beta(u,u)=\int_K\int_K\frac{(u(x)-u(y))^2}{|x-y|^{\alpha+\beta}}\nu(\md x)\nu(\md y),\\
&\calF_\beta=\myset{u\in L^2(K;\nu):\calE_\beta(u,u)<+\infty},
\end{aligned}
$$
where $\alpha=\log8/\log3$ is Hausdorff dimension of the SC, $\beta>0$ is so far arbitrary. Then $(\calE_\beta,\calF_\beta)$ is a quadratic form on $L^2(K;\nu)$ for all $\beta\in(0,+\infty)$. Note that $(\calE_\beta,\calF_\beta)$ is not necessary to be a regular Dirichlet form on $L^2(K;\nu)$ related to a stable-like jump process. The \emph{walk dimension of the SC} is defined as
$$\beta_*:=\sup\myset{\beta>0:(\calE_\beta,\calF_\beta)\text{ is a regular Dirichlet form on }L^2(K;\nu)}.$$

We give a new semi-norm $E_\beta$ as follows.
$$E_\beta(u,u):=\sum_{n=1}^\infty3^{(\beta-\alpha)n}\sum_{w\in W_n}
{\sum_{\mbox{\tiny
$
\begin{subarray}{c}
p,q\in V_w\\
|p-q|=2^{-1}\cdot3^{-n}
\end{subarray}
$
}}}
(u(p)-u(q))^2.$$

Our first result is as follows.

\begin{mylem}\label{SC_con_lem_equiv}
For all $\beta\in(\alpha,+\infty)$, for all $u\in C(K)$, we have
$$E_\beta(u,u)\asymp\calE_\beta(u,u).$$
\end{mylem}

We have established similar equivalence on the SG, see Theorem \ref{SG_app_thm_main}.

We use Lemma \ref{SC_con_lem_equiv} to give bound of the walk dimension of the SC as follows.

\begin{mythm}\label{SC_con_thm_bound}
\begin{equation}\label{SC_con_eqn_bound_beta}
\beta_*\in\left[\frac{\log\left(8\cdot\frac{7}{6}\right)}{\log3},\frac{\log\left(8\cdot\frac{3}{2}\right)}{\log3}\right].
\end{equation}
\end{mythm}

This estimate follows also from the results of \cite{BB90} and \cite{BBS90} where the same bound 
for $\beta^*$ was obtained by means of shorting and cutting techniques, while the identity $\beta_{*}=\beta^{*}$ 
follows from the sub-Gaussian heat kernel estimates by means of subordination technique. 
Here we prove the estimate (\ref{SC_con_eqn_bound_beta}) of $\beta_*$ directly, without using heat kernel or subordination technique.

We give a direct proof of the following result.

\begin{mythm}\label{SC_con_thm_walk}
$$\beta_*=\beta^*:=\frac{\log(8\rho)}{\log3},$$
where $\rho$ is some parameter in resistance estimates.
\end{mythm}

Hino and Kumagai \cite{HK06} established other equivalent semi-norms as follows.

For all $n\ge1,u\in L^2(K;\nu)$, let
$$P_nu(w)=\frac{1}{\nu(K_w)}\int_{K_w}u(x)\nu(\md x),w\in W_n.$$
For all $w^{(1)},w^{(2)}\in W_n$, denote $w^{(1)}\sim_nw^{(2)}$ if $\mathrm{dim}_{\mathcal{H}}(K_{w^{(1)}}\cap K_{w^{(2)}})=1$. Let
$$\frakE_\beta(u,u):=\sum_{n=1}^\infty3^{(\beta-\alpha)n}
{\sum_{\mbox{\tiny
$
\begin{subarray}{c}
w^{(1)}\sim_nw^{(2)}
\end{subarray}
$
}}}
\left(P_nu(w^{(1)})-P_nu(w^{(2)})\right)^2.$$

\begin{mylem}\label{SC_con_lem_equivHK}(\cite[Lemma 3.1]{HK06})
For all $\beta\in(0,+\infty),u\in L^2(K;\nu)$, we have
$$\frakE_\beta(u,u)\asymp\calE_\beta(u,u).$$
\end{mylem}

We combine $E_\beta$ and $\frakE_\beta$ to construct a local regular Dirichlet form on $K$ using $\Gamma$-convergence technique as follows.

\begin{mythm}\label{SC_con_thm_BM}
There exists a self-similar strongly local regular Dirichlet form $(\calE_{\loc},\calF_{\loc})$ on $L^2(K;\nu)$ satisfying
\begin{align}
&\calE_{\loc}(u,u)\asymp\sup_{n\ge1}3^{(\beta^*-\alpha)n}\sum_{w\in W_n}
{\sum_{\mbox{\tiny
$
\begin{subarray}{c}
p,q\in V_w\\
|p-q|=2^{-1}\cdot3^{-n}
\end{subarray}
$
}}}
(u(p)-u(q))^2,\label{SC_con_eqn_Kigami}\\
&\calF_{\loc}=\myset{u\in C(K):\calE_\loc(u,u)<+\infty}.\nonumber
\end{align}
\end{mythm}
By the uniqueness result in \cite{BBKT10}, we have the above local regular Dirichlet form coincides with that given by \cite{BB89} and \cite{KZ92}.

We have a direct corollary that non-local Dirichlet forms can approximate local Dirichlet form as follows.

\begin{mycor}\label{SC_con_cor_approx}
There exists some positive constant $C$ such that for all $u\in\calF_\loc$, we have
$$\frac{1}{C}\calE_\loc(u,u)\le\varliminf_{\beta\uparrow\beta^*}(\beta^*-\beta)\calE_\beta(u,u)\le\varlimsup_{\beta\uparrow\beta^*}(\beta^*-\beta)\calE_{\beta}(u,u)\le C\calE_\loc(u,u).$$
\end{mycor}

We characterize $(\calE_\loc,\calF_\loc)$ on $L^2(K;\nu)$ as follows.

\begin{mythm}\label{SC_con_thm_Besov}
$\calF_\loc=B_{\alpha,\beta^*}^{2,\infty}(K)$ and $\calE_\loc(u,u)\asymp[u]_{B^{2,\infty}_{\alpha,\beta^*}(K)}$ for all $u\in\calF_\loc$.
\end{mythm}

We give a direct proof of this theorem using (\ref{SC_con_eqn_Kigami}) and thus avoiding heat kernel estimates, while using some geometric properties of the SC.

Finally, using (\ref{SC_con_eqn_Kigami}) of Theorem \ref{SC_con_thm_BM}, we give an alternative proof of sub-Gaussian heat kernel estimates as follows.

\begin{mythm}\label{SC_con_thm_hk}
$(\calE_\loc,\calF_\loc)$ on $L^2(K;\nu)$ has a heat kernel $p_t(x,y)$ satisfying
$$p_t(x,y)\asymp\frac{C}{t^{\alpha/\beta^*}}\exp\left(-c\left(\frac{|x-y|}{t^{1/\beta^*}}\right)^{\frac{\beta^*}{\beta^*-1}}\right),$$
for all $x,y\in K,t\in(0,1)$.
\end{mythm}

This chapter is organized as follows. In Section \ref{SC_con_sec_equiv}, we prove Lemma \ref{SC_con_lem_equiv}. In Section \ref{SC_con_sec_bound}, we prove Theorem \ref{SC_con_thm_bound}. In Section \ref{SC_con_sec_resistance}, we give resistance estimates. In Section \ref{SC_con_sec_harnack}, we give uniform Harnack inequality. In Section \ref{SC_con_sec_monotone}, we give two weak monotonicity results. In Section \ref{SC_con_sec_good}, we construct one good function. In Section \ref{SC_con_sec_walk}, we prove Theorem \ref{SC_con_thm_walk}. In Section \ref{SC_con_sec_BM}, we prove Theorem \ref{SC_con_thm_BM}. In Section \ref{SC_con_sec_Besov}, we prove Theorem \ref{SC_con_thm_Besov}. In Section \ref{SC_con_sec_hk}, we prove Theorem \ref{SC_con_thm_hk}.

\section{Proof of Lemma \ref{SC_con_lem_equiv}}\label{SC_con_sec_equiv}

We need some preparation as follows.

\begin{mylem}\label{SC_con_lem_equiv1}
For all $u\in L^2(K;\nu)$, we have
$$\int_K\int_K\frac{(u(x)-u(y))^2}{|x-y|^{\alpha+\beta}}\nu(\md x)\nu(\md y)\asymp\sum_{n=0}^\infty3^{(\alpha+\beta)n}\int_K\int_{B(x,3^{-n})}(u(x)-u(y))^2\nu(\md y)\nu(\md x).$$
\end{mylem}

\begin{mycor}\label{SC_con_cor_arbi}
Fix arbitrary integer $N\ge0$ and real number $c>0$. For all $u\in L^2(K;\nu)$, we have
$$\int_K\int_K\frac{(u(x)-u(y))^2}{|x-y|^{\alpha+\beta}}\nu(\md x)\nu(\md y)\asymp\sum_{n=N}^\infty3^{(\alpha+\beta)n}\int_K\int_{B(x,c3^{-n})}(u(x)-u(y))^2\nu(\md y)\nu(\md x).$$
\end{mycor}

The proofs of the above results are essentially the same as those of Lemma \ref{SG_app_lem_equiv1} and Corollary \ref{SG_app_cor_arbi} except that contraction ratio $1/2$ is replaced by $1/3$. We also need the fact that the SC satisfies the chain condition, see \cite[Definition 3.4]{GHL03}.

We divide Lemma \ref{SC_con_lem_equiv} into the following Theorem \ref{SC_con_thm_equiv1} and Theorem \ref{SC_con_thm_equiv2}. The idea of the proofs of these theorems comes form \cite{Jon96}. We do need to pay special attention to the difficulty brought by non-p.c.f. property.

\begin{mythm}\label{SC_con_thm_equiv1}
For all $u\in C(K)$, we have
$$\sum_{n=1}^\infty3^{(\beta-\alpha)n}\sum_{w\in W_n}
{\sum_{\mbox{\tiny
$
\begin{subarray}{c}
p,q\in V_w\\
|p-q|=2^{-1}\cdot3^{-n}
\end{subarray}
$
}}}
(u(p)-u(q))^2\lesssim\int_K\int_K\frac{(u(x)-u(y))^2}{|x-y|^{\alpha+\beta}}\nu(\md x)\nu(\md y).$$
\end{mythm}

\begin{proof}
First, fix $n\ge1,w=w_1\ldots w_n\in W_n$, consider
$$
{\sum_{\mbox{\tiny
$
\begin{subarray}{c}
p,q\in V_w\\
|p-q|=2^{-1}\cdot3^{-n}
\end{subarray}
$
}}}
(u(p)-u(q))^2.$$
For all $x\in K_w$, we have
$$(u(p)-u(q))^2\le2(u(p)-u(x))^2+2(u(x)-u(q))^2.$$
Integrating with respect to $x\in K_w$ and dividing by $\nu(K_w)$, we have
$$(u(p)-u(q))^2\le\frac{2}{\nu(K_w)}\int_{K_w}(u(p)-u(x))^2\nu(\md x)+\frac{2}{\nu(K_w)}\int_{K_w}(u(x)-u(q))^2\nu(\md x),$$
hence
$$
{\sum_{\mbox{\tiny
$
\begin{subarray}{c}
p,q\in V_w\\
|p-q|=2^{-1}\cdot3^{-n}
\end{subarray}
$
}}}
(u(p)-u(q))^2\le2\cdot2\cdot2\sum_{p\in V_w}\frac{1}{\nu(K_w)}\int_{K_w}(u(p)-u(x))^2\nu(\md x).$$
Consider $(u(p)-u(x))^2$, $p\in V_w$, $x\in K_w$. There exists $w_{n+1}\in\myset{0,\ldots,7}$ such that $p=f_{w_1}\circ\ldots\circ f_{w_n}(p_{w_{n+1}})$. Let $k,l\ge1$ be integers to be determined, let
$$w^{(i)}=w_1\ldots w_nw_{n+1}\ldots w_{n+1}$$
with $ki$ terms of $w_{n+1}$, $i=0,\ldots,l$. For all $x^{(i)}\in K_{w^{(i)}}$, $i=0,\ldots,l$, we have
$$
\begin{aligned}
(u(p)-u(x^{(0)}))^2&\le2(u(p)-u(x^{(l)}))^2+2(u(x^{(0)})-u(x^{(l)}))^2\\
&\le2(u(p)-u(x^{(l)}))^2+2\left[2(u(x^{(0)})-u(x^{(1)}))^2+2(u(x^{(1)})-u(x^{(l)}))^2\right]\\
&=2(u(p)-u(x^{(l)}))^2+2^2(u(x^{(0)})-u(x^{(1)}))^2+2^2(u(x^{(1)})-u(x^{(l)}))^2\\
&\le\ldots\le2(u(p)-u(x^{(l)}))^2+2^2\sum_{i=0}^{l-1}2^i(u(x^{(i)})-u(x^{(i+1)}))^2.
\end{aligned}
$$
Integrating with respect to $x^{(0)}\in K_{w^{(0)}}$, \ldots, $x^{(l)}\in K_{w^{(l)}}$ and dividing by $\nu(K_{w^{(0)}})$, \ldots, $\nu(K_{w^{(l)}})$, we have
$$
\begin{aligned}
&\frac{1}{\nu(K_{w^{(0)}})}\int_{K_{w^{(0)}}}(u(p)-u(x^{(0)}))^2\nu(\md x^{(0)})\\
\le&\frac{2}{\nu(K_{w^{(l)}})}\int_{K_{w^{(l)}}}(u(p)-u(x^{(l)}))^2\nu(\md x^{(l)})\\
&+2^2\sum_{i=0}^{l-1}\frac{2^i}{\nu(K_{w^{(i)}})\nu(K_{w^{(i+1)}})}\int_{K_{w^{(i)}}}\int_{K_{w^{(i+1)}}}(u(x^{(i)})-u(x^{(i+1)}))^2\nu(\md x^{(i)})\nu(\md x^{(i+1)}).
\end{aligned}
$$
Now let us use $\nu(K_{w^{(i)}})=(1/8)^{n+ki}=3^{-\alpha(n+ki)}$. For the first term, by Lemma \ref{lem_SC_holder}, we have
$$
\begin{aligned}
\frac{1}{\nu(K_{w^{(l)}})}\int_{K_{w^{(l)}}}(u(p)-u(x^{(l)}))^2\nu(\md x^{(l)})&\le \frac{cE(u)}{\nu(K_{w^{(l)}})}\int_{K_{w^{(l)}}}|p-x^{(l)}|^{\beta-\alpha}\nu(\md x^{(l)})\\
&\le{2}^{(\beta-\alpha)/2}cE(u){3}^{-(\beta-\alpha)(n+kl)}.
\end{aligned}
$$
For the second term, for all $x^{(i)}\in K_{w^{(i)}},x^{(i+1)}\in K_{w^{(i+1)}}$, we have
$$|x^{(i)}-x^{(i+1)}|\le\sqrt{2}\cdot3^{-(n+ki)},$$
hence
$$
\begin{aligned}
&\sum_{i=0}^{l-1}\frac{2^i}{\nu(K_{w^{(i)}})\nu(K_{w^{(i+1)}})}\int_{K_{w^{(i)}}}\int_{K_{w^{(i+1)}}}(u(x^{(i)})-u(x^{(i+1)}))^2\nu(\md x^{(i)})\nu(\md x^{(i+1)})\\
\le&\sum_{i=0}^{l-1}{2^{i}\cdot3^{\alpha k+2\alpha(n+ki)}}\int\limits_{K_{w^{(i)}}}\int\limits_{|x^{(i+1)}-x^{(i)}|\le\sqrt{2}\cdot3^{-(n+ki)}}(u(x^{(i)})-u(x^{(i+1)}))^2\nu(\md x^{(i)})\nu(\md x^{(i+1)}),
\end{aligned}
$$
and
$$
\begin{aligned}
&\frac{1}{\nu(K_w)}\int_{K_w}(u(p)-u(x))^2\nu(\md x)=\frac{1}{\nu(K_{w^{(0)}})}\int_{K_{w^{(0)}}}(u(p)-u(x^{(0)}))^2\nu(\md x^{(0)})\\
\le& 2\cdot{2}^{(\beta-\alpha)/2}cE(u)3^{-(\beta-\alpha)(n+kl)}\\
&+4\sum_{i=0}^{l-1}{2^{i}\cdot3^{\alpha k+2\alpha(n+ki)}}\int\limits_{K_{w^{(i)}}}\int\limits_{|x-y|\le\sqrt{2}\cdot3^{-(n+ki)}}(u(x)-u(y))^2\nu(\md x)\nu(\md y).
\end{aligned}
$$
Hence
$$
\begin{aligned}
&\sum_{w\in W_n}
{\sum_{\mbox{\tiny
$
\begin{subarray}{c}
p,q\in V_w\\
|p-q|=2^{-1}\cdot3^{-n}
\end{subarray}
$
}}}
(u(p)-u(q))^2\\
\le&8\sum_{w\in W_n}\sum_{p\in V_w}\frac{1}{\nu(K_w)}\int_{K_w}(u(p)-u(x))^2\nu(\md x)\\
\le&8\sum_{w\in W_n}\sum_{p\in V_w}\left(2\cdot{2}^{(\beta-\alpha)/2}cE(u)3^{-(\beta-\alpha)(n+kl)}\right.\\
&\left.+4\sum_{i=0}^{l-1}{2^{i}\cdot3^{\alpha k+2\alpha(n+ki)}}\int\limits_{K_{w^{(i)}}}\int\limits_{|x-y|\le\sqrt{2}\cdot3^{-(n+ki)}}(u(x)-u(y))^2\nu(\md x)\nu(\md y)\right).
\end{aligned}
$$
For the first term, we have
$$\sum_{w\in W_n}\sum_{p\in V_w}3^{-(\beta-\alpha)(n+kl)}=8\cdot8^n\cdot3^{-(\beta-\alpha)(n+kl)}=8\cdot3^{\alpha n-(\beta-\alpha)(n+kl)}.$$
For the second term, fix $i=0,\ldots,l-1$, different $p\in V_w$, $w\in W_n$ correspond to different $K_{w^{(i)}}$, hence
$$
\begin{aligned}
&\sum_{i=0}^{l-1}\sum_{w\in W_n}\sum_{p\in V_w}2^{i}\cdot3^{\alpha k+2\alpha(n+ki)}\int\limits_{K_{w^{(i)}}}\int\limits_{|x-y|\le\sqrt{2}\cdot3^{-(n+ki)}}(u(x)-u(y))^2\nu(\md x)\nu(\md y)\\
\le&\sum_{i=0}^{l-1}2^{i}\cdot3^{\alpha k+2\alpha(n+ki)}\int\limits_{K}\int\limits_{|x-y|\le\sqrt{2}\cdot3^{-(n+ki)}}(u(x)-u(y))^2\nu(\md x)\nu(\md y)\\
=&3^{\alpha k}\sum_{i=0}^{l-1}2^{i}\cdot3^{-(\beta-\alpha)(n+ki)}\left(3^{(\alpha+\beta)(n+ki)}\int\limits_{K}\int\limits_{|x-y|\le\sqrt{2}\cdot3^{-(n+ki)}}(u(x)-u(y))^2\nu(\md x)\nu(\md y)\right).
\end{aligned}
$$
For simplicity, denote
$$E_{n}(u)=3^{(\alpha+\beta)n}\int_{K}\int_{|x-y|\le\sqrt{2}\cdot3^{-n}}(u(x)-u(y))^2\nu(\md x)\nu(\md y).$$
We have
\begin{equation}\label{SC_con_eqn_equiv1_1}
\begin{aligned}
&\sum_{w\in W_n}
{\sum_{\mbox{\tiny
$
\begin{subarray}{c}
p,q\in V_w\\
|p-q|=2^{-1}\cdot3^{-n}
\end{subarray}
$
}}}
(u(p)-u(q))^2\\
\le&128\cdot{2}^{(\beta-\alpha)/2}cE(u)3^{\alpha n-(\beta-\alpha)(n+kl)}+32\cdot3^{\alpha k}\sum_{i=0}^{l-1}2^{i}\cdot3^{-(\beta-\alpha)(n+ki)}E_{n+ki}(u).
\end{aligned}
\end{equation}
Hence
$$
\begin{aligned}
&\sum_{n=1}^\infty3^{(\beta-\alpha)n}\sum_{w\in W_n}
{\sum_{\mbox{\tiny
$
\begin{subarray}{c}
p,q\in V_w\\
|p-q|=2^{-1}\cdot3^{-n}
\end{subarray}
$
}}}
(u(p)-u(q))^2\\
\le&128\cdot{2}^{(\beta-\alpha)/2}cE(u)\sum_{n=1}^\infty3^{\beta n-(\beta-\alpha)(n+kl)}+32\cdot3^{\alpha k}\sum_{n=1}^\infty\sum_{i=0}^{l-1}2^{i}\cdot3^{-(\beta-\alpha)ki}E_{n+ki}(u).
\end{aligned}
$$
Take $l=n$, then
$$
\begin{aligned}
&\sum_{n=1}^\infty3^{(\beta-\alpha)n}\sum_{w\in W_n}
{\sum_{\mbox{\tiny
$
\begin{subarray}{c}
p,q\in V_w\\
|p-q|=2^{-1}\cdot3^{-n}
\end{subarray}
$
}}}
(u(p)-u(q))^2\\
\le&128\cdot{2}^{(\beta-\alpha)/2}cE(u)\sum_{n=1}^\infty3^{\left[\beta-(\beta-\alpha)(k+1)\right]n}+32\cdot3^{\alpha k}\sum_{n=1}^\infty\sum_{i=0}^{n-1}2^{i}\cdot3^{-(\beta-\alpha)ki}E_{n+ki}(u)\\
=&128\cdot{2}^{(\beta-\alpha)/2}cE(u)\sum_{n=1}^\infty3^{\left[\beta-(\beta-\alpha)(k+1)\right]n}+32\cdot3^{\alpha k}\sum_{i=0}^\infty2^{i}\cdot3^{-(\beta-\alpha)ki}\sum_{n=i+1}^{\infty}E_{n+ki}(u)\\
\le&128\cdot{2}^{(\beta-\alpha)/2}cE(u)\sum_{n=1}^\infty3^{\left[\beta-(\beta-\alpha)(k+1)\right]n}+32\cdot3^{\alpha k}\sum_{i=0}^\infty3^{\left[1-(\beta-\alpha)k\right]i}C_1E(u),
\end{aligned}
$$
where $C_1$ is some positive constant from Corollary \ref{SC_con_cor_arbi}. Take $k\ge1$ sufficiently large such that $\beta-(\beta-\alpha)(k+1)<0$ and $1-(\beta-\alpha)k<0$, then the above two series converge, hence
$$\sum_{n=1}^\infty3^{(\beta-\alpha)n}\sum_{w\in W_n}
{\sum_{\mbox{\tiny
$
\begin{subarray}{c}
p,q\in V_w\\
|p-q|=2^{-1}\cdot3^{-n}
\end{subarray}
$
}}}
(u(p)-u(q))^2\lesssim\int_K\int_K\frac{(u(x)-u(y))^2}{|x-y|^{\alpha+\beta}}\nu(\md x)\nu(\md y).$$
\end{proof}

\begin{mythm}\label{SC_con_thm_equiv2}
For all $u\in C(K)$, we have
\begin{equation}\label{SC_con_eqn_equiv2_1}
\int_K\int_K\frac{(u(x)-u(y))^2}{|x-y|^{\alpha+\beta}}\nu(\md x)\nu(\md y)\lesssim\sum_{n=1}^\infty3^{(\beta-\alpha)n}\sum_{w\in W_n}
{\sum_{\mbox{\tiny
$
\begin{subarray}{c}
p,q\in V_w\\
|p-q|=2^{-1}\cdot3^{-n}
\end{subarray}
$
}}}
(u(p)-u(q))^2,
\end{equation}
or equivalently, for all $c\in(0,1)$, we have
\begin{equation}\label{SC_con_eqn_equiv2_2}
\begin{aligned}
&\sum_{n=2}^\infty3^{(\alpha+\beta)n}\int\limits_K\int\limits_{B(x,c3^{-n})}(u(x)-u(y))^2\nu(\md y)\nu(\md x)\\
\lesssim&\sum_{n=1}^\infty3^{(\beta-\alpha)n}\sum_{w\in W_n}
{\sum_{\mbox{\tiny
$
\begin{subarray}{c}
p,q\in V_w\\
|p-q|=2^{-1}\cdot3^{-n}
\end{subarray}
$
}}}
(u(p)-u(q))^2.
\end{aligned}
\end{equation}
\end{mythm}

\begin{proof}
Note $V_n=\cup_{w\in W_n}V_w$, it is obvious that its cardinal $\#V_n\asymp8^n=3^{\alpha n}$. Let $\nu_n$ be the measure on $V_n$ which assigns $1/\#V_n$ on each point of $V_n$, then $\nu_n$ converges weakly to $\nu$.

First, for $n\ge2,m>n$, we estimate
$$3^{(\alpha+\beta)n}\int_K\int_{B(x,c3^{-n})}(u(x)-u(y))^2\nu_m(\md y)\nu_m(\md x).$$
Note that
$$
\begin{aligned}
\int\limits_K\int\limits_{B(x,c3^{-n})}(u(x)-u(y))^2\nu_m(\md y)\nu_m(\md x)=\sum\limits_{w\in W_n}\int\limits_{K_w}\int\limits_{B(x,c3^{-n})}(u(x)-u(y))^2\nu_m(\md y)\nu_m(\md x).
\end{aligned}
$$
Fix $w\in W_n$, there exist at most nine $\tilde{w}\in W_n$ such that $K_{\tilde{w}}\cap K_w\ne\emptyset$, see Figure \ref{SC_con_fig_Kw}.

\begin{figure}[ht]
\centering
\begin{tikzpicture}[scale=0.5]

\draw (0,0)--(6,0)--(6,6)--(0,6)--cycle;
\draw (2,0)--(2,6);
\draw (4,0)--(4,6);
\draw (0,2)--(6,2);
\draw (0,4)--(6,4);
\draw[very thick] (2,2)--(4,2)--(4,4)--(2,4)--cycle;

\draw (3,3) node {$K_w$};

\end{tikzpicture}
\caption{A Neighborhood of $K_w$}\label{SC_con_fig_Kw}
\end{figure}
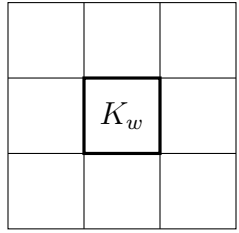

Let
$$K_w^*=
{\bigcup_{\mbox{\tiny
$
\begin{subarray}{c}
\tilde{w}\in W_n\\
K_{\tilde{w}}\cap K_w\ne\emptyset
\end{subarray}
$
}}}
K_{\tilde{w}}.$$
For all $x\in K_w$, $y\in B(x,c3^{-n})$, we have $y\in K_w^*$, hence
$$
\begin{aligned}
&\int_{K_w}\int_{B(x,c3^{-n})}(u(x)-u(y))^2\nu_m(\md y)\nu_m(\md x)\le\int_{K_w}\int_{K_w^*}(u(x)-u(y))^2\nu_m(\md y)\nu_m(\md x)\\
&=
{\sum_{\mbox{\tiny
$
\begin{subarray}{c}
\tilde{w}\in W_n\\
K_{\tilde{w}}\cap K_w\ne\emptyset
\end{subarray}
$
}}}
\int_{K_w}\int_{K_{\tilde{w}}}(u(x)-u(y))^2\nu_m(\md y)\nu_m(\md x).
\end{aligned}
$$
Note $\myset{P_w}=K_w\cap V_{n-1}$ for all $w\in W_n$. Fix $\tilde{w},w\in W_n$ with $K_{\tilde{w}}\cap K_w\ne\emptyset$. If $P_{\tilde{w}}\ne P_w$, then $|P_{\tilde{w}}-P_w|=2^{-1}\cdot3^{-(n-1)}$ or there exists a unique $z\in V_{n-1}$ such that

\begin{equation}\label{SC_con_eqn_med}
\lvert P_{\tilde{w}}-z\rvert=\lvert P_w-z\rvert=2^{-1}\cdot3^{-(n-1)}.
\end{equation}
Let $z_1=P_{\tilde{w}}$, $z_3=P_w$ and
$$z_2=
\begin{cases}
P_{\tilde{w}}=P_w,&\text{if }P_{\tilde{w}}=P_w,\\
P_{\tilde{w}},&\text{if }|P_{\tilde{w}}-P_w|=2^{-1}\cdot3^{-(n-1)},\\
z,&\text{if }P_{\tilde{w}}\ne P_w\text{ and }z \text{ is given by Equation (\ref{SC_con_eqn_med})}.
\end{cases}
$$
Then for all $x\in K_w$, $y\in K_{\tilde{w}}$, we have
$$
\begin{aligned}
&(u(x)-u(y))^2\\
\le&4\left[(u(y)-u(z_1))^2+(u(z_1)-u(z_2))^2+(u(z_2)-u(z_3))^2+(u(z_3)-u(x))^2\right].
\end{aligned}
$$
For $i=1,2$, we have
$$
\begin{aligned}
&\int_{K_w}\int_{K_{\tilde{w}}}(u(z_i)-u(z_{i+1}))^2\nu_m(\md y)\nu_m(\md x)=(u(z_i)-u(z_{i+1}))^2\left(\frac{\#(K_w\cap V_m)}{\#V_m}\right)^2\\
\asymp&(u(z_i)-u(z_{i+1}))^2\left(\frac{8^{m-n}}{8^m}\right)^2=3^{-2\alpha n}(u(z_i)-u(z_{i+1}))^2.
\end{aligned}
$$
Hence
$$
\begin{aligned}
&\sum_{w\in W_n}
{\sum_{\mbox{\tiny
$
\begin{subarray}{c}
\tilde{w}\in W_n\\
K_{\tilde{w}}\cap K_w\ne\emptyset
\end{subarray}
$
}}}
\int_{K_w}\int_{K_{\tilde{w}}}(u(x)-u(y))^2\nu_m(\md y)\nu_m(\md x)\\
\lesssim&3^{-\alpha n}\sum\limits_{w\in W_n}\int\limits_{K_w}(u(x)-u(P_w))^2\nu_m(\md x)+3^{-2\alpha n}\sum_{w\in W_{n-1}}
{\sum_{\mbox{\tiny
$
\begin{subarray}{c}
p,q\in V_w\\
|p-q|=2^{-1}\cdot3^{-(n-1)}
\end{subarray}
$
}}}
(u(p)-u(q))^2\\
\asymp&3^{-\alpha(m+n)}\sum\limits_{w\in W_n}\sum\limits_{x\in K_w\cap V_m}(u(x)-u(P_w))^2\\
&+3^{-2\alpha n}\sum_{w\in W_{n-1}}
{\sum_{\mbox{\tiny
$
\begin{subarray}{c}
p,q\in V_w\\
|p-q|=2^{-1}\cdot3^{-(n-1)}
\end{subarray}
$
}}}
(u(p)-u(q))^2.
\end{aligned}
$$
Let us estimate $(u(x)-u(P_w))^2$ for $x\in K_w\cap V_m$. We construct a finite sequence
$$p_1,\ldots,p_{4(m-n+1)},p_{4(m-n+1)+1}$$
such that
\begin{itemize}
\item $p_1=P_w$ and $p_{4(m-n+1)+1}=x$.
\item For all $k=0,\ldots,m-n$, we have
$$p_{4k+1},p_{4k+2},p_{4k+3},p_{4k+4},p_{4(k+1)+1}\in V_{n+k}.$$
\item For all $i=1,2,3,4$, we have
$$\lvert p_{4k+i}-p_{4k+i+1}\rvert=0\text{ or }2^{-1}\cdot3^{-(n+k)}.$$
\end{itemize}
Then
$$
\begin{aligned}
\left(u(x)-u(P_w)\right)^2\lesssim\sum_{k=0}^{m-n}4^{k}&\left[(u(p_{4k+1})-u(p_{4k+2}))^2+(u(p_{4k+2})-u(p_{4k+3}))^2\right.\\
&\left.+(u(p_{4k+3})-u(p_{4k+4}))^2+(u(p_{4k+4})-u(p_{4(k+1)+1}))^2\right].
\end{aligned}
$$
For all $k=n,\ldots,m$, for all $p,q\in V_k\cap K_w$ with $|p-q|=2^{-1}\cdot 3^{-k}$, the term $(u(p)-u(q))^2$ occurs in the sum with times of the order $8^{m-k}=3^{\alpha(m-k)}$, hence
$$
\begin{aligned}
&3^{-\alpha(m+n)}\sum\limits_{w\in W_n}\sum\limits_{x\in K_w\cap V_m}(u(x)-u(P_w))^2\\
\lesssim&3^{-\alpha(m+n)}\sum_{k=n}^{m}4^{k-n}\cdot3^{\alpha(m-k)}\sum_{w\in W_k}
{\sum_{\mbox{\tiny
$
\begin{subarray}{c}
p,q\in V_w\\
|p-q|=2^{-1}\cdot3^{-k}
\end{subarray}
$
}}}
(u(p)-u(q))^2\\
=&\sum_{k=n}^{m}4^{k-n}\cdot3^{-\alpha(n+k)}\sum_{w\in W_k}
{\sum_{\mbox{\tiny
$
\begin{subarray}{c}
p,q\in V_w\\
|p-q|=2^{-1}\cdot3^{-k}
\end{subarray}
$
}}}
(u(p)-u(q))^2.
\end{aligned}
$$
Hence
$$
\begin{aligned}
&\int_K\int_{B(x,c3^{-n})}(u(x)-u(y))^2\nu_m(\md y)\nu_m(\md x)\\
\lesssim&\sum_{k=n}^{m}4^{k-n}\cdot3^{-\alpha(n+k)}\sum_{w\in W_k}
{\sum_{\mbox{\tiny
$
\begin{subarray}{c}
p,q\in V_w\\
|p-q|=2^{-1}\cdot3^{-k}
\end{subarray}
$
}}}
(u(p)-u(q))^2\\
&+3^{-2\alpha n}\sum_{w\in W_{n-1}}
{\sum_{\mbox{\tiny
$
\begin{subarray}{c}
p,q\in V_w\\
|p-q|=2^{-1}\cdot3^{-(n-1)}
\end{subarray}
$
}}}
(u(p)-u(q))^2.
\end{aligned}
$$
Letting $m\to+\infty$, we have
\begin{equation}\label{SC_con_eqn_equiv2_3}
\begin{aligned}
&\int_K\int_{B(x,c3^{-n})}(u(x)-u(y))^2\nu(\md y)\nu(\md x)\\
\lesssim&\sum_{k=n}^\infty4^{k-n}\cdot3^{-\alpha(n+k)}\sum_{w\in W_k}
{\sum_{\mbox{\tiny
$
\begin{subarray}{c}
p,q\in V_w\\
|p-q|=2^{-1}\cdot3^{-k}
\end{subarray}
$
}}}
(u(p)-u(q))^2\\
&+3^{-2\alpha n}\sum_{w\in W_{n-1}}
{\sum_{\mbox{\tiny
$
\begin{subarray}{c}
p,q\in V_w\\
|p-q|=2^{-1}\cdot3^{-(n-1)}
\end{subarray}
$
}}}
(u(p)-u(q))^2.
\end{aligned}
\end{equation}
Hence
$$
\begin{aligned}
&\sum_{n=2}^\infty3^{(\alpha+\beta)n}\int_K\int_{B(x,c3^{-n})}(u(x)-u(y))^2\nu(\md y)\nu(\md x)\\
\lesssim&\sum_{n=2}^\infty\sum_{k=n}^\infty4^{k-n}\cdot3^{\beta n-\alpha k}\sum_{w\in W_k}
{\sum_{\mbox{\tiny
$
\begin{subarray}{c}
p,q\in V_w\\
|p-q|=2^{-1}\cdot3^{-k}
\end{subarray}
$
}}}
(u(p)-u(q))^2\\
&+\sum_{n=2}^\infty3^{(\beta-\alpha)n}\sum_{w\in W_{n-1}}
{\sum_{\mbox{\tiny
$
\begin{subarray}{c}
p,q\in V_w\\
|p-q|=2^{-1}\cdot3^{-(n-1)}
\end{subarray}
$
}}}
(u(p)-u(q))^2\\
\lesssim&\sum_{k=2}^\infty\sum_{n=2}^k4^{k-n}\cdot3^{\beta n-\alpha k}\sum_{w\in W_k}
{\sum_{\mbox{\tiny
$
\begin{subarray}{c}
p,q\in V_w\\
|p-q|=2^{-1}\cdot3^{-k}
\end{subarray}
$
}}}
(u(p)-u(q))^2\\
&+\sum_{n=1}^\infty3^{(\beta-\alpha)n}\sum_{w\in W_{n}}
{\sum_{\mbox{\tiny
$
\begin{subarray}{c}
p,q\in V_w\\
|p-q|=2^{-1}\cdot3^{-n}
\end{subarray}
$
}}}
(u(p)-u(q))^2\\
\lesssim&\sum_{n=1}^\infty3^{(\beta-\alpha)n}\sum_{w\in W_{n}}
{\sum_{\mbox{\tiny
$
\begin{subarray}{c}
p,q\in V_w\\
|p-q|=2^{-1}\cdot3^{-n}
\end{subarray}
$
}}}
(u(p)-u(q))^2.
\end{aligned}
$$
\end{proof}

\section{Proof of Theorem \ref{SC_con_thm_bound}}\label{SC_con_sec_bound}

Firstly, we consider lower bound. We need some preparation.

\begin{myprop}\label{SC_con_prop_lower}
Assume that $\beta\in(\alpha,+\infty)$. Let $f:[0,1]\to\R$ be a strictly increasing continuous function. Assume that the function $U(x,y)=f(x)$, $(x,y)\in K$ satisfies $\calE_{\beta}(U,U)<+\infty$. Then $(\calE_\beta,\calF_\beta)$ is a regular Dirichlet form on $L^2(K;\nu)$.
\end{myprop}

\begin{myrmk}
The above proposition means that only \emph{one} good enough function contained in the domain can ensure that the domain is large enough.
\end{myrmk}

\begin{proof}
We only need to show that $\calF_\beta$ is uniformly dense in $C(K)$. Then $\calF_\beta$ is dense in $L^2(K;\nu)$. Using Fatou's lemma, we have $\calF_\beta$ is complete under $(\calE_\beta)_1$-norm. It is obvious that $\calE_\beta$ has Markovian property. Hence $(\calE_\beta,\calF_\beta)$ is a Dirichlet form on $L^2(K;\nu)$. Moreover, $\calF_\beta\cap C(K)=\calF_\beta$ is trivially $(\calE_\beta)_1$-dense in $\calF_\beta$ and uniformly dense in $C(K)$. Hence $(\calE_\beta,\calF_\beta)$ on $L^2(K;\nu)$ is regular.

Indeed, by assumption, $U\in\calF_\beta$, $\calF_\beta\ne\emptyset$. It is obvious that $\calF_\beta$ is a sub-algebra of $C(K)$, that is, for all $u,v\in\calF_\beta$, $c\in\R$, we have $u+v,cu,uv\in\calF_\beta$. We show that $\calF_\beta$ separates points. For all distinct $(x^{(1)},y^{(1)}),(x^{(2)},y^{(2)})\in K$, we have $x^{(1)}\ne x^{(2)}$ or $y^{(1)}\ne y^{(2)}$.

If $x^{(1)}\ne x^{(2)}$, then since $f$ is strictly increasing, we have
$$U(x^{(1)},y^{(1)})=f(x^{(1)})\ne f(x^{(2)})=U(x^{(2)},y^{(2)}).$$
If $y^{(1)}\ne y^{(2)}$, then let $V(x,y)=f(y)$, $(x,y)\in K$, we have $V\in\calF_\beta$ and
$$V(x^{(1)},y^{(1)})=f(y^{(1)})\ne f(y^{(2)})=V(x^{(2)},y^{(2)}).$$
By Stone-Weierstrass theorem, we have $\calF_\beta$ is uniformly dense in $C(K)$.
\end{proof}

Now, we give lower bound.

\begin{proof}[Proof of Lower Bound]
The point is to construct an explicit function. We define $f:[0,1]\to\R$ as follows. Let $f(0)=0$ and $f(1)=1$. First, we determine the values of $f$ at $1/3$ and $2/3$. We consider the minimum of the following function
$$\vphi(x,y)=3x^2+2(x-y)^2+3(1-y)^2,x,y\in\R.$$
By elementary calculation, $\vphi$ attains minimum $6/7$ at $(x,y)=(2/7,5/7)$. Assume that we have defined $f$ on $i/3^n$, $i=0,1,\ldots,3^n$. Then, for $n+1$, for all $i=0,1,\ldots,3^{n}-1$, we define
$$
f(\frac{3i+1}{3^{n+1}})=\frac{5}{7}f(\frac{i}{3^n})+\frac{2}{7}f(\frac{i+1}{3^n}),f(\frac{3i+2}{3^{n+1}})=\frac{2}{7}f(\frac{i}{3^n})+\frac{5}{7}f(\frac{i+1}{3^n}).
$$
By induction principle, we have the definition of $f$ on all triadic points. It is obvious that $f$ is uniformly continuous on the set of all triadic points. We extend $f$ to be continuous on $[0,1]$. It is obvious that $f$ is increasing. For all $x,y\in[0,1]$ with $x<y$, there exist triadic points $i/3^n,(i+1)/3^n\in(x,y)$, then $f(x)\le f(i/3^n)<f((i+1)/3^n)\le f(y)$, hence $f$ is strictly increasing.

Let $U(x,y)=f(x)$, $(x,y)\in K$. By induction, we have
$$\sum_{w\in W_{n+1}}
{\sum_{\mbox{\tiny
$
\begin{subarray}{c}
p,q\in V_w\\
|p-q|=2^{-1}\cdot3^{-(n+1)}
\end{subarray}
$
}}}
(U(p)-U(q))^2=\frac{6}{7}\sum_{w\in W_{n}}
{\sum_{\mbox{\tiny
$
\begin{subarray}{c}
p,q\in V_w\\
|p-q|=2^{-1}\cdot3^{-n}
\end{subarray}
$
}}}
(U(p)-U(q))^2\text{ for all }n\ge1.$$
Hence
\begin{equation}\label{SC_con_eqn_lower}
\sum_{w\in W_{n}}
{\sum_{\mbox{\tiny
$
\begin{subarray}{c}
p,q\in V_w\\
|p-q|=2^{-1}\cdot3^{-n}
\end{subarray}
$
}}}
(U(p)-U(q))^2=\left(\frac{6}{7}\right)^n\text{ for all }n\ge1.
\end{equation}
For all $\beta\in(\log8/\log3,\log(8\cdot7/6)/\log3)$, we have $3^{\beta-\alpha}<7/6$. By Equation (\ref{SC_con_eqn_lower}), we have
$$\sum_{n=1}^\infty3^{(\beta-\alpha)n}\sum_{w\in W_n}
{\sum_{\mbox{\tiny
$
\begin{subarray}{c}
p,q\in V_w\\
|p-q|=2^{-1}\cdot3^{-n}
\end{subarray}
$
}}}
(U(p)-U(q))^2<+\infty.$$
By Lemma \ref{SC_con_lem_equiv}, $\calE_\beta(U,U)<+\infty$. By Proposition \ref{SC_con_prop_lower}, $(\calE_\beta,\calF_\beta)$ is a regular Dirichlet form on $L^2(K;\nu)$ for all $\beta\in(\log8/\log3,\log(8\cdot7/6)/\log3)$. Hence
$$\beta_*\ge\frac{\log(8\cdot\frac{7}{6})}{\log3}.$$
\end{proof}

\begin{myrmk}
\normalfont
The construction of the above function is similar to that given in the proof of \cite[Theorem 2.6]{Bar13}. Indeed, the above function is constructed in a self-similar way. Let $f_n:[0,1]\to\R$ be given by $f_0(x)=x$, $x\in[0,1]$ and for all $n\ge0$
$$
f_{n+1}(x)=
\begin{cases}
\frac{2}{7}f_n(3x),&\text{if }0\le x\le\frac{1}{3},\\
\frac{3}{7}f_n(3x-1)+\frac{2}{7},&\text{if }\frac{1}{3}<x\le\frac{2}{3},\\
\frac{2}{7}f_n(3x-2)+\frac{5}{7},&\text{if }\frac{2}{3}<x\le1.
\end{cases}
$$
See Figure \ref{SC_con_fig_f012} for the figures of $f_0,f_1,f_2$.

\begin{figure}[ht]
\centering
\subfigure[$f_0$]{
\begin{tikzpicture}[scale=2.5]
\draw[->] (0,0)--(1.5,0);
\draw[->] (0,0)--(0,1.5);

\draw[dashed] (1,0)--(1,1)--(0,1);

\draw[thick] (0,0)--(1,1);

\draw (-0.2,-0.2) node {\footnotesize{$O$}};
\draw (1,-0.2) node {\footnotesize{$1$}};
\draw (-0.2,1) node {\footnotesize{$1$}};
\end{tikzpicture}
}
\hspace{0.00in}
\subfigure[$f_1$]{
\begin{tikzpicture}[scale=2.5]
\draw[->] (0,0)--(1.5,0);
\draw[->] (0,0)--(0,1.5);

\draw[dashed] (1,0)--(1,1)--(0,1);

\draw[dashed] (1/3,0)--(1/3,2/7)--(0,2/7);
\draw[dashed] (2/3,0)--(2/3,5/7)--(0,5/7);

\draw[thick] (0,0)--(1/3,2/7)--(2/3,5/7)--(1,1);

\draw (-0.2,-0.2) node {\footnotesize{$O$}};
\draw (1,-0.2) node {\footnotesize{$1$}};
\draw (-0.2,1) node {\footnotesize{$1$}};

\draw (-0.2,2/7) node {\scriptsize{${2}/{7}$}};
\draw (-0.2,5/7) node {\scriptsize{${5}/{7}$}};
\draw (1/3,-0.2) node {\scriptsize{$\frac{1}{3}$}};
\draw (2/3,-0.2) node {\scriptsize{$\frac{2}{3}$}};
\end{tikzpicture}
}
\hspace{0.0in}
\subfigure[$f_2$]{
\begin{tikzpicture}[scale=2.5]
\draw[->] (0,0)--(1.5,0);
\draw[->] (0,0)--(0,1.5);

\draw[dashed] (1,0)--(1,1)--(0,1);

\draw[dashed] (1/3,0)--(1/3,2/7)--(0,2/7);
\draw[dashed] (2/3,0)--(2/3,5/7)--(0,5/7);

\draw[dashed] (1/9,0)--(1/9,4/49)--(0,4/49);
\draw[dashed] (2/9,0)--(2/9,10/49)--(0,10/49);

\draw[dashed] (4/9,0)--(4/9,20/49)--(0,20/49);
\draw[dashed] (5/9,0)--(5/9,29/49)--(0,29/49);

\draw[dashed] (7/9,0)--(7/9,39/49)--(0,39/49);
\draw[dashed] (8/9,0)--(8/9,45/49)--(0,45/49);

\draw[thick] (0,0)--(1/9,4/49)--(2/9,10/49)--(1/3,2/7)--(4/9,2/7+3/7*2/7)--(5/9,2/7+3/7*5/7)--(2/3,5/7)--(7/9,39/49)--(8/9,45/49)--(1,1);

\draw (-0.2,-0.2) node {\footnotesize{$O$}};
\draw (1,-0.2) node {\tiny{$1$}};
\draw (-0.2,1) node {\tiny{$1$}};

\draw (-0.2,2/7) node {\tiny{${2}/{7}$}};
\draw (-0.2,5/7) node {\tiny{${5}/{7}$}};
\draw (1/3,-0.2) node {\tiny{$\frac{1}{3}$}};
\draw (2/3,-0.2) node {\tiny{$\frac{2}{3}$}};

\draw (1/9,-0.2) node {\tiny{$\frac{1}{9}$}};
\draw (2/9,-0.2) node {\tiny{$\frac{2}{9}$}};

\draw (4/9,-0.2) node {\tiny{$\frac{4}{9}$}};
\draw (5/9,-0.2) node {\tiny{$\frac{5}{9}$}};

\draw (7/9,-0.2) node {\tiny{$\frac{7}{9}$}};
\draw (8/9,-0.2) node {\tiny{$\frac{8}{9}$}};

\draw (-0.2,4/49) node {\tiny{${4}/{49}$}};
\draw (-0.2,10/49) node {\tiny{${10}/{49}$}};

\draw (-0.2,20/49) node {\tiny{${20}/{49}$}};
\draw (-0.2,29/49) node {\tiny{${29}/{49}$}};

\draw (-0.2,39/49) node {\tiny{${39}/{49}$}};
\draw (-0.2,45/49) node {\tiny{${45}/{49}$}};
\end{tikzpicture}
}
\caption{The Figures of $f_0,f_1,f_2$}\label{SC_con_fig_f012}
\end{figure}
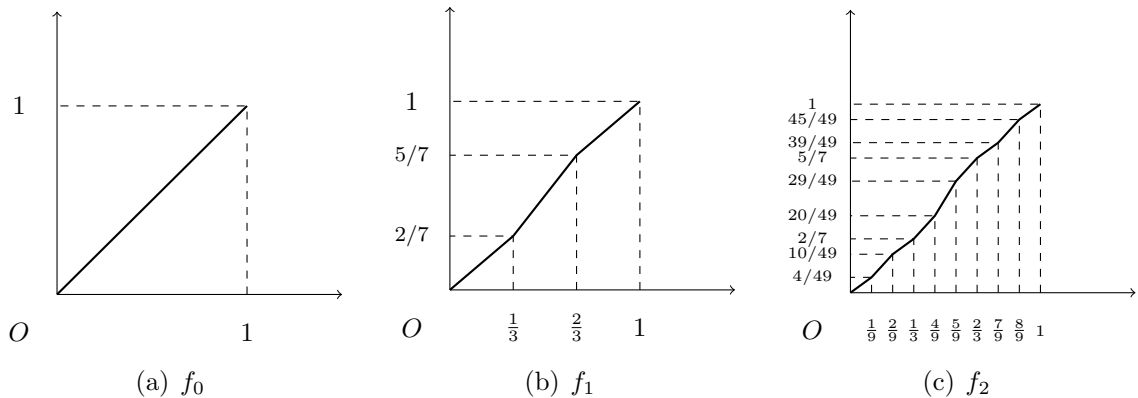

It is obvious that
$$f_n(\frac{i}{3^n})=f(\frac{i}{3^n})\text{ for all }i=0,\ldots,3^n,n\ge0,$$
and
$$\max_{x\in[0,1]}\lvert f_{n+1}(x)-f_n(x)\rvert\le\frac{3}{7}\max_{x\in[0,1]}\lvert f_{n}(x)-f_{n-1}(x)\rvert\text{ for all }n\ge1,$$
hence $f_n$ converges uniformly to $f$ on $[0,1]$. Let $g_1,g_2,g_3:\R^2\to\R^2$ be given by
$$g_1(x,y)=\left(\frac{1}{3}x,\frac{2}{7}y\right),g_2(x,y)=\left(\frac{1}{3}x+\frac{1}{3},\frac{3}{7}y+\frac{2}{7}\right),g_3(x,y)=\left(\frac{1}{3}x+\frac{2}{3},\frac{2}{7}y+\frac{5}{7}\right).$$
Then $\myset{(x,f(x)):x\in[0,1]}$ is the unique non-empty compact set $G$ in $\R^2$ satisfying 
$$G=g_1(G)\cup g_2(G)\cup g_3(G).$$
\end{myrmk}

Secondly, we consider upper bound. We shrink the SC to another fractal. Denote $\mathcal{C}$ as the Cantor ternary set in $[0,1]$. Then $[0,1]\times\mathcal{C}$ is the unique non-empty compact set $\tilde{K}$ in $\R^2$ satisfying
$$\tilde{K}=\cup_{i=0,1,2,4,5,6}f_i(\tilde{K}).$$
Let
$$\tilde{V}_0=\myset{p_0,p_1,p_2,p_4,p_5,p_6},\tilde{V}_{n+1}=\cup_{i=0,1,2,4,5,6}f_i(\tilde{V}_n)\text{ for all }n\ge0.$$
Then $\myset{\tilde{V}_n}$ is an increasing sequence of finite sets and $[0,1]\times\mathcal{C}$ is the closure of $\cup_{n=0}^\infty\tilde{V}_n$. Let $\tilde{W}_0=\myset{\emptyset}$ and
$$\tilde{W}_n=\myset{w=w_1\ldots w_n:w_i=0,1,2,4,5,6,i=1,\ldots,n}\text{ for all }n\ge1.$$
For all $w=w_1\ldots w_n\in\tilde{W}_n$, let
$$\tilde{V}_w=f_{w_1}\circ\ldots\circ f_{w_n}(\tilde{V}_0).$$

\begin{proof}[Proof of Upper Bound]
Assume that $(\calE_\beta,\calF_\beta)$ is a regular Dirichlet form on $L^2(K;\nu)$, then there exists $u\in\calF_\beta$ such that $u|_{\myset{0}\times[0,1]}=0$ and $u|_{\myset{1}\times[0,1]}=1$. By Lemma \ref{SC_con_lem_equiv}, we have
\begin{equation}\label{SC_con_eqn_upper}
\begin{aligned}
+\infty&>\sum_{n=1}^\infty3^{(\beta-\alpha)n}\sum_{w\in W_n}
{\sum_{\mbox{\tiny
$
\begin{subarray}{c}
p,q\in V_w\\
|p-q|=2^{-1}\cdot3^{-n}
\end{subarray}
$
}}}
(u(p)-u(q))^2\\
&\ge\sum_{n=1}^\infty3^{(\beta-\alpha)n}\sum_{w\in\tilde{W}_n}
{\sum_{\mbox{\tiny
$
\begin{subarray}{c}
p,q\in\tilde{V}_w\\
|p-q|=2^{-1}\cdot3^{-n}
\end{subarray}
$
}}}
(u(p)-u(q))^2\\
&=\sum_{n=1}^\infty3^{(\beta-\alpha)n}\sum_{w\in\tilde{W}_n}
{\sum_{\mbox{\tiny
$
\begin{subarray}{c}
p,q\in\tilde{V}_w\\
|p-q|=2^{-1}\cdot3^{-n}
\end{subarray}
$
}}}
((u|_{[0,1]\times\mathcal{C}})(p)-(u|_{[0,1]\times\mathcal{C}})(q))^2\\
&\ge\sum_{n=1}^\infty3^{(\beta-\alpha)n}\sum_{w\in\tilde{W}_n}
{\sum_{\mbox{\tiny
$
\begin{subarray}{c}
p,q\in\tilde{V}_w\\
|p-q|=2^{-1}\cdot3^{-n}
\end{subarray}
$
}}}
(\tilde{u}(p)-\tilde{u}(q))^2,
\end{aligned}
\end{equation}
where $\tilde{u}$ is the function on $[0,1]\times\mathcal{C}$ that is the minimizer of 
$$\sum_{n=1}^\infty3^{(\beta-\alpha)n}\sum_{w\in\tilde{W}_n}
{\sum_{\mbox{\tiny
$
\begin{subarray}{c}
p,q\in\tilde{V}_w\\
|p-q|=2^{-1}\cdot3^{-n}
\end{subarray}
$
}}}
(\tilde{u}(p)-\tilde{u}(q))^2:\tilde{u}|_{\myset{0}\times\mathcal{C}}=0,\tilde{u}|_{\myset{1}\times\mathcal{C}}=1,\tilde{u}\in C([0,1]\times\mathcal{C}).$$
By symmetry of $[0,1]\times\mathcal{C}$, $\tilde{u}(x,y)=x,(x,y)\in [0,1]\times\mathcal{C}$. By induction, we have
$$\sum_{w\in\tilde{W}_{n+1}}
{\sum_{\mbox{\tiny
$
\begin{subarray}{c}
p,q\in\tilde{V}_w\\
|p-q|=2^{-1}\cdot3^{-(n+1)}
\end{subarray}
$
}}}
(\tilde{u}(p)-\tilde{u}(q))^2=\frac{2}{3}\sum_{w\in\tilde{W}_n}
{\sum_{\mbox{\tiny
$
\begin{subarray}{c}
p,q\in\tilde{V}_w\\
|p-q|=2^{-1}\cdot3^{-n}
\end{subarray}
$
}}}
(\tilde{u}(p)-\tilde{u}(q))^2\text{ for all }n\ge1,$$
hence
$$\sum_{w\in\tilde{W}_{n}}
{\sum_{\mbox{\tiny
$
\begin{subarray}{c}
p,q\in\tilde{V}_w\\
|p-q|=2^{-1}\cdot3^{-n}
\end{subarray}
$
}}}
(\tilde{u}(p)-\tilde{u}(q))^2=\left(\frac{2}{3}\right)^n\text{ for all }n\ge1.$$
By Equation (\ref{SC_con_eqn_upper}), we have
$$\sum_{n=1}^\infty3^{(\beta-\alpha)n}\left(\frac{2}{3}\right)^n<+\infty,$$
hence, $\beta<\log(8\cdot3/2)/\log3$. Hence
$$\beta_*\le\frac{\log(8\cdot\frac{3}{2})}{\log3}.$$
\end{proof}

\section{Resistance Estimates}\label{SC_con_sec_resistance}

In this section, we give resistance estimates using electrical network techniques.

We consider two sequences of finite graphs related to $V_n$ and $W_n$, respectively.

For all $n\ge1$. Let $\calV_n$ be the graph with vertex set $V_n$ and edge set given by
$$\myset{(p,q):p,q\in V_n,|p-q|=2^{-1}\cdot3^{-n}}.$$
For example, we have the figure of $\calV_2$ in Figure \ref{SC_con_fig_V2}.

\begin{figure}[ht]
\centering
\begin{tikzpicture}

\foreach \x in {0,1,...,9}
\draw (\x,0)--(\x,9);

\foreach \y in {0,1,...,9}
\draw (0,\y)--(9,\y);

\draw[fill=white] (3,3)--(6,3)--(6,6)--(3,6)--cycle;

\foreach \x in {0,1,...,9}
\foreach \y in {0,0.5,1,...,9}
\draw[fill=black] (\x,\y) circle (0.06);

\foreach \y in {0,1,...,9}
\foreach \x in {0,0.5,1,...,9}
\draw[fill=black] (\x,\y) circle (0.06);

\draw[fill=white,draw=white] (3.25,3.25)--(5.75,3.25)--(5.75,5.75)--(3.25,5.75)--cycle;

\end{tikzpicture}
\caption{$\mathcal{V}_2$}\label{SC_con_fig_V2}
\end{figure}
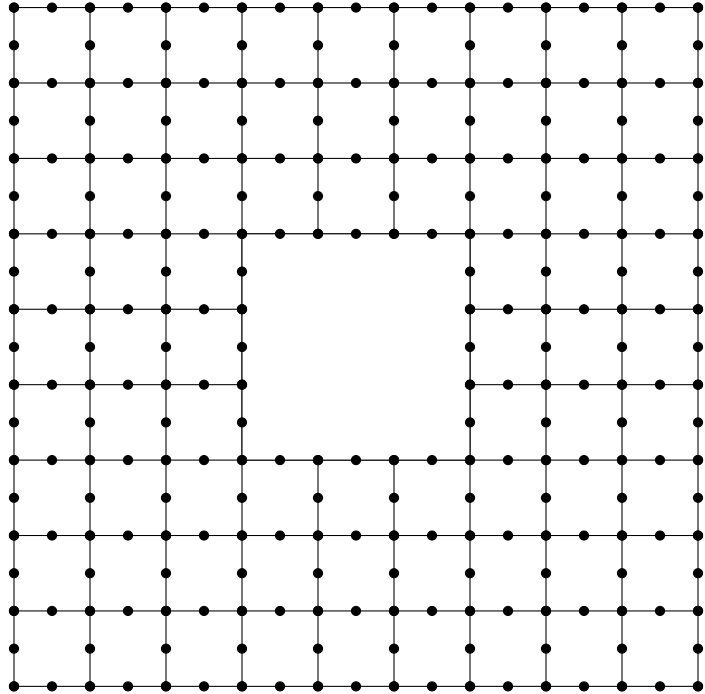

Let $\calW_n$ be the graph with vertex set $W_n$ and edge set given by
$$\myset{(w^{(1)},w^{(2)}):w^{(1)},w^{(2)}\in W_n,\mathrm{dim}_{\mathcal{H}}\left(K_{w^{(1)}}\cap K_{w^{(2)}}\right)=1}.$$
For example, we have the figure of $\calW_2$ in Figure \ref{SC_con_fig_W2}.




\begin{figure}[ht]
\centering
\begin{tikzpicture}
\draw (0,0)--(8,0)--(8,8)--(0,8)--cycle;
\draw (2,0)--(2,8);
\draw (6,0)--(6,8);
\draw (0,2)--(8,2);
\draw (0,6)--(8,6);
\draw (3,0)--(3,2);
\draw (5,0)--(5,2);
\draw (0,3)--(2,3);
\draw (0,5)--(2,5);
\draw (6,3)--(8,3);
\draw (6,5)--(8,5);
\draw (3,6)--(3,8);
\draw (5,6)--(5,8);
\draw (2,1)--(3,1);
\draw (5,1)--(6,1);
\draw (1,2)--(1,3);
\draw (7,2)--(7,3);
\draw (1,5)--(1,6);
\draw (7,5)--(7,6);
\draw (2,7)--(3,7);
\draw (5,7)--(6,7);

\draw[fill=black] (0,0) circle (0.06);
\draw[fill=black] (1,0) circle (0.06);
\draw[fill=black] (2,0) circle (0.06);
\draw[fill=black] (3,0) circle (0.06);
\draw[fill=black] (4,0) circle (0.06);
\draw[fill=black] (5,0) circle (0.06);
\draw[fill=black] (6,0) circle (0.06);
\draw[fill=black] (7,0) circle (0.06);
\draw[fill=black] (8,0) circle (0.06);

\draw[fill=black] (0,1) circle (0.06);
\draw[fill=black] (2,1) circle (0.06);
\draw[fill=black] (3,1) circle (0.06);
\draw[fill=black] (5,1) circle (0.06);
\draw[fill=black] (6,1) circle (0.06);
\draw[fill=black] (8,1) circle (0.06);

\draw[fill=black] (0,2) circle (0.06);
\draw[fill=black] (1,2) circle (0.06);
\draw[fill=black] (2,2) circle (0.06);
\draw[fill=black] (3,2) circle (0.06);
\draw[fill=black] (4,2) circle (0.06);
\draw[fill=black] (5,2) circle (0.06);
\draw[fill=black] (6,2) circle (0.06);
\draw[fill=black] (7,2) circle (0.06);
\draw[fill=black] (8,2) circle (0.06);

\draw[fill=black] (0,3) circle (0.06);
\draw[fill=black] (1,3) circle (0.06);
\draw[fill=black] (2,3) circle (0.06);
\draw[fill=black] (6,3) circle (0.06);
\draw[fill=black] (7,3) circle (0.06);
\draw[fill=black] (8,3) circle (0.06);

\draw[fill=black] (0,4) circle (0.06);
\draw[fill=black] (2,4) circle (0.06);
\draw[fill=black] (6,4) circle (0.06);
\draw[fill=black] (8,4) circle (0.06);

\draw[fill=black] (0,5) circle (0.06);
\draw[fill=black] (1,5) circle (0.06);
\draw[fill=black] (2,5) circle (0.06);
\draw[fill=black] (6,5) circle (0.06);
\draw[fill=black] (7,5) circle (0.06);
\draw[fill=black] (8,5) circle (0.06);

\draw[fill=black] (0,6) circle (0.06);
\draw[fill=black] (1,6) circle (0.06);
\draw[fill=black] (2,6) circle (0.06);
\draw[fill=black] (3,6) circle (0.06);
\draw[fill=black] (4,6) circle (0.06);
\draw[fill=black] (5,6) circle (0.06);
\draw[fill=black] (6,6) circle (0.06);
\draw[fill=black] (7,6) circle (0.06);
\draw[fill=black] (8,6) circle (0.06);

\draw[fill=black] (0,7) circle (0.06);
\draw[fill=black] (2,7) circle (0.06);
\draw[fill=black] (3,7) circle (0.06);
\draw[fill=black] (5,7) circle (0.06);
\draw[fill=black] (6,7) circle (0.06);
\draw[fill=black] (8,7) circle (0.06);

\draw[fill=black] (0,8) circle (0.06);
\draw[fill=black] (1,8) circle (0.06);
\draw[fill=black] (2,8) circle (0.06);
\draw[fill=black] (3,8) circle (0.06);
\draw[fill=black] (4,8) circle (0.06);
\draw[fill=black] (5,8) circle (0.06);
\draw[fill=black] (6,8) circle (0.06);
\draw[fill=black] (7,8) circle (0.06);
\draw[fill=black] (8,8) circle (0.06);
\end{tikzpicture}
\caption{$\mathcal{W}_2$}\label{SC_con_fig_W2}
\end{figure}
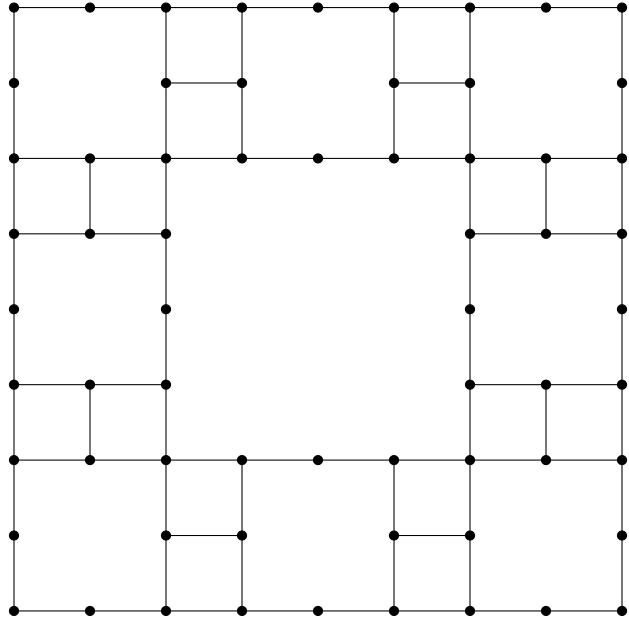


On $\calV_n$, the energy
$$
{\sum_{\mbox{\tiny
$
\begin{subarray}{c}
p,q\in V_n\\
|p-q|=2^{-1}\cdot3^{-n}
\end{subarray}
$
}}}
(u(p)-u(q))^2,u\in l(V_n),$$
is related to a weighted graph with the conductances of all edges equal to $1$. While the energy
$$\sum_{w\in W_n}
{\sum_{\mbox{\tiny
$
\begin{subarray}{c}
p,q\in V_w\\
|p-q|=2^{-1}\cdot3^{-n}
\end{subarray}
$
}}}
(u(p)-u(q))^2,u\in l(V_n),$$
is related to a weighted graph with the conductances of some edges equal to $1$ and the conductances of other edges equal to $2$, since the term $(u(p)-u(q))^2$ is added either once or twice.

Since
$$
\begin{aligned}
{\sum_{\mbox{\tiny
$
\begin{subarray}{c}
p,q\in V_n\\
|p-q|=2^{-1}\cdot3^{-n}
\end{subarray}
$
}}}
(u(p)-u(q))^2&\le\sum_{w\in W_n}
{\sum_{\mbox{\tiny
$
\begin{subarray}{c}
p,q\in V_w\\
|p-q|=2^{-1}\cdot3^{-n}
\end{subarray}
$
}}}
(u(p)-u(q))^2\\
&\le 2
{\sum_{\mbox{\tiny
$
\begin{subarray}{c}
p,q\in V_n\\
|p-q|=2^{-1}\cdot3^{-n}
\end{subarray}
$
}}}
(u(p)-u(q))^2,
\end{aligned}
$$
we use
$$D_n(u,u):=\sum_{w\in W_n}
{\sum_{\mbox{\tiny
$
\begin{subarray}{c}
p,q\in V_w\\
|p-q|=2^{-1}\cdot3^{-n}
\end{subarray}
$
}}}
(u(p)-u(q))^2,u\in l(V_n),$$
as the energy on $\calV_n$. Assume that $A,B$ are two disjoint subsets of $V_n$. Let
$$R_n(A,B)=\inf\myset{D_n(u,u):u|_A=0,u|_B=1,u\in l(V_n)}^{-1}.$$
Denote
$$R_n^V=R_n(V_n\cap\myset{0}\times[0,1],V_n\cap\myset{1}\times[0,1]),$$
$$R_n(x,y)=R_n(\myset{x},\myset{y}),x,y\in V_n.$$
It is obvious that $R_n$ is a metric on $V_n$, hence
$$R_n(x,y)\le R_n(x,z)+R_n(z,y)\text{ for all }x,y,z\in V_n.$$

On $\calW_n$, the energy
$$\frakD_n(u,u):=\sum_{w^{(1)}\sim_nw^{(2)}}(u(w^{(1)})-u(w^{(2)}))^2,u\in l(W_n),$$
is related to a weighted graph with the conductances of all edges equal to $1$. Assume that $A,B$ are two disjoint subsets of $W_n$. Let
$$\frakR_n(A,B)=\inf\myset{\frakD_n(u,u):u|_A=0,u|_B=1,u\in l(W_n)}^{-1}.$$
Denote
$$\frakR_n(w^{(1)},w^{(2)})=\frakR_n(\myset{w^{(1)}},\myset{w^{(2)}}),w^{(1)},w^{(2)}\in W_n.$$
It is obvious that $\frakR_n$ is a metric on $W_n$, hence
$$\frakR_n(w^{(1)},w^{(2)})\le\frakR_n(w^{(1)},w^{(3)})+\frakR_n(w^{(3)},w^{(2)})\text{ for all }w^{(1)},w^{(2)},w^{(3)}\in W_n.$$

The main result of this section is as follows.

\begin{mythm}\label{SC_con_thm_resist}
There exists some positive constant $\rho\in\left[7/6,3/2\right]$ such that for all $n\ge1$
$$R_n^V\asymp\rho^n,$$
$$R_n(p_0,p_1)=\ldots=R_n(p_6,p_7)=R_n(p_7,p_0)\asymp\rho^n,$$
$$\frakR_n(0^n,1^n)=\ldots=\frakR_n(6^n,7^n)=\frakR_n(7^n,0^n)\asymp\rho^n.$$
\end{mythm}

\begin{myrmk}
By triangle inequality, for all $i,j=0,\ldots,7,n\ge1$
$$R_{n}(p_i,p_j)\lesssim\rho^n,$$
$$\frakR_{n}(i^n,j^n)\lesssim\rho^n.$$
\end{myrmk}

We have a direct corollary as follows.

\begin{mycor}\label{SC_con_cor_resist_upper}
For all $n\ge1,p,q\in V_n,w^{(1)},w^{(2)}\in W_n$
$$R_n(p,q)\lesssim\rho^n,$$
$$\frakR_n(w^{(1)},w^{(2)})\lesssim\rho^n.$$
\end{mycor}

\begin{proof}
We only need to show that $\frakR_n(w,0^n)\lesssim\rho^n$ for all $w\in W_n,n\ge1$. Then for all $w^{(1)},w^{(2)}\in W_n$
$$\frakR_n(w^{(1)},w^{(2)})\le\frakR_n(w^{(1)},0^n)+\frakR_n(w^{(2)},0^n)\lesssim\rho^n.$$
Similarly, we have the proof of $R_n(p,q)\lesssim\rho^n$ for all $p,q\in V_n,n\ge1$.

Indeed, for all $n\ge1,w=w_1\ldots w_n\in W_n$, we construct a finite sequence as follows.
$$
\begin{aligned}
w^{(1)}&=w_1\ldots w_{n-2}w_{n-1}w_n=w,\\
w^{(2)}&=w_1\ldots w_{n-2}w_{n-1}w_{n-1},\\
w^{(3)}&=w_1\ldots w_{n-2}w_{n-2}w_{n-2},\\
&\ldots\\
w^{(n)}&=w_1\ldots w_1w_1w_1,\\
w^{(n+1)}&=0\ldots 000=0^n.
\end{aligned}
$$
For all $i=1,\ldots,n-1$, by cutting technique
$$
\begin{aligned}
&\frakR_n(w^{(i)},w^{(i+1)})=\frakR_n(w_1\ldots w_{n-i}w_{n-i+1}\ldots w_{n-i+1},w_1\ldots w_{n-i}w_{n-i}\ldots w_{n-i})\\
\le&\frakR_i(w_{n-i+1}\ldots w_{n-i+1},w_{n-i}\ldots w_{n-i})=\frakR_i(w_{n-i+1}^i,w_{n-i}^i)\lesssim\rho^i.
\end{aligned}
$$
Since $\frakR_n(w^{(n)},w^{(n+1)})=\frakR_n(w_1^n,0^n)\lesssim\rho^n$, we have
$$\frakR_n(w,0^n)=\frakR_n(w^{(1)},w^{(n+1)})\le\sum_{i=1}^n\frakR_n(w^{(i)},w^{(i+1)})\lesssim\sum_{i=1}^n\rho^i\lesssim\rho^n.$$
\end{proof}

We need the following results for preparation.

Firstly, we have resistance estimates for some symmetric cases.

\begin{mythm}\label{SC_con_thm_resist1}
There exists some positive constant $\rho\in[7/6,3/2]$ such that for all $n\ge1$
$$R_n^V\asymp\rho^n,$$
$$R_n(p_1,p_5)=R_n(p_3,p_7)\asymp\rho^n,$$
$$R_n(p_0,p_4)=R_n(p_2,p_6)\asymp\rho^n.$$
\end{mythm}

\begin{proof}
The proof is similar to \cite[Theorem 5.1]{BB90} and \cite[Theorem 6.1]{McG02} where flow technique and potential technique are used. We need discrete version instead of continuous version.

Hence there exists some positive constant $C$ such that
$$\frac{1}{C}x_nx_m\le x_{n+m}\le Cx_nx_m\text{ for all }n,m\ge1,$$
where $x$ is any of the above resistances. Since the above resistances share the same complexity, there exists \emph{one} positive constant $\rho$ such that they are equivalent to $\rho^n$ for all $n\ge1$.

By shorting and cutting technique, we have $\rho\in[7/6,3/2]$, see \cite[Equation (2.6)]{Bar13} or \cite[Remarks 5.4]{BB99a}.
\end{proof}

Secondly, by symmetry and shorting technique, we have the following relations.

\begin{myprop}\label{SC_con_prop_resist2}
For all $n\ge1$
$$R_n(p_0,p_1)\le\frakR_n(0^n,1^n),$$
$$R_n^V\le R_n(p_1,p_5)=R_n(p_3,p_7)\le\frakR_n(1^n,5^n)=\frakR_n(3^n,7^n),$$
$$R_n^V\le R_n(p_0,p_4)=R_n(p_2,p_6)\le\frakR_n(0^n,4^n)=\frakR_n(2^n,6^n).$$
\end{myprop}

Thirdly, we have the following relations.

\begin{myprop}\label{SC_con_prop_resist3}
For all $n\ge1$
$$\frakR_n(0^n,1^n)\lesssim R_n(p_0,p_1),$$
$$\frakR_n(1^n,5^n)=\frakR_n(3^n,7^n)\lesssim R_n(p_1,p_5)=R_n(p_3,p_7),$$
$$\frakR_n(0^n,4^n)=\frakR_n(2^n,6^n)\lesssim R_n(p_0,p_4)=R_n(p_2,p_6).$$
\end{myprop}

\begin{proof}
The idea is to use electrical network transformations to \emph{increase} resistances to transform weighted graph $\calW_n$ to weighted graph $\calV_{n-1}$.

Firstly, we do the transformation in Figure \ref{SC_con_fig_trans1} where the resistances of the resistors in the new network only depend on the shape of the networks in Figure \ref{SC_con_fig_trans1} such that we obtain the weighted graph in Figure \ref{SC_con_fig_trans2} where the resistances between any two points are larger than those in the weighted graph $\calW_n$. For $\frakR_n(i^n,j^n)$, we have the equivalent weighted graph in Figure \ref{SC_con_fig_trans3}.




\begin{figure}[ht]
\centering
\begin{tikzpicture}
\draw (0,0)--(2,0)--(2,1)--(0,1)--cycle;
\draw (1,0)--(1,1);

\draw[fill=black] (0,0) circle (0.06);
\draw[fill=black] (1,0) circle (0.06);
\draw[fill=black] (2,0) circle (0.06);
\draw[fill=black] (0,1) circle (0.06);
\draw[fill=black] (1,1) circle (0.06);
\draw[fill=black] (2,1) circle (0.06);

\draw (3,0)--(3,1);
\draw (4,0)--(4,1);
\draw (5,0)--(5,1);
\draw (3,0.5)--(5,0.5);

\draw[fill=black] (3,0) circle (0.06);
\draw[fill=black] (4,0) circle (0.06);
\draw[fill=black] (5,0) circle (0.06);
\draw[fill=black] (3,1) circle (0.06);
\draw[fill=black] (4,1) circle (0.06);
\draw[fill=black] (5,1) circle (0.06);

\draw (2.5,0.5) node {$\Rightarrow$};

\end{tikzpicture}
\caption{First Transformation}\label{SC_con_fig_trans1}
\end{figure}
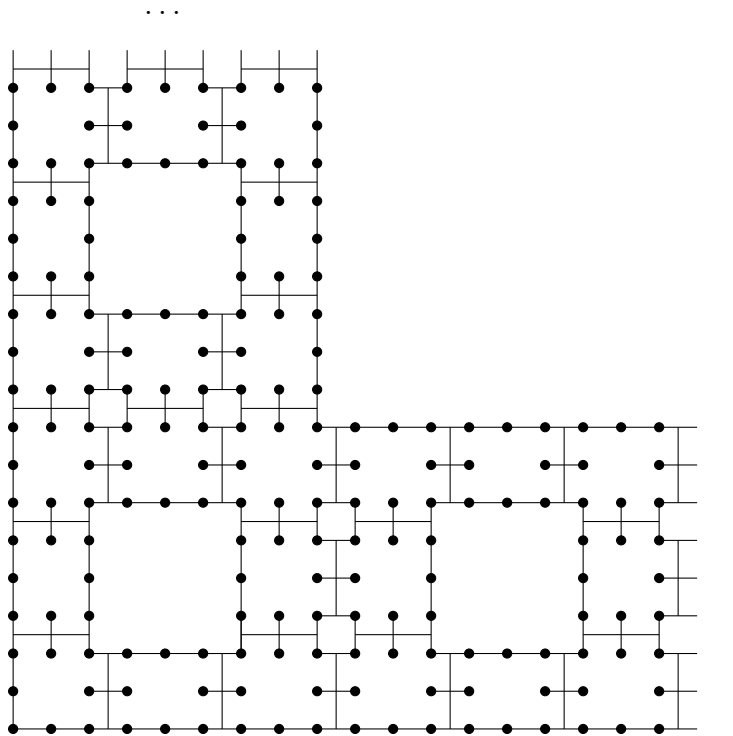




\begin{figure}[ht]
\centering
\begin{tikzpicture}

\draw[fill=black] (0,0) circle (0.06);
\draw[fill=black] (0.5,0) circle (0.06);
\draw[fill=black] (1,0) circle (0.06);
\draw[fill=black] (1.5,0) circle (0.06);
\draw[fill=black] (2,0) circle (0.06);
\draw[fill=black] (2.5,0) circle (0.06);
\draw[fill=black] (3,0) circle (0.06);
\draw[fill=black] (3.5,0) circle (0.06);
\draw[fill=black] (4,0) circle (0.06);
\draw[fill=black] (4.5,0) circle (0.06);
\draw[fill=black] (5,0) circle (0.06);
\draw[fill=black] (5.5,0) circle (0.06);
\draw[fill=black] (6,0) circle (0.06);
\draw[fill=black] (6.5,0) circle (0.06);
\draw[fill=black] (7,0) circle (0.06);
\draw[fill=black] (7.5,0) circle (0.06);
\draw[fill=black] (8,0) circle (0.06);
\draw[fill=black] (8.5,0) circle (0.06);

\draw[fill=black] (0,0.5) circle (0.06);
\draw[fill=black] (1,0.5) circle (0.06);
\draw[fill=black] (1.5,0.5) circle (0.06);
\draw[fill=black] (2.5,0.5) circle (0.06);
\draw[fill=black] (3,0.5) circle (0.06);
\draw[fill=black] (4,0.5) circle (0.06);
\draw[fill=black] (4.5,0.5) circle (0.06);
\draw[fill=black] (5.5,0.5) circle (0.06);
\draw[fill=black] (6,0.5) circle (0.06);
\draw[fill=black] (7,0.5) circle (0.06);
\draw[fill=black] (7.5,0.5) circle (0.06);
\draw[fill=black] (8.5,0.5) circle (0.06);

\draw[fill=black] (0,1) circle (0.06);
\draw[fill=black] (0.5,1) circle (0.06);
\draw[fill=black] (1,1) circle (0.06);
\draw[fill=black] (1.5,1) circle (0.06);
\draw[fill=black] (2,1) circle (0.06);
\draw[fill=black] (2.5,1) circle (0.06);
\draw[fill=black] (3,1) circle (0.06);
\draw[fill=black] (3.5,1) circle (0.06);
\draw[fill=black] (4,1) circle (0.06);
\draw[fill=black] (4.5,1) circle (0.06);
\draw[fill=black] (5,1) circle (0.06);
\draw[fill=black] (5.5,1) circle (0.06);
\draw[fill=black] (6,1) circle (0.06);
\draw[fill=black] (6.5,1) circle (0.06);
\draw[fill=black] (7,1) circle (0.06);
\draw[fill=black] (7.5,1) circle (0.06);
\draw[fill=black] (8,1) circle (0.06);
\draw[fill=black] (8.5,1) circle (0.06);

\draw[fill=black] (0,1.5) circle (0.06);
\draw[fill=black] (0.5,1.5) circle (0.06);
\draw[fill=black] (1,1.5) circle (0.06);
\draw[fill=black] (3,1.5) circle (0.06);
\draw[fill=black] (3.5,1.5) circle (0.06);
\draw[fill=black] (4,1.5) circle (0.06);
\draw[fill=black] (4.5,1.5) circle (0.06);
\draw[fill=black] (5,1.5) circle (0.06);
\draw[fill=black] (5.5,1.5) circle (0.06);
\draw[fill=black] (7.5,1.5) circle (0.06);
\draw[fill=black] (8,1.5) circle (0.06);
\draw[fill=black] (8.5,1.5) circle (0.06);

\draw[fill=black] (0,2) circle (0.06);
\draw[fill=black] (1,2) circle (0.06);
\draw[fill=black] (3,2) circle (0.06);
\draw[fill=black] (4,2) circle (0.06);
\draw[fill=black] (4.5,2) circle (0.06);
\draw[fill=black] (5.5,2) circle (0.06);
\draw[fill=black] (7.5,2) circle (0.06);
\draw[fill=black] (8.5,2) circle (0.06);

\draw[fill=black] (0,2.5) circle (0.06);
\draw[fill=black] (0.5,2.5) circle (0.06);
\draw[fill=black] (1,2.5) circle (0.06);
\draw[fill=black] (3,2.5) circle (0.06);
\draw[fill=black] (3.5,2.5) circle (0.06);
\draw[fill=black] (4,2.5) circle (0.06);
\draw[fill=black] (4.5,2.5) circle (0.06);
\draw[fill=black] (5,2.5) circle (0.06);
\draw[fill=black] (5.5,2.5) circle (0.06);
\draw[fill=black] (7.5,2.5) circle (0.06);
\draw[fill=black] (8,2.5) circle (0.06);
\draw[fill=black] (8.5,2.5) circle (0.06);

\draw[fill=black] (0,3) circle (0.06);
\draw[fill=black] (0.5,3) circle (0.06);
\draw[fill=black] (1,3) circle (0.06);
\draw[fill=black] (1.5,3) circle (0.06);
\draw[fill=black] (2,3) circle (0.06);
\draw[fill=black] (2.5,3) circle (0.06);
\draw[fill=black] (3,3) circle (0.06);
\draw[fill=black] (3.5,3) circle (0.06);
\draw[fill=black] (4,3) circle (0.06);
\draw[fill=black] (4.5,3) circle (0.06);
\draw[fill=black] (5,3) circle (0.06);
\draw[fill=black] (5.5,3) circle (0.06);
\draw[fill=black] (6,3) circle (0.06);
\draw[fill=black] (6.5,3) circle (0.06);
\draw[fill=black] (7,3) circle (0.06);
\draw[fill=black] (7.5,3) circle (0.06);
\draw[fill=black] (8,3) circle (0.06);
\draw[fill=black] (8.5,3) circle (0.06);

\draw[fill=black] (0,3.5) circle (0.06);
\draw[fill=black] (1,3.5) circle (0.06);
\draw[fill=black] (1.5,3.5) circle (0.06);
\draw[fill=black] (2.5,3.5) circle (0.06);
\draw[fill=black] (3,3.5) circle (0.06);
\draw[fill=black] (4,3.5) circle (0.06);
\draw[fill=black] (4.5,3.5) circle (0.06);
\draw[fill=black] (5.5,3.5) circle (0.06);
\draw[fill=black] (6,3.5) circle (0.06);
\draw[fill=black] (7,3.5) circle (0.06);
\draw[fill=black] (7.5,3.5) circle (0.06);
\draw[fill=black] (8.5,3.5) circle (0.06);

\draw[fill=black] (0,4) circle (0.06);
\draw[fill=black] (0.5,4) circle (0.06);
\draw[fill=black] (1,4) circle (0.06);
\draw[fill=black] (1.5,4) circle (0.06);
\draw[fill=black] (2,4) circle (0.06);
\draw[fill=black] (2.5,4) circle (0.06);
\draw[fill=black] (3,4) circle (0.06);
\draw[fill=black] (3.5,4) circle (0.06);
\draw[fill=black] (4,4) circle (0.06);
\draw[fill=black] (4.5,4) circle (0.06);
\draw[fill=black] (5,4) circle (0.06);
\draw[fill=black] (5.5,4) circle (0.06);
\draw[fill=black] (6,4) circle (0.06);
\draw[fill=black] (6.5,4) circle (0.06);
\draw[fill=black] (7,4) circle (0.06);
\draw[fill=black] (7.5,4) circle (0.06);
\draw[fill=black] (8,4) circle (0.06);
\draw[fill=black] (8.5,4) circle (0.06);

\draw[fill=black] (0,4.5) circle (0.06);
\draw[fill=black] (0.5,4.5) circle (0.06);
\draw[fill=black] (1,4.5) circle (0.06);
\draw[fill=black] (1.5,4.5) circle (0.06);
\draw[fill=black] (2,4.5) circle (0.06);
\draw[fill=black] (2.5,4.5) circle (0.06);
\draw[fill=black] (3,4.5) circle (0.06);
\draw[fill=black] (3.5,4.5) circle (0.06);
\draw[fill=black] (4,4.5) circle (0.06);

\draw[fill=black] (0,5) circle (0.06);
\draw[fill=black] (1,5) circle (0.06);
\draw[fill=black] (1.5,5) circle (0.06);
\draw[fill=black] (2.5,5) circle (0.06);
\draw[fill=black] (3,5) circle (0.06);
\draw[fill=black] (4,5) circle (0.06);

\draw[fill=black] (0,5.5) circle (0.06);
\draw[fill=black] (0.5,5.5) circle (0.06);
\draw[fill=black] (1,5.5) circle (0.06);
\draw[fill=black] (1.5,5.5) circle (0.06);
\draw[fill=black] (2,5.5) circle (0.06);
\draw[fill=black] (2.5,5.5) circle (0.06);
\draw[fill=black] (3,5.5) circle (0.06);
\draw[fill=black] (3.5,5.5) circle (0.06);
\draw[fill=black] (4,5.5) circle (0.06);

\draw[fill=black] (0,6) circle (0.06);
\draw[fill=black] (0.5,6) circle (0.06);
\draw[fill=black] (1,6) circle (0.06);
\draw[fill=black] (3,6) circle (0.06);
\draw[fill=black] (3.5,6) circle (0.06);
\draw[fill=black] (4,6) circle (0.06);

\draw[fill=black] (0,6.5) circle (0.06);
\draw[fill=black] (1,6.5) circle (0.06);
\draw[fill=black] (3,6.5) circle (0.06);
\draw[fill=black] (4,6.5) circle (0.06);

\draw[fill=black] (0,7) circle (0.06);
\draw[fill=black] (0.5,7) circle (0.06);
\draw[fill=black] (1,7) circle (0.06);
\draw[fill=black] (3,7) circle (0.06);
\draw[fill=black] (3.5,7) circle (0.06);
\draw[fill=black] (4,7) circle (0.06);

\draw[fill=black] (0,7.5) circle (0.06);
\draw[fill=black] (0.5,7.5) circle (0.06);
\draw[fill=black] (1,7.5) circle (0.06);
\draw[fill=black] (1.5,7.5) circle (0.06);
\draw[fill=black] (2,7.5) circle (0.06);
\draw[fill=black] (2.5,7.5) circle (0.06);
\draw[fill=black] (3,7.5) circle (0.06);
\draw[fill=black] (3.5,7.5) circle (0.06);
\draw[fill=black] (4,7.5) circle (0.06);

\draw[fill=black] (0,8) circle (0.06);
\draw[fill=black] (1,8) circle (0.06);
\draw[fill=black] (1.5,8) circle (0.06);
\draw[fill=black] (2.5,8) circle (0.06);
\draw[fill=black] (3,8) circle (0.06);
\draw[fill=black] (4,8) circle (0.06);

\draw[fill=black] (0,8.5) circle (0.06);
\draw[fill=black] (0.5,8.5) circle (0.06);
\draw[fill=black] (1,8.5) circle (0.06);
\draw[fill=black] (1.5,8.5) circle (0.06);
\draw[fill=black] (2,8.5) circle (0.06);
\draw[fill=black] (2.5,8.5) circle (0.06);
\draw[fill=black] (3,8.5) circle (0.06);
\draw[fill=black] (3.5,8.5) circle (0.06);
\draw[fill=black] (4,8.5) circle (0.06);

\draw (0,0)--(8.5,0);
\draw (0,0)--(0,8.5);

\draw (1,1)--(1.5,1);
\draw (1,1)--(1,1.5);
\draw (1.25,0)--(1.25,1);
\draw (0,1.25)--(1,1.25);
\draw (1,0.5)--(1.5,0.5);
\draw (0.5,1)--(0.5,1.5);

\draw (1.5,1)--(3,1);
\draw (2.75,1)--(2.75,0);
\draw (2.5,0.5)--(3,0.5);

\draw (3,1)--(3,1.5);
\draw (4,1)--(4,1.5);
\draw (3,1.25)--(4,1.25);
\draw (3.5,1)--(3.5,1.5);
\draw (4,1)--(4.5,1);
\draw (4,0.5)--(4.5,0.5);
\draw (4.25,0)--(4.25,1);

\draw (4.5,1)--(4.5,1.5);
\draw (5,1)--(5,1.5);
\draw (5.5,1)--(5.5,1.5);
\draw (4.5,1.25)--(5.5,1.25);
\draw (5.5,1)--(6,1);
\draw (5.5,0.5)--(6,0.5);
\draw (5.75,0)--(5.75,1);
\draw (6,1)--(7,1);
\draw (7,1)--(7.5,1);
\draw (7,0.5)--(7.5,0.5);
\draw (7.25,0)--(7.25,1);
\draw (7.5,1)--(7.5,1.5);
\draw (8,1)--(8,1.5);
\draw (8.5,1)--(8.5,1.5);
\draw (7.5,1.25)--(8.5,1.25);

\draw (8.5,0)--(9,0);
\draw (8.5,0.5)--(9,0.5);
\draw (8.5,1)--(9,1);
\draw (8.75,0)--(8.75,1);

\draw (1,1.5)--(1,3);
\draw (0.5,2.5)--(0.5,3);
\draw (0,2.75)--(1,2.75);

\draw (1,3)--(3,3);
\draw (3,3)--(3,1);

\draw (3.5,2.5)--(3.5,3);
\draw (4,2.5)--(4,3);
\draw (3,2.75)--(4,2.75);
\draw (4,1.5)--(4.5,1.5);
\draw (4,2)--(4.5,2);
\draw (4,2.5)--(4.5,2.5);
\draw (4.25,1.5)--(4.25,2.5);

\draw (4.5,2.5)--(4.5,3);
\draw (5,2.5)--(5,3);
\draw (5.5,2.5)--(5.5,3);
\draw (4.5,2.75)--(5.5,2.75);
\draw (5.5,1.5)--(5.5,2.5);
\draw (5.5,3)--(7.5,3);
\draw (7.5,3)--(7.5,1.5);

\draw (8,2.5)--(8,3);
\draw (8.5,2.5)--(8.5,3);
\draw (7.5,2.75)--(8.5,2.75);

\draw (8.5,1.5)--(9,1.5);
\draw (8.5,2)--(9,2);
\draw (8.5,2.5)--(9,2.5);
\draw (8.75,1.5)--(8.75,2.5);

\draw (8.5,3)--(9,3);
\draw (8.5,3.5)--(9,3.5);
\draw (8.5,4)--(9,4);
\draw (8.75,3)--(8.75,4);

\draw (7,3.5)--(7.5,3.5);
\draw (7,4)--(7.5,4);
\draw (7.25,3)--(7.25,4);

\draw (5.5,3.5)--(6,3.5);
\draw (5.5,4)--(6,4);
\draw (5.75,3)--(5.75,4);

\draw (6,4)--(7,4);
\draw (4.5,4)--(5.5,4);

\draw (4,3)--(4.5,3);
\draw (4,3.5)--(4.5,3.5);
\draw (4,4)--(4.5,4);
\draw (4.25,3)--(4.25,4);

\draw (4,4)--(4,4.5);
\draw (3.5,4)--(3.5,4.5);
\draw (3,4)--(3,4.5);
\draw (3,4.25)--(4,4.25);

\draw (2.5,4)--(3,4);
\draw (2.5,3.5)--(3,3.5);
\draw (2.75,3)--(2.75,4);

\draw (1,4)--(1.5,4);
\draw (1,3.5)--(1.5,3.5);
\draw (1.25,3)--(1.25,4);

\draw (7.5,4)--(8.5,4);

\draw (0.5,4)--(0.5,4.5);
\draw (1,4)--(1,4.5);
\draw (0,4.25)--(1,4.25);

\draw (1.5,4)--(1.5,4.5);
\draw (2,4)--(2,4.5);
\draw (2.5,4)--(2.5,4.5);
\draw (1.5,4.25)--(2.5,4.25);

\draw (1,4.5)--(1.5,4.5);
\draw (1,5)--(1.5,5);
\draw (1,5.5)--(1.5,5.5);
\draw (1.25,4.5)--(1.25,5.5);

\draw (2.5,4.5)--(3,4.5);
\draw (2.5,5)--(3,5);
\draw (2.5,5.5)--(3,5.5);
\draw (2.75,4.5)--(2.75,5.5);

\draw (4,4.5)--(4,5.5);
\draw (1.5,5.5)--(2.5,5.5);

\draw (0.5,5.5)--(0.5,6);
\draw (1,5.5)--(1,6);
\draw (0,5.75)--(1,5.75);

\draw (3,5.5)--(3,6);
\draw (3.5,5.5)--(3.5,6);
\draw (4,5.5)--(4,6);
\draw (3,5.75)--(4,5.75);

\draw (1,6)--(1,7);
\draw (3,6)--(3,7);
\draw (4,6)--(4,7);

\draw (0.5,7)--(0.5,7.5);
\draw (1,7)--(1,7.5);
\draw (0,7.25)--(1,7.25);

\draw (3,7)--(3,7.5);
\draw (3.5,7)--(3.5,7.5);
\draw (4,7)--(4,7.5);
\draw (3,7.25)--(4,7.25);

\draw (1,7.5)--(3,7.5);

\draw (1,8)--(1.5,8);
\draw (1,8.5)--(1.5,8.5);
\draw (1.25,7.5)--(1.25,8.5);

\draw (2.5,8)--(3,8);
\draw (2.5,8.5)--(3,8.5);
\draw (2.75,7.5)--(2.75,8.5);

\draw (0,8.5)--(0,9);
\draw (0.5,8.5)--(0.5,9);
\draw (1,8.5)--(1,9);
\draw (1.5,8.5)--(1.5,9);
\draw (2,8.5)--(2,9);
\draw (2.5,8.5)--(2.5,9);
\draw (3,8.5)--(3,9);
\draw (3.5,8.5)--(3.5,9);
\draw (4,8.5)--(4,9);

\draw (0,8.75)--(1,8.75);
\draw (1.5,8.75)--(2.5,8.75);
\draw (3,8.75)--(4,8.75);

\draw (4,8.5)--(4,7.5);

\draw (2,9.5) node {\ldots};
\draw (9.5,2) node {\vdots};
\end{tikzpicture}
\caption{First Transformation}\label{SC_con_fig_trans2}
\end{figure}





\begin{figure}[ht]
\centering
\begin{tikzpicture}

\draw[fill=black] (0,0) circle (0.06);
\draw[fill=black] (0.5,0) circle (0.06);
\draw[fill=black] (1,0) circle (0.06);
\draw[fill=black] (1.5,0) circle (0.06);
\draw[fill=black] (2,0) circle (0.06);
\draw[fill=black] (2.5,0) circle (0.06);
\draw[fill=black] (3,0) circle (0.06);
\draw[fill=black] (3.5,0) circle (0.06);
\draw[fill=black] (4,0) circle (0.06);
\draw[fill=black] (4.5,0) circle (0.06);
\draw[fill=black] (5,0) circle (0.06);
\draw[fill=black] (5.5,0) circle (0.06);
\draw[fill=black] (6,0) circle (0.06);
\draw[fill=black] (6.5,0) circle (0.06);
\draw[fill=black] (7,0) circle (0.06);
\draw[fill=black] (7.5,0) circle (0.06);
\draw[fill=black] (8,0) circle (0.06);
\draw[fill=black] (8.5,0) circle (0.06);

\draw[fill=black] (0,0.5) circle (0.06);
\draw[fill=black] (1,0.5) circle (0.06);
\draw[fill=black] (1.5,0.5) circle (0.06);
\draw[fill=black] (2.5,0.5) circle (0.06);
\draw[fill=black] (3,0.5) circle (0.06);
\draw[fill=black] (4,0.5) circle (0.06);
\draw[fill=black] (4.5,0.5) circle (0.06);
\draw[fill=black] (5.5,0.5) circle (0.06);
\draw[fill=black] (6,0.5) circle (0.06);
\draw[fill=black] (7,0.5) circle (0.06);
\draw[fill=black] (7.5,0.5) circle (0.06);
\draw[fill=black] (8.5,0.5) circle (0.06);

\draw[fill=black] (0,1) circle (0.06);
\draw[fill=black] (0.5,1) circle (0.06);
\draw[fill=black] (1,1) circle (0.06);
\draw[fill=black] (1.5,1) circle (0.06);
\draw[fill=black] (2,1) circle (0.06);
\draw[fill=black] (2.5,1) circle (0.06);
\draw[fill=black] (3,1) circle (0.06);
\draw[fill=black] (3.5,1) circle (0.06);
\draw[fill=black] (4,1) circle (0.06);
\draw[fill=black] (4.5,1) circle (0.06);
\draw[fill=black] (5,1) circle (0.06);
\draw[fill=black] (5.5,1) circle (0.06);
\draw[fill=black] (6,1) circle (0.06);
\draw[fill=black] (6.5,1) circle (0.06);
\draw[fill=black] (7,1) circle (0.06);
\draw[fill=black] (7.5,1) circle (0.06);
\draw[fill=black] (8,1) circle (0.06);
\draw[fill=black] (8.5,1) circle (0.06);

\draw[fill=black] (0,1.5) circle (0.06);
\draw[fill=black] (0.5,1.5) circle (0.06);
\draw[fill=black] (1,1.5) circle (0.06);
\draw[fill=black] (3,1.5) circle (0.06);
\draw[fill=black] (3.5,1.5) circle (0.06);
\draw[fill=black] (4,1.5) circle (0.06);
\draw[fill=black] (4.5,1.5) circle (0.06);
\draw[fill=black] (5,1.5) circle (0.06);
\draw[fill=black] (5.5,1.5) circle (0.06);
\draw[fill=black] (7.5,1.5) circle (0.06);
\draw[fill=black] (8,1.5) circle (0.06);
\draw[fill=black] (8.5,1.5) circle (0.06);

\draw[fill=black] (0,2) circle (0.06);
\draw[fill=black] (1,2) circle (0.06);
\draw[fill=black] (3,2) circle (0.06);
\draw[fill=black] (4,2) circle (0.06);
\draw[fill=black] (4.5,2) circle (0.06);
\draw[fill=black] (5.5,2) circle (0.06);
\draw[fill=black] (7.5,2) circle (0.06);
\draw[fill=black] (8.5,2) circle (0.06);

\draw[fill=black] (0,2.5) circle (0.06);
\draw[fill=black] (0.5,2.5) circle (0.06);
\draw[fill=black] (1,2.5) circle (0.06);
\draw[fill=black] (3,2.5) circle (0.06);
\draw[fill=black] (3.5,2.5) circle (0.06);
\draw[fill=black] (4,2.5) circle (0.06);
\draw[fill=black] (4.5,2.5) circle (0.06);
\draw[fill=black] (5,2.5) circle (0.06);
\draw[fill=black] (5.5,2.5) circle (0.06);
\draw[fill=black] (7.5,2.5) circle (0.06);
\draw[fill=black] (8,2.5) circle (0.06);
\draw[fill=black] (8.5,2.5) circle (0.06);

\draw[fill=black] (0,3) circle (0.06);
\draw[fill=black] (0.5,3) circle (0.06);
\draw[fill=black] (1,3) circle (0.06);
\draw[fill=black] (1.5,3) circle (0.06);
\draw[fill=black] (2,3) circle (0.06);
\draw[fill=black] (2.5,3) circle (0.06);
\draw[fill=black] (3,3) circle (0.06);
\draw[fill=black] (3.5,3) circle (0.06);
\draw[fill=black] (4,3) circle (0.06);
\draw[fill=black] (4.5,3) circle (0.06);
\draw[fill=black] (5,3) circle (0.06);
\draw[fill=black] (5.5,3) circle (0.06);
\draw[fill=black] (6,3) circle (0.06);
\draw[fill=black] (6.5,3) circle (0.06);
\draw[fill=black] (7,3) circle (0.06);
\draw[fill=black] (7.5,3) circle (0.06);
\draw[fill=black] (8,3) circle (0.06);
\draw[fill=black] (8.5,3) circle (0.06);

\draw[fill=black] (0,3.5) circle (0.06);
\draw[fill=black] (1,3.5) circle (0.06);
\draw[fill=black] (1.5,3.5) circle (0.06);
\draw[fill=black] (2.5,3.5) circle (0.06);
\draw[fill=black] (3,3.5) circle (0.06);
\draw[fill=black] (4,3.5) circle (0.06);
\draw[fill=black] (4.5,3.5) circle (0.06);
\draw[fill=black] (5.5,3.5) circle (0.06);
\draw[fill=black] (6,3.5) circle (0.06);
\draw[fill=black] (7,3.5) circle (0.06);
\draw[fill=black] (7.5,3.5) circle (0.06);
\draw[fill=black] (8.5,3.5) circle (0.06);

\draw[fill=black] (0,4) circle (0.06);
\draw[fill=black] (0.5,4) circle (0.06);
\draw[fill=black] (1,4) circle (0.06);
\draw[fill=black] (1.5,4) circle (0.06);
\draw[fill=black] (2,4) circle (0.06);
\draw[fill=black] (2.5,4) circle (0.06);
\draw[fill=black] (3,4) circle (0.06);
\draw[fill=black] (3.5,4) circle (0.06);
\draw[fill=black] (4,4) circle (0.06);
\draw[fill=black] (4.5,4) circle (0.06);
\draw[fill=black] (5,4) circle (0.06);
\draw[fill=black] (5.5,4) circle (0.06);
\draw[fill=black] (6,4) circle (0.06);
\draw[fill=black] (6.5,4) circle (0.06);
\draw[fill=black] (7,4) circle (0.06);
\draw[fill=black] (7.5,4) circle (0.06);
\draw[fill=black] (8,4) circle (0.06);
\draw[fill=black] (8.5,4) circle (0.06);

\draw[fill=black] (0,4.5) circle (0.06);
\draw[fill=black] (0.5,4.5) circle (0.06);
\draw[fill=black] (1,4.5) circle (0.06);
\draw[fill=black] (1.5,4.5) circle (0.06);
\draw[fill=black] (2,4.5) circle (0.06);
\draw[fill=black] (2.5,4.5) circle (0.06);
\draw[fill=black] (3,4.5) circle (0.06);
\draw[fill=black] (3.5,4.5) circle (0.06);
\draw[fill=black] (4,4.5) circle (0.06);

\draw[fill=black] (0,5) circle (0.06);
\draw[fill=black] (1,5) circle (0.06);
\draw[fill=black] (1.5,5) circle (0.06);
\draw[fill=black] (2.5,5) circle (0.06);
\draw[fill=black] (3,5) circle (0.06);
\draw[fill=black] (4,5) circle (0.06);

\draw[fill=black] (0,5.5) circle (0.06);
\draw[fill=black] (0.5,5.5) circle (0.06);
\draw[fill=black] (1,5.5) circle (0.06);
\draw[fill=black] (1.5,5.5) circle (0.06);
\draw[fill=black] (2,5.5) circle (0.06);
\draw[fill=black] (2.5,5.5) circle (0.06);
\draw[fill=black] (3,5.5) circle (0.06);
\draw[fill=black] (3.5,5.5) circle (0.06);
\draw[fill=black] (4,5.5) circle (0.06);

\draw[fill=black] (0,6) circle (0.06);
\draw[fill=black] (0.5,6) circle (0.06);
\draw[fill=black] (1,6) circle (0.06);
\draw[fill=black] (3,6) circle (0.06);
\draw[fill=black] (3.5,6) circle (0.06);
\draw[fill=black] (4,6) circle (0.06);

\draw[fill=black] (0,6.5) circle (0.06);
\draw[fill=black] (1,6.5) circle (0.06);
\draw[fill=black] (3,6.5) circle (0.06);
\draw[fill=black] (4,6.5) circle (0.06);

\draw[fill=black] (0,7) circle (0.06);
\draw[fill=black] (0.5,7) circle (0.06);
\draw[fill=black] (1,7) circle (0.06);
\draw[fill=black] (3,7) circle (0.06);
\draw[fill=black] (3.5,7) circle (0.06);
\draw[fill=black] (4,7) circle (0.06);

\draw[fill=black] (0,7.5) circle (0.06);
\draw[fill=black] (0.5,7.5) circle (0.06);
\draw[fill=black] (1,7.5) circle (0.06);
\draw[fill=black] (1.5,7.5) circle (0.06);
\draw[fill=black] (2,7.5) circle (0.06);
\draw[fill=black] (2.5,7.5) circle (0.06);
\draw[fill=black] (3,7.5) circle (0.06);
\draw[fill=black] (3.5,7.5) circle (0.06);
\draw[fill=black] (4,7.5) circle (0.06);

\draw[fill=black] (0,8) circle (0.06);
\draw[fill=black] (1,8) circle (0.06);
\draw[fill=black] (1.5,8) circle (0.06);
\draw[fill=black] (2.5,8) circle (0.06);
\draw[fill=black] (3,8) circle (0.06);
\draw[fill=black] (4,8) circle (0.06);

\draw[fill=black] (0,8.5) circle (0.06);
\draw[fill=black] (0.5,8.5) circle (0.06);
\draw[fill=black] (1,8.5) circle (0.06);
\draw[fill=black] (1.5,8.5) circle (0.06);
\draw[fill=black] (2,8.5) circle (0.06);
\draw[fill=black] (2.5,8.5) circle (0.06);
\draw[fill=black] (3,8.5) circle (0.06);
\draw[fill=black] (3.5,8.5) circle (0.06);
\draw[fill=black] (4,8.5) circle (0.06);

\draw (0,0)--(8.5,0);
\draw (0,0)--(0,8.5);

\draw (1,1)--(1.5,1);
\draw (1,1)--(1,1.5);
\draw (1.25,0)--(1.25,1);
\draw (0,1.25)--(1,1.25);

\draw (1.5,1)--(3,1);
\draw (2.75,1)--(2.75,0);

\draw (3,1)--(3,1.5);
\draw (4,1)--(4,1.5);
\draw (3,1.25)--(4,1.25);
\draw (4,1)--(4.5,1);
\draw (4.25,0)--(4.25,1);

\draw (4.5,1)--(4.5,1.5);
\draw (5.5,1)--(5.5,1.5);
\draw (4.5,1.25)--(5.5,1.25);
\draw (5.5,1)--(6,1);
\draw (5.75,0)--(5.75,1);
\draw (6,1)--(7,1);
\draw (7,1)--(7.5,1);
\draw (7.25,0)--(7.25,1);
\draw (7.5,1)--(7.5,1.5);
\draw (8.5,1)--(8.5,1.5);
\draw (7.5,1.25)--(8.5,1.25);

\draw (8.5,0)--(9,0);
\draw (8.5,1)--(9,1);
\draw (8.75,0)--(8.75,1);

\draw (1,1.5)--(1,3);
\draw (0,2.75)--(1,2.75);

\draw (1,3)--(3,3);
\draw (3,3)--(3,1);

\draw (4,2.5)--(4,3);
\draw (3,2.75)--(4,2.75);
\draw (4,1.5)--(4.5,1.5);
\draw (4,2.5)--(4.5,2.5);
\draw (4.25,1.5)--(4.25,2.5);

\draw (4.5,2.5)--(4.5,3);
\draw (5.5,2.5)--(5.5,3);
\draw (4.5,2.75)--(5.5,2.75);
\draw (5.5,1.5)--(5.5,2.5);
\draw (5.5,3)--(7.5,3);
\draw (7.5,3)--(7.5,1.5);

\draw (8.5,2.5)--(8.5,3);
\draw (7.5,2.75)--(8.5,2.75);

\draw (8.5,1.5)--(9,1.5);
\draw (8.5,2.5)--(9,2.5);
\draw (8.75,1.5)--(8.75,2.5);

\draw (8.5,3)--(9,3);
\draw (8.5,4)--(9,4);
\draw (8.75,3)--(8.75,4);

\draw (7,4)--(7.5,4);
\draw (7.25,3)--(7.25,4);

\draw (5.5,4)--(6,4);
\draw (5.75,3)--(5.75,4);

\draw (6,4)--(7,4);
\draw (4.5,4)--(5.5,4);

\draw (4,3)--(4.5,3);
\draw (4,4)--(4.5,4);
\draw (4.25,3)--(4.25,4);

\draw (4,4)--(4,4.5);
\draw (3,4)--(3,4.5);
\draw (3,4.25)--(4,4.25);

\draw (2.5,4)--(3,4);
\draw (2.75,3)--(2.75,4);

\draw (1,4)--(1.5,4);
\draw (1.25,3)--(1.25,4);

\draw (7.5,4)--(8.5,4);

\draw (1,4)--(1,4.5);
\draw (0,4.25)--(1,4.25);

\draw (1.5,4)--(1.5,4.5);
\draw (2.5,4)--(2.5,4.5);
\draw (1.5,4.25)--(2.5,4.25);

\draw (1,4.5)--(1.5,4.5);
\draw (1,5.5)--(1.5,5.5);
\draw (1.25,4.5)--(1.25,5.5);

\draw (2.5,4.5)--(3,4.5);
\draw (2.5,5.5)--(3,5.5);
\draw (2.75,4.5)--(2.75,5.5);

\draw (4,4.5)--(4,5.5);
\draw (1.5,5.5)--(2.5,5.5);

\draw (1,5.5)--(1,6);
\draw (0,5.75)--(1,5.75);

\draw (3,5.5)--(3,6);
\draw (4,5.5)--(4,6);
\draw (3,5.75)--(4,5.75);

\draw (1,6)--(1,7);
\draw (3,6)--(3,7);
\draw (4,6)--(4,7);

\draw (1,7)--(1,7.5);
\draw (0,7.25)--(1,7.25);

\draw (3,7)--(3,7.5);
\draw (4,7)--(4,7.5);
\draw (3,7.25)--(4,7.25);

\draw (1,7.5)--(3,7.5);

\draw (1,8.5)--(1.5,8.5);
\draw (1.25,7.5)--(1.25,8.5);

\draw (2.5,8.5)--(3,8.5);
\draw (2.75,7.5)--(2.75,8.5);

\draw (0,8.5)--(0,9);
\draw (1,8.5)--(1,9);
\draw (1.5,8.5)--(1.5,9);
\draw (2.5,8.5)--(2.5,9);
\draw (3,8.5)--(3,9);
\draw (4,8.5)--(4,9);

\draw (0,8.75)--(1,8.75);
\draw (1.5,8.75)--(2.5,8.75);
\draw (3,8.75)--(4,8.75);

\draw (4,8.5)--(4,7.5);

\draw (2,9.5) node {\ldots};
\draw (9.5,2) node {\vdots};
\end{tikzpicture}
\caption{First Transformation}\label{SC_con_fig_trans3}
\end{figure}





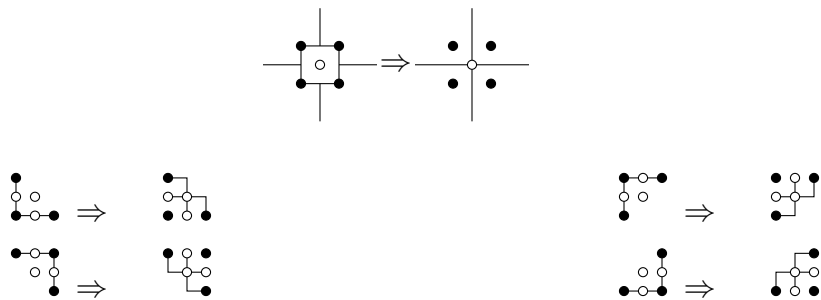
\begin{figure}[ht]
\centering
\begin{tikzpicture}

\draw (0.25-1,0.25)--(0.25-1,-0.25)--(-0.25-1,-0.25)--(-0.25-1,0.25)--cycle;
\draw (0-1,0.25)--(0-1,0.75);
\draw (0-1,-0.25)--(0-1,-0.75);
\draw (0.25-1,0)--(0.75-1,0);
\draw (-0.25-1,0)--(-0.75-1,0);

\draw[fill=black] (0.25-1,0.25) circle (0.06);
\draw[fill=black] (0.25-1,-0.25) circle (0.06);
\draw[fill=black] (-0.25-1,-0.25) circle (0.06);
\draw[fill=black] (-0.25-1,0.25) circle (0.06);
\draw[fill=white] (-1,0) circle (0.06);

\draw (0,0) node {$\Rightarrow$};

\draw (0.25,0)--(1.75,0);
\draw (1,-0.75)--(1,0.75);

\draw[fill=black] (0.25+1,0.25) circle (0.06);
\draw[fill=black] (0.25+1,-0.25) circle (0.06);
\draw[fill=black] (-0.25+1,-0.25) circle (0.06);
\draw[fill=black] (-0.25+1,0.25) circle (0.06);
\draw[fill=white] (1,0) circle (0.06);

\draw (-5,-1.5)--(-5,-2)--(-4.5,-2);
\draw[fill=black] (-5,-1.5) circle (0.06);
\draw[fill=black] (-5,-2) circle (0.06);
\draw[fill=black] (-4.5,-2) circle (0.06);
\draw[fill=white] (-5,-1.75) circle (0.06);
\draw[fill=white] (-4.75,-2) circle (0.06);
\draw[fill=white] (-4.75,-1.75) circle (0.06);

\draw (-4,-2) node {$\Rightarrow$};

\draw (-3,-1.75)--(-2.5,-1.75)--(-2.5,-2);
\draw (-2.75,-2)--(-2.75,-1.5)--(-3,-1.5);
\draw[fill=black] (-5+2,-1.5) circle (0.06);
\draw[fill=black] (-5+2,-2) circle (0.06);
\draw[fill=black] (-4.5+2,-2) circle (0.06);
\draw[fill=white] (-5+2,-1.75) circle (0.06);
\draw[fill=white] (-4.75+2,-2) circle (0.06);
\draw[fill=white] (-4.75+2,-1.75) circle (0.06);

\draw (3,-2)--(3,-1.5)--(3.5,-1.5);
\draw[fill=black] (3,-2) circle (0.06);
\draw[fill=black] (3,-1.5) circle (0.06);
\draw[fill=black] (3.5,-1.5) circle (0.06);
\draw[fill=white] (3,-1.75) circle (0.06);
\draw[fill=white] (3.25,-1.5) circle (0.06);
\draw[fill=white] (3.25,-1.75) circle (0.06);

\draw (4,-2) node {$\Rightarrow$};

\draw (5,-1.75)--(5.5,-1.75)--(5.5,-1.5);
\draw (5.25,-1.5)--(5.25,-2)--(5,-2);
\draw[fill=black] (5,-2) circle (0.06);
\draw[fill=black] (5,-1.5) circle (0.06);
\draw[fill=black] (5.5,-1.5) circle (0.06);
\draw[fill=white] (5,-1.75) circle (0.06);
\draw[fill=white] (5.25,-1.5) circle (0.06);
\draw[fill=white] (5.25,-1.75) circle (0.06);

\draw (-5,-2.5)--(-4.5,-2.5)--(-4.5,-3);
\draw[fill=black] (-5,-2.5) circle (0.06);
\draw[fill=black] (-4.5,-2.5) circle (0.06);
\draw[fill=black] (-4.5,-3) circle (0.06);
\draw[fill=white] (-4.75,-2.5) circle (0.06);
\draw[fill=white] (-4.5,-2.75) circle (0.06);
\draw[fill=white] (-4.75,-2.75) circle (0.06);

\draw (-4,-3) node {$\Rightarrow$};

\draw (-2.75,-2.5)--(-2.75,-3)--(-2.5,-3);
\draw (-3,-2.5)--(-3,-2.75)--(-2.5,-2.75);
\draw[fill=black] (-5+2,-2.5) circle (0.06);
\draw[fill=black] (-4.5+2,-2.5) circle (0.06);
\draw[fill=black] (-4.5+2,-3) circle (0.06);
\draw[fill=white] (-4.75+2,-2.5) circle (0.06);
\draw[fill=white] (-4.5+2,-2.75) circle (0.06);
\draw[fill=white] (-4.75+2,-2.75) circle (0.06);

\draw (3,-3)--(3.5,-3)--(3.5,-2.5);
\draw[fill=black] (3,-3) circle (0.06);
\draw[fill=black] (3.5,-3) circle (0.06);
\draw[fill=black] (3.5,-2.5) circle (0.06);
\draw[fill=white] (3.25,-3) circle (0.06);
\draw[fill=white] (3.5,-2.75) circle (0.06);
\draw[fill=white] (3.25,-2.75) circle (0.06);

\draw (4,-3) node {$\Rightarrow$};

\draw (5,-3)--(5,-2.75)--(5.5,-2.75);
\draw (5.25,-3)--(5.25,-2.5)--(5.5,-2.5);
\draw[fill=black] (3+2,-3) circle (0.06);
\draw[fill=black] (3.5+2,-3) circle (0.06);
\draw[fill=black] (3.5+2,-2.5) circle (0.06);
\draw[fill=white] (3.25+2,-3) circle (0.06);
\draw[fill=white] (3.5+2,-2.75) circle (0.06);
\draw[fill=white] (3.25+2,-2.75) circle (0.06);
\end{tikzpicture}
\caption{Second Transformation}\label{SC_con_fig_trans4}
\end{figure}


Secondly, we do the transformations in Figure \ref{SC_con_fig_trans4} where the resistances of the resistors in the new networks only depend on the shape of the networks in Figure \ref{SC_con_fig_trans4} such that we obtain a weighted graph with vertex set $V_{n-1}$ and all conductances equivalent to $1$. Moreover, the resistances between any two points are larger than those in the weighted graph $\calW_n$, hence we obtain the desired result.
\end{proof}

Now, we estimate $R_n(p_0,p_1)$ and $\frakR_n(0^n,1^n)$ as follows.

\begin{proof}[Proof of Theorem \ref{SC_con_thm_resist}]
The idea is that replacing one point by one network should increase resistances by multiplying the resistance of an individual network.

By Proposition \ref{SC_con_prop_resist2} and Proposition \ref{SC_con_prop_resist3}, we have for all $n\ge1$
$$R_n(p_0,p_1)\asymp\frakR_n(0^n,1^n).$$
By Theorem \ref{SC_con_thm_resist1} and Proposition \ref{SC_con_prop_resist2}, we have for all $n\ge1$
$$\frakR_n(0^n,1^n)\ge R_n(p_0,p_1)\ge\frac{1}{4}R_n(p_1,p_5)\asymp\rho^n.$$
We only need to show that for all $n\ge1$
$$\frakR_n(0^n,1^n)\lesssim\rho^n.$$

Firstly, we estimate $\frakR_{n+1}(0^{n+1},12^n)$. Cutting certain edges in $\calW_{n+1}$, we obtain the electrical network in Figure \ref{SC_con_fig_resist1} which is equivalent to the electrical networks in Figure \ref{SC_con_fig_resist2}.




\begin{figure}[ht]
\centering
\begin{tikzpicture}
\draw (0,0)--(2,0)--(2,2)--(0,2)--cycle;
\draw (3,0)--(5,0)--(5,2)--(3,2)--cycle;
\draw (6,0)--(8,0)--(8,2)--(6,2)--cycle;
\draw (0,3)--(2,3)--(2,5)--(0,5)--cycle;
\draw (6,3)--(8,3)--(8,5)--(6,5)--cycle;
\draw (0,6)--(2,6)--(2,8)--(0,8)--cycle;
\draw (3,6)--(5,6)--(5,8)--(3,8)--cycle;
\draw (6,6)--(8,6)--(8,8)--(6,8)--cycle;

\draw (2,2)--(2,3);
\draw (2,2)--(3,2);
\draw (0,5)--(0,6);
\draw (2,8)--(3,8);
\draw (5,6)--(6,6);
\draw (6,6)--(6,5);
\draw (8,3)--(8,2);
\draw (5,0)--(6,0);

\draw[fill=black] (0,0) circle (0.06);
\draw[fill=black] (5,0) circle (0.06);

\draw (0,-0.5) node {$0^{n+1}$};
\draw (5,-0.5) node {$12^n$};

\draw (1,1) node {$0W_n$};
\draw (4,1) node {$1W_n$};
\draw (7,1) node {$2W_n$};
\draw (7,4) node {$3W_n$};
\draw (7,7) node {$4W_n$};
\draw (4,7) node {$5W_n$};
\draw (1,7) node {$6W_n$};
\draw (1,4) node {$7W_n$};

\end{tikzpicture}
\caption{An Equivalent Electrical Network}\label{SC_con_fig_resist1}
\end{figure}





\begin{figure}[ht]
\centering
\subfigure{
\begin{tikzpicture}[scale=0.5]
\draw (0,0)--(2,2);
\draw (5,0)--(3,2);
\draw (6,0)--(8,2);
\draw (2,3)--(0,5);
\draw (8,3)--(6,5);
\draw (0,6)--(2,8);
\draw (5,6)--(3,8);

\draw (2,2)--(2,3);
\draw (2,2)--(3,2);
\draw (0,5)--(0,6);
\draw (2,8)--(3,8);
\draw (5,6)--(6,6);
\draw (6,6)--(6,5);
\draw (8,3)--(8,2);
\draw (5,0)--(6,0);

\draw[fill=black] (0,0) circle (0.06);
\draw[fill=black] (5,0) circle (0.06);

\draw (0,-0.5) node {$0^{n+1}$};
\draw (5,-0.5) node {$12^n$};

\draw (0,1.4) node {\tiny{$\mathfrak{R}_n(0^n,4^n)$}};
\draw (2,4.4) node {\tiny{$\mathfrak{R}_n(2^n,6^n)$}};
\draw (0,7.4) node {\tiny{$\mathfrak{R}_n(0^n,4^n)$}};
\draw (5,7.4) node {\tiny{$\mathfrak{R}_n(2^n,6^n)$}};
\draw (8,4.4) node {\tiny{$\mathfrak{R}_n(2^n,6^n)$}};
\draw (5,1.4) node {\tiny{$\mathfrak{R}_n(2^n,6^n)$}};
\draw (8,0.6) node {\tiny{$\mathfrak{R}_n(0^n,4^n)$}};

\draw (2.5,1.7) node {\tiny{$1$}};
\draw (1.7,2.5) node {\tiny{$1$}};
\draw (-0.3,5.5) node {\tiny{$1$}};
\draw (2.5,8.3) node {\tiny{$1$}};
\draw (5.5,6.3) node {\tiny{$1$}};
\draw (6.3,5.5) node {\tiny{$1$}};
\draw (8.3,2.5) node {\tiny{$1$}};
\draw (5.5,0.3) node {\tiny{$1$}};

\end{tikzpicture}}
\hspace{0.01pt}
\subfigure{
\begin{tikzpicture}
\draw (0,0)--(2,0);
\draw (2,0)..controls (3,1) and (4,1) ..(5,0);
\draw (2,0)..controls (3,-1) and (4,-1) ..(5,0);

\draw[fill=black] (0,0) circle (0.06);
\draw[fill=black] (5,0) circle (0.06);

\draw (0,-0.5) node {$0^{n+1}$};
\draw (5,-0.5) node {$12^n$};

\draw (1,0.5) node {$\mathfrak{R}_n(0^n,4^n)$};
\draw (3.5,1) node {$5\mathfrak{R}_n(0^n,4^n)+7$};
\draw (3.5,-1) node {$\mathfrak{R}_n(0^n,4^n)+1$};

\end{tikzpicture}
}
\caption{Equivalent Electrical Networks}\label{SC_con_fig_resist2}
\end{figure}


Hence
$$
\begin{aligned}
\frakR_{n+1}(0^{n+1},12^n)&\le\frakR_n(0^n,4^n)+\frac{\left(5\frakR_n(0^n,4^n)+7\right)\left(\frakR_n(0^n,4^n)+1\right)}{\left(5\frakR_n(0^n,4^n)+7\right)+\left(\frakR_n(0^n,4^n)+1\right)}\\
&\lesssim\frakR_n(0^n,4^n)+\frac{5}{6}\frakR_n(0^n,4^n)=\frac{11}{6}\frakR_n(0^n,4^n)\lesssim\rho^{n+1}.
\end{aligned}
$$
Secondly, from $0^{n+1}$ to $1^{n+1}$, we construct a finite sequence as follows. For $i=1,\ldots,n+2$,
$$
w^{(i)}=
\begin{cases}
1^{i-1}0^{n+2-i},\text{ if }i\text{ is an odd number},\\
1^{i-1}2^{n+2-i},\text{ if }i\text{ is an even number}.\\
\end{cases}
$$
By cutting technique, if $i$ is an odd number, then
$$
\begin{aligned}
&\frakR_{n+1}(w^{(i)},w^{(i+1)})=\frakR_{n+1}(1^{i-1}0^{n+2-i},1^{i}2^{n+1-i})\\
\le&\frakR_{n+2-i}(0^{n+2-i},12^{n+1-i})\lesssim\rho^{n+2-i}.
\end{aligned}
$$
If $i$ is an even number, then
$$
\begin{aligned}
&\frakR_{n+1}(w^{(i)},w^{(i+1)})=\frakR_{n+1}(1^{i-1}2^{n+2-i},1^{i}0^{n+1-i})\\
\le&\frakR_{n+2-i}(2^{n+2-i},10^{n+1-i})=\frakR_{n+2-i}(0^{n+2-i},12^{n+1-i})\lesssim\rho^{n+2-i}.
\end{aligned}
$$
Hence
$$
\begin{aligned}
&\frakR_{n+1}(0^{n+1},1^{n+1})=\frakR_{n+1}(w^{(1)},w^{(n+2)})\\
\le&\sum_{i=1}^{n+1}\frakR_{n+1}(w^{(i)},w^{(i+1)})\lesssim\sum_{i=1}^{n+1}\rho^{n+2-i}=\sum_{i=1}^{n+1}\rho^{i}\lesssim\rho^{n+1}.
\end{aligned}
$$
\end{proof}

\section{Uniform Harnack Inequality}\label{SC_con_sec_harnack}

In this section, we give uniform Harnack inequality as follows.

\begin{mythm}\label{SC_con_thm_harnack}
There exist some constants $C\in(0,+\infty),\delta\in(0,1)$ such that for all $n\ge1,x\in K,r>0$, for all nonnegative harmonic function $u$ on $V_n\cap B(x,r)$, we have
$$\max_{V_n\cap B(x,\delta r)}u\le C\min_{V_n\cap B(x,\delta r)}u.$$
\end{mythm}

\begin{myrmk}
The point of the above theorem is that the constant $C$ is \emph{uniform} in $n$.
\end{myrmk}

The idea is as follows. Firstly, we use resistance estimates in finite graphs $V_n$ to obtain resistance estimates in an infinite graph $V_\infty$. Secondly, we obtain Green function estimates in $V_\infty$. Thirdly, we obtain elliptic Harnack inequality in $V_\infty$. Finally, we transfer elliptic Harnack inequality in $V_\infty$ to uniform Harnack inequality in $V_n$.

Let $\calV_\infty$ be the graph with vertex set $V_\infty=\cup_{n=0}^\infty3^nV_n$ and edge set given by
$$\myset{(p,q):p,q\in V_\infty,|p-q|=2^{-1}}.$$
We have the figure of $\calV_\infty$ in Figure \ref{fig_graphSC}.

Locally, $\calV_\infty$ is like $\calV_n$. Let the conductances of all edges be $1$. Let $d$ be an integer-valued metric, that is, $d(p,q)$ is the minimum of the lengths of all paths connecting $p$ and $q$. It is obvious that
$$d(p,q)\asymp|p-q|\text{ for all }p,q\in V_\infty.$$

By shorting and cutting technique, we reduce $\calV_\infty$ to $\calV_n$ to obtain resistance estimates as follows.
$$R(x,y)\asymp\rho^{\frac{\log d(x,y)}{\log 3}}=d(x,y)^{\frac{\log\rho}{\log3}}=d(x,y)^\gamma\text{ for all }x,y\in V_\infty,$$
where $\gamma=\log\rho/\log3$.

Let $g_B$ be the Green function in a ball $B$. We have Green function estimates as follows.

\begin{mythm}(\cite[Proposition 6.11]{GHL14})\label{SC_con_thm_green}
There exist some constants $C\in(0,+\infty),\eta\in(0,1)$ such that for all $z\in V_\infty,r>0$, we have
$$g_{B(z,r)}(x,y)\le Cr^\gamma\text{ for all }x,y\in B(z,r),$$
$$g_{B(z,r)}(z,y)\ge\frac{1}{C}r^\gamma\text{ for all }y\in B(z,\eta r).$$
\end{mythm}

We obtain elliptic Harnack inequality in $V_\infty$ as follows.

\begin{mythm}(\cite[Lemma 10.2]{GT01},\cite[Theorem 3.12]{GH14a})\label{SC_con_thm_harnack_infinite}
There exist some constants $C\in(0,+\infty)$, $\delta\in(0,1)$ such that for all $z\in V_\infty,r>0$, for all nonnegative harmonic function $u$ on $V_\infty\cap B(z,r)$, we have
$$\max_{B(z,\delta r)}u\le C\min_{B(z,\delta r)}u.$$
\end{mythm}

\begin{myrmk}
We give an alternative approach as follows. It was proved in \cite{BCK05} that sub-Gaussian heat kernel estimates are equivalent to resistance estimates for random walks on fractal graph under strongly recurrent condition. Hence we obtain sub-Gaussian heat kernel estimates, see \cite[Example 4]{BCK05}. It was proved in \cite[Theorem 3.1]{GT02} that sub-Gaussian heat kernel estimates imply elliptic Harnack inequality. Hence we obtain elliptic Harnack inequality in $V_\infty$.
\end{myrmk}

Now we obtain Theorem \ref{SC_con_thm_harnack} directly.

\section{Weak Monotonicity Results}\label{SC_con_sec_monotone}

In this section, we give two weak monotonicity results.

For all $n\ge1$, let
$$a_n(u)=\rho^n\sum_{w\in W_n}
{\sum_{\mbox{\tiny
$
\begin{subarray}{c}
p,q\in V_w\\
|p-q|=2^{-1}\cdot3^{-n}
\end{subarray}
$
}}}
(u(p)-u(q))^2,u\in l(V_n).$$

We have one weak monotonicity result as follows.

\begin{mythm}\label{SC_con_thm_monotone1}
There exists some positive constant $C$ such that for all $n,m\ge1,u\in l(V_{n+m})$, we have
$$a_n(u)\le Ca_{n+m}(u).$$
\end{mythm}
\begin{proof}
For all $w\in W_n,p,q\in V_w$ with $|p-q|=2^{-1}\cdot3^{-n}$, by cutting technique and Corollary \ref{SC_con_cor_resist_upper}
$$
\begin{aligned}
\left(u(p)-u(q)\right)^2&\le R_m(f_w^{-1}(p),f_w^{-1}(q))
\sum_{v\in W_m}
{\sum_{\mbox{\tiny
$
\begin{subarray}{c}
x,y\in V_{wv}\\
|x-y|=2^{-1}\cdot3^{-(n+m)}
\end{subarray}
$
}}}
(u(x)-u(y))^2\\
&\le C\rho^m\sum_{v\in W_m}
{\sum_{\mbox{\tiny
$
\begin{subarray}{c}
x,y\in V_{wv}\\
|x-y|=2^{-1}\cdot3^{-(n+m)}
\end{subarray}
$
}}}
(u(x)-u(y))^2.
\end{aligned}
$$
Hence
$$
\begin{aligned}
a_n(u)&=\rho^n\sum_{w\in W_n}
{\sum_{\mbox{\tiny
$
\begin{subarray}{c}
p,q\in V_w\\
|p-q|=2^{-1}\cdot3^{-n}
\end{subarray}
$
}}}
(u(p)-u(q))^2\\
&\le\rho^n\sum_{w\in W_n}
{\sum_{\mbox{\tiny
$
\begin{subarray}{c}
p,q\in V_w\\
|p-q|=2^{-1}\cdot3^{-n}
\end{subarray}
$
}}}
\left(C\rho^m\sum_{v\in W_m}
{\sum_{\mbox{\tiny
$
\begin{subarray}{c}
x,y\in V_{wv}\\
|x-y|=2^{-1}\cdot3^{-(n+m)}
\end{subarray}
$
}}}
(u(x)-u(y))^2\right)\\
&=C\rho^{n+m}\sum_{w\in W_{n+m}}
{\sum_{\mbox{\tiny
$
\begin{subarray}{c}
p,q\in V_w\\
|p-q|=2^{-1}\cdot3^{-(n+m)}
\end{subarray}
$
}}}
(u(p)-u(q))^2=Ca_{n+m}(u).
\end{aligned}
$$
\end{proof}

For all $n\ge1$, let
$$b_n(u)=\rho^n
{\sum_{\mbox{\tiny
$
\begin{subarray}{c}
w^{(1)}\sim_nw^{(2)}
\end{subarray}
$
}}}
(P_nu(w^{(1)})-P_nu(w^{(2)}))^2,u\in L^2(K;\nu).$$

We have another weak monotonicity result as follows.
\begin{mythm}\label{SC_con_thm_monotone2}
There exists some positive constant $C$ such that for all $n,m\ge1,u\in L^2(K;\nu)$, we have
$$b_n(u)\le Cb_{n+m}(u).$$
\end{mythm}

\begin{myrmk}
This result was also obtained in \cite[Proposition 5.2]{KZ92}. Here we give a direct proof using resistance estimates.
\end{myrmk}

This result can be reduced as follows.

For all $n\ge1$, let
$$B_n(u)=\rho^n\sum_{w^{(1)}\sim_nw^{(2)}}(u(w^{(1)})-u(w^{(2)}))^2,u\in l(W_n).$$

For all $n,m\ge1$, let $M_{n,m}:l(W_{n+m})\to l(W_n)$ be a mean value operator given by
$$(M_{n,m}u)(w)=\frac{1}{8^m}\sum_{v\in W_m}u(wv),w\in W_n,u\in l(W_{n+m}).$$

\begin{mythm}\label{SC_con_thm_monotonicity2}
There exists some positive constant $C$ such that for all $n,m\ge1,u\in l(W_{n+m})$, we have
$$B_n(M_{n,m}u)\le CB_{n+m}(u).$$
\end{mythm}

\begin{proof}[Proof of Theorem \ref{SC_con_thm_monotone2} using Theorem \ref{SC_con_thm_monotonicity2}]
For all $u\in L^2(K;\nu)$, note that
$$P_nu=M_{n,m}(P_{n+m}u),$$
hence
$$
\begin{aligned}
b_n(u)&=\rho^n\sum_{w^{(1)}\sim_nw^{(2)}}(P_nu(w^{(1)})-P_nu(w^{(2)}))^2=B_n(P_nu)\\
&=B_n(M_{n,m}(P_{n+m}u))\le CB_{n+m}(P_{n+m}u)\\
&=C\rho^{n+m}\sum_{w^{(1)}\sim_{n+m}w^{(2)}}(P_{n+m}u(w^{(1)})-P_{n+m}u(w^{(2)}))^2=Cb_{n+m}(u).
\end{aligned}
$$
\end{proof}

\begin{proof}[Proof of Theorem \ref{SC_con_thm_monotonicity2}]
Fix $n\ge1$. Assume that $W\subseteq W_n$ is connected, that is, for all $w^{(1)},w^{(2)}\in W$, there exists a finite sequence $\myset{v^{(1)},\ldots,v^{(k)}}\subseteq W$ such that $v^{(1)}=w^{(1)},v^{(k)}=w^{(2)}$ and $v^{(i)}\sim_nv^{(i+1)}$ for all $i=1,\ldots,k-1$. Let
$$\frakD_W(u,u):=
{\sum_{\mbox{\tiny
$
\begin{subarray}{c}
w^{(1)},w^{(2)}\in W\\
w^{(1)}\sim_nw^{(2)}
\end{subarray}
$
}}}
(u(w^{(1)})-u(w^{(2)}))^2,u\in l(W).$$
For all $w^{(1)},w^{(2)}\in W$, let
$$
\begin{aligned}
\frakR_W(w^{(1)},w^{(2)})&=\inf\myset{\frakD_W(u,u):u(w^{(1)})=0,u(w^{(2)})=1,u\in l(W)}^{-1}\\
&=\sup\myset{\frac{(u(w^{(1)})-u(w^{(2)}))^2}{\frakD_W(u,u)}:\frakD_W(u,u)\ne0,u\in l(W)}.
\end{aligned}
$$
It is obvious that
$$(u(w^{(1)})-u(w^{(2)}))^2\le\frakR_W(w^{(1)},w^{(2)})\frakD_W(u,u)\text{ for all }w^{(1)},w^{(2)}\in W,u\in l(W),$$
and $\frakR_W$ is a metric on $W$, hence
$$\frakR_W(w^{(1)},w^{(2)})\le\frakR_W(w^{(1)},w^{(3)})+\frakR_W(w^{(3)},w^{(2)})\text{ for all }w^{(1)},w^{(2)},w^{(3)}\in W.$$

Fix $w^{(1)}\sim_nw^{(2)}$, there exist $i,j=0,\ldots,7$ such that $w^{(1)}i^m\sim_{n+m}w^{(2)}j^m$, see Figure \ref{SC_con_fig_monotonicity}.




\begin{figure}[ht]
\centering
\begin{tikzpicture}

\draw (0,0)--(3,0)--(3,3)--(0,3)--cycle;
\draw (4,0)--(7,0)--(7,3)--(4,3)--cycle;
\draw (3,0)--(4,0);
\draw (3,3)--(4,3);

\draw (3.5,1.5) node {$\vdots$};

\draw[fill=black] (3,0) circle (0.06);
\draw[fill=black] (4,0) circle (0.06);

\draw (1.5,1.5) node {$w^{(1)}W_m$};
\draw (5.5,1.5) node {$w^{(2)}W_m$};
\draw (2.5,-0.5) node {$w^{(1)}i^m$};
\draw (4.5,-0.5) node {$w^{(2)}j^m$};

\draw (1,2.5) node {$w^{(1)}v$};
\draw[fill=black] (1,2.2) circle (0.06);
\draw (5,2.5) node {$w^{(2)}v$};
\draw[fill=black] (5,2.2) circle (0.06);

\end{tikzpicture}
\caption{$w^{(1)}W_m$ and $w^{(2)}W_m$}\label{SC_con_fig_monotonicity}
\end{figure}


Fix $v\in W_m$
$$(u(w^{(1)}v)-u(w^{(2)}v))^2\le\frakR_{w^{(1)}W_m\cup w^{(2)}W_m}(w^{(1)}v,w^{(2)}v)\frakD_{w^{(1)}W_m\cup w^{(2)}W_m}(u,u).$$
By cutting technique and Corollary \ref{SC_con_cor_resist_upper}
$$
\begin{aligned}
&\frakR_{w^{(1)}W_m\cup w^{(2)}W_m}(w^{(1)}v,w^{(2)}v)\\
\le&\frakR_{w^{(1)}W_m\cup w^{(2)}W_m}(w^{(1)}v,w^{(1)}i^m)+\frakR_{w^{(1)}W_m\cup w^{(2)}W_m}(w^{(1)}i^m,w^{(2)}j^m)\\
&+\frakR_{w^{(1)}W_m\cup w^{(2)}W_m}(w^{(2)}j^m,w^{(2)}v)\\
\le&\frakR_m(v,i^m)+1+\frakR_m(v,j^m)\lesssim\rho^m.
\end{aligned}
$$
Hence
$$
\begin{aligned}
&(u(w^{(1)}v)-u(w^{(2)}v))^2\lesssim\rho^m\frakD_{w^{(1)}W_m\cup w^{(2)}W_m}(u,u)\\
=&\rho^m\left(\frakD_{w^{(1)}W_m}(u,u)+\frakD_{w^{(2)}W_m}(u,u)\right.\\
&\left.+
{\sum_{\mbox{\tiny
$
\begin{subarray}{c}
v^{(1)},v^{(2)}\in W_m\\
w^{(1)}v^{(1)}\sim_{n+m}w^{(2)}v^{(2)}
\end{subarray}
$
}}}
(u(w^{(1)}v^{(1)})-u(w^{(2)}v^{(2)}))^2\right).
\end{aligned}
$$
Hence
$$
\begin{aligned}
&\left(M_{n,m}u(w^{(1)})-M_{n,m}u(w^{(2)})\right)^2=\left(\frac{1}{8^m}\sum_{v\in W_m}\left(u(w^{(1)}v)-u(w^{(2)}v)\right)\right)^2\\
\le&\frac{1}{8^m}\sum_{v\in W_m}\left(u(w^{(1)}v)-u(w^{(2)}v)\right)^2\\
\lesssim&\rho^m\left(\frakD_{w^{(1)}W_m}(u,u)+\frakD_{w^{(2)}W_m}(u,u)\right.\\
&\left.+
{\sum_{\mbox{\tiny
$
\begin{subarray}{c}
v^{(1)},v^{(2)}\in W_m\\
w^{(1)}v^{(1)}\sim_{n+m}w^{(2)}v^{(2)}
\end{subarray}
$
}}}
(u(w^{(1)}v^{(1)})-u(w^{(2)}v^{(2)}))^2\right).
\end{aligned}
$$
In the summation with respect to $w^{(1)}\sim_nw^{(2)}$, the terms $\frakD_{w^{(1)}W_m}(u,u),\frakD_{w^{(2)}W_m}(u,u)$ are summed at most $8$ times, hence
$$
\begin{aligned}
B_n(M_{n,m}u)&=\rho^n\sum_{w^{(1)}\sim_nw^{(2)}}\left(M_{n,m}u(w^{(1)})-M_{n,m}u(w^{(2)})\right)^2\\
&\lesssim\rho^n\sum_{w^{(1)}\sim_nw^{(2)}}\rho^m\left(\frakD_{w^{(1)}W_m}(u,u)+\frakD_{w^{(2)}W_m}(u,u)\right.\\
&\left.+
{\sum_{\mbox{\tiny
$
\begin{subarray}{c}
v^{(1)},v^{(2)}\in W_m\\
w^{(1)}v^{(1)}\sim_{n+m}w^{(2)}v^{(2)}
\end{subarray}
$
}}}
(u(w^{(1)}v^{(1)})-u(w^{(2)}v^{(2)}))^2\right)\\
&\le8\rho^{n+m}\sum_{w^{(1)}\sim_{n+m}w^{(2)}}\left(u(w^{(1)})-u(w^{(2)})\right)^2=8B_{n+m}(u).
\end{aligned}
$$
\end{proof}

\section{One Good Function}\label{SC_con_sec_good}

In this section, we construct \emph{one} good function with energy property and separation property.

By standard argument, we have H\"older continuity from Harnack inequality as follows.

\begin{mythm}\label{SC_con_thm_holder}
For all $0\le\delta_1<\veps_1<\veps_2<\delta_2\le1$, there exist some positive constants $\theta=\theta(\delta_1,\delta_2,\veps_1,\veps_2)$, $C=C(\delta_1,\delta_2,\veps_1,\veps_2)$ such that for all $n\ge1$, for all bounded harmonic function $u$ on $V_n\cap(\delta_1,\delta_2)\times[0,1]$, we have
$$|u(x)-u(y)|\le C|x-y|^\theta\left(\max_{V_n\cap[\delta_1,\delta_2]\times[0,1]}|u|\right)\text{ for all }x,y\in V_n\cap[\veps_1,\veps_2]\times[0,1].$$
\end{mythm}
\begin{proof}
The proof is similar to \cite[Theorem 3.9]{BB89}.
\end{proof}

For all $n\ge1$. Let $u_n\in l(V_n)$ satisfy $u_n|_{V_n\cap\myset{0}\times[0,1]}=0,u_n|_{V_n\cap\myset{1}\times[0,1]}=1$ and
$$D_n(u_n,u_n)=\sum_{w\in W_n}
{\sum_{\mbox{\tiny
$
\begin{subarray}{c}
p,q\in V_w\\
|p-q|=2^{-1}\cdot3^{-n}
\end{subarray}
$
}}}
(u_n(p)-u_n(q))^2=(R_n^V)^{-1}.$$
Then $u_n$ is harmonic on $V_n\cap(0,1)\times[0,1]$, $u_n(x,y)=1-u_n(1-x,y)=u_n(x,1-y)$ for all $(x,y)\in V_n$ and
$$u_n|_{V_n\cap\myset{\frac{1}{2}}\times[0,1]}=\frac{1}{2},u_n|_{V_n\cap[0,\frac{1}{2})\times[0,1]}<\frac{1}{2},u_n|_{V_n\cap(\frac{1}{2},1]\times[0,1]}>\frac{1}{2}.$$
By Arzel\`a-Ascoli theorem, Theorem \ref{SC_con_thm_holder} and diagonal argument, there exist some subsequence still denoted by $\myset{u_n}$ and some function $u$ on $K$ with $u|_{\myset{0}\times[0,1]}=0$ and $u|_{\myset{1}\times[0,1]}=1$ such that $u_n$ converges uniformly to $u$ on $K\cap[\veps_1,\veps_2]\times[0,1]$ for all $0<\veps_1<\veps_2<1$. Hence $u$ is continuous on $K\cap(0,1)\times[0,1]$, $u_n(x)\to u(x)$ for all $x\in K$ and $u(x,y)=1-u(1-x,y)=u(x,1-y)$ for all $(x,y)\in K$.

\begin{myprop}\label{SC_con_prop_u}
The function $u$ given above has the following properties.
\begin{enumerate}[(1)]
\item There exists some positive constant $C$ such that
$$a_n(u)\le C\text{ for all }n\ge1.$$
\item For all $\beta\in(\alpha,\log(8\rho)/\log3)$, we have
$$E_{\beta}(u,u)<+\infty.$$
Hence $u\in C^{\frac{\beta-\alpha}{2}}(K)$.
\item
$$u|_{K\cap\myset{\frac{1}{2}}\times[0,1]}=\frac{1}{2},u|_{K\cap[0,\frac{1}{2})\times[0,1]}<\frac{1}{2},u|_{K\cap(\frac{1}{2},1]\times[0,1]}>\frac{1}{2}.$$
\end{enumerate}
\end{myprop}

\begin{proof}
(1) By Theorem \ref{SC_con_thm_resist1} and Theorem \ref{SC_con_thm_monotone1}, for all $n\ge1$, we have
$$
\begin{aligned}
&a_n(u)=\lim_{m\to+\infty}a_{n}(u_{n+m})\le C\varliminf_{m\to+\infty}a_{n+m}(u_{n+m})\\
=&C\varliminf_{m\to+\infty}\rho^{n+m}D_{n+m}(u_{n+m},u_{n+m})=C\varliminf_{m\to+\infty}\rho^{n+m}\left(R_{n+m}^V\right)^{-1}\le C.
\end{aligned}
$$

(2) By (1), for all $\beta\in(\alpha,\log(8\rho)/\log3)$, we have
$$E_{\beta}(u,u)=\sum_{n=1}^\infty\left(3^{\beta-\alpha}\rho^{-1}\right)^na_n(u)\le C\sum_{n=1}^\infty\left(3^{\beta-\alpha}\rho^{-1}\right)^n<+\infty.$$
By Lemma \ref{SC_con_lem_equiv} and Lemma \ref{lem_SC_holder}, we have $u\in C^{\frac{\beta-\alpha}{2}}(K)$.

(3) It is obvious that
$$u|_{K\cap\myset{\frac{1}{2}}\times[0,1]}=\frac{1}{2},u|_{K\cap[0,\frac{1}{2})\times[0,1]}\le\frac{1}{2},u|_{K\cap(\frac{1}{2},1]\times[0,1]}\ge\frac{1}{2}.$$
By symmetry, we only need to show that
$$u|_{K\cap(\frac{1}{2},1]\times[0,1]}>\frac{1}{2}.$$
Suppose there exists $(x,y)\in K\cap(1/2,1)\times[0,1]$ such that $u(x,y)=1/2$. Since $u_n-\frac{1}{2}$ is a nonnegative harmonic function on $V_n\cap(\frac{1}{2},1)\times[0,1]$, by Theorem \ref{SC_con_thm_harnack}, for all $1/2<\veps_1<x<\veps_2<1$, there exists some positive constant $C=C(\veps_1,\veps_2)$ such that for all $n\ge1$
$$\max_{V_n\cap[\veps_1,\veps_2]\times[0,1]}\left(u_n-\frac{1}{2}\right)\le C\min_{V_n\cap[\veps_1,\veps_2]\times[0,1]}\left(u_n-\frac{1}{2}\right).$$
Since $u_n$ converges uniformly to $u$ on $K\cap[\veps_1,\veps_2]\times[0,1]$, we have
$$\sup_{K\cap[\veps_1,\veps_2]\times[0,1]}\left(u-\frac{1}{2}\right)\le C\inf_{K\cap[\veps_1,\veps_2]\times[0,1]}\left(u-\frac{1}{2}\right)=0.$$
Hence
$$u-\frac{1}{2}=0\text{ on }K\cap[\veps_1,\veps_2]\times[0,1]\text{ for all }\frac{1}{2}<\veps_1<x<\veps_2<1.$$
Hence
$$u=\frac{1}{2}\text{ on }K\cap(\frac{1}{2},1)\times[0,1].$$
By continuity, we have
$$u=\frac{1}{2}\text{ on }K\cap[\frac{1}{2},1]\times[0,1],$$
contradiction!
\end{proof}

\section{Proof of Theorem \ref{SC_con_thm_walk}}\label{SC_con_sec_walk}
Firstly, we consider upper bound. Assume that $(\calE_\beta,\calF_\beta)$ is a regular Dirichlet form on $L^2(K;\nu)$, then there exists $u\in\calF_\beta$ such that $u|_{\myset{0}\times[0,1]}=0$ and $u|_{\myset{1}\times[0,1]}=1$. Hence
$$
\begin{aligned}
+\infty&>E_\beta(u,u)=\sum_{n=1}^\infty3^{(\beta-\alpha)n}D_n(u,u)\ge\sum_{n=1}^\infty3^{(\beta-\alpha)n}D_n(u_n,u_n)\\
&=\sum_{n=1}^\infty3^{(\beta-\alpha)n}\left(R_n^V\right)^{-1}\ge C\sum_{n=1}^\infty\left(3^{\beta-\alpha}\rho^{-1}\right)^n.
\end{aligned}
$$
Hence $3^{\beta-\alpha}\rho^{-1}<1$, that is, $\beta<{\log\left(8\rho\right)}/{\log3}=\beta^*$. Hence $\beta_*\le\beta^*$.

Secondly, we consider lower bound. Similar to the proof of Proposition \ref{SC_con_prop_lower}, to show that $(\calE_\beta,\calF_\beta)$ is a regular Dirichlet form on $L^2(K;\nu)$ for all $\beta\in(\alpha,\beta^*)$, we only need to show that $\calF_\beta$ separates points.

Let $u\in C(K)$ be the function in Proposition \ref{SC_con_prop_u}. By Proposition \ref{SC_con_prop_u} (2), we have $E_{\beta}(u,u)<+\infty$, hence $u\in\calF_\beta$. 

For all distinct $z_1=(x_1,y_1),z_2=(x_2,y_2)\in K$, without lose of generality, we may assume that $x_1<x_2$. Replacing $z_i$ by $f_w^{-1}(z_i)$ with some $w\in W_n$ and some $n\ge1$, we only have the following cases.
\begin{enumerate}[(1)]
\item $x_1\in[0,\frac{1}{2}),x_2\in[\frac{1}{2},1]$.
\item $x_1\in[0,\frac{1}{2}],x_2\in(\frac{1}{2},1]$.
\item $x_1,x_2\in[0,\frac{1}{2})$, there exist distinct $w_1,w_2\in\myset{0,1,5,6,7}$ such that
$$z_1\in K_{w_1}\backslash K_{w_2}\text{ and }z_2\in K_{w_2}\backslash K_{w_1}.$$
\item $x_1,x_2\in(\frac{1}{2},1]$, there exist distinct $w_1,w_2\in\myset{1,2,3,4,5}$ such that
$$z_1\in K_{w_1}\backslash K_{w_2}\text{ and }z_2\in K_{w_2}\backslash K_{w_1}.$$
\end{enumerate}
For the first case, $u(z_1)<{1}/{2}\le u(z_2)$. For the second case, $u(z_1)\le{1}/{2}<u(z_2)$.




\begin{figure}[ht]
\centering
\begin{tikzpicture}[scale=0.5]

\draw (0,0)--(6,0)--(6,6)--(0,6)--cycle;
\draw (2,0)--(2,6);
\draw (4,0)--(4,6);
\draw (0,2)--(6,2);
\draw (0,4)--(6,4);

\draw (1,1) node {$K_0$};
\draw (3,1) node {$K_1$};
\draw (5,1) node {$K_2$};
\draw (5,3) node {$K_3$};
\draw (5,5) node {$K_4$};
\draw (3,5) node {$K_5$};
\draw (1,5) node {$K_6$};
\draw (1,3) node {$K_7$};

\end{tikzpicture}
\caption{The Location of $z_1,z_2$}\label{SC_con_fig_characterization}
\end{figure}


For the third case. If $w_1,w_2$ do \emph{not} belong to the same one of the following sets
$$\myset{0,1},\myset{7},\myset{5,6},$$
then we construct a function $w$ as follows. Let $v(x,y)=u(y,x)$ for all $(x,y)\in K$, then
$$v|_{[0,1]\times\myset{0}}=0,v|_{[0,1]\times\myset{1}}=1,$$
$$v(x,y)=v(1-x,y)=1-v(x,1-y)\text{ for all }(x,y)\in K,$$
$$E_\beta(v,v)=E_\beta(u,u)<+\infty.$$
Let
$$w=
\begin{cases}
v\circ f_i^{-1}-1,&\text{on }K_i,i=0,1,2,\\
v\circ f_i^{-1},&\text{on }K_i,i=3,7,\\
v\circ f_i^{-1}+1,&\text{on }K_i,i=4,5,6,\\
\end{cases}
$$
then $w\in C(K)$ is well-defined and $E_\beta(w,w)<+\infty$, hence $w\in\calF_\beta$. Moreover, $w(z_1)\ne w(z_2)$, $w|_{[0,1]\times\myset{0}}=-1,w|_{[0,1]\times\myset{1}}=2,w(x,y)=w(1-x,y)=1-w(x,1-y)$ for all $(x,y)\in K$.

If $w_1,w_2$ \emph{do} belong to the same one of the following sets
$$\myset{0,1},\myset{7},\myset{5,6},$$
then it can only happen that $w_1,w_2\in\myset{0,1}$ or $w_1,w_2\in\myset{5,6}$, without lose of generality, we may assume that $w_1=0$ and $w_2=1$, then $z_1\in K_0\backslash K_1$ and $z_2\in K_1\backslash K_0$.

Let
$$
w=
\begin{cases}
u\circ f_i^{-1}-1,&\text{on }K_i,i=0,6,7,\\
u\circ f_i^{-1},&\text{on }K_i,i=1,5,\\
u\circ f_i^{-1}+1,&\text{on }K_i,i=2,3,4,\\
\end{cases}
$$
then $w\in C(K)$ is well-defined and $E_{\beta}(w,w)<+\infty$, hence $w\in\calF_\beta$. Moreover $w(z_1)\ne w(z_2)$, $w|_{\myset{0}\times[0,1]}=-1,w|_{\myset{1}\times[0,1]}=2,w(x,y)=w(x,1-y)=1-w(1-x,y)$ for all $(x,y)\in K$.

For the forth case, by reflection about $\myset{\frac{1}{2}}\times[0,1]$, we reduce to the third case.

Hence $\calF_\beta$ separates points, hence $(\calE_\beta,\calF_\beta)$ is a regular Dirichlet form on $L^2(K;\nu)$ for all $\beta\in(\alpha,\beta^*)$, hence $\beta_*\ge\beta^*$.

In conclusion, $\beta_*=\beta^*$.

\section{Proof of Theorem \ref{SC_con_thm_BM}}\label{SC_con_sec_BM}

In this section, we use $\Gamma$-convergence technique to construct a local regular Dirichlet form on $L^2(K;\nu)$ which corresponds to the BM. The idea of this construction is from \cite{KS05}.

The construction of local Dirichlet forms on p.c.f. self-similar sets relies heavily on some monotonicity result which is ensured by some compatibility condition, see \cite{Kig93,Kig01}. Our key observation is that even with some weak monotonicity results, we still apply $\Gamma$-convergence technique to obtain some limit.

Take $\myset{\beta_n}\subseteq(\alpha,\beta^*)$ with $\beta_n\uparrow\beta^*$. By Proposition \ref{prop_gamma}, there exist some subsequence still denoted by $\myset{\beta_n}$ and some closed form $(\calE,\calF)$ on $L^2(K;\nu)$ in the wide sense such that $(\beta^*-\beta_n)\frakE_{\beta_n}$ is $\Gamma$-convergent to $\calE$. Without lose of generality, we may assume that
$$0<\beta^*-\beta_n<\frac{1}{n+1}\text{ for all }n\ge1.$$

We have the characterization of $(\calE,\calF)$ on $L^2(K;\nu)$ as follows.

\begin{mythm}\label{SC_con_thm_gamma}
$$
\begin{aligned}
&\calE(u,u)\asymp\sup_{n\ge1}3^{(\beta^*-\alpha)n}\sum_{w\in W_n}
{\sum_{\mbox{\tiny
$
\begin{subarray}{c}
p,q\in V_w\\
|p-q|=2^{-1}\cdot3^{-n}
\end{subarray}
$
}}}
(u(p)-u(q))^2,\\
&\calF=\myset{u\in C(K):\sup_{n\ge1}3^{(\beta^*-\alpha)n}\sum_{w\in W_n}
{\sum_{\mbox{\tiny
$
\begin{subarray}{c}
p,q\in V_w\\
|p-q|=2^{-1}\cdot3^{-n}
\end{subarray}
$
}}}
(u(p)-u(q))^2<+\infty}.
\end{aligned}
$$
Moreover, $(\calE,\calF)$ is a regular closed form on $L^2(K;\nu)$.
\end{mythm}

\begin{proof}
Recall that $\rho=3^{\beta^*-\alpha}$, then
$$
\begin{aligned}
E_{\beta}(u,u)&=\sum_{n=1}^\infty3^{(\beta-\alpha)n}\sum_{w\in W_n}
{\sum_{\mbox{\tiny
$
\begin{subarray}{c}
p,q\in V_w\\
|p-q|=2^{-1}\cdot3^{-n}
\end{subarray}
$
}}}
(u(p)-u(q))^2=\sum_{n=1}^\infty3^{(\beta-\beta^*)n}a_n(u),\\
\frakE_\beta(u,u)&=\sum_{n=1}^\infty3^{(\beta-\alpha)n}
{\sum_{\mbox{\tiny
$
\begin{subarray}{c}
w^{(1)}\sim_nw^{(2)}
\end{subarray}
$
}}}
\left(P_nu(w^{(1)})-P_nu(w^{(2)})\right)^2=\sum_{n=1}^\infty3^{(\beta-\beta^*)n}b_n(u).
\end{aligned}
$$

We use weak monotonicity results Theorem \ref{SC_con_thm_monotone1}, Theorem \ref{SC_con_thm_monotone2} and elementary result Proposition \ref{prop_ele2}.

For any $u\in L^2(K;\nu)$, there exists $\myset{u_n}\subseteq L^2(K;\nu)$ converging strongly to $u$ in $L^2(K;\nu)$ such that
\begin{align*}
\calE(u,u)&\ge\varlimsup_{n\to+\infty}(\beta^*-\beta_n)\frakE_{\beta_n}(u_n,u_n)\\
&=\varlimsup_{n\to+\infty}(\beta^*-\beta_n)\sum_{k=1}^\infty3^{(\beta_n-\beta^*)k}b_k(u_n)\\
&\ge\varlimsup_{n\to+\infty}(\beta^*-\beta_n)\sum_{k=n+1}^\infty3^{(\beta_n-\beta^*)k}b_k(u_n)\\
&\ge C\varlimsup_{n\to+\infty}(\beta^*-\beta_n)\sum_{k=n+1}^\infty3^{(\beta_n-\beta^*)k}b_n(u_n)\\
&=C\varlimsup_{n\to+\infty}\left\{b_n(u_n)\left[(\beta^*-\beta_n)\frac{3^{(\beta_n-\beta^*)(n+1)}}{1-3^{\beta_n-\beta^*}}\right]\right\}.
\end{align*}
Since $0<\beta^*-\beta_n<1/(n+1)$, we have $3^{(\beta_n-\beta^*)(n+1)}>1/3$. Since
$$\lim_{n\to+\infty}\frac{\beta^*-\beta_n}{1-3^{\beta_n-\beta^*}}=\frac{1}{\log3},$$
there exists some positive constant $C$ such that 
$$(\beta^*-\beta_n)\frac{3^{(\beta_n-\beta^*)(n+1)}}{1-3^{\beta_n-\beta^*}}\ge C\text{ for all }n\ge1.$$
Hence
$$\calE(u,u)\ge C\varlimsup_{n\to+\infty}b_n(u_n).$$
Since $u_n\to u$ in $L^2(K;\nu)$, for all $k\ge1$, we have
$$b_k(u)=\lim_{n\to+\infty}b_k(u_n)=\lim_{k\le n\to+\infty}b_k(u_n)\le C\varliminf_{n\to+\infty}b_n(u_n).$$
For all $m\ge1$, we have
$$
\begin{aligned}
(\beta^*-\beta_m)\sum_{k=1}^\infty3^{(\beta_m-\beta^*)k}b_k(u)&\le C(\beta^*-\beta_m)\sum_{k=1}^\infty3^{(\beta_m-\beta^*)k}\varliminf_{n\to+\infty}b_n(u_n)\\
&=C(\beta^*-\beta_m)\frac{3^{\beta_m-\beta^*}}{1-3^{\beta_m-\beta^*}}\varliminf_{n\to+\infty}b_n(u_n).
\end{aligned}
$$
Hence $\calE(u,u)<+\infty$ implies $\frakE_{\beta_m}(u,u)<+\infty$, by Lemma \ref{lem_SC_holder}, we have $\calF\subseteq C(K)$. Hence
$$\varliminf_{m\to+\infty}(\beta^*-\beta_m)\sum_{k=1}^\infty3^{(\beta_m-\beta^*)k}b_k(u)\le C\varliminf_{n\to+\infty}b_n(u_n).$$
Hence for all $u\in\calF\subseteq C(K)$, we have
$$
\begin{aligned}
\calE(u,u)&\ge C\varlimsup_{n\to+\infty}b_n(u_n)\ge C\varliminf_{n\to+\infty}b_n(u_n)\\
&\ge C\varliminf_{m\to+\infty}(\beta^*-\beta_m)\sum_{k=1}^\infty3^{(\beta_m-\beta^*)k}b_k(u)\\
&\ge C\varliminf_{m\to+\infty}(\beta^*-\beta_m)\sum_{k=1}^\infty3^{(\beta_m-\beta^*)k}a_k(u)\\
&\ge C\sup_{n\ge1}a_n(u).
\end{aligned}
$$

On the other hand, for all $u\in\calF\subseteq C(K)$, we have
$$
\begin{aligned}
\calE(u,u)&\le\varliminf_{n\to+\infty}(\beta^*-\beta_n)\frakE_{\beta_n}(u,u)\\
&\le C\varliminf_{n\to+\infty}(\beta^*-\beta_n)E_{\beta_n}(u,u)\\
&=C\varliminf_{n\to+\infty}(\beta^*-\beta_n)\sum_{k=1}^\infty3^{(\beta_n-\beta^*)k}a_k(u)\\
&=C\varliminf_{n\to+\infty}\frac{\beta^*-\beta_n}{1-3^{\beta_n-\beta^*}}(1-3^{\beta_n-\beta^*})\sum_{k=1}^\infty3^{(\beta_n-\beta^*)k}a_k(u)\\
&\le C\sup_{n\ge1}a_n(u).
\end{aligned}
$$
Therefore, for all $u\in\calF\subseteq C(K)$, we have
$$\calE(u,u)\asymp\sup_{n\ge1}a_n(u)=\sup_{n\ge1}3^{(\beta^*-\alpha)n}\sum_{w\in W_n}
{\sum_{\mbox{\tiny
$
\begin{subarray}{c}
p,q\in V_w\\
|p-q|=2^{-1}\cdot3^{-n}
\end{subarray}
$
}}}
(u(p)-u(q))^2,$$
and
$$\calF=\myset{u\in C(K):\sup_{n\ge1}3^{(\beta^*-\alpha)n}\sum_{w\in W_n}
{\sum_{\mbox{\tiny
$
\begin{subarray}{c}
p,q\in V_w\\
|p-q|=2^{-1}\cdot3^{-n}
\end{subarray}
$
}}}
(u(p)-u(q))^2<+\infty}.$$

It is obvious that the function $u\in C(K)$ in Proposition \ref{SC_con_prop_u} is in $\calF$. Similar to the proof of Theorem \ref{SC_con_thm_walk}, we have $\calF$ is uniformly dense in $C(K)$. Hence $(\calE,\calF)$ is a regular closed form on $L^2(K;\nu)$.
\end{proof}

Now we prove Theorem \ref{SC_con_thm_BM} as follows.

\begin{proof}[Proof of Theorem \ref{SC_con_thm_BM}]
For all $n\ge1,u\in l(V_{n+1})$, we have
$$
\begin{aligned}
&\rho\sum_{i=0}^7a_n(u\circ f_i)=\rho\sum_{i=0}^7\rho^n\sum_{w\in W_n}
{\sum_{\mbox{\tiny
$
\begin{subarray}{c}
p,q\in V_w\\
|p-q|=2^{-1}\cdot3^{-n}
\end{subarray}
$
}}}
(u\circ f_i(p)-u\circ f_i(q))^2\\
=&\rho^{n+1}\sum_{w\in W_{n+1}}
{\sum_{\mbox{\tiny
$
\begin{subarray}{c}
p,q\in V_w\\
|p-q|=2^{-1}\cdot3^{-(n+1)}
\end{subarray}
$
}}}
(u(p)-u(q))^2=a_{n+1}(u).
\end{aligned}
$$
Hence for all $n,m\ge1,u\in l(V_{n+m})$, we have
$$\rho^m\sum_{w\in W_m}a_n(u\circ f_w)=a_{n+m}(u).$$
For all $u\in\calF,n\ge1,w\in W_n$, we have
$$\sup_{k\ge1}a_k(u\circ f_w)\le\sup_{k\ge1}\sum_{w\in W_n}a_k(u\circ f_w)=\rho^{-n}\sup_{k\ge1}a_{n+k}(u)\le\rho^{-n}\sup_{k\ge1}a_{k}(u)<+\infty,$$
hence $u\circ f_w\in\calF$.

Let
$$\mybar{\calE}^{(n)}(u,u)=\rho^n\sum_{w\in W_n}\calE(u\circ f_w,u\circ f_w),u\in\calF,n\ge1.$$
Then
$$
\begin{aligned}
\mybar{\calE}^{(n)}(u,u)&\ge C\rho^n\sum_{w\in W_n}\varlimsup_{k\to+\infty}a_k(u\circ f_w)\ge C\rho^n\varlimsup_{k\to+\infty}\sum_{w\in W_n}a_k(u\circ f_w)\\
&=C\varlimsup_{k\to+\infty}a_{n+k}(u)\ge C\sup_{k\ge1}a_k(u).
\end{aligned}
$$
Similarly
$$
\begin{aligned}
\mybar{\calE}^{(n)}(u,u)&\le C\rho^n\sum_{w\in W_n}\varliminf_{k\to+\infty}a_k(u\circ f_w)\le C\rho^n\varliminf_{k\to+\infty}\sum_{w\in W_n}a_k(u\circ f_w)\\
&=C\varliminf_{k\to+\infty}a_{n+k}(u)\le C\sup_{k\ge1}a_k(u).
\end{aligned}
$$
Hence
$$\mybar{\calE}^{(n)}(u,u)\asymp\sup_{k\ge1}a_k(u)\text{ for all }u\in\calF,n\ge1.$$
Moreover, for all $u\in\calF$, $n\ge1$, we have
$$
\begin{aligned}
\mybar{\calE}^{(n+1)}(u,u)&=\rho^{n+1}\sum_{w\in W_{n+1}}\calE(u\circ f_w,u\circ f_w)\\
&=\rho^{n+1}\sum_{i=0}^7\sum_{w\in W_n}\calE(u\circ f_i\circ f_w,u\circ f_i\circ f_w)\\
&=\rho\sum_{i=0}^7\left(\rho^n\sum_{w\in W_n}\calE((u\circ f_i)\circ f_w,(u\circ f_i)\circ f_w)\right)\\
&=\rho\sum_{i=0}^7\mybar{\calE}^{(n)}(u\circ f_i,u\circ f_i).
\end{aligned}
$$
Let
$$\tilde{\calE}^{(n)}(u,u)=\frac{1}{n}\sum_{l=1}^n\mybar{\calE}^{(l)}(u,u),u\in\calF,n\ge1.$$
It is obvious that
$$\tilde{\calE}^{(n)}(u,u)\asymp\sup_{k\ge1}a_k(u)\text{ for all }u\in\calF,n\ge1.$$
Since $(\calE,\calF)$ is a regular closed form on $L^2(K;\nu)$, by \cite[Definition 1.3.8, Remark 1.3.9, Definition 1.3.10, Remark 1.3.11]{CF12}, we have $(\calF,\calE_1)$ is a separable Hilbert space. Let $\{u_i\}_{i\ge1}$ be a dense subset of $(\calF,\calE_1)$. For all $i\ge1$, $\{\tilde{\calE}^{(n)}(u_i,u_i)\}_{n\ge1}$ is a bounded sequence. By diagonal argument, there exists a subsequence $\{n_k\}_{k\ge1}$ such that $\{\tilde{\calE}^{(n_k)}(u_i,u_i)\}_{k\ge1}$ converges for all $i\ge1$. Since
$$\tilde{\calE}^{(n)}(u,u)\asymp\sup_{k\ge1}a_k(u)\asymp\calE(u,u)\text{ for all }u\in\calF,n\ge1,$$
we have $\{\tilde{\calE}^{(n_k)}(u,u)\}_{k\ge1}$ converges for all $u\in\calF$. Let
$$\calE_{\loc}(u,u)=\lim_{k\to+\infty}\tilde{\calE}^{(n_k)}(u,u)\text{ for all }u\in\calF_{\loc}:=\calF.$$
Then
$$\calE_{\loc}(u,u)\asymp\sup_{k\ge1}a_k(u)\asymp\calE(u,u)\text{ for all }u\in\calF_\loc=\calF.$$
Hence $(\calE_\loc,\calF_\loc)$ is a regular closed form on $L^2(K;\nu)$. It is obvious that $1\in\calF_\loc$ and $\calE_\loc(1,1)=0$, by \cite[Lemma 1.6.5, Theorem 1.6.3]{FOT11}, we have $(\calE_\loc,\calF_\loc)$ on $L^2(K;\nu)$ is conservative.

For all $u\in\calF_\loc=\calF$, we have $u\circ f_i\in\calF=\calF_\loc$ for all $i=0,\ldots,7$ and
$$
\begin{aligned}
\rho\sum_{i=0}^7\calE_\loc(u\circ f_i,u\circ f_i)&=\rho\sum_{i=0}^7\lim_{k\to+\infty}\tilde{\calE}^{(n_k)}(u\circ f_i,u\circ f_i)\\
&=\rho\sum_{i=0}^7\lim_{k\to+\infty}\frac{1}{n_k}\sum_{l=1}^{n_k}\mybar{\calE}^{(l)}(u\circ f_i,u\circ f_i)\\
&=\lim_{k\to+\infty}\frac{1}{n_k}\sum_{l=1}^{n_k}\left[\rho\sum_{i=0}^7\mybar{\calE}^{(l)}(u\circ f_i,u\circ f_i)\right]\\
&=\lim_{k\to+\infty}\frac{1}{n_k}\sum_{l=1}^{n_k}\mybar{\calE}^{(l+1)}(u,u)\\
&=\lim_{k\to+\infty}\frac{1}{n_k}\sum_{l=2}^{n_k+1}\mybar{\calE}^{(l)}(u,u)\\
&=\lim_{k\to+\infty}\left[\frac{1}{n_k}\sum_{l=1}^{n_k}\mybar{\calE}^{(l)}(u,u)+\frac{1}{n_k}\mybar{\calE}^{(n_k+1)}(u,u)-\frac{1}{n_k}\mybar{\calE}^{(1)}(u,u)\right]\\
&=\lim_{k\to+\infty}\tilde{\calE}^{(n_k)}(u,u)=\calE_\loc(u,u).
\end{aligned}
$$
Hence $(\calE_\loc,\calF_\loc)$ on $L^2(K;\nu)$ is self-similar.

For all $u,v\in\calF_\loc$ satisfying $\mathrm{supp}(u),\mathrm{supp}(v)$ are compact and $v$ is constant in an open neighborhood $U$ of $\mathrm{supp}(u)$, we have $K\backslash U$ is compact and $\mathrm{supp}(u)\cap(K\backslash U)=\emptyset$, hence
$$\delta=\mathrm{dist}(\mathrm{supp}(u),K\backslash U)>0.$$
Taking sufficiently large $n\ge1$ such that $3^{1-n}<\delta$, by self-similarity, we have
$$\calE_\loc(u,v)=\rho^n\sum_{w\in W_n}\calE_\loc(u\circ f_w,v\circ f_w).$$
For all $w\in W_n$, we have $u\circ f_w=0$ or $v\circ f_w$ is constant, hence $\calE_\loc(u\circ f_w,v\circ f_w)=0$, hence $\calE_\loc(u,v)=0$, that is, $(\calE_\loc,\calF_\loc)$ on $L^2(K;\nu)$ is strongly local.

For all $u\in\calF_\loc$, it is obvious that $u^+,u^-,1-u,\mybar{u}=(0\vee u)\wedge1\in\calF_\loc$ and
$$\calE_\loc(u,u)=\calE_\loc(1-u,1-u).$$
Since $u^+u^-=0$ and $(\calE_\loc,\calF_\loc)$ on $L^2(K;\nu)$ is strongly local, we have $\calE_\loc(u^+,u^-)=0$. Hence
$$
\begin{aligned}
\calE_\loc(u,u)&=\calE_\loc(u^+-u^-,u^+-u^-)\\
&=\calE_\loc(u^+,u^+)+\calE_\loc(u^-,u^-)-2\calE_\loc(u^+,u^-)\\
&=\calE_\loc(u^+,u^+)+\calE_\loc(u^-,u^-)\\
&\ge\calE_\loc(u^+,u^+)=\calE_\loc(1-u^+,1-u^+)\\
&\ge\calE_\loc((1-u^+)^+,(1-u^+)^+)=\calE_\loc(1-(1-u^+)^+,1-(1-u^+)^+)\\
&=\calE_\loc(\mybar{u},\mybar{u}),
\end{aligned}
$$
that is, $(\calE_\loc,\calF_\loc)$ on $L^2(K;\nu)$ is Markovian. Hence $(\calE_\loc,\calF_\loc)$ is a self-similar strongly local regular Dirichlet form on $L^2(K;\nu)$.
\end{proof}

\begin{myrmk}
The idea of the construction of $\mybar{\calE}^{(n)},\tilde{\calE}^{(n)}$ is from \cite[Section 6]{KZ92}. The proof of Markovain property is from the proof of \cite[Theorem 2.1]{BBKT10}.
\end{myrmk}

\section{Proof of Theorem \ref{SC_con_thm_Besov}}\label{SC_con_sec_Besov}

Theorem \ref{SC_con_thm_Besov} is a special case of the following result.

\begin{myprop}\label{SC_con_prop_equiv_local}
For all $\beta\in(\alpha,+\infty)$, for all $u\in C(K)$, we have
$$\sup_{n\ge1}3^{(\beta-\alpha)n}\sum_{w\in W_n}
{\sum_{\mbox{\tiny
$
\begin{subarray}{c}
p,q\in V_w\\
|p-q|=2^{-1}\cdot3^{-n}
\end{subarray}
$
}}}
(u(p)-u(q))^2\asymp[u]_{B^{2,\infty}_{\alpha,\beta}(K)}.$$
\end{myprop}

\begin{proof}[Proof of Proposition \ref{SC_con_prop_equiv_local}]
The proof is very similar to that of Lemma \ref{SC_con_lem_equiv}. We only point out the differences. To show that LHS$\lesssim$RHS, by the proof of Theorem \ref{SC_con_thm_equiv1}, we still have Equation (\ref{SC_con_eqn_equiv1_1}) where $E(u)$ is replaced by $F(u)$. Then
$$
\begin{aligned}
&3^{(\beta-\alpha)n}\sum_{w\in W_n}
{\sum_{\mbox{\tiny
$
\begin{subarray}{c}
p,q\in V_w\\
|p-q|=2^{-1}\cdot3^{-n}
\end{subarray}
$
}}}
(u(p)-u(q))^2\\
\le&128\cdot2^{(\beta-\alpha)/2}cF(u)3^{\beta n-(\beta-\alpha)(n+kl)}+32\cdot3^{\alpha k}\sum_{i=0}^{l-1}2^i\cdot3^{-(\beta-\alpha)ki}E_{n+ki}(u).
\end{aligned}
$$
Take $l=n$, then
$$
\begin{aligned}
&3^{(\beta-\alpha)n}\sum_{w\in W_n}
{\sum_{\mbox{\tiny
$
\begin{subarray}{c}
p,q\in V_w\\
|p-q|=2^{-1}\cdot3^{-n}
\end{subarray}
$
}}}
(u(p)-u(q))^2\\
\le&128\cdot2^{(\beta-\alpha)/2}cF(u)3^{[\beta-(\beta-\alpha)(k+1)]n}+32\cdot3^{\alpha k}\sum_{i=0}^{n-1}2^i\cdot3^{-(\beta-\alpha)ki}E_{n+ki}(u)\\
\le&128\cdot2^{(\beta-\alpha)/2}cF(u)3^{[\beta-(\beta-\alpha)(k+1)]n}+32\cdot3^{\alpha k}\sum_{i=0}^{\infty}3^{[1-(\beta-\alpha)k]i}\left(\sup_{n\ge1}E_{n}(u)\right).
\end{aligned}
$$
Take $k\ge1$ sufficiently large such that $\beta-(\beta-\alpha)(k+1)<0$ and $1-(\beta-\alpha)k<0$, then
$$
\begin{aligned}
&\sup_{n\ge1}3^{(\beta-\alpha)n}\sum_{w\in W_n}
{\sum_{\mbox{\tiny
$
\begin{subarray}{c}
p,q\in V_w\\
|p-q|=2^{-1}\cdot3^{-n}
\end{subarray}
$
}}}
(u(p)-u(q))^2\\
\lesssim&\sup_{n\ge1}3^{(\alpha+\beta)n}\int_K\int_{B(x,3^{-n})}(u(x)-u(y))^2\nu(\md y)\nu(\md x).
\end{aligned}
$$

To show that LHS$\gtrsim$RHS, by the proof of Theorem \ref{SC_con_thm_equiv2}, we still have Equation (\ref{SC_con_eqn_equiv2_3}). Then
$$
\begin{aligned}
&\sup_{n\ge2}3^{(\alpha+\beta)n}\int_K\int_{B(x,c3^{-n})}(u(x)-u(y))^2\nu(\md y)\nu(\md x)\\
\lesssim&\sup_{n\ge2}\sum_{k=n}^\infty4^{k-n}\cdot3^{\beta n-\alpha k}\sum_{w\in W_k}
{\sum_{\mbox{\tiny
$
\begin{subarray}{c}
p,q\in V_w\\
|p-q|=2^{-1}\cdot3^{-k}
\end{subarray}
$
}}}
(u(p)-u(q))^2\\
&+\sup_{n\ge2}3^{(\beta-\alpha)n}\sum_{w\in W_{n-1}}
{\sum_{\mbox{\tiny
$
\begin{subarray}{c}
p,q\in V_w\\
|p-q|=2^{-1}\cdot3^{-(n-1)}
\end{subarray}
$
}}}
(u(p)-u(q))^2\\
\lesssim&\sup_{n\ge2}\sum_{k=n}^\infty4^{k-n}\cdot3^{\beta(n-k)}\left(\sup_{k\ge1}3^{(\beta-\alpha)k}\sum_{w\in W_k}
{\sum_{\mbox{\tiny
$
\begin{subarray}{c}
p,q\in V_w\\
|p-q|=2^{-1}\cdot3^{-k}
\end{subarray}
$
}}}
(u(p)-u(q))^2\right)\\
&+\sup_{n\ge1}3^{(\beta-\alpha)n}\sum_{w\in W_n}
{\sum_{\mbox{\tiny
$
\begin{subarray}{c}
p,q\in V_w\\
|p-q|=2^{-1}\cdot3^{-n}
\end{subarray}
$
}}}
(u(p)-u(q))^2\\
\lesssim&\sup_{n\ge1}3^{(\beta-\alpha)n}\sum_{w\in W_n}
{\sum_{\mbox{\tiny
$
\begin{subarray}{c}
p,q\in V_w\\
|p-q|=2^{-1}\cdot3^{-n}
\end{subarray}
$
}}}
(u(p)-u(q))^2.
\end{aligned}
$$

\end{proof}

We have the following properties of Besov spaces for large exponent.

\begin{mycor}\label{SC_con_cor_chara}
$B^{2,2}_{\alpha,\beta^*}(K)=\myset{\text{constant functions}}$ but $B^{2,\infty}_{\alpha,\beta^*}(K)$ is uniformly dense in $C(K)$. $B^{2,2}_{\alpha,\beta}(K)=B^{2,\infty}_{\alpha,\beta}(K)=\myset{\text{constant functions}}$ for all $\beta\in(\beta^*,+\infty)$.
\end{mycor}

\begin{proof}
By Theorem \ref{SC_con_thm_BM} and Theorem \ref{SC_con_thm_Besov}, we have $B^{2,\infty}_{\alpha,\beta^*}(K)$ is uniformly dense in $C(K)$. Assume that $u\in C(K)$ is non-constant, then there exists $N\ge1$ such that $a_N(u)>0$. By Theorem \ref{SC_con_thm_monotone1}, for all $\beta\in[\beta^*,+\infty)$, we have
\begin{align*}
&\sum_{n=1}^\infty3^{(\beta-\beta^*)n}a_n(u)\ge\sum_{n=N+1}^\infty3^{(\beta-\beta^*)n}a_n(u)\\
\ge&C\sum_{n=N+1}^\infty3^{(\beta-\beta^*)n}a_N(u)=+\infty,
\end{align*}
for all $\beta\in(\beta^*,+\infty)$, we have
\begin{align*}
&\sup_{n\ge1}3^{(\beta-\beta^*)n}a_n(u)\ge\sup_{n\ge N+1}3^{(\beta-\beta^*)n}a_n(u)\\
\ge&C\sup_{n\ge N+1}3^{(\beta-\beta^*)n}a_N(u)=+\infty.
\end{align*}

By Lemma \ref{SC_con_lem_equiv} and Proposition \ref{SC_con_prop_equiv_local}, we have $B^{2,2}_{\alpha,\beta}(K)=\myset{\text{constant functions}}$ for all $\beta\in[\beta^*,+\infty)$ and $B^{2,\infty}_{\alpha,\beta}(K)=\myset{\text{constant functions}}$ for all $\beta\in(\beta^*,+\infty)$.
\end{proof}

\section{Proof of Theorem \ref{SC_con_thm_hk}}\label{SC_con_sec_hk}

We use effective resistance as follows.

Let $(M,d,\mu)$ be a metric measure space and $(\calE,\calF)$ a regular Dirichlet form on $L^2(M;\mu)$. Assume that $A,B$ are two disjoint subsets of $M$. Define \emph{effective resistance} as
$$R(A,B)=\inf\myset{\calE(u,u):u|_A=0,u|_B=1,u\in\calF\cap C_0(M)}^{-1}.$$
Denote
$$R(x,B)=R(\myset{x},B),R(x,y)=R(\myset{x},\myset{y}),x,y\in M.$$
It is obvious that if $A_1\subseteq A_2$, $B_1\subseteq B_2$, then
$$R(A_1,B_1)\ge R(A_2,B_2).$$

\begin{proof}[Proof of Theorem \ref{SC_con_thm_hk}]
First, we show that
$$R(x,y)\asymp|x-y|^{\beta^*-\alpha}\text{ for all }x,y\in K.$$
By Lemma \ref{lem_SC_holder}, we have
$$(u(x)-u(y))^2\le c\calE_\loc(u,u)|x-y|^{\beta^*-\alpha}\text{ for all }x,y\in K,u\in\calF_\loc,$$
hence
$$R(x,y)\lesssim|x-y|^{\beta^*-\alpha}\text{ for all }x,y\in K.$$
On the other hand, we claim
$$R(x,B(x,r)^c)\asymp r^{\beta^*-\alpha}\text{ for all }x\in K,r>0\text{ with }B(x,r)^c\ne\emptyset.$$
Indeed, fix $C>0$. If $u\in\calF_\loc$ satisfies $u(x)=1$, $u|_{B(x,r)^c}=0$, then $\tilde{u}:y\mapsto u(x+C(y-x))$ satisfies $\tilde{u}\in\calF_\loc$, $\tilde{u}(x)=1$, $\tilde{u}|_{B(x,Cr)^c}=0$. By Theorem \ref{SC_con_thm_BM}, it is obvious that
$$\calE_\loc(\tilde{u},\tilde{u})\asymp C^{-(\beta^*-\alpha)}\calE_\loc(u,u),$$
hence
$$R(x,B(x,Cr)^c)\asymp C^{\beta^*-\alpha}R(x,B(x,r)^c).$$
Hence
$$R(x,B(x,r)^c)\asymp r^{\beta^*-\alpha}.$$
For all $x,y\in K$, we have
$$R(x,y)\ge R(x,B(x,|x-y|)^c)\asymp|x-y|^{\beta^*-\alpha}.$$

Then, we follow a standard analytic approach as follows. First, we obtain Green function estimates as in \cite[Proposition 6.11]{GHL14}. Then, we obtain heat kernel estimates as in \cite[Theorem 3.14]{GH14a}. Note that we are dealing with compact set, the final estimates only hold for some finite time $t\in(0,1)$.
\end{proof}

\newpage

\fancyhead[RE,LO]{\textit{BIBLIOGRAPHY}}

\bibliographystyle{plain}

\def\cprime{$'$}

\addcontentsline{toc}{chapter}{Bibliography}

\end{document}